\documentclass[reqno,12pt,letterpaper]{amsart}
\usepackage[proof]{sdmacros}
\usepackage[T1]{fontenc}

\newcommand{\eps}{\epsilon}

\newcommand{\cB}{\mathcal B}
\newcommand{\cE}{\mathcal E}

\newcommand{\cO}{\mathcal O}
\newcommand{\cQ}{\mathcal Q}
\newcommand{\cJ}{\mathcal J}

\newcommand{\cV}{\mathcal V}

\newcommand{\cZ}{\mathcal Z}
\newcommand{\cZC}{\mathcal Z^\complement}
\newcommand{\bv}{\mathbf v}
\newcommand{\bw}{\mathbf w}
\newcommand{\bp}{\mathbf p}
\newcommand{\bq}{\mathbf q}

\newcommand{\IR}{\mathbb{R}}
\newcommand{\IN}{\mathbb{N}}

\newcommand{\sA}{\mathscr{A}}

\title[Control of eigenfunctions in variable curvature]%
{Control of eigenfunctions\\
on surfaces of variable curvature}
\author{Semyon Dyatlov}
\email{dyatlov@math.berkeley.edu}
\address{Department of Mathematics,
University of California,
Berkeley, CA 94720}
\address{Department of Mathematics,
Massachusetts Institute of Technology,
Cambridge, MA 02139}
\author{Long Jin}
\email{jinlong@mail.tsinghua.edu.cn}
\address{Yau Mathematical Sciences Center, Tsinghua University, China}
\author{St\'ephane Nonnenmacher}
\email{stephane.nonnenmacher@u-psud.fr}
\address{Laboratoire de Math\'ematiques d'Orsay, Univ. Paris-Sud, CNRS, Universit\'e Paris-Saclay, 91405 Orsay cedex, France}
\begin{document}

\begin{abstract}
We prove a microlocal lower bound on the mass of high energy eigenfunctions of the Laplacian
on compact surfaces of negative curvature, and more generally on surfaces with Anosov
geodesic flows. This implies controllability for the Schr\"odinger equation
by any nonempty open set, and shows that every semiclassical measure has full support.
We also prove exponential energy decay for solutions to the damped
wave equation on such surfaces, for any
nontrivial damping coefficient. These results extend previous works~\cite{meassupp,JinDWE}, 
which considered the setting of surfaces of constant negative curvature.

The proofs use the strategy of~\cite{meassupp,JinDWE} and rely on the
fractal uncertainty principle of~\cite{fullgap}. However, in the variable curvature case
the stable/unstable foliations are not smooth, so we can no longer associate to these foliations
a pseudodifferential calculus of the type used in~\cite{hgap}. Instead, our argument uses Egorov's Theorem up to local Ehrenfest time and the hyperbolic parametrix of~\cite{NZ09}, together with the $C^{1+}$ regularity
of the stable/unstable foliations.
\end{abstract}

\maketitle

%%%%%%%%%%%%%%%%%%%%%%%%%%%%%%%%%%%%%%%%%%%%%%%%%%%%%%%%%%%%%%%%%%%%%%%%%%%%%%%%
%                                 INTRODUCTION                                 %
%%%%%%%%%%%%%%%%%%%%%%%%%%%%%%%%%%%%%%%%%%%%%%%%%%%%%%%%%%%%%%%%%%%%%%%%%%%%%%%%
\addtocounter{section}{1}
\addcontentsline{toc}{section}{1. Introduction}

Let $(M,g)$ be a compact smooth Riemannian manifold. The Laplace--Beltrami operator $\Delta$
admits a complete set of eigenfunctions
$$
u_j\in C^\infty(M),\quad
(-\Delta-\lambda_j^2)u_j=0,\quad
\|u_j\|_{L^2(M)}=1.
$$
These can be interpreted as stationary states of a quantum
particle evolving freely on~$M$,
with $\lambda_j^2$ being the energy of the particle, and $|u_j(x)|^2$ the probability density
of finding the particle at the point~$x$. One fundamental question in
the field of spectral geometry
is to understand the structure of the eigenfunctions $u_j$ in the
high-energy r\'egime $\lambda_j\to \infty$, using some information on
the geodesic flow on~$M$  (this flow corresponds to the dynamics of a classical
particle evolving freely on $M$). In particular, the field of Quantum
Chaos focuses on situations where the geodesic flow on~$M$ has chaotic
behavior.

In this paper we assume that $(M,g)$ is a compact connected Riemannian
surface without boundary, whose geodesic flow has the Anosov property (see~\S\ref{s:hyperbolics} for
definitions and properties); we will refer to such $(M,g)$ as an
Anosov surface.
Anosov flows form a standard mathematical model of systems with
strongly chaotic behavior, in some sense they are the ``purest'' form
of chaotic systems. A large
family of examples is provided by the surfaces of negative Gauss
curvature. 
Our first result gives a lower bound on the mass distribution of $u_j$, showing that the probability
of finding the quantum particle in any fixed open set is bounded away
from zero uniformly in the high-energy limit:
%%%%%%%%%%%%%%%%%%%%%%%%%%%%%%%%%%%%%%%%%%%%%%%%%%%%%%%%%%%%%%%%%%%%%%%%%%%%%%%%
\begin{theo}
  \label{t:coneig} Assume that $(M,g)$ is
an Anosov surface. Choose $\Omega\subset M$ open and nonempty. Then there exists a constant
$c_\Omega>0$ such that any eigenfunction $u_j$ of the Laplace--Beltrami
operator on $(M,g)$ satisfies
\begin{equation}
  \label{e:coneig}
\|u_j\|_{L^2(\Omega)}\geq c_\Omega\,.
\end{equation}
\end{theo}
%%%%%%%%%%%%%%%%%%%%%%%%%%%%%%%%%%%%%%%%%%%%%%%%%%%%%%%%%%%%%%%%%%%%%%%%%%%%%%%%
On any Riemannian manifold, the unique continuation principle shows
that a positive lower bound~\eqref{e:coneig} holds if one allows
$c_\Omega$ to depend on $\lambda_j$; see e.g. Lebeau--Robbiano~\cite[Corollaire~2]{LeRo}; an introduction to 
quantitative unique continuation for eigenfunctions of the Schr\"odinger operators on $\IR^d$ can be found in~\cite[Theorem~7.7]{e-z}. 
In general, the lower bound decays
exponentially fast as $\lambda_j\to\infty$, as can be seen in the
case of the round sphere, where one can construct Gaussian beam eigenstates
concentrating on a closed geodesic and exponentially small away from
this geodesic.
Note that related {\it propagation of smallness} results for solutions
of elliptic equations were also obtained for 
any set $\Omega$ of positive Lebesgue measure $\vol(\Omega)$ by Logunov--Malinnikova~\cite[\S1.7]{Remez}, who showed that
$$
\sup_\Omega |u_j|\geq \bigg({\vol(\Omega)\over C}\bigg)^{C\lambda_j}\sup_M |u_j|
$$
for some constant $C$ depending on $(M,g)$, but not on $\Omega$ or $j$.
In our situation, the energy-independent lower bound~\eqref{e:coneig} strongly
relies on the chaotic behavior of the geodesic flow. 

The proof of Theorem~\ref{t:coneig} gives
a stronger result featuring the localization of~$u_j$
in both position and Fourier spaces.
Let $\Op_h$ be a semiclassical quantization procedure on $M$,
and $S^0(T^*M)$ be the standard symbol class, see~\S\ref{s:semiclassics}.
Denote by $S^*M\subset T^*M$ the cosphere bundle.
%%%%%%%%%%%%%%%%%%%%%%%%%%%%%%%%%%%%%%%%%%%%%%%%%%%%%%%%%%%%%%%%%%%%%%%%%%%%%%%%
\begin{theo}
  \label{t:eig}
Assume that $a\in S^0(T^*M)$ and $a|_{S^*M}\not\equiv 0$.
Then there exist constants $C>0$ and $h_0>0$ depending only on $a$, such that
for all $h\in(0,h_0)$ and all $u\in H^2(M)$ we have the estimate
\begin{equation}
  \label{e:eig}
\|u\|_{L^2(M)}\leq C\|\Op_h(a)u\|_{L^2(M)}+{C\log(1/h)\over h}
\big\|(-h^2\Delta-I)u\big\|_{L^2(M)}.
\end{equation}
\end{theo}
%%%%%%%%%%%%%%%%%%%%%%%%%%%%%%%%%%%%%%%%%%%%%%%%%%%%%%%%%%%%%%%%%%%%%%%%%%%%%%%%
If $a=a(x)$ is a function on $M$, then $\Op_h(a)$ is the multiplication operator by~$a$.
Hence Theorem~\ref{t:eig} implies Theorem~\ref{t:coneig} by taking $a(x)$ supported inside $\Omega$
and putting $h:=\lambda_j^{-1}$, $u:=u_j$. More generally, the lower bound~\eqref{e:coneig}
holds for quasimodes $u_h$ of the Laplacian of the following type:
\begin{equation}
  \label{e:quasimode}
\|(-h^2\Delta-I)u_h\|_{L^2(M)}= o\big(h/\log(1/h)\big),\quad
h\to 0;\quad
\|u_h\|_{L^2(M)}=1.
\end{equation}
On the opposite, the lower bound~\eqref{e:coneig} may fail
for quasimodes of error $\cO(h/\log(1/h))$: for $(M,g)$ a surface of constant
negative curvature (also known as a \emph{hyperbolic surface}),
Brooks~\cite{Brooks} constructed quasimodes of such strength localized along a closed geodesic;
the construction was extended to more general two-dimensional quantum systems by
Eswarathasan--Nonnenmacher~\cite{Suresh1}, and in higher dimension to quasimodes localized
on an invariant submanifold of~$M$ by Eswarathasan--Silberman~\cite{Suresh2}.

%%%%%%%%%%%%%%%%%%%%%%%%%%%%%%%%%%%%%%%%%%%%%%%%%%%%%%%%%%%%%%%%%%%%%%%%%%%%%%%%
\subsection{Application to semiclassical measures}

We now discuss two applications of Theorem~\ref{t:eig}.
The first one concerns \emph{semiclassical measures}, which describe
asymptotic macroscopic distribution of
subsequences of eigenfunctions. More precisely, if $(u_{j_k})_{k\in\IN}$ is a sequence
of eigenfunctions with $\lambda_{j_k}\to\infty$ and $h_{j_k}:=\lambda_{j_k}^{-1}$,
then we say that $(u_{j_k})_{k}$ converges to a measure $\mu$ on $T^*M$ if
\begin{equation}
  \label{e:semimes}
\langle\Op_{h_{j_k}}(a)u_{j_k},u_{j_k}\rangle_{L^2(M)}\xrightarrow{k\to\infty}\int_{T^*M}a\,d\mu\quad\text{for all}\quad
a\in S^0(T^*M).
\end{equation}
The measure $\mu$ is called a semiclassical measure of the manifold
$(M,g)$, it describes the asymptotic microlocal properties of the
eigenstates along the sequence $(u_{j_k})$ of eigenfunctions.
A compactness argument shows that, from any sequence of eigenstates
$(u_{j_k})$, it is always possible to extract a subsequence which
converges to a semiclassical measure. Any semiclassical measure is a
probability measure supported inside $S^*M$, which is invariant under the geodesic flow,
see~\cite[Chapter~5]{e-z}. 

From~\eqref{e:semimes} and the semiclassical calculus we see that
$\|\Op_{h_{j_k}}(a)u_{j_k}\|_{L^2(M)}^2$ converges to $\int
|a|^2\,d\mu$. Thus Theorem~\ref{t:eig} implies the following
%%%%%%%%%%%%%%%%%%%%%%%%%%%%%%%%%%%%%%%%%%%%%%%%%%%%%%%%%%%%%%%%%%%%%%%%%%%%%%%%
\begin{theo}
\label{t:semimes}
Let $\mu$ be a semiclassical measure associated to a sequence
of Laplacian eigenfunctions on $M$. Then $\supp\mu=S^*M$, that is
$\mu(U)>0$ for any open nonempty $U\subset S^*M$.
\end{theo}
%%%%%%%%%%%%%%%%%%%%%%%%%%%%%%%%%%%%%%%%%%%%%%%%%%%%%%%%%%%%%%%%%%%%%%%%%%%%%%%%
While we do not provide an explicit formula for the lower bound on $\mu(U)$ in terms
of~$U$, we show that this lower bound only depends on a certain dynamical quantity associated to~$U$:
%%%%%%%%%%%%%%%%%%%%%%%%%%%%%%%%%%%%%%%%%%%%%%%%%%%%%%%%%%%%%%%%%%%%%%%%%%%%%%%%
\begin{theo}
\label{t:eig-quant}
There exists $\varepsilon_0>0$ depending only on $(M,g)$ such that the following holds.
Assume that $U\subset S^*M$ is an open set which is $(L_0,L_1)$-dense
in both unstable and stable directions in the sense of Definition~\ref{d:stun-dense} below,
and has diameter less than~$\varepsilon_0$.
Then for each semiclassical measure $\mu$ we have
$\mu(U)\geq c$, where the constant $c>0$ depends only on $(M,g)$
and on the lengths $(L_0,L_1)$.
\end{theo}
%%%%%%%%%%%%%%%%%%%%%%%%%%%%%%%%%%%%%%%%%%%%%%%%%%%%%%%%%%%%%%%%%%%%%%%%%%%%%%%%
Theorem~\ref{t:eig-quant} follows by analyzing the dependence of various parameters
in the proof of Theorem~\ref{t:eig}. We indicate the required changes
in various remarks throughout the paper, with the proof of
Theorem~\ref{t:eig-quant} explained at the end of~\S\ref{s:uncon-long}.
Let us remark that Theorems~\ref{t:semimes} and~\ref{t:eig-quant} also apply to semiclassical measures associated with quasimodes of the form~\eqref{e:quasimode}.

We believe that our results are not specific to the Laplacian, but can be extended to operators of the form $P=-\Delta + P_1 + P_0$ on $(M,g)$, where $P_i$ are symmetric differential operators of order $i$ with smooth coefficients. `
One could also consider semiclassical Schr\"odinger operators $P_h=-h^2\Delta + V$ with $V\in C^\infty(M;\mathbb R)$, and study families of eigenstates $P_h u_h =E(h) u_h$, with eigenvalues $E(h)\to 1$ when $h\to 0$. If the potential $V$ is sufficiently small, the Hamiltonian flow generated by the symbol $p(x,\xi)=|\xi|^2_g+V(x)$, restricted to the energy hypersurface $p^{-1}(1)$, will still enjoy the Anosov property, due to the structural stability of that property. We then believe that the eigenstates $(u_h)_{h\to 0}$, as well as the associated semiclassical measures, will satisfy similar delocalization properties as in Theorems~\ref{t:coneig}--\ref{t:eig-quant}.

To put Theorems~\ref{t:eig}--\ref{t:eig-quant} into context, let us
give a brief historical review, referring to
the expository articles of Marklof~\cite{MarklofReview},
Zelditch~\cite{ZelditchReview}, and Sarnak~\cite{SarnakQUE}
for more information. The Quantum Ergodicity theorem of Shnirelman~\cite{Shnirelman,Shnirelman2},
Zelditch~\cite{ZelditchQE}, and Colin de Verdi\`ere~\cite{CdV} states that
when the geodesic flow on $S^*M$ is ergodic (with respect to the
Liouville measure $\mu_L$), there exists a density one sequence
$(u_{j_k})$ which \emph{asymptotically equidistributes}, namely which
converges to the Liouville measure $\mu_L$ in the sense of \eqref{e:semimes}. The Quantum Unique Ergodicity (QUE) conjecture
formulated by Rudnick--Sarnak~\cite{RudnickSarnak} states that on any Anosov manifold,
the full sequence of eigenfunctions
equidistributes, that is $\mu_L$ is the unique semiclassical measure. So far this
conjecture has only been established for hyperbolic surfaces possessing arithmetic
symmetries~\cite{LindenstraussQUE}. On the other hand, there exist toy models of
quantized Anosov maps on the two-dimensional torus, where the corresponding QUE conjecture fails, see
Faure--Nonnenmacher--de Bi\`evre~\cite{FNB} and Anantharaman--Nonnenmacher~\cite{AnantharamanNonnenmacherBaker}.
On a similar Anosov toy model on a higher dimensional torus, Kelmer~\cite{Kelmer} exhibited counterexamples to QUE, but also to our full delocalization result, featuring semiclassical measure supported on proper submanifolds.

With QUE seeming out of reach, it is natural to wonder which flow invariant probability
measures on $S^*M$ can arise as semiclassical measures; in other
words, does quantum mechanics select certain invariant measures, or
allow all of them? The first restrictions
on semiclassical measures were proved by Anantharaman~\cite{AnAnn},
Anantharaman--Nonnenmacher~\cite{AN07}, Rivi\`ere~\cite{Riv10}, and
Anantharaman--Silberman~\cite{AnantharamanSilberman}, in the form of
positive lower bounds on the Kolmogorov--Sinai entropy of $\mu$. The
entropy is a nonnegative number associated with each invariant
measure, representing the information theoretic complexity of the
measure. Low-entropy measures therefore have low complexity. 
These lower bounds on the entropy exclude, for instance, the extreme case when $\mu$ is a $\delta$ measure
on a closed geodesic. 
Our Theorem~\ref{t:semimes} gives a different type of restriction on
$\mu$.
As explained in~\cite{meassupp}, there exist invariant measures which are excluded by Theorem~\ref{t:semimes}
but not by entropy bounds, and vice versa. For instance, on any Anosov
surface one can construct flow invariant fractal subsets $F\subsetneq
S^*M$ of Hausdorff dimension close to~$3$, which support invariant measures
of large entropy. Conversely, an invariant measure of the form $\epsilon
\mu_L+(1-\epsilon)\delta_\gamma$, with $\delta_\gamma$ the delta
measure on a closed geodesic and $0<\eps\ll 1$, will have full support
but small entropy.

In the special case of hyperbolic surfaces, Theorems~\ref{t:coneig}--\ref{t:semimes} were proved
by Dyatlov--Jin~\cite{meassupp}; see also the reviews~\cite{fwlEDP,ICMPReview}.
The proofs in the present paper partially
use the strategy of~\cite{meassupp}, in particular they rely on the fractal uncertainty
principle (FUP) established by Bourgain--Dyatlov~\cite{fullgap}. However, many new difficulties
arise in the variable curvature case, in particular from the fact that the
stable and unstable foliations on $S^*M$ are not smooth, see~\S\S\ref{s:remarks-proofs},\ref{s:outline} below.

%%%%%%%%%%%%%%%%%%%%%%%%%%%%%%%%%%%%%%%%%%%%%%%%%%%%%%%%%%%%%%%%%%%%%%%%%%%%%%%%
\subsection{Application to control theory}

The second application of Theorem~\ref{t:eig} is to observability
and exact null-controllability
for the (nonsemiclassical) Schr\"odinger equation:
%%%%%%%%%%%%%%%%%%%%%%%%%%%%%%%%%%%%%%%%%%%%%%%%%%%%%%%%%%%%%%%%%%%%%%%%%%%%%%%%
\begin{theo}
  \label{t:control}
Assume that $\Omega\subset M$ is open and nonempty, and fix $T>0$. Then:
\begin{itemize}
\item (Observability)
There exists a constant $K>0$ depending only on $M$, $\Omega$, and~$T$, such that for any $u_0\in L^2(M)$, we have
\begin{equation}
\label{e:observe}
\|u_0\|_{L^2(M)}^2\leq K\int_0^T\|e^{it\Delta}u_0\|_{L^2(\Omega)}^2dt;
\end{equation}
\item (Control) For any $u_0\in L^2(M)$, there exists $f\in L^2((0,T)\times\Omega)$ such that the solution to the equation
\begin{equation*}
(i\partial_t+\Delta)u(t,x)=f\mathbf 1_{(0,T)\times\Omega}(t,x),
\quad 
u(0,x)=u_0(x)
\end{equation*}
satisfies
\begin{equation*}
u(T,x)\equiv 0.
\end{equation*}
\end{itemize}
\end{theo}
%%%%%%%%%%%%%%%%%%%%%%%%%%%%%%%%%%%%%%%%%%%%%%%%%%%%%%%%%%%%%%%%%%%%%%%%%%%%%%%%
The proof that the above statements follow from Theorem~\ref{t:eig} is identical
to the one in Jin~\cite{JinControl}, so we will not reproduce it here.

For a general manifold, such observability/control is known to hold if the open set $\Omega$ 
satisfies the geometric control condition of Bardos--Lebeau--Rauch~\cite{BLR92,Le92}, namely if every geodesic ray intersects $\Omega$. Yet, it may hold as well if this geometric condition is violated, for instance on compact manifolds of negative sectional curvature, provided the set of geodesics never meeting $\Omega$ is ``sufficiently thin'', see Anantharaman--Rivi\`ere~\cite{AnRi12}. The novelty in the above two-dimensional result, is that this control holds for \emph{any} open set $\Omega$, now matter how thick the set of uncontrolled geodesics.
So far the only other family of manifolds for which observability/control was known to hold
for any $\Omega$ were the flat tori, see Haraux~\cite{Ha89} and Jaffard~\cite{Ja90}.  Further references on this question may be found in Burq--Zworski~\cite{BZ04} and Jin~\cite{JinControl}.
%%%%%%%%%%%%%%%%%%%%%%%%%%%%%%%%%%%%%%%%%%%%%%%%%%%%%%%%%%%%%%%%%%%%%%%%%%%%%%%%
\subsection{Damped wave equation}

Our final result concerns the long time behavior of solutions to the damped wave equation on $M$, with damping function $b\in C^\infty(M)$, $b\geq 0$, $b\not\equiv 0$:
\begin{equation}
	\label{e:dampwave}
(\partial_t^2-\Delta+2b(x)\partial_t)v(t,x)=0,
\quad v|_{t=0}=v_0(x),\quad
 \partial_tv|_{t=0}=v_1(x).
\end{equation}
Semigroup theory shows that for initial data $(v_0,v_1)\in \mathcal{H}^0:=H^1(M)\times L^2(M)$,
the above equation has a unique solution in
$C(\mathbb{R}^+;H^1(M))\cap C^1(\mathbb R^+;L^2(M))$.
The energy of this solution at time $t\geq 0$ is defined by
\begin{equation}
	\label{e:energy}
E(v(t)):=\frac{1}{2}\int_M|\partial_tv(t,x)|^2+|\nabla_x v(t,x)|^2\,dx.
\end{equation}
It is well-known that on every compact Riemannian manifold, this energy
decays to zero when $t\to\infty$. However, the rate of decay depends
on a subtle interplay between the geodesic flow and the
support of the damping function, see Lebeau~\cite{Le96}. In particular,
exponential decay (the fastest possible decay) always holds if the
damping function satisfies the geometric control condition, that is
any geodesic intersects the set $\{b > 0\}$. 
In the case of an Anosov
surface with any damping function $b$, we obtain exponential decay
without requiring this geometric condition:
%%%%%%%%%%%%%%%%%%%%%%%%%%%%%%%%%%%%%%%%%%%%%%%%%%%%%%%%%%%%%%%%%%%%%%%%%%%%%%%
\begin{theo}
  \label{t:dwe}
Assume that $b\geq0$ but $b\not\equiv0$.
Then for every $s>0$, there exist constants $C$ and
$\gamma=\gamma(s)>0$ such that for any $(v_0,v_1)\in
\mathcal{H}^s:=H^{s+1}(M)\times H^s(M)$, the energy of the solution decays exponentially:
\begin{equation}
	\label{e:energy-decay}
E(v(t))\leq C\,e^{-\gamma t}\,\|(v_0,v_1)\|_{\mathcal{H}^s}^2\,.
\end{equation}
\end{theo}
%%%%%%%%%%%%%%%%%%%%%%%%%%%%%%%%%%%%%%%%%%%%%%%%%%%%%%%%%%%%%%%%%%%%%%%%%%%%%%%
We remark that on any compact manifold, the decay~\eqref{e:energy-decay} holds for $s=0$ if and only if the set $\{b>0\}$ satisfies the geometric control condition,
see Rauch--Taylor~\cite{RaTa75}. 
On manifolds of negative curvature, an exponential decay controlled by a higher Sobolev norm $s>0$ has been proved in situations where the set of undamped trajectories is sufficiently ``thin'', see Schenck~\cite{Schenk}.

To our knowledge, Theorem~\ref{t:dwe} gives the first class of manifolds (of
dimension $\geq 2$) for which the energy decays exponentially (under a control by a higher Sobolev norm), no matter how small the support of the damping is. As a comparison, in the case of flat tori, in absence of geometric control of the region $\{b>0\}$,  the decay is instead algebraic in time, see Anantharaman--L\'eautaud~\cite{AnLe14}. For an account on previous results on the rate of energy decay for damped waves, the reader may consult the introduction to~\cite{JinDWE} and the references therein.

The proof of Theorem~\ref{t:dwe} uses many of the ingredients of the proof of Theorem~\ref{t:eig}, including the key estimate, Proposition~\ref{l:longdec-0}.
In the special case of hyperbolic surfaces, Theorem~\ref{t:dwe} was proved by Jin~\cite{JinDWE} using the methods of~\cite{meassupp}.

%%%%%%%%%%%%%%%%%%%%%%%%%%%%%%%%%%%%%%%%%%%%%%%%%%%%%%%%%%%%%%%%%%%%%%%%%%%%%%%%
\subsection{Structure of the article}
\label{s:remarks-proofs}

\begin{itemize}
\item In~\S\ref{s:ingredients} we review various ingredients used in the proof.
Those include: hyperbolic (Anosov) dynamics and stable/unstable manifolds (\S\ref{s:hyperbolics});
pseudodifferential operators with mildly exotic symbols and Egorov's theorem (\S\ref{s:semiclassics}); Lagrangian distributions/Fourier integral operators (\S\ref{s:prelim-fio});
fractal uncertainty principle (\S\ref{s:fups});
proof of porosity of dynamically defined sets (\S\ref{s:porosity-basic}).
%%%%%
\item In~\S\ref{s:proofs-of-theorems} we give the proofs of Theorems~\ref{t:eig}, \ref{t:eig-quant}
(\S\ref{s:proofs-1}), and~\ref{t:dwe} (\S\ref{s:proofs-2}). The
strategy of proof is
similar to the one used in~\cite{meassupp,JinDWE} in the constant
curvature case. It starts from a microlocal partition of the identity,
quantizing the partition of~$S^*M$ into the controlled vs. uncontrolled
regions. Using the wave group, we may refine this microlocal partition up to a time $N$, 
each element of the refined partition being an operator
$A_{\bw}=A_{w_{N}}(N)\cdots A_{w_1}(1)A_{w_0}$ indexed by a word
$\bw=w_0\ldots w_N$, each symbol $w_j$ indicating whether the system
sits in the controlled or uncontrolled region at the time $j$. We need to push
this refinement up to a time $N\sim C\log(1/h)$ exceeding the
Ehrenfest time, which implies that the operators $A_{\bw}$ are no longer
pseudodifferential operators. The core of the proof then
consists in a key estimate on these ``long'' operators $A_{\bw}$, given in Proposition~\ref{l:longdec-0}.
%%%%%
\item \S\ref{s:long-word-fup} is devoted to the proof of this key
  Proposition. It proceeds by transforming this estimate into a
  collection of fractal uncertainty principles. 
This part of the proof is very different from the constant curvature case,
due to the fact that the Ehrenfest time is not uniform, but depends on
the trajectory; the difficulty also comes from the low regularity of
the stable/unstable foliations, which are not
$C^\infty$, but only $C^{2-\eps}$.
An outline of the proof is provided in~\S\ref{s:outline}.
%%%%%
\item In~\S\ref{s:ops-Aq} we complete the analysis of the operators
  $A_{\bw}$, by splitting them into more elementary pieces, which we
  may precisely analyze through a version of Egorov's Theorem up to
  the local Ehrenfest time. Similar elementary pieces were already introduced
  in the proofs of entropic lower bounds~\cite{AnAnn,NZ09,Riv10}; we will need 
 a somewhat more precise description of these operators for our aims. 
%%%%%
\item Appendix~\ref{s:semi-detail} contains quantitative estimates for
  the semiclassical pseudodifferential calculus on a compact surface,
used in~\S\ref{s:semiclassics} and~\S\ref{s:ops-Aq}.
\end{itemize}

%%%%%%%%%%%%%%%%%%%%%%%%%%%%%%%%%%%%%%%%%%%%%%%%%%%%%%%%%%%%%%%%%%%%%%%%%%%%%%%%
%%%%%%%%%%%%%%%%%%%%%%%%%%%%%%%%%%%%%%%%%%%%%%%%%%%%%%%%%%%%%%%%%%%%%%%%%%%%%%%%
\section{Ingredients}
  \label{s:ingredients}

In this section we review some of the ingredients used in the proof:
hyperbolic dynamics (\S\ref{s:hyperbolics}), semiclassical analysis (\S\S\ref{s:semiclassics}--\ref{s:prelim-fio}),
fractal uncertainty principle (\S\ref{s:fups}),
and porosity properties in the stable/unstable directions (\S\ref{s:porosity-basic}).

%%%%%%%%%%%%%%%%%%%%%%%%%%%%%%%%%%%%%%%%%%%%%%%%%%%%%%%%%%%%%%%%%%%%%%%%%%%%%%%%
\subsection{Hyperbolic dynamics}
  \label{s:hyperbolics}
  
Let $(M,g)$ be a compact connected Riemannian surface. Denote
$$
\begin{aligned}
T^*M\setminus 0& := \{(x,\xi)\in T^*M\colon \xi\neq 0\},\\
S^*M&:=\{(x,\xi)\in T^*M\colon |\xi|_g=1\}.
\end{aligned}
$$
Define the smooth function
\begin{equation}
  \label{e:p-def}
p:T^*M\setminus 0\to \mathbb R,\quad
p(x,\xi):=|\xi|_g.
\end{equation}
The Hamiltonian flow of $p$,
\begin{equation}
  \label{e:phi-def}
\varphi_t:=\exp(tH_p):T^*M\setminus 0\to T^*M\setminus 0
\end{equation}
is the homogeneous geodesic flow, note that it preserves $S^*M$.

We assume that the restriction of $\varphi_t$ to $S^*M$ is an \emph{Anosov flow},
namely for each $\rho\in S^*M$ there is a splitting
of the tangent space $T_\rho(S^*M)$ into one-dimensional spaces
$$
T_\rho (S^*M)=E_0(\rho)\oplus E_s(\rho)\oplus E_u(\rho)
$$
such that:
\begin{itemize}
\item $E_0(\rho)=\mathbb RH_p(\rho)$ is the flow direction;
\item $E_s,E_u$ are invariant under $d\varphi_t$;
\item $E_s$ is \emph{stable} and $E_u$ is \emph{unstable} in the following
sense: for any choice of continuous metric $|\bullet|$ on the fibers of $T(S^*M)$, there
exist $C,\theta>0$ such that
\begin{equation}
  \label{e:stun-exp}
|d\varphi_t(\rho)v|\leq Ce^{-\theta|t|}|v|,\quad\begin{cases}
v\in E_s(\rho),& t\geq 0;\\
v\in E_u(\rho),& t\leq 0.
\end{cases}
\end{equation}
\end{itemize}
The Anosov assumption holds in particular if $(M,g)$ has everywhere negative Gauss curvature,
see~\cite[Theorem~17.6.2]{KaHa}, \cite[Theorem~3.9.1]{Klingenberg}, or~\cite[Theorem~6 in~\S5.1]{stunnote}.
In the present setting the dependence of the spaces $E_s,E_u$ (and the stable/unstable manifolds defined in~\S\ref{s:stun} below)
on the base point $\rho$ is $C^{2-}$ but (unless $M$ has constant curvature) not $C^2$, see
 Remark~1 following Lemma~\ref{l:stun-straight}.

Since $\varphi_t$ is a homogeneous Hamiltonian flow, it preserves the
canonical 1-form $\xi\,dx$ (which is the symplectic dual of the dilation field
$\xi\cdot\partial_\xi$). By~\eqref{e:stun-exp} we see that
$\xi\,dx$ annihilates $E_s\oplus E_u$, that is
\begin{equation}
  \label{e:contact-structure}
E_s\oplus E_u=\ker(dp)\cap \ker(\xi\,dx).
\end{equation}
We fix \emph{adapted metrics} $|\bullet|_s$, $|\bullet|_u$, which are smooth Riemannian
metrics on $S^*M$, so that
the following stronger version of~\eqref{e:stun-exp} holds for some $\Lambda_0>0$:
\begin{equation}
  \label{e:stun-exp-2}
\begin{aligned}
|d\varphi_t(\rho)v|_s\leq e^{-\Lambda_0|t|}|v|_s,& \quad v\in E_s(\rho),\quad t\geq 0;\\
|d\varphi_t(\rho)v|_u\leq e^{-\Lambda_0|t|}|v|_u,& \quad v\in E_u(\rho),\quad t\leq 0.
\end{aligned}
\end{equation}
See for instance~\cite[Lemma~4.7]{stunnote} for the construction of such metrics.
By homogeneity we extend the spaces $E_0,E_s,E_u$ to $T^*M\setminus 0$.
We also extend $|\bullet|_s$, $|\bullet|_u$ to homogeneous metrics of degree~0
on~$T^*M\setminus 0$.

For each $\rho \in T^*M\setminus 0$ and $t\in\mathbb R$ we define the \emph{stable/unstable expansion rates}
(since $E_s,E_u$ are one-dimensional these coincide with the stable/unstable Jacobians):
\begin{equation}
  \label{e:stun-J-def}
\begin{aligned}
|d\varphi_t(\rho)v|_s= J^s_t(\rho)|v|_s,&\quad
v\in E_s(\rho);\\
|d\varphi_t(\rho)v|_u= J^u_t(\rho)|v|_u,&\quad
v\in E_u(\rho).
\end{aligned}
\end{equation}
From the stable/unstable decomposition and the homogeneity of the flow we see that
for all $\rho\in \{{1\over 4}\leq |\xi|_g\leq 4\}$ and all~$t$
\begin{equation}
  \label{e:exp-rate}
\begin{aligned}
\|d\varphi_t(\rho)\|\leq CJ^u_t(\rho),&\quad t\geq 0;\\
\|d\varphi_t(\rho)\|\leq CJ^s_t(\rho),&\quad t\leq 0.
\end{aligned}
\end{equation}
Since $E_0$ is spanned by $H_p$ and $E_s,E_u$ are tangent to the level sets of $p$,
we see that the weak stable/unstable spaces $E_s\oplus E_0$,
$E_u\oplus E_0$ are Lagrangian with respect to the standard symplectic
form $\omega$ on $T^*M\setminus 0$
and $E_s\oplus E_u$ is symplectic.
Since $\varphi_t$ are symplectomorphisms,
there exists a constant $C$ such that
for all $\rho\in T^*M\setminus 0$ and $t\in \mathbb R$
\begin{equation}
  \label{e:J-su}
C^{-1}\leq J^s_t(\rho)J^u_t(\rho)\leq C.
\end{equation}
Moreover, $J^s_t$ and $J^u_t$ are invariant under a short time evolution
by the flow $\varphi_t$ up to a multiplicative constant: for all
$\rho\in T^*M\setminus 0$, $t'\in [-1,1]$, and $t\in\mathbb R$
\begin{equation}
  \label{e:J-flow}
C^{-1}J^s_t(\rho)\leq J^s_t(\varphi_{t'}(\rho))\leq CJ^s_t(\rho),\quad
C^{-1}J^u_t(\rho)\leq J^u_t(\varphi_{t'}(\rho))\leq CJ^u_t(\rho).
\end{equation}
By~\eqref{e:stun-exp-2},
$J^s_t$ is exponentially decaying in time, and $J^u_t$ cannot grow faster than exponentially due to the compactness of $M$.
As a result, there exist constants%
\footnote{We can think of $\Lambda_0$ as the \emph{minimal expansion rate} and $\Lambda_1$ as
the \emph{maximal expansion rate} but strictly speaking this is not the case:
instead one should take as $\Lambda_0$ any number smaller than the
minimal expansion rate, and as $\Lambda_1$ any number larger than the maximal
expansion rate.}
$0<\Lambda_0\leq\Lambda_1$ such that
for all $\rho\in T^*M\setminus 0$
\begin{equation}
  \label{e:Lambda-0-1}
\begin{aligned}
e^{\Lambda_0 |t|}\leq J^u_t(\rho)\leq e^{\Lambda_1 |t|},&\quad
e^{-\Lambda_1 |t|}\leq J^s_t(\rho)\leq e^{-\Lambda_0 |t|}
&\quad\text{for all}\quad t\geq 0;\\
e^{-\Lambda_1 |t|}\leq J^u_t(\rho)\leq e^{-\Lambda_0 |t|},&\quad
e^{\Lambda_0 |t|}\leq J^s_t(\rho)\leq e^{\Lambda_1 |t|}
&\quad\text{for all}\quad t\leq 0.
\end{aligned}
\end{equation}
For technical reasons (in the proof of Lemma~\ref{l:cq-log}) we choose to take $\Lambda_1\geq 1$.

Define also
\begin{equation}
  \label{e:Lambda-def}
\Lambda:=\bigg\lceil{\Lambda_1\over\Lambda_0}\bigg\rceil\in\mathbb N.
\end{equation}  

%%%%%%%%%%%%%%%%%%%%%%%%%%%%%%%%%%%%%%%%%%%%%%%%%%%%%%%%%%%%%%%%%%%%%%%%%%%%%%%%
\subsubsection{Stable/unstable manifolds}
  \label{s:stun}

For $\rho\in S^*M$, denote by
$$
W_s(\rho),W_u(\rho)\subset S^*M
$$
the \emph{local stable/unstable leaves passing through $\rho$}.
These are $C^\infty$-embedded one dimensional disks (i.e. intervals)
tangent to $E_s$, $E_u$. Their definition depends on arbitrary choices
(because of the freedom of choosing where to end the interval)
however their behavior near each point depends only on $(M,g)$.
For the construction of $W_s(\rho),W_u(\rho)$
and their properties we refer to~\cite[Theorem~17.4.3]{KaHa},
\cite[Theorem~3.9.2]{Klingenberg},
or~\cite[Theorem~5 in~\S4.6]{stunnote}. We can ajust the definition of
these local laves such that they satisfy the following invariance properties under the flow $\varphi_t$:
\begin{equation}
  \label{e:hyprop-0}
\forall \rho\in S^*M,\qquad \varphi_1(W_s(\rho))\subset W_s(\varphi_1(\rho)),\quad
\varphi_{-1}(W_u(\rho))\subset W_u(\varphi_{-1}(\rho)).
\end{equation}
We also use the
\emph{local weak stable/unstable leaves}
\begin{equation}
  \label{e:weak-leaves}
W_{0s}(\rho):=\bigcup_{|t|\leq\tilde\varepsilon}\varphi_t(W_s(\rho)),\quad
W_{0u}(\rho):=\bigcup_{|t|\leq\tilde\varepsilon}\varphi_t(W_u(\rho)),
\end{equation}
which are $C^\infty$-embedded two dimensional rectangles inside $S^*M$ tangent
to the weak stable/unstable spaces $E_0\oplus E_s$, $E_0\oplus E_u$. 
Here $\tilde\varepsilon>0$ is fixed small,
depending only on $(M,g)$.
We extend $W_s,W_u,W_{0s},W_{0u}$ to $T^*M\setminus 0$ by homogeneity, however
for simplicity
the lemmas below are stated on~$S^*M$.

The stable/unstable manifolds are related to the dynamics of $\varphi_t$ by the following
lemma. To state it we introduce the following piece of notation: for $A,B>0$
\begin{equation}
A\sim B\quad\text{iff}\quad
C^{-1}A\leq B\leq CA\quad\text{for some}\ C>0
\ \text{depending only on}\ (M,g).
\end{equation}
%%%%%%%%%%%%%%%%%%%%%%%%%%%%%%%%%%%%%%%%%%%%%%%%%%%%%%%%%%%%%%%%%%%%%%%%%%%%%%%%
\begin{lemm}
  \label{l:stun-main}
Fix a Riemannian metric on $S^*M$ which
induces a distance function $d(\bullet,\bullet)$. Then
there exist $C,\varepsilon_0>0$ such that for all $\rho,\tilde\rho\in S^*M$ we have:
%%%%%%%%%%%%
\begin{enumerate}
\item if $\tilde \rho\in W_s(\rho)$, then
\begin{equation}
  \label{e:st-main-1}
d(\varphi_t(\rho),\varphi_t(\tilde \rho))\leq C J^s_t(\rho)d(\rho,\tilde\rho)\quad\text{for all}\quad t\geq 0;
\end{equation}
%%%%%%
\item if $\tilde\rho\in W_u(\rho)$, then 
\begin{equation}
  \label{e:un-main-1}
d(\varphi_t(\rho),\varphi_t(\tilde \rho))\leq C J^u_t(\rho)d(\rho,\tilde\rho)\quad\text{for all}\quad t\leq 0;
\end{equation}
%%%%%%
\item if $\tilde\rho \in W_{0s}(\rho)$, then
$J^s_t(\rho)\sim J^s_t(\tilde\rho)$ and
$J^u_t(\rho)\sim J^u_t(\tilde\rho)$ for all $t\geq 0$;
%%%%%%
\item if $\tilde\rho \in W_{0u}(\rho)$, then
$J^s_t(\rho)\sim J^s_t(\tilde\rho)$ and
$J^u_t(\rho)\sim J^u_t(\tilde\rho)$ for all $t\leq 0$;
%%%%%%
\item\label{l:point5} if $T\in\mathbb N_0$ and $d(\varphi_t(\rho),\varphi_t(\tilde\rho))\leq \varepsilon_0$ for all
integers $t\in [0,T]$, then
\begin{equation}
  \label{e:st-main-2}
d(\tilde\rho,W_{0s}(\rho))\leq {C/ J^u_T(\rho)}
\end{equation}
and $J^s_t(\rho)\sim J^s_t(\tilde\rho)$,
$J^u_t(\rho)\sim J^u_t(\tilde\rho)$ for all $t\in[0,T]$;
%%%%%%
\item\label{l:point6} if $T\in\mathbb N_0$ and $d(\varphi_t(\rho),\varphi_t(\tilde\rho))\leq \varepsilon_0$ for all
integers $t\in [-T,0]$, then
\begin{equation}
  \label{e:un-main-2}
d(\tilde\rho,W_{0u}(\rho))\leq {C/ J^s_{-T}(\rho)}
\end{equation}
and $J^s_t(\rho)\sim J^s_t(\tilde\rho)$, $J^u_t(\rho)\sim J^u_t(\tilde\rho)$ for all $t\in[-T,0]$.
\end{enumerate}
%%%%%%%%%%%%
\end{lemm}
%%%%%%%%%%%%%%%%%%%%%%%%%%%%%%%%%%%%%%%%%%%%%%%%%%%%%%%%%%%%%%%%%%%%%%%%%%%%%%%%
\Remarks 1. The difference between Lemma~\ref{l:stun-main} and standard facts from
hyperbolic dynamics (see for instance~\cite[Theorem~17.4.3]{KaHa}) is that
our estimates involve the \emph{local} expansion rates for the point $\rho$
rather than the minimal expansion rate. This will be important later in our analysis.

2. By~\eqref{e:J-su} we have $J^s_t(\rho)\sim 1/ J^u_t(\rho)$.
However the present lemma does not rely on $\varphi_t$ being symplectomorphisms
which is why we choose to keep both the stable and unstable Jacobians in the estimates.
%%%%%%%%%%%%%%%%%%%%%%%%%%%%%%%%%%%%%%%%%%%%%%%%%%%%%%%%%%%%%%%%%%%%%%%%%%%%%%%%
\begin{proof}
We only prove parts~(1), (3), (5), with parts~(2), (4), (6) proved similarly.

(1) Without loss of generality we may assume that the distance
function $d(\bullet,\bullet)$ is induced by the metric $|\bullet|_s$
used in~\eqref{e:stun-J-def} to define $J^s_t(\rho)$.
Since the tangent space to $W_s(\rho)$ at $\rho$ is $E_s(\rho)$, there exists
a constant $C$ such that
for every $\rho\in S^*M$ and $\tilde\rho\in W_s(\rho)$
\begin{equation}
  \label{e:sma-1}
\big| d(\varphi_1(\rho),\varphi_1(\tilde\rho))-J^s_1(\rho)d(\rho,\tilde\rho)\big|
\leq C d(\rho,\tilde\rho)^2.
\end{equation}
That is, when $\tilde\rho$ is close to $\rho$ the dilation factor of the distance
$d(\rho,\tilde\rho)$ by the map $\varphi_1$ is well-approximated by the
norm of the differential $d\varphi_1(\rho)$ on $E_s(\rho)$.

Since $\tilde\rho\in W_s(\rho)$, there exist constants $C,\theta>0$ such that
(see for instance~\cite[Theorem~17.4.3(3)]{KaHa} or~\cite[(4.67)]{stunnote})
\begin{equation}
  \label{e:hyprop-1}
d(\varphi_t(\rho),\varphi_t(\tilde \rho))\leq C e^{-\theta t}d(\rho,\tilde\rho)\quad\text{for all}\quad
t\geq 0.
\end{equation}
For each integer $t\geq 0$, we have $\varphi_t(\tilde\rho)\in W_s(\varphi_t(\rho))$
by~\eqref{e:hyprop-0}. Applying~\eqref{e:sma-1} with $\rho,\tilde\rho$
replaced by $\varphi_t(\rho),\varphi_t(\tilde\rho)$ we have
\begin{equation}
  \label{e:sma-2}
\begin{aligned}
d(\varphi_{t+1}(\rho),\varphi_{t+1}(\tilde\rho))&\leq J^s_1(\varphi_t(\rho))d(\varphi_t(\rho),\varphi_t(\tilde\rho))
+Cd(\varphi_t(\rho),\varphi_t(\tilde\rho))^2\\
&\leq
(1+Ce^{-\theta t})J^s_1(\varphi_t(\rho))d(\varphi_t(\rho),\varphi_t(\tilde\rho)).
\end{aligned}
\end{equation}
By the chain rule we have for all integers $t\geq 0$
\begin{equation}
  \label{e:sma-chain}
J^s_t(\rho)=J^s_1(\rho)J^s_1(\varphi_1(\rho))\cdots J^s_1(\varphi_{t-1}(\rho)).
\end{equation}
Iterating~\eqref{e:sma-2} and using that the product $\prod_{j=0}^\infty (1+Ce^{-\theta j})$ converges,
we get~\eqref{e:st-main-1} for all integer $t\geq 0$, which immediately implies it
for all $t\geq 0$. 

\smallskip

(3) We show that $J^s_t(\rho)\sim J^s_t(\tilde\rho)$, with the statement
$J^u_t(\rho)\sim J^u_t(\tilde\rho)$ proved similarly. Assume first that $\tilde\rho\in W_s(\rho)$.
The map $\rho\mapsto E_s(\rho)$ is in the H\"older class $C^\gamma$ for some~$\gamma>0$
(see for instance~\cite[Lemma~4.3]{stunnote}; in~\S\ref{s:stun-regularity} below we see that
in our setting it is in fact $C^{2-}$). Recalling~\eqref{e:stun-J-def} we have for all $\rho,\tilde\rho\in S^*M$
$$
|J^s_1(\rho)-J^s_1(\tilde\rho)|\leq C d(\rho,\tilde\rho)^\gamma.
$$
Applying this with $\rho$, $\tilde\rho$ replaced
by $\varphi_t(\rho)$, $\varphi_t(\tilde\rho)$ and using~\eqref{e:hyprop-1} we get
for all $t\geq 0$
\begin{equation}
  \label{e:baguette}
(1+Ce^{-\gamma\theta t})^{-1}J^s_1(\varphi_t(\rho))\leq J^s_1(\varphi_t(\tilde\rho))\leq (1+Ce^{-\gamma\theta t})J^s_1(\varphi_t(\rho)).
\end{equation}
Using the chain rule~\eqref{e:sma-chain} and iterating~\eqref{e:baguette}, we get
$J^s_t(\rho)\sim J^s_t(\tilde\rho)$ for all $t\geq 0$.
The general weak stable case $\tilde\rho\in W_{0s}(\rho)$ follows since
$J^s_t(\varphi_{t'}(\rho))\sim J^s_t(\rho)$ for all $\rho\in S^*M$ and $t'\in [-1,1]$
by~\eqref{e:J-flow}.

\smallskip

(5) Since $E_0\oplus E_s$ is transversal to $E_u$,
for $\varepsilon_0$ small enough and all $\rho,\tilde\rho\in S^*M$ such that
$d(\rho,\tilde\rho)\leq\varepsilon_0$, there exists a point (see Figure~\ref{f:sma})
%%%%%%%%%%%%%%%%%%%%%%%%%%%%%%%%%%%%%%%%%%%%%%%%%%%%%%%%%%%%%%%%%%%%%%%%%%%%%%%%
\begin{figure}
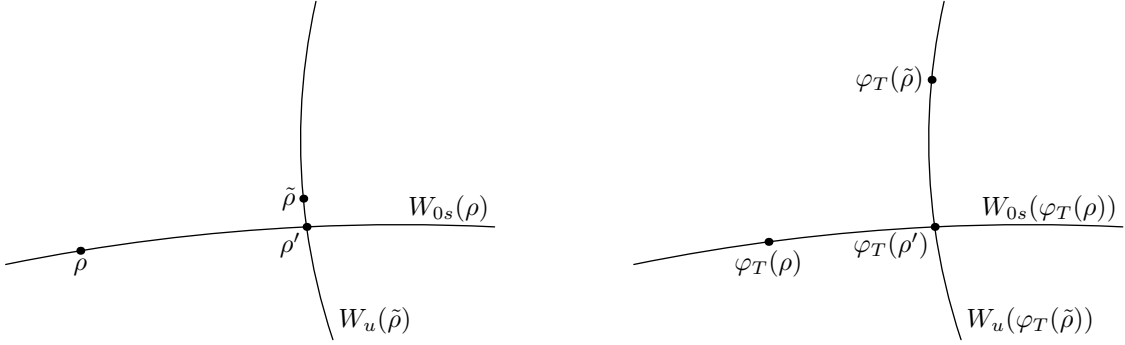

\includegraphics{varfup.1}
\qquad\qquad
\includegraphics{varfup.2}
\caption{Left: the points $\rho,\tilde\rho,\rho'$ in the proof of part~(5) of Lemma~\ref{l:stun-main},
with the flow direction removed.
Right: the image of the left half by $\varphi_T$.}
\label{f:sma}
\end{figure}
%%%%%%%%%%%%%%%%%%%%%%%%%%%%%%%%%%%%%%%%%%%%%%%%%%%%%%%%%%%%%%%%%%%%%%%%%%%%%%%%
\begin{equation}
  \label{e:sma-3}
\rho'\in W_{0s}(\rho)\cap W_u(\tilde\rho),\quad
d(\rho,\rho')\leq C\varepsilon_0.
\end{equation}
See for instance~\cite[Proposition~6.4.13]{KaHa} (in the related case of maps) or~\cite[(4.66)]{stunnote}.
Since $\rho'\in W_{0s}(\rho)$, by~\eqref{e:hyprop-1} there exists a constant $C_0\geq 1$ such that
\begin{equation}
  \label{e:sma-4}
d(\varphi_t(\rho'),\varphi_t(\rho))\leq C_0\varepsilon_0\quad\text{for all}\quad t\geq 0.
\end{equation}
By~\eqref{e:hyprop-0},
for $\varepsilon_0$ small enough we have (denoting by $B_d$ balls with respect to the distance function
$d(\bullet,\bullet)$)
\begin{equation}
  \label{e:sma-5}
\varphi_1(W_u(\hat\rho))\cap B_d(\varphi_1(\hat\rho),2C_0\varepsilon_0)
\subset W_u(\varphi_1(\hat\rho))
\quad\text{for all}\quad \hat\rho\in S^*M.
\end{equation}
Now, assume that $\rho,\tilde\rho\in S^*M$ and $d(\varphi_t(\rho),\varphi_t(\tilde\rho))\leq\varepsilon_0$
for all integers $t\in [0,T]$. Choose $\rho'$ satisfying~\eqref{e:sma-3}.
If $\varepsilon_0$ is small enough, then by the local uniqueness of unstable leaves
we have $\tilde\rho\in W_u(\rho')$.
By~\eqref{e:sma-4} we have for all integers $t\in [0,T]$
$$
d(\varphi_t(\rho'),\varphi_t(\tilde\rho))
\leq d(\varphi_t(\rho'),\varphi_t(\rho))
+d(\varphi_t(\rho),\varphi_t(\tilde\rho))\leq 2C_0\varepsilon_0.
$$
Applying~\eqref{e:sma-5} with $\hat\rho:=\varphi_t(\rho')$, we see by induction on~$t$
that
$$
\varphi_t(\tilde \rho)\in W_u(\varphi_t(\rho'))\quad\text{for all integer}\quad t\in [0,T].
$$
In particular, $\varphi_T(\tilde\rho)\in W_u(\varphi_T(\rho'))$.
Applying~\eqref{e:un-main-1} with $t:=-T$ and $\rho,\tilde\rho$
replaced by $\varphi_T(\rho'),\varphi_T(\tilde\rho)$,
$$
d(\rho',\tilde\rho)=d\big(\varphi_{-T}(\varphi_T(\rho')),\varphi_{-T}(\varphi_T(\tilde\rho))\big)
\leq {C J^u_{-T}(\varphi_T(\rho'))}
={C\over J^u_T(\rho')}
\leq {C\over J^u_T(\rho)}
$$
where the last inequality follows from
part~(3) of the present lemma.
Since $\rho'\in W_{0s}(\rho)$ this proves~\eqref{e:st-main-2}.

It remains to show that $J^s_t(\rho)\sim J^s_t(\tilde\rho)$,
$J^u_t(\rho)\sim J^u_t(\tilde\rho)$ for all $t\in [0,T]$.
As before, we prove the first statement with the second one proved similarly.
We can moreover restrict ourselves to integer values of $t$.
By part~(4) of the present lemma applied to the points $\varphi_t(\rho')$,
$\varphi_t(\tilde\rho)\in W_u(\varphi_t(\rho'))$
and propagation time $-t$, we have
$J^s_{-t}(\varphi_t(\rho'))\sim J^s_{-t}(\varphi_t(\tilde\rho))$.
Since $J^s_t(\rho')=1/J^s_{-t}(\varphi_t(\rho'))$
this implies that $J^s_t(\rho')\sim J^s_t(\tilde\rho)$.
On the other hand by part~(3) of the present lemma we have
$J^s_t(\rho)\sim J^s_t(\rho')$. Combining the last two statements
we get $J^s_t(\rho)\sim J^s_t(\tilde\rho)$ as needed.
\end{proof}
%%%%%%%%%%%%%%%%%%%%%%%%%%%%%%%%%%%%%%%%%%%%%%%%%%%%%%%%%%%%%%%%%%%%%%%%%%%%%%%%
Parts~(5) and~(6) of Lemma~\ref{l:stun-main} applied to $\tilde\rho:=\varphi_t(\rho)$
together with~\eqref{e:Lambda-0-1} give
%%%%%%%%%%%%%%%%%%%%%%%%%%%%%%%%%%%%%%%%%%%%%%%%%%%%%%%%%%%%%%%%%%%%%%%%%%%%%%%%
\begin{figure}
\includegraphics[width=3.2cm]{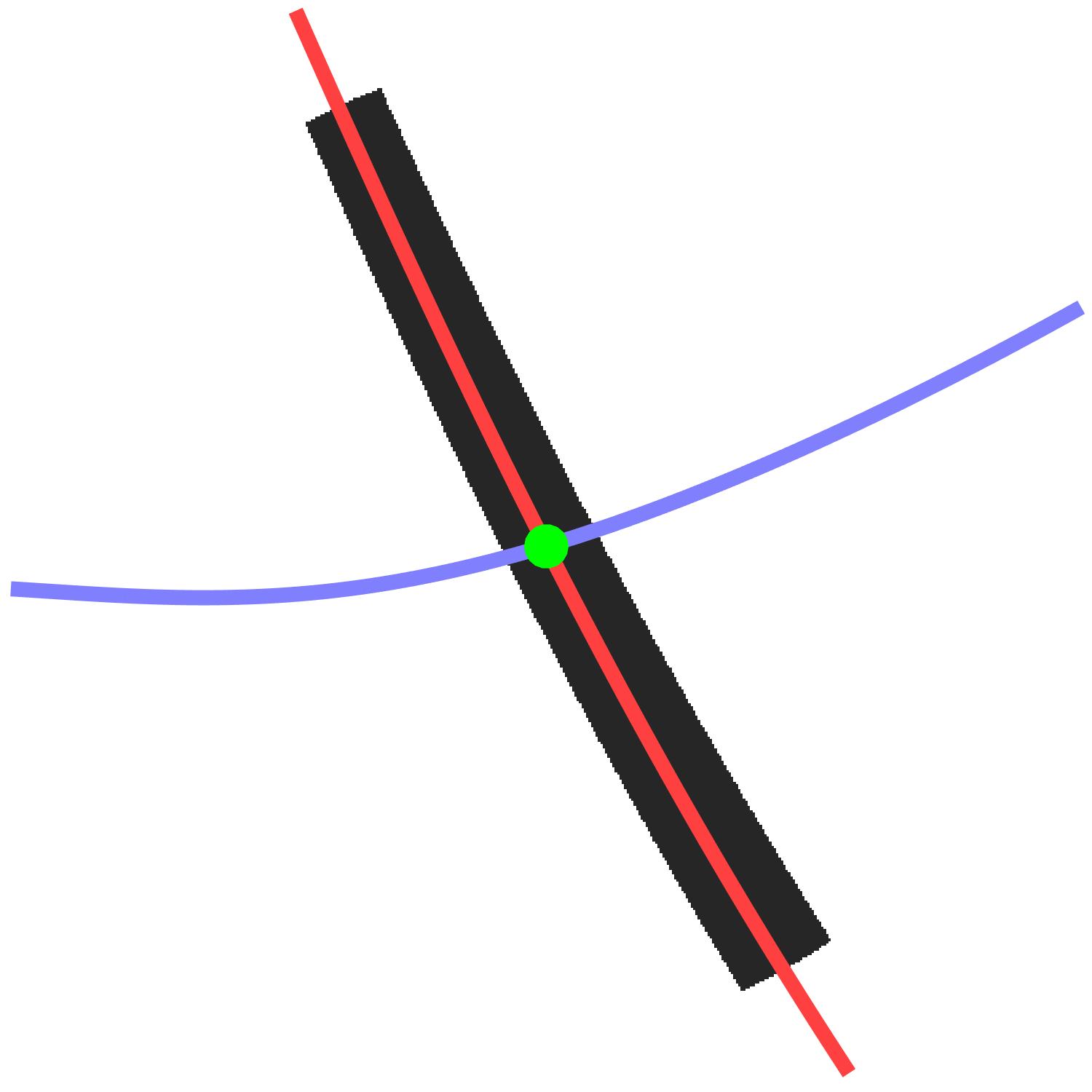}\qquad
\includegraphics[width=3.2cm]{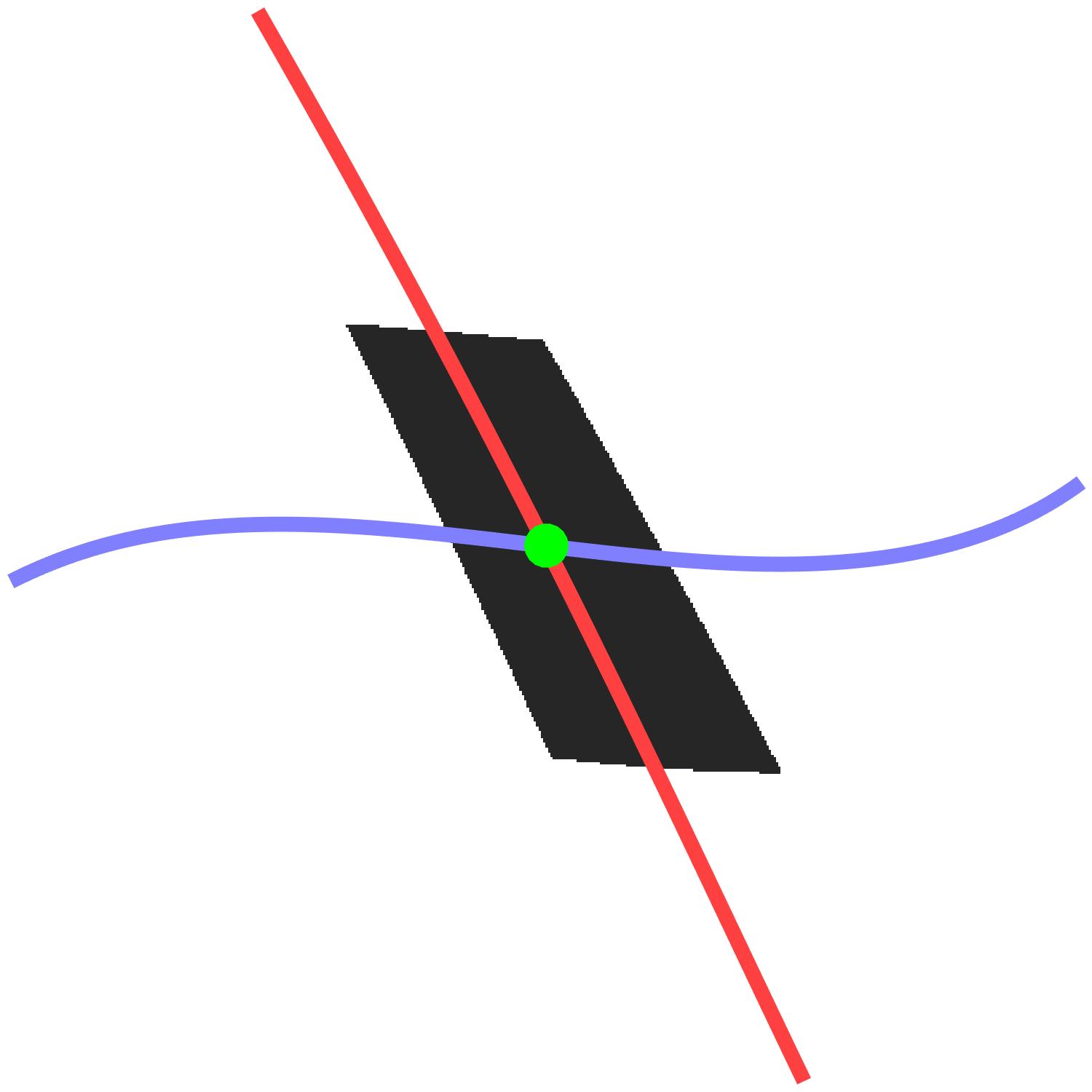}\qquad
\includegraphics[width=3.2cm]{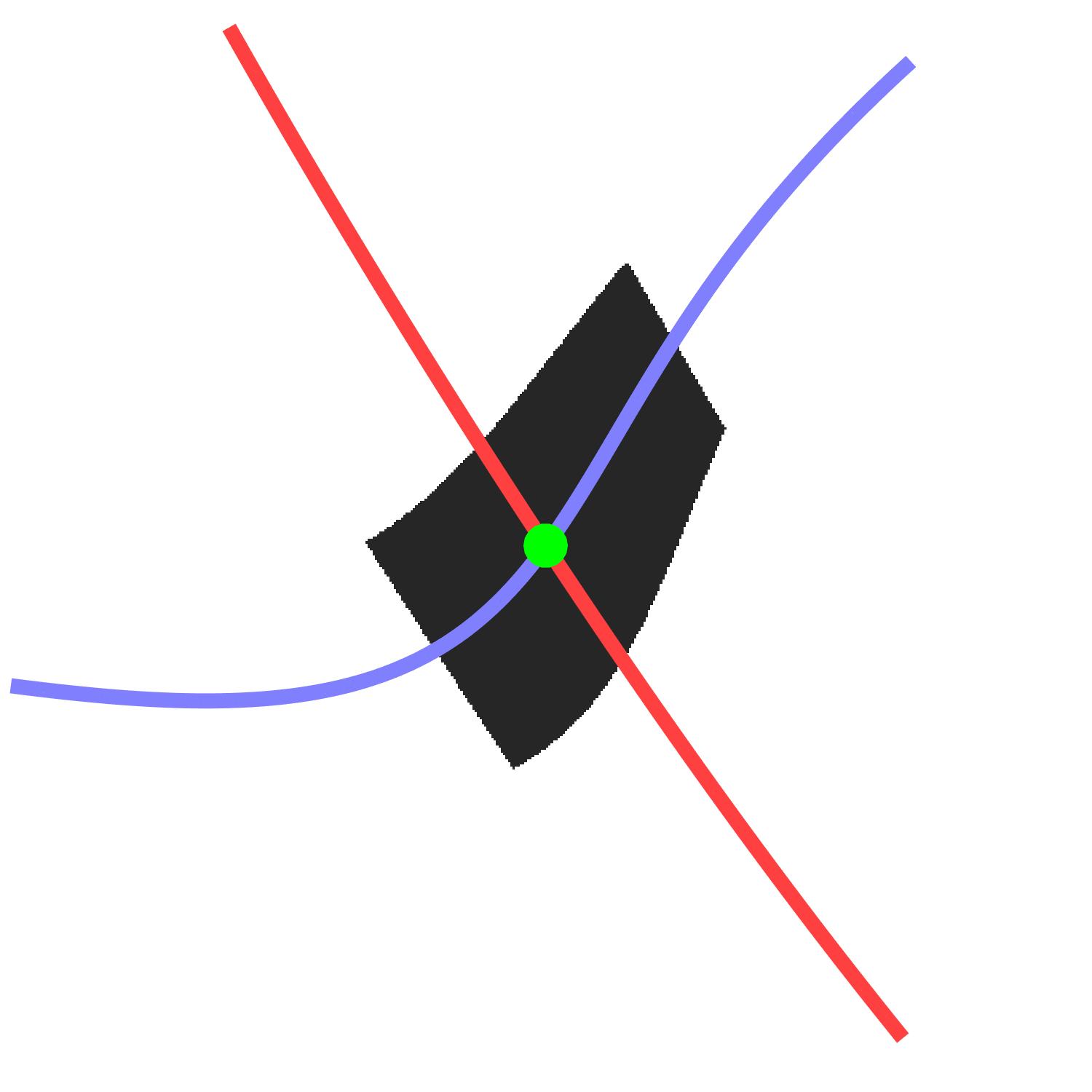}\qquad
\includegraphics[width=3.2cm]{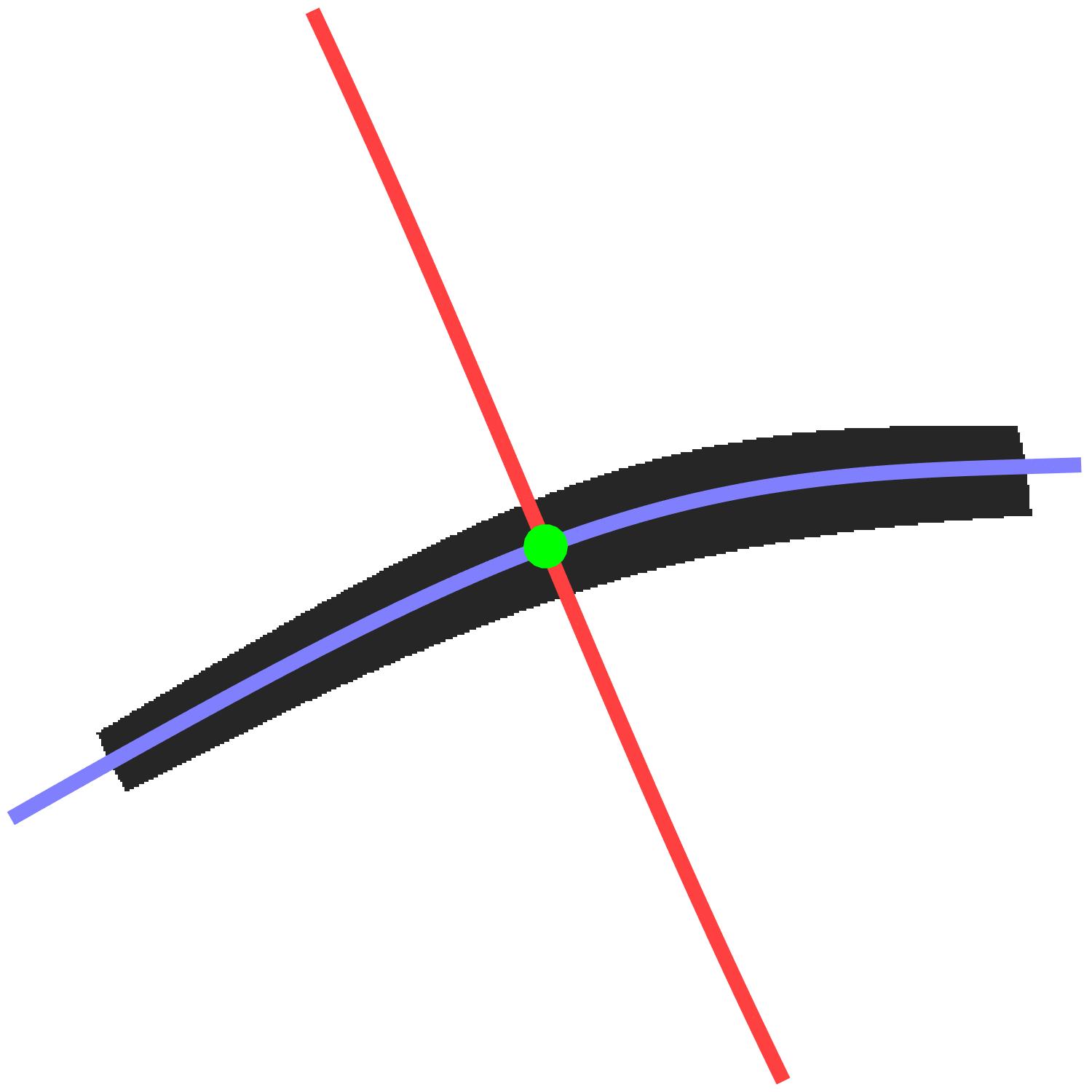}
\hbox to\hsize{\hss\quad
$t=0$ \hss\hss\quad
$t=1$ \hss\hss\quad
$t=2$ \hss\hss\quad
$t=3$
\hss}
\caption{An illustration of Corollary~\ref{l:stun-main-cor}
for $T=3$ with the flow direction removed. The green points are $\varphi_t(\rho_0)$,
the curves are the local stable (red) and unstable (blue) manifolds of these points,
and the black rectangles are the sets $\varphi_t(\mathcal V)$.}
\label{f:stun-cor}
\end{figure}
%%%%%%%%%%%%%%%%%%%%%%%%%%%%%%%%%%%%%%%%%%%%%%%%%%%%%%%%%%%%%%%%%%%%%%%%%%%%%%%%
\begin{corr}
  \label{l:stun-main-cor}
Let $d(\bullet,\bullet)$ and $\varepsilon_0>0$ be fixed in Lemma~\ref{l:stun-main}.
Fix $\rho_0\in S^*M$, $T\in\mathbb N_0$, and consider the set
$$
\mathcal V:=\big\{ \rho\in S^*M \mid d(\varphi_t(\rho),\varphi_t(\rho_0))\leq \varepsilon_0\quad\text{for all integer}\ t\in [0,T]\big\}.
$$
Then we have for all $\rho\in\mathcal V$ and $t\in [0,T]$
\begin{equation}
  \label{e:stun-main-rect}
\begin{aligned}
d\big(\varphi_t(\rho),W_{0s}(\varphi_t(\rho_0))\big)&\leq C/J^u_{T-t}(\rho_0)\leq Ce^{-\Lambda_0 (T-t)},\\
d\big(\varphi_t(\rho),W_{0u}(\varphi_t(\rho_0))\big)&\leq CJ^s_{t}(\varphi_t(\rho_0))\leq Ce^{-\Lambda_0 t}.
\end{aligned}
\end{equation}
\end{corr}
%%%%%%%%%%%%%%%%%%%%%%%%%%%%%%%%%%%%%%%%%%%%%%%%%%%%%%%%%%%%%%%%%%%%%%%%%%%%%%%%
Roughly speaking~\eqref{e:stun-main-rect} implies that
$\varphi_t(\mathcal V)$ lies inside
an $\varepsilon_0\times e^{-\Lambda_0 t}\times e^{-\Lambda_0(T-t)}$ sized
rectangle (with dimensions along $E_0,E_s,E_u$ respectively)
centered at $\varphi_t(\rho_0)$~-- see
Figure~\ref{f:stun-cor}.

%%%%%%%%%%%%%%%%%%%%%%%%%%%%%%%%%%%%%%%%%%%%%%%%%%%%%%%%%%%%%%%%%%%%%%%%%%%%%%%%
\subsubsection{Straightening out the weak unstable foliation}
\label{s:stun-regularity}

In~\S\ref{s:longtime} and \S\ref{s:fup-normal} below (most crucially in the proof
of Lemma~\ref{l:close-unstable}) we rely on the following construction of normal coordinates
which straighten out a given weak unstable leaf.
Similarly to Lemma~\ref{l:stun-main} we fix a distance function $d(\bullet,\bullet)$ on $S^*M$.
%%%%%%%%%%%%%%%%%%%%%%%%%%%%%%%%%%%%%%%%%%%%%%%%%%%%%%%%%%%%%%%%%%%%%%%%%%%%%%%%
\begin{lemm}
  \label{l:stun-straight}
For $\varepsilon_0>0$ small enough and for any $\rho_0\in S^*M$ there exists a $C^\infty$ symplectomorphism
$$
\varkappa=\varkappa_{\rho_0}:U_{\rho_0}\to V_{\rho_0},\quad
U_{\rho_0}\subset T^*M\setminus 0,\quad
V_{\rho_0}\subset T^*\mathbb R^2\setminus 0,
$$
such that, denoting points in $T^*M$ by $(x,\xi)$ and points in $T^*\mathbb R^2$ by $(y,\eta)$,
we have:
\begin{enumerate}
\item $U_{\rho_0},V_{\rho_0}$ are conic sets and the ball $B_d(\rho_0,\varepsilon_0)$
is contained in $U_{\rho_0}\cap S^*M$;
\item $\varkappa$ is homogeneous, namely it maps the vector field
$\xi\cdot\partial_\xi$ to $\eta\cdot\partial_\eta$;
\item $\varkappa(\rho_0)=(0,0,0,1)$, $d\varkappa(\rho_0)E_u(\rho_0)=\mathbb R\partial_{y_1}$,
and $d\varkappa(\rho_0)E_s(\rho_0)=\mathbb R\partial_{\eta_1}$;
\item putting $p(x,\xi):=|\xi|_g$, we have $p=\eta_2\circ\varkappa$ on $U_{\rho_0}$;
\item for each $\tilde\rho\in U_{\rho_0}$, the weak unstable leaf
$W_{0u}(\tilde\rho)$ satisfies for some $\tilde\zeta=Z(\tilde\rho)\in\mathbb R$
\begin{equation}
  \label{e:stun-straight}
\varkappa(W_{0u}(\tilde\rho)\cap U_{\rho_0})=\big\{(y_1,y_2,p(\tilde\rho)F(y_1,\tilde\zeta),p(\tilde\rho))\mid
(y_1,\tilde\zeta)\in \Omega,\ y_2\in\mathbb R\big\}\cap V_{\rho_0}
\end{equation}
where $F=F_{\rho_0}$ is a function from an open set $\Omega=\Omega_{\rho_0}\subset \mathbb R^2$
to $\mathbb R$ lying in the H\"older class $C^{3/2}(\Omega)$,
the map $y_1\mapsto F(y_1,\zeta)$ is $C^\infty$ for every~$\zeta$,
and $Z:U_{\rho_0}\to\mathbb R$ is homogeneous of degree~0,
in the class $C^{3/2}$ 
on $U_{\rho_0}\cap S^*M$,
and constant on each local weak unstable leaf;
\item $Z(\rho_0)=0$, $F(y_1,0)=0$, and $F(0,\zeta)=\zeta$;
\item $\partial_\zeta F(y_1,0)=1$.
\item there exists $C_{\rho_0}>0$ such that $|F(y_1, \zeta) - \zeta | \leq C_{\rho_0}\,|\zeta|^{3/2}$.
\end{enumerate}
The derivatives of all orders of $\varkappa_{\rho_0}$ and the constant $C_{\rho_0}$
are bounded independently of~$\rho_0$. 
\end{lemm}
%%%%%%%%%%%%%%%%%%%%%%%%%%%%%%%%%%%%%%%%%%%%%%%%%%%%%%%%%%%%%%%%%%%%%%%%%%%%%%%%
\begin{figure}
\includegraphics{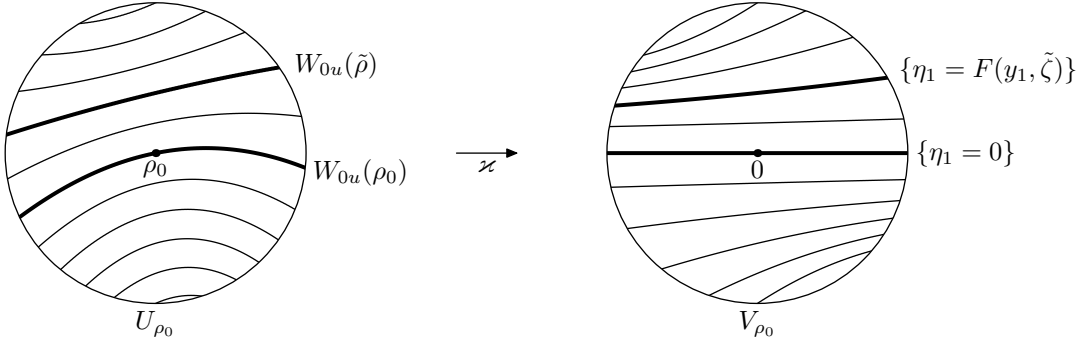}
\caption{An illustration of Lemma~\ref{l:stun-straight}, restricted to $S^*M$ and with the flow
direction removed. The curves on the left are the (weak) unstable manifolds and
the curves on the right are their images under $\varkappa$.}
\label{f:reg}
\end{figure}
%%%%%%%%%%%%%%%%%%%%%%%%%%%%%%%%%%%%%%%%%%%%%%%%%%%%%%%%%%%%%%%%%%%%%%%%%%%%%%%%
\Remarks 1. The statements (1)--(7) of Lemma~\ref{l:stun-straight} rely on the $C^{3/2}$ regularity of the unstable distribution $(E_u(\rho))_{\rho\in S^*M}$, proved by Hurder--Katok \cite[Theorem~3.1]{HK90}. They actually proved that for a generic surface of negative curvature, the distribution has regularity $C^{2-}$, but not better: by~\cite[Theorem~3.2 and Corollary~3.7]{HK90}, if the regularity is $C^2$, then
$(M,g)$ must have constant curvature. For our application the regularity $C^{1+\epsilon_0}$
for some $\epsilon_0>0$ would suffice, but we will use the $C^{3/2}$
regularity to simplify the expressions.

\noindent 2. The point (6) in the Lemma shows that the weak unstable manifold $W_{0u}(\rho_0)$ is
represented, in the coordinates given by~$\varkappa$, by the horizontal plane $\{\eta_1=0,\,\eta_2=1\}$, see \eqref{e:regw-1}.
The nearby unstable leaves $W_{0u}(\tilde\rho)$ will then be approximately horizontal, that is close
to planes $\{\eta_1=\zeta=\text{const},\, \eta_2=\text{const}\}$. The statements (7)--(8) express this almost
horizontality more precisely. In~\S\ref{s:long-word-fup} this almost horizontality will allow us to
apply the (``straight'') fractal uncertainty principle to families of almost-horizontal
unstable manifolds. The statement (8), which relies on the $C^{3/2}$ regularity, will be directly used in Lemma~\ref{l:close-unstable}.

To prove Lemma~\ref{l:stun-straight} we start by constructing a local coordinate frame under slightly weaker conditions:
%%%%%%%%%%%%%%%%%%%%%%%%%%%%%%%%%%%%%%%%%%%%%%%%%%%%%%%%%%%%%%%%%%%%%%%%%%%%%%%%
\begin{lemm}
  \label{l:stun-straight-weak}
Under the assumptions of Lemma~\ref{l:stun-straight} there exists a symplectomorphism $\varkappa_0$
having properties~(1)--(6) of that lemma.
\end{lemm}
%%%%%%%%%%%%%%%%%%%%%%%%%%%%%%%%%%%%%%%%%%%%%%%%%%%%%%%%%%%%%%%%%%%%%%%%%%%%%%%%
\begin{proof}
To construct $\varkappa_0$ we need to define a system of symplectic coordinates
$(y_1,y_2,\eta_1,\eta_2)$ on a conic neighborhood of $\rho_0$ which are homogeneous
(more precisely $y_1,y_2$ are homogeneous of degree~0 and $\eta_1,\eta_2$ are homogeneous of degree~1).
Put $\eta_2:=p$ and let $\eta_1|_{S^*M}$ be a defining function of the leaf $W_{0u}(\rho_0)$
(namely $\eta_1$ vanishes on $W_{0u}(\rho_0)$ and its differential is nondegenerate on that
submanifold) satisfying $H_p\eta_1=0$; this is possible since $H_p$ is tangent to $W_{0u}(\rho_0)$.
Extending $\eta_1$ to be homogeneous of degree~1, we see that the Poisson bracket
$\{\eta_1,\eta_2\}$ vanishes in a conic neighborhood of $\rho_0$.
The existence of the system of coordinates $(y_1,y_2,\eta_1,\eta_2)$ now follows
from the Darboux Theorem~\cite[Theorem~21.1.9]{Hormander3}, where we can arrange that
$y_1(\rho_0)=y_2(\rho_0)=0$. 

Since $\varkappa_0$ is homogeneous, it sends the canonical 1-form $\xi\,dx$ on $T^*M$
to the canonical 1-form $\eta\,dy$ on $T^*\mathbb R^2$. By~\eqref{e:contact-structure} we then have
$$
d\varkappa_0(\rho_0)(E_s(\rho_0)\oplus E_u(\rho_0))=\ker(d\eta_2)\cap \ker (d y_2).
$$
Since $E_u(\rho_0)$ is tangent to $W_{0u}(\rho_0)$, we see that
$d\varkappa_0(E_u(\rho_0))=\mathbb R \partial_{y_1}$.
To ensure that $d\varkappa_0(E_s(\rho_0))=\mathbb R\partial_{\eta_1}$ we compose $\varkappa_0$
with the nonlinear shear
$$
(y,\eta)\mapsto (y+d\mathcal F(\eta),\eta),\quad
\mathcal F(\eta_1,\eta_2):=\theta{\eta_1^2\over\eta_2}
$$
for an appropriate choice of $\theta\in\mathbb R$.

Properties~(1)--(4) of Lemma~\ref{l:stun-straight} follow immediately from the
discussion above. For property~(5),
we first note that by construction
\begin{equation}
  \label{e:regw-1}
\varkappa_0(W_{0u}(\rho_0))=\{\eta_1=0,\ \eta_2=1\}.
\end{equation}
Since the tangent spaces $E_{0u}(\rho)$ to the leaves $W_{0u}(\rho)$ depend continuously
on $\rho$, we see that for $\tilde\rho\in S^*M$ near $\rho$
the images $\varkappa_0(W_{0u}(\tilde\rho))$ project diffeomorphically onto the $(y_1,y_2)$
variables. Therefore we can locally write
$$
\varkappa_0(W_{0u}(\tilde\rho))=\{\eta_1=F_0(y_1,y_2,\tilde\zeta),\ \eta_2=1\}
$$
for some function $F_0(y_1,y_2,\zeta)$ and some $\tilde\zeta=Z_0(\tilde\rho)$ depending on $\tilde\rho$,
and we can assume that $F_0(0,0,\zeta)=\zeta$ which uniquely determines the functions $F_0,Z_0$.
Since $W_{0u}(\tilde\rho)$ is a $C^\infty$ submanifold, the function $y\mapsto F_0(y,\zeta)$
is $C^\infty$ for each $\zeta$. 
Since $H_p$ is tangent to each $W_{0u}(\tilde\rho)$ and is mapped by $\varkappa_0$
to $\partial_{y_2}$, we see that $\partial_{y_2}F_0=0$, thus $F_0$ is a function of $(y_1,\zeta)$ only.
This shows that~\eqref{e:stun-straight} holds for all $\tilde\rho\in U_{\rho_0}\cap S^*M$,
and it is easy to see that it holds for all $\tilde\rho\in U_{\rho_0}$
by homogeneity, with $Z_0$ homogeneous of degree~0.
Property~(6) follows from~\eqref{e:regw-1}. 

It remains to prove that the functions $F_0,Z_0$ have regularity $C^{3/2}$. According to \cite[Definition~4.1 and Theorem~4.2]{HK90}, the function $F_0$ is $C^\infty$ in the variable $y_1$ (this shows that each unstable leaf is smooth submanifold), and is $C^1$ w.r.t. $\zeta$. Besides, \cite[Theorem~3.1]{HK90} shows that the distribution $E_u(\rho)$ depends $C^{3/2}$ on $\rho$. In our coordinates~$\varkappa_0$, this regularity means that the ``slope function'' $e_u(y_1,\eta_1)$ of the unstable distribution has regularity $C^{3/2}$ w.r.t. its variables. Now, the function $F_0$ is a solution of the differential equation
$$
\frac{d}{dy_1} F_0(y_1,\zeta)= e_u\big(y_1, F_0(y_1,\zeta)\big),\quad\text{with initial condition}\ \ F_0(0,\zeta)=\zeta.
$$
Standard results on ODEs \cite[Chapter~V]{Ha02} show that the unique solution to such an ODE with $C^k$ function $e_u$ will depend in a $C^k$ way of the initial condition $\zeta$. The proof of \cite[Theorem~3.1]{Ha02} can be easily adapted to show that a $C^{3/2}$ function $e_u$ induces a solution $F_0$ with regularity $C^{3/2}$.
\end{proof}
%%%%%%%%%%%%%%%%%%%%%%%%%%%%%%%%%%%%%%%%%%%%%%%%%%%%%%%%%%%%%%%%%%%%%%%%%%%%%%%%

We now modify the map $\varkappa_0$ from Lemma~\ref{l:stun-straight-weak}
to obtain a map $\varkappa$ satisfying also the condition (7) of Lemma~\ref{l:stun-straight}.
Let $F_0$ be the function constructed in the proof of Lemma~\ref{l:stun-straight-weak}.
We have for every $\zeta$
\begin{equation}
  \label{e:reg-better}
y_1\mapsto \partial_\zeta F_0(y_1,\zeta)\quad\text{lies in }C^\infty.
\end{equation}
This follows from the existence of $C^\infty$-adapted transverse coordinates,
see~\cite[Point~2 in Definition~4.1 and Proposition~4.2]{HK90}. 

From the normalization $F_0(0,\zeta)=\zeta$ we see that
$\partial_\zeta F_0(y_1,\zeta)>0$ for $y_1$ close to~0.
Take the diffeomorphism $\psi$ of neighborhoods of $0$ in $\mathbb R$ defined by
$$
\psi(y_1)=\int_0^{y_1} \partial_\zeta F_0(s,0)\,ds.
$$
We define $\varkappa$ as the composition
$\varkappa:=\Psi\circ\varkappa_0$
where $\Psi$ is the symplectic lift of $\psi$:
$$
\Psi(y_1,y_2,\eta_1,\eta_2)=(\psi(y_1),y_2,\eta_1/\psi'(y_1),\eta_2).
$$
Then $\varkappa$ satisfies all the properties in Lemma~\ref{l:stun-straight},
with the function
$$
F(y_1,\zeta)={F_0(\psi^{-1}(y_1),\zeta)\over \partial_\zeta F_0(\psi^{-1}(y_1),0)},\quad
Z=Z_0\,.
$$
Like $F_0$, the function $F$ is $C^{3/2}$ w.r.t. the variable $\zeta$. 
We now use this regularity to prove part~(8) of Lemma~\ref{l:stun-straight}. This $C^{3/2}$ regularity, together with the property (7), implies the Taylor expansion of $F$ at the point $(y_1,0)$:
\begin{align*}
F(y_1,\zeta) &= F(y_1,0) + \zeta \partial_\zeta F(y_1,0) + \cO(\zeta^{3/2})\\
& = \zeta + \cO(\zeta^{3/2}),
\end{align*}
with the implied constant being uniform w.r.t. $y_1$. The second line is the point (8) of the Lemma: the leaf $W_{0u}(\rho)$ at ``height'' $\zeta$ from the reference horizontal leaf $W_{0u}(\rho_0)$ is contained in horizontal rectangle of thickness $\cO(\zeta^{3/2})$.

Finally, the fact that the derivatives of all orders of $\varkappa_{\rho_0}$ 
are bounded uniformly in $\rho_0$
follows directly from the arguments above and the fact
that the leaf $W_{0u}(\rho_0)$ depends continuously on $\rho_0$ as an embedded $C^\infty$
submanifold of~$S^*M$. It also shows that the constant $C_{\rho_0}$ in item (8) is uniformly bounded w.r.t $\rho_0$.
\hfill$\square$

%%%%%%%%%%%%%%%%%%%%%%%%%%%%%%%%%%%%%%%%%%%%%%%%%%%%%%%%%%%%%%%%%%%%%%%%%%%
\subsection{Pseudodifferential operators}
  \label{s:semiclassics}
  
Let $M$ be a manifold.
We use the standard semiclassical symbol class $S^k_h(T^*M)$ whose elements
$a(x,\xi;h)$
satisfy uniform derivative bounds on every compact subset $K\subset M$:
$$
|\partial^\alpha_x\partial^\beta_\xi a(x,\xi;h)|\leq C_{\alpha\beta K}\langle\xi\rangle^{k-|\beta|},\quad
x\in K,\ \xi\in T^*_xM,
$$
and admit an expansion in powers of $h$ and $|\xi|$. See for instance~\cite[Definition~E.3]{dizzy}
or~\cite[\S2.1]{hgap}. Denote by $S^k(T^*M)$ the class of $h$-independent symbols
in $S^k_h(T^*M)$.
We fix a (noncanonical) quantization procedure $\Op_h$ on~$M$,
see~\eqref{e:Op-h-M} below and~\cite[Proposition~E.15]{dizzy}.
Denote the class of semiclassical pseudodifferential operators 
with symbols in $S^k_h(T^*M)$ by $\Psi^k_h(M)$ and the (canonical)
principal symbol map by $\sigma_h:\Psi^k_h(M)\to S^k(T^*M)$.
See for instance~\cite[\S E.1.7]{dizzy} or~\cite[\S14.2]{e-z}.

If $M$ is noncompact, then we do not impose any restrictions on the growth
of $a(x,\xi;h)\in S^k_h(T^*M)$ as $|x|\to\infty$ and likewise do not
say anything about the asymptotic behavior of operators in $\Psi^k_h(M)$
as we approach the infinity of~$M$. Therefore in general
operators in $\Psi^k_h(M)$ are bounded (uniformly in~$h$)
acting $H^s_{h,\comp}(M)\to H^{s-k}_{h,\loc}(M)$ where
$H^s_{h,\loc}(M)$ denotes the space
of distributions locally in the semiclassical Sobolev space $H^s_h$
and $H^s_{h,\comp}(M)$ consists of the compactly supported elements of~$H^s_{h,\loc}(M)$.
See~\cite[\S E.1.8]{dizzy} or~\cite[\S8.3.1]{e-z}.
We will typically use operators in $\Psi^k_h(M)$ which are \emph{properly
supported}, mapping $H^s_{h,\comp}(M)\to H^{s-k}_{h,\comp}(M)$
and $H^s_{h,\loc}(M)\to H^{s-k}_{h,\loc}(M)$.
The quantization procedure $\Op_h$ is chosen so that
$\Op_h(a)$ is properly supported for every $a$
and $\Op_h(a)$ is \emph{compactly supported} (i.e. it has a compactly supported Schwartz kernel)
for symbols $a$ which are compactly supported in the $x$ variable. 
Of course if~$M$ is a compact manifold (which will mostly be the case in this paper),
then $H^s_{h,\loc}(M)$ and~$H^s_{h,\comp}(M)$ are the same space, denoted by $H^s_h(M)$.
We will mostly use the space $H^0_h(M)=L^2(M)$.

For $A\in \Psi^k_h(M)$ we denote by $\WFh(A)$ its wavefront set and by $\Ell_h(A)$
its elliptic set. Both are subsets of the fiber-radially compactified
cotangent bundle $\overline T^*M$. See for instance~\cite[\S E.2]{dizzy} or~\cite[\S2.1]{hgap}.
For $A\in\Psi^k_h(M)$, $B\in\Psi^\ell_h(M)$ we say that
$$
A=B+\mathcal O(h^\infty)\quad\text{microlocally on some open set $U\subset \overline T^*M$}
$$
if $\WFh(A-B)\cap U=\emptyset$.

We also use the notion of the wavefront set $\WFh(u)\subset\overline T^*M$ of an $h$-dependent tempered family of distributions $u=u(h)\in\mathcal D'(M)$
and the wavefront set $\WFh'(B)\subset \overline T^*(M_1\times M_2)$ of an $h$-dependent tempered family of
operators $B=B(h):\CIc(M_2)\to \mathcal D'(M_1)$, see~\cite[\S E.2.3]{dizzy}.

%%%%%%%%%%%%%%%%%%%%%%%%%%%%%%%%%%%%%%%%%%%%%%%%%%%%%%%%%%%%%%%%%%%%%%%%%%%%%%%%
\subsubsection{Mildly exotic symbols}
  \label{s:mildly-exotic}

We also use the mildly exotic symbol class $S^{\comp}_\delta(T^*M)$, $\delta\in[0,{1\over 2})$, consisting of symbols $a(x,\xi;h)$ such that:
\begin{itemize}
\item the $(x,\xi)$-support of $a$ is contained in an $h$-independent compact
subset of~$T^*M$;
\item the symbol $a$ satisfies derivative bounds
$$
|\partial^\alpha_{(x,\xi)} a(x,\xi;h)|\leq C_{\alpha\beta} h^{-\delta|\alpha|}.
$$
\end{itemize}
The operator class corresponding to $S^{\comp}_\delta(T^*M)$ is denoted by $\Psi_\delta^{\comp}(M)$.
We require operators in~$\Psi_\delta^{\comp}(M)$ to be compactly supported.
We use the same quantization procedure $\Op_h$ for this class
and note that compactly supported elements of $S^k_h(T^*M)$ lie in $S^{\comp}_0(T^*M)$.
See~\cite[\S4.4]{e-z} or~\cite[\S3.1]{qeefun}.

Operators in the class $\Psi^{\comp}_\delta(M)$ satisfy the following version of the \emph{sharp G\r arding inequality}
for all $u\in L^2(M)$:
\begin{equation}
  \label{e:gaarding}
a\in S^{\comp}_\delta(T^*M),\
\Re a\geq 0
\quad\Longrightarrow\quad
\Re\langle \Op_h(a) u,u\rangle_{L^2}\geq -Ch^{1-2\delta}\|u\|_{L^2}^2
\end{equation}
where the constant $C$ depends only on a certain $S^{\comp}_\delta(T^*M)$ seminorm of $a$.
The inequality~\eqref{e:gaarding} can be reduced to the case of the standard quantization on $\mathbb R^n$;
the latter case is proved by applying the standard sharp G\r arding inequality~\cite[Theorem~4.32]{e-z}
to the rescaled symbol $\tilde a(x,\xi):=a(h^\delta x,h^\delta\xi)$ and using the identity
$\Op_h(a)=T^{-1}\Op_{h^{1-2\delta}}(\tilde a)T$
where $T u(x)=u(h^\delta x)$.

We also have the following norm bound when $M$ is compact:
\begin{equation}
  \label{e:precise-norm}
a\in S^{\comp}_{\delta}(T^*M)
\quad\Longrightarrow\quad
\|\Op_h(a)\|_{L^2\to L^2}\leq\sup_{T^*M}|a|+Ch^{{1\over 2}-\delta}
\end{equation}
where the constant $C$ depends only on some $S^{\comp}_{\delta}(T^*M)$ seminorm of $a$.
To show~\eqref{e:precise-norm} it suffices to apply~\eqref{e:gaarding}
to the operator $c^2-\Op_h(a)^*\Op_h(a)=\Op_h(c^2-|a|^2)+\mathcal O(h^{1-2\delta})_{L^2\to L^2}$
where $c=c(h):=\sup_{T^*M}|a|$.

\noindent\textbf{Notation:} We remark that there is a slight conflict of notation between the classes~$S^k_h$ ($h$-dependent
symbols of order~$k$ in~$\xi$ which are polyhomogeneous in both~$\xi$ and~$h$) and~$S^{\comp}_\delta$ ($h$-dependent compactly supported
symbols losing $h^{-\delta}$ with each differentiation). A more proper notation would be
$$
S^k_{h,\text{phg}}(T^*M):=S^k_h(T^*M),\quad
S^{\comp}_{h,\delta}(T^*M):=S^{\comp}_\delta(T^*M).
$$
We however keep the shorter notation to reduce the number of indices used.
For $\delta\in [0,{1\over 2})$ we define the symbol class
$$
S^{\comp}_{\delta+}(T^*M)=\bigcap_{\epsilon>0} S^{\comp}_{\delta+\epsilon}(T^*M).
$$
We also use the following notation:
$$
f(h)=\mathcal O(h^{\alpha-})\quad\text{if}\quad
f(h)=\mathcal O_\eps(h^{\alpha-\epsilon})\quad\text{for all }\epsilon>0.
$$
When writing $a\in \CIc(T^*M)$ for a symbol $a$, we assume that
$a$ is $h$-independent unless stated otherwise.

%%%%%%%%%%%%%%%%%%%%%%%%%%%%%%%%%%%%%%%%%%%%%%%%%%%%%%%%%%%%%%%%%%%%%%%%%%%%%%%%
\subsubsection{Egorov's Theorem}
\label{s:egorov}

We now specialize to the case when $(M,g)$ is a compact Anosov surface
as in~\S\ref{s:hyperbolics}.
Since $\sigma_h(-h^2\Delta)=p^2$ where $p(x,\xi)=|\xi|_g$, by
the functional calculus of pseudodifferential operators
(see~\cite[Theorem~14.9]{e-z} or~\cite[\S8]{DimassiSjostrand})
we have
\begin{equation}
  \label{e:funcal}
\begin{gathered}
\psi\in \CIc(\mathbb R)\quad\Longrightarrow\quad
\psi(-h^2\Delta)\in \Psi^{-\infty}_h(M),\\
\WFh(\psi(-h^2\Delta))\subset \supp \psi(p^2),\quad
\sigma_h(\psi(-h^2\Delta))=\psi(p^2).
\end{gathered}
\end{equation}
We now discuss conjugation of pseudodifferential operators by the wave group.
Similarly to~\cite[\S2.2]{meassupp}, to avoid technical issues coming
from the zero section, instead of the true half-wave propagator
$e^{-it\sqrt{-\Delta}}$ we use the unitary operator
\begin{equation}
  \label{e:U-t-def}
U(t):=\exp(-itP/h),\quad
P:=\psi_P(-h^2\Delta)\in\Psi^{-\infty}_h(M),\quad
P^*=P,
\end{equation}
where we fixed some function
$$
\psi_P\in \CIc((0,\infty);\mathbb R),\quad
\supp\psi_P\subset \{\textstyle{1\over 25}<\lambda<25\},\quad
\psi_P(\lambda)=\sqrt\lambda\quad\text{for }{\textstyle{1\over 16}\leq\lambda\leq 16}.
$$
For a bounded operator $A$ on $L^2(M)$, we define the
Heisenberg-evolved operators
\begin{equation}
  \label{e:A-t-def}
A(t):=U(-t)AU(t),\qquad  t\in\IR\,.
\end{equation}
Assume that $a\in \CIc(T^*M)$ and $\supp a\subset \{{1\over 4}<|\xi|_g<4\}$. Then
Egorov's Theorem~\cite[Theorem~11.1]{e-z} implies that for $t$ bounded
independently of $h$ we have
\begin{equation}
  \label{e:egorov-basic}
A=\Op_h(a)\quad\Longrightarrow\quad
A(t)=\Op_h(a\circ\varphi_t)+\mathcal O(h)_{L^2\to L^2}
\end{equation}
where $\varphi_t=\exp(tH_p)$ is the homogeneous geodesic flow. In fact, the proof
in~\cite{e-z} gives the following stronger statement (see e.g.~\cite[\S C.2]{qeefun}
or Lemma~\ref{l:egorov-precise} below
for details): for each time $t$ there exists $a_t(x,\xi;h)\in
S^{\comp}_{0}(T^*M)$
such that
\begin{equation}
  \label{e:egorov-basic-more}
A(t)=\Op_h(a_t)+\mathcal O(h^\infty)_{\Psi^{-\infty}},\quad
a_t=a\circ\varphi_t+\mathcal O(h),\quad
\supp a_t\subset \varphi_{-t}(\supp a).
\end{equation}
We next extend~\eqref{e:egorov-basic} to the case of $t$ bounded by a small constant times
$\log(1/h)$, using the mildly exotic symbol classes described in~\S\ref{s:mildly-exotic}.
Let $\Lambda_1>0$ be the `maximal expansion rate' from~\eqref{e:Lambda-0-1}.
It follows from~\eqref{e:exp-rate} and~\eqref{e:Lambda-0-1} that
\begin{equation}
  \label{e:flow-long-bound}
\sup_{\rho\in \{{1\over 4}\leq |\xi|_g\leq 4\}}\|d\varphi_t(\rho)\|\leq Ce^{\Lambda_1|t|}\quad\text{for all }
t\in\mathbb R.
\end{equation}
%%%%%%%%%%%%%%%%%%%%%%%%%%%%%%%%%%%%%%%%%%%%%%%%%%%%%%%%%%%%%%%%%%%%%%%%%%%%%%%%
\begin{lemm}
  \label{l:egorov-mild}
Assume that $a\in \CIc(T^*M)$ and $\supp a\subset \{{1\over 4}\leq |\xi|_g\leq 4\}$;
put $A:=\Op_h(a)$.
Fix $\delta\in (0,{1\over 2})$.
Then we have uniformly in $t$ satisfying
$|t|\leq \delta\Lambda_1^{-1}\log(1/h)$:
\begin{enumerate}
\item $a\circ\varphi_t\in S^{\comp}_{\delta+}(T^*M)$;
\item $A(t)=\Op_h(a\circ\varphi_t)+\mathcal O(h^{1-2\delta-})_{L^2\to L^2}$.
\end{enumerate}
\end{lemm}
%%%%%%%%%%%%%%%%%%%%%%%%%%%%%%%%%%%%%%%%%%%%%%%%%%%%%%%%%%%%%%%%%%%%%%%%%%%%%%%%
\Remarks
1. A stronger statement similar to~\eqref{e:egorov-basic-more}, which shows that
the remainder $\mathcal O(h^{1-2\delta-})$ is actually pseudodifferential,
is proved for instance in~\cite[Proposition~3.9]{qeefun}.

\noindent 2. Lemma~\ref{l:egorov-mild} shows that Egorov's theorem holds for all times
$t$ which are smaller (by at least $\epsilon\log(1/h)$ for some $\epsilon>0$) than the \emph{minimal Ehrenfest time} $\log(1/h)\over 2\Lambda_1$.
Later we will show a finer version of Egorov's theorem, up to the (potentially much longer) \emph{local Ehrenfest time}~--
see Proposition~\ref{l:ehrenfest-prop}.
%%%%%%%%%%%%%%%%%%%%%%%%%%%%%%%%%%%%%%%%%%%%%%%%%%%%%%%%%%%%%%%%%%%%%%%%%%%%%%%%
\begin{proof}
(1) The estimate~\eqref{e:flow-long-bound} implies the following bounds on higher derivatives:
for all $t\in \mathbb R$, all multiindices $\alpha$, and all $\epsilon>0$
\begin{equation}
  \label{e:egorov-mild-int-1}
\sup_{T^*M}|\partial^\alpha (a\circ\varphi_t)|\leq C_{\alpha,\epsilon} e^{(\Lambda_1+\epsilon)|\alpha|\cdot |t|}.
\end{equation}
See for instance~\cite[Lemma~C.1]{qeefun}, whose proof applies directly to the present situation;
alternatively one could use the proof of Lemma~\ref{l:symbols-bounds} below in the special case $k=0$.
Under the condition $|t|\leq \delta\Lambda_1^{-1}\log(1/h)$ the bound~\eqref{e:egorov-mild-int-1}
implies that $a\circ\varphi_t\in S^{\comp}_{\delta+}(T^*M)$ uniformly in~$t$.

(2)
We use the following commutator formula
valid for all $\tilde a\in S^{\comp}_{\delta+}(T^*M)$
with $\supp \tilde a\subset \{{1\over 4}\leq |\xi|_g\leq 4\}$:
\begin{equation}
  \label{e:commfor}
[P,\Op_h(\tilde a)]=-ih\Op_h(H_p \tilde a)+\mathcal O(h^{2-2\delta-})_{L^2\to L^2}.
\end{equation}
Here it is important that $p\in S^{\comp}_0(T^*M)$
and we use the same quantization procedure $\Op_h$
on both sides of the equation; the $S^{\comp}_\delta$ calculus
would only give an $\mathcal O(h^{2-4\delta-})$ remainder.
See Remark~2 following Lemma~\ref{l:prod-mfld} for the proof.

Using~\eqref{e:commfor} and part~(1) we compute
for $|t|\leq \delta\Lambda_1^{-1}\log(1/h)$
$$
\begin{aligned}
\partial_t\big(U(t)\Op_h(a\circ\varphi_t)U(-t)\big)
&=U(t)\big(-ih^{-1}[P,\Op_h(a\circ\varphi_t)]+\Op_h(\partial_t(a\circ\varphi_t))\big)U(-t)\\
&=\mathcal O(h^{1-2\delta-})_{L^2\to L^2}.
\end{aligned}
$$
Integrating this from~0 to~$t$, we get $U(t)\Op_h(a\circ\varphi_t)U(-t)=\Op_h(a)+\mathcal O(h^{1-2\delta-})_{L^2\to L^2}$ which finishes the proof since $U(t)$ is unitary.
\end{proof}
%%%%%%%%%%%%%%%%%%%%%%%%%%%%%%%%%%%%%%%%%%%%%%%%%%%%%%%%%%%%%%%%%%%%%%%%%%%%%%%%
We will also need to control products of many pseudodifferential
operators. The following Lemma considers products of logarithmically many
pseudodifferential operators; it is proved in the same way as~\cite[Lemmas~A.1 and~A.6]{meassupp}
using the norm bound~\eqref{e:precise-norm}:
%%%%%%%%%%%%%%%%%%%%%%%%%%%%%%%%%%%%%%%%%%%%%%%%%%%%%%%%%%%%%%%%%%%%%%%%%%%%%%%%
\begin{lemm}
  \label{l:log-product}
Let $C$ be an arbitrary fixed constant, $\delta\in [0,{1\over 2})$, and assume that the symbols
$$
a_1,\dots, a_N\in S^{\comp}_{\delta}(T^*M),\quad
N\leq C\log(1/h),\quad
\sup|a_j|\leq 1
$$
have each $S^{\comp}_\delta$ seminorm bounded uniformly in~$j$.
Assume also that we are given operators
$A_j=\Op_h(a_j)+\mathcal O(h^{1-2\delta})_{L^2\to L^2}$
with the remainders bounded uniformly in~$j$. Then:
\begin{enumerate}
\item $a_1\cdots a_N\in S^{\comp}_{\delta+}(T^*M)$;
\item $A_1\cdots A_N=\Op_h(a_1\cdots a_N)+\mathcal O(h^{1-2\delta-})_{L^2\to L^2}$.
\end{enumerate}
That is, the product of these symbols (resp. operators) is essentially
in the same symbol class (resp. operator class) as the individual factors.
\end{lemm}
%%%%%%%%%%%%%%%%%%%%%%%%%%%%%%%%%%%%%%%%%%%%%%%%%%%%%%%%%%%%%%%%%%%%%%%%%%%%%%%%

%%%%%%%%%%%%%%%%%%%%%%%%%%%%%%%%%%%%%%%%%%%%%%%%%%%%%%%%%%%%%%%%%%%%%%%%%%%%%%%%
\subsection{Lagrangian distributions and Fourier integral operators}
  \label{s:prelim-fio}

In this section we review the theory of semiclassical
Lagrangian distributions and Fourier integral operators.
These are used in~\S\ref{s:longtime} to describe propagation of Lagrangian
states beyond the Ehrenfest time. In particular we use that the wave propagator $U(t)$
defined in~\eqref{e:U-t-def} is, after appropriate cutoffs, a Fourier integral operator
associated to the geodesic flow~$\varphi_t$, see~\eqref{e:U-1-fio}.
Fourier integral operators are also used in~\S\ref{s:micro-conjugation} to quantize a symplectomorphism which
locally straightens out unstable leaves.

We keep the presentation brief, referring the reader to~\cite{AlexandrovaFIO},
\cite[Chapter~5]{Guillemin-SternbergBook1}, and \cite[Chapter~8]{Guillemin-SternbergBook2}
for details. For other reviews (bearing some similarities to the one here)
see~\cite[\S3.2]{fwl}, \cite[\S3.2]{qeefun}, \cite[\S3.2]{nhp}, \cite[\S2.2]{hgap},
and~\cite[\S4.1]{NZ09}.
For the related nonsemiclassical case, see~\cite[Chapter~25]{Hormander4} and~\cite[Chapters~10--11]{Grigis-Sjostrand}.

%%%%%%%%%%%%%%%%%%%%%%%%%%%%%%%%%%%%%%%%%%%%%%%%%%%%%%%%%%%%%%%%%%%%%%%%%%%%%%%%
\subsubsection{Lagrangian manifolds and phase functions}
  \label{s:lagr-mflds}

Let $M$ be a smooth $n$-dimensional manifold (in this subsection we do not assume $M$ to be compact). Denote by
$\xi\,dx$ the canonical 1-form on $T^*M$, then the symplectic form
is given by
$$
\omega:=d(\xi\,dx).
$$
An embedded $n$-dimensional submanifold $\mathscr L\subset T^*M$
is called a \emph{Lagrangian submanifold} if the pullback of $\omega$
to $\mathscr L$ is zero; that is, the pullback of $\xi\,dx$
to $\mathscr L$ is a closed 1-form. A Lagrangian submanifold is called
\emph{exact} if the pullback of $\xi\,dx$ to $\mathscr L$
is equal to $dF$ for some function $F\in C^\infty(\mathscr L;\mathbb R)$,
called an \emph{antiderivative} on $\mathscr L$. We henceforth
define an exact Lagrangian submanifold as the pair $(\mathscr L,F)$
but still often denote it by~$\mathscr L$ for simplicity.

We note that $\mathscr L$ is exact in particular if it is \emph{conic},
namely the generator of dilations $\xi\cdot\partial_\xi$ is tangent to $\mathscr L$.
In this case the pullback of $\xi\,dx$ to $\mathscr L$ is equal to~0
(since $\omega(\xi\cdot\partial_\xi,v)=\langle \xi\,dx,v\rangle=0$
for any tangent vector $v\in T\mathscr L$), thus it is natural
to fix the antiderivative equal to~0 as well.

One way to obtain an exact Lagrangian submanifold is by using a phase function.
More precisely, if $U\subset M_x\times\mathbb R^m_\theta$ is an open set
(for some $m\in\mathbb N_0$) then we call a function
$\Phi(x,\theta)\in C^\infty(U;\mathbb R)$ a \emph{nondegenerate phase function}
if:
\begin{enumerate}
\item the differentials
$d(\partial_{\theta_j}\Phi)_{1\leq j\leq m}$
are linearly independent on the \emph{critical set}
$$
\mathcal C_\Phi:=\{(x,\theta)\in U\mid \partial_\theta \Phi(x,\theta)=0\}
$$
which is then an $n$-dimensional embedded submanifold of $U$; and
\item the following map is a smooth embedding:
$$
j_\Phi:\mathcal C_\Phi\to T^*M,\quad
j_\Phi(x,\theta)=(x,\partial_x\Phi(x,\theta)).
$$
\end{enumerate}
We call $\theta$ the \emph{oscillatory variables}.

Under the conditions~(1)--(2) above the manifold
\begin{equation}
  \label{e:L-Phi}
\mathscr L_\Phi:=j_\Phi(\mathcal C_\Phi)\ \subset\ T^*M
\end{equation}
is exact Lagrangian, with the antiderivative $F_\Phi\in
C^\infty(\mathscr L_\Phi;\mathbb R)$ given by the restriction of the
phase function on the critical set:
$$
F_\Phi(j_\Phi(x,\theta))=\Phi(x,\theta),\quad
(x,\theta)\in\mathcal C_\Phi.
$$
For an exact Lagrangian submanifold $(\mathscr L,F)$
we say that a nondegenerate phase function $\Phi$ \emph{generates} $\mathscr L$,
if $\mathscr L=\mathscr L_\Phi$ and $F=F_\Phi$.

Every exact Lagrangian submanifold $(\mathscr L,F)$ is locally generated by phase
functions: that is, each point $\rho\in\mathscr L$ has a neighborhood generated
by some phase function; see~\cite[Proposition~5.1]{Guillemin-SternbergBook1}. The simplest case is when the projection
$\pi:\mathscr L\to M$ is a diffeomorphism onto its image, in which case $\mathscr L$
is given by
\begin{equation}
  \label{e:basic-lagrangian}
\mathscr L=\mathscr L_\Phi=\{(x,\partial_x\Phi(x))\mid x\in U\},\quad
U:=\pi(\mathscr L)\subset M,
\end{equation}
where the function $\Phi\in C^\infty(U;\mathbb R)$ is defined by
$F(x,\xi)=\Phi(x)$ for all $(x,\xi)\in\mathscr L$.

Another important case is when $\mathscr L\subset T^*M\setminus 0$ is conic. In this case
each point $\rho\in\mathscr L$ has a conic neighborhood which is generated
by some phase function $\Phi(x,\theta)$, $(x,\theta)\in U$,
where $U\subset M\times\mathbb R^m$ is conic and $\Phi$ is
homogeneous of degree~1 in the $\theta$ variables. For the proof see~\cite[Proposition~5.2]{Guillemin-SternbergBook1}, \cite[Theorem~21.2.16]{Hormander3}, or~\cite[Proposition~11.4]{Grigis-Sjostrand}.

%%%%%%%%%%%%%%%%%%%%%%%%%%%%%%%%%%%%%%%%%%%%%%%%%%%%%%%%%%%%%%%%%%%%%%%%%%%%%%%%
\subsubsection{Lagrangian distributions}
  \label{s:intro-lagr}
  
Let $(\mathscr L,F)$ be an exact Lagrangian submanifold of~$T^*M$. We use
the class $I^{\comp}_h(\mathscr L)$ of (compactly microlocalized semiclassical) \emph{Lagrangian distributions} associated
to $\mathscr L$. Elements of $I^{\comp}_h(\mathscr L)$ are $h$-dependent
families of functions in $\CIc(M)$, with support contained in some $h$-independent compact set.
We give a definition and some properties of the class $I^{\comp}_h(\mathscr L)$ below.

If $\mathscr L=\mathscr L_\Phi$ is generated by some phase function $\Phi(x,\theta)$,
$(x,\theta)\in U\subset M\times \mathbb R^m$,
in the sense of~\eqref{e:L-Phi}, then $I^{\comp}_h(\mathscr L)$
consists of distributions of the form
\begin{equation}
  \label{e:lag-dist}
u(x;h)=(2\pi h)^{-m/2}\int_{\mathbb R^m} e^{i\Phi(x,\theta)/h}a(x,\theta;h)\,d\theta
+\mathcal O(h^\infty)_{\CIc(M)}.
\end{equation}
Here the amplitude $a(x,\theta;h)\in \CIc(U)$ is a \emph{classical symbol}; that is,
$\supp a$ is contained in an $h$-independent compact subset of $U$
and we have the asymptotic expansion in~$\CIc(U)$
$$
a(x,\theta;h)\sim \sum_{j=0}^\infty h^ja_j(x,\theta)\quad\text{as}\quad h\to 0
$$
for some $a_0,a_1,\ldots\in\CIc(U)$.

In the special case when $\Phi$ has no oscillatory variables (i.e. $\mathscr L$ is given by~\eqref{e:basic-lagrangian}) the expression~\eqref{e:lag-dist} simplifies to
\begin{equation}
  \label{e:lag-dist-basic}
u(x;h)=e^{i\Phi(x)/h}a(x;h)+\mathcal O(h^\infty)_{\CIc(M)}.
\end{equation}
The class of functions defined by~\eqref{e:lag-dist} does not depend
on the choice of the phase function generating $\mathscr L$.
That is, if $\Phi,\Phi'$ are two phase functions with
$\mathscr L=\mathscr L_{\Phi}=\mathscr L_{\Phi'}$ and $u$ is given by~\eqref{e:lag-dist}
for the phase function $\Phi$ and some amplitude $a$, then
$u$ is also given by~\eqref{e:lag-dist} for the phase function $\Phi'$
and some other amplitude $a'$. The simplest case of this statement is when
$\Phi'$ has no oscillatory variables (that is, $\mathscr L$ is constructed
from~$\Phi'$ using~\eqref{e:basic-lagrangian}) as we can then write
(ignoring the $\mathcal O(h^\infty)$ remainder in~\eqref{e:lag-dist})
\begin{equation}
  \label{e:lag-repas}
a'(x;h)=e^{-i\Phi'(x)/h}u(x;h)=(2\pi h)^{-m/2}\int_{\mathbb R^m} e^{{i\over h}(\Phi(x,\theta)-\Phi'(x))}
a(x,\theta;h)\,d\theta
\end{equation}
and show that $a'$ is a classical symbol using the method of stationary phase.
The proof in the general case also uses stationary phase but is more involved,
see~\cite[\S8.1.2]{Guillemin-SternbergBook2};
for the nonsemiclassical case see~\cite[\S6.4]{Guillemin-SternbergBook1}, \cite[Proposition~25.1.5]{Hormander4}, or~\cite[Theorem~11.5]{Grigis-Sjostrand}.

For general Lagrangians $\mathscr L$ (not parametrized by a single phase function)
we define $I^{\comp}_h(\mathscr L)$ as consisting of sums
$u_1+\dots+u_k$ where $u_j\in I^{\comp}_h(\mathscr L_j)$ and
each $\mathscr L_j\subset\mathscr L$ is generated by some phase function.
Here are two important properties of Lagrangian distributions:
%%%%%%%%%%%%%%%%%%%%%%%%%%%%%%%%%%%%%%%%%%%%%%%%%%%%%%%%%%%%%%%%%%%%%%%%%%%%%%%%
\begin{enumerate}
\item If $u\in I^{\comp}_h(\mathscr L)$ and $A\in \Psi^k_h(M)$ is compactly supported (which means that its Schwartz kernel is compactly supported) then
$Au\in I^{\comp}_h(\mathscr L)$;
\item If $u\in I^{\comp}_h(\mathscr L)$ then
$\WFh(u)\subset \mathscr L$; that is, for any compactly supported $A\in\Psi^k_h(M)$
with $\WFh(A)\cap \mathscr L=\emptyset$ we have $Au=\mathcal O(h^\infty)_{\CIc(M)}$.
\end{enumerate}
%%%%%%%%%%%%%%%%%%%%%%%%%%%%%%%%%%%%%%%%%%%%%%%%%%%%%%%%%%%%%%%%%%%%%%%%%%%%%%%%
To show these, we first use a partition of unity to reduce to the case when $M=\mathbb R^n$
and $\mathscr L$ is generated by some phase function $\Phi$.
We next write for $b\in \CIc(T^*\mathbb R^n)$
and $u$ given by~\eqref{e:lag-dist}
$$
\Op_h(b) u(x)=(2\pi h)^{-{m\over 2}-n}\int_{\mathbb R^{2n+m}}
e^{{i\over h}(\langle x-y,\xi\rangle+\Phi(y,\theta))} b(x,\xi)a(y,\theta;h)\,dyd\xi d\theta.
$$
We now apply stationary phase in the $(y,\xi)$ variables to get an expression
of the form~\eqref{e:lag-dist} with the phase function $\Phi(x,\theta)$ and
some amplitude which is a classical symbol.
On the other hand, if $b$ is a symbol in $S^k_h(T^*\mathbb R^n)$ and
$\supp b\cap \mathscr L=\emptyset$ then the method of nonstationary
phase in the $(y,\xi,\theta)$ variables shows that $\Op_h(b)u(x)=\mathcal O(h^\infty)_{C^\infty}$.

%%%%%%%%%%%%%%%%%%%%%%%%%%%%%%%%%%%%%%%%%%%%%%%%%%%%%%%%%%%%%%%%%%%%%%%%%%%%%%%%
\subsubsection{Fourier integral operators}
  \label{s:prelim-fio-s}

We next discuss Fourier integral operators associated to symplectomorphisms.
Let $M_1,M_2$ be two manifolds of the same dimension~$n$, $U_j\subset T^*M_j$ be
two open sets, and $\varkappa:U_2\to U_1$
be a symplectomorphism. The flipped graph
\begin{equation}
  \label{e:L-varkappa}
\mathscr L_\varkappa:=\{(x_1,\xi_1,x_2,-\xi_2)\mid (x_2,\xi_2)\in U_2,\
\varkappa(x_2,\xi_2)=(x_1,\xi_1)\}\subset T^*(M_1\times M_2)
\end{equation}
is a Lagrangian submanifold. We further assume that $\varkappa$ is \emph{exact},
namely $\mathscr L_\varkappa$ is an exact Lagrangian submanifold.
As before, we fix an antiderivative
for $\mathscr L_\varkappa$ but suppress it in the notation. The exactness condition
holds in particular if $\varkappa$ is homogeneous, that is it sends
$\xi_2\cdot\partial_{\xi_2}$ to $\xi_1\cdot\partial_{\xi_1}$;
indeed, $\mathscr L_\varkappa$ is conic and we fix the antiderivative to be~0.

We say that an $h$-dependent family of operators $B=B(h):\mathcal D'(M_2)\to \CIc(M_1)$
is a (compactly microlocalized semiclassical) \emph{Fourier integral operator} associated to $\varkappa$,
and write $B\in I^{\comp}_h(\varkappa)$, if the corresponding integral kernel
$K_B(x_1,x_2;h)\in \CIc(M_1\times M_2)$ satisfies
$K_B\in h^{-n/2}I^{\comp}_h(\mathscr L_\varkappa)$. Here
$I^{\comp}_h(\mathscr L_\varkappa)$ is the class of Lagrangian distributions defined in~\S\ref{s:intro-lagr} above.
In particular, the wavefront set $\WF'_h(B)$ is contained in the graph of~$\varkappa$.

An important special case is when $M_2=\mathbb R^n$
and the projection $\pi:\mathscr L_\varkappa\to M_1\times\mathbb R^n$ onto
the $(x_1,\xi_2)$ variables is a diffeomorphism onto its image. If $F$ is the
antiderivative on $\mathscr L_\varkappa$, then we can write
\begin{equation}
  \label{e:ct-std-par}
\mathscr L_\varkappa=\{(x_1,\partial_{x_1}S(x_1,\xi_2),\partial_{\xi_2}S(x_1,\xi_2),-\xi_2)\mid
(x_1,\xi_2)\in U\}
\end{equation}
where $U:=\{(x_1,\xi_2)\mid (x_1,\xi_1,x_2,-\xi_2)\in\mathscr L_\varkappa\}$
and $S\in C^\infty(U;\mathbb R)$ is given by
$$
F(x_1,\xi_1,x_2,-\xi_2)=S(x_1,\xi_2)-\langle x_2,\xi_2\rangle,\quad
(x_1,\xi_1,x_2,-\xi_2)\in\mathscr L_\varkappa.
$$
That is, $\mathscr L_\varkappa$ is generated by the phase
function $\Phi(x_1,x_2,\theta)=S(x_1,\theta)-\langle x_2,\theta\rangle$
in the sense of~\eqref{e:L-Phi}. Then every operator $B\in I^{\comp}_h(\varkappa)$
has the following form modulo an $\mathcal O(h^\infty)_{\mathcal D'(\mathbb R^n)\to \CIc(M_1)}$
remainder:
\begin{equation}
  \label{e:fio-std-par}
Bf(x_1)=(2\pi h)^{-n}\int_{\mathbb R^{2n}} e^{{i\over h}(S(x_1,\theta)-\langle x_2,\theta\rangle)}
b(x_1,x_2,\theta;h)f(x_2)\,dx_2d\theta
\end{equation}
for some classical symbol $b\in \CIc(U_{(x_1,\theta)}\times\mathbb R^n_{x_2})$.

We list several fundamental properties of the class $I^{\comp}_h(\varkappa)$:
%%%%%%%%%%%%%%%%%%%%%%%%%%%%%%%%%%%%%%%%%%%%%%%%%%%%%%%%%%%%%%%%%%%%%%%%%%%%%%%%
\begin{enumerate}
\item If $B\in I^{\comp}_h(\varkappa)$, then $B:L^2(M_2)\to L^2(M_1)$
is bounded in norm uniformly in~$h$;
\item If $\varkappa$ is the identity map on $T^*M$, then
$B\in I^{\comp}_h(\varkappa)$ if and only if $B$ is a compactly supported pseudodifferential
operator in $\Psi^k_h(M)$ and $\WFh(B)\subset T^*M$ is compact;
\item If $B\in I^{\comp}_h(\varkappa)$ and $u\in I^{\comp}_h(\mathscr L)$
is a Lagrangian distribution, then $Bu$ is a Lagrangian distribution in~$I^{\comp}_h(\varkappa(\mathscr L))$;
\item If $B_1\in I^{\comp}_h(\varkappa_1)$, $B_2\in I^{\comp}_h(\varkappa_2)$,
then the composition $B_1B_2$ is a Fourier integral operator in~$I^{\comp}_h(\varkappa_1\circ\varkappa_2)$;
\item If $B\in I^{\comp}_h(\varkappa)$, then the adjoint $B^*$ lies in~$I^{\comp}_h(\varkappa^{-1})$.\end{enumerate}
%%%%%%%%%%%%%%%%%%%%%%%%%%%%%%%%%%%%%%%%%%%%%%%%%%%%%%%%%%%%%%%%%%%%%%%%%%%%%%%%
Here in property~(2) we let the antiderivative equal to~0 (as the identity map
is homogeneous). In property~(3) we define the antiderivative $F_{\varkappa(\mathscr L)}$ on $\varkappa(\mathscr L)$
using the antiderivatives $F_{\mathscr L},F_{\mathscr \varkappa}$
on $\mathscr L,\mathscr L_\varkappa$ by
\begin{equation}
  \label{e:anti-composition}
F_{\varkappa(\mathscr L)}(x_1,\xi_1)=F_\varkappa(x_1,\xi_1,x_2,-\xi_2)+F_{\mathscr L}(x_2,\xi_2)\quad\text{where}\quad
(x_1,\xi_1,x_2,-\xi_2)\in\mathscr L_\varkappa
\end{equation}
and in property~(4) the antiderivative on $\mathscr L_{\varkappa_1\circ\varkappa_2}$
is defined similarly. In property~(5) the antiderivative on $\mathscr L_{\varkappa^{-1}}$
is minus the antiderivative on $\mathscr L_\varkappa$.

We briefly explain how the above properties are proven:
%%%%%%%%%%%%%%%%%%%%%%%%%%%%%%%%%%%%%%%%%%%%%%%%%%%%%%%%%%%%%%%%%%%%%%%%%%%%%%%%
\begin{itemize}
\item For property~(2), we can use a partition of unity to reduce to the case $M=\mathbb R^n$.
The flipped graph of the identity map is given by~\eqref{e:ct-std-par}
with $S(x_1,\xi_2)=\langle x_1,\xi_2\rangle$.
The corresponding expression~\eqref{e:fio-std-par} gives the class of pseudodifferential
operators with compactly supported symbols (see~\cite[Theorem~4.20]{e-z}).
\item For property~(3), we reduce to the case when
$\mathscr L=\mathscr L_\Phi$ and $\mathscr L_\varkappa=\mathscr L_\Psi$
are generated by some phase functions $\Phi(x_2,\theta_2)$ and $\Psi(x_1,x_2,\theta_1)$,
where $\theta_j\in\mathbb R^{m_j}$. Using
the corresponding representations~\eqref{e:lag-dist} for $u$ and $B$
(with some amplitudes $a$ and $b$) we get
\begin{equation}
  \label{e:stator}
\begin{aligned}
Au(x_1)=\,\,&(2\pi h)^{-{n+m_1+m_2\over 2}}\int_{\mathbb R^{n+m_1+m_2}}
e^{{i\over h}(\Psi(x_1,x_2,\theta_1)+\Phi(x_2,\theta_2))}\times
\\&
\qquad a(x_2,\theta_2;h)b(x_1,x_2,\theta_1;h)\,d\theta_1d\theta_2 dx_2.
\end{aligned}
\end{equation}
This is an expression of the form~\eqref{e:lag-dist} for
the phase function $\Psi(x_1,x_2,\theta_1)+\Phi(x_2,\theta_2)$,
with $(\theta_1,\theta_2,x_2)$ treated as oscillatory variables,
and this phase function generates the Lagrangian $\varkappa(\mathscr L)$.
See also~\cite[Lemma~4.1]{NZ09}.
\item Property~(4) is proved similarly to property~(3), see~\cite[\S8.13]{Guillemin-SternbergBook2}.
Property~(5) is immediate by writing an expression of the form~\eqref{e:lag-dist}
for the integral kernel of the adjoint of $B$.
\item Finally, to show property~(1) we note that $B^*B$ is a semiclassical pseudodifferential
operator (and thus bounded on $L^2$) by properties~(2), (4), and~(5).
\end{itemize}
%%%%%%%%%%%%%%%%%%%%%%%%%%%%%%%%%%%%%%%%%%%%%%%%%%%%%%%%%%%%%%%%%%%%%%%%%%%%%%%%
We now discuss the conjugation by Fourier integral operators.
Assume that $\varkappa:U_2\to U_1$, $U_j\subset T^*M_j$, is an exact symplectomorphism
and
\begin{equation}
  \label{e:bassumer}
B\in I^{\comp}_h(\varkappa),\quad
B'\in I^{\comp}_h(\varkappa^{-1}).
\end{equation}
By properties~(2) and~(4) above we see that
$BB'\in \Psi^0_h(M_1)$, $B'B\in\Psi^0_h(M_2)$ are pseudodifferential
operators with wavefront sets compactly contained in $T^*M_j$.
Moreover, if $a\in S^{\comp}_\delta(T^*M_2)$, $\delta\in [0,{1\over 2})$
(see~\S\ref{s:mildly-exotic}), then
\begin{equation}
  \label{e:egorov-gen}
\begin{gathered}
B\Op_h(a)B'=\Op_h(\tilde a)+\mathcal O(h^\infty)_{\Psi^{-\infty}}\quad\text{for some}\quad
\tilde a\in S^{\comp}_\delta(T^*M_1),\\
\tilde a=(a\circ\varkappa^{-1})\sigma_h(BB')+\mathcal O(h^{1-2\delta})_{S^{\comp}_\delta},\quad
\supp \tilde a\subset \varkappa(\supp a).
\end{gathered}
\end{equation}
Indeed, we may reduce to the case $M_1=M_2=\mathbb R^n$.
By oscillatory testing~\cite[Theorem~4.19]{e-z} the symbol of $B\Op_h(a)B'$ as a pseudodifferential operator is given by
$$
\tilde a(x_1,\xi_1;h)=e^{-i\langle x_1,\xi_1\rangle/h}B\Op_h(a)B'(e^{i\langle\bullet,\xi_1\rangle/h}).
$$
Taking generating functions
$\Phi(x_1,x_2,\theta)$ of $\mathscr L_\varkappa$ and $-\Phi(x_1,x_2,\theta)$ of $\mathscr L_{\varkappa^{-1}}$ we write
\begin{equation}
  \label{e:lola}
\begin{aligned}
\tilde a(x_1,\xi_1;h)=\,\,&(2\pi h)^{-2n-m}\int_{\mathbb R^{4n+2m}}e^{{i\over h}(\langle x'_1-x_1,\xi_1\rangle+\langle x_2-x'_2,\xi_2\rangle
+\Phi(x_1,x_2,\theta)-\Phi(x'_1,x'_2,\theta'))}\\
&b(x_1,x_2,\theta;h)a(x_2,\xi_2;h)b'(x'_1,x'_2,\theta';h)
\,d\theta d\theta' dx'_1 dx_2dx'_2 d\xi_2 
\end{aligned}
\end{equation}
for some classical symbols $b(x_1,x_2,\theta;h)$, $b'(x'_1,x'_2,\theta';h)$.
Using the method of stationary phase we get that
$\tilde a$ is a symbol in $S^{\comp}_\delta(T^*\mathbb R^n)$. The principal term in the stationary phase expansion is equal
to $(a\circ\varkappa^{-1})\sigma_h(BB')$, as can be seen by formally putting $a\equiv 1$.
The support property (modulo $\mathcal O(h^\infty)$) follows immediately from the expansion,
finishing the proof of~\eqref{e:egorov-gen}.
See also~\cite[\S8.9.3]{Guillemin-SternbergBook2}.

If $V_j\subset U_j$, $j=1,2$, are compact sets with $\varkappa(V_2)=V_1$
and $B,B'$ are Fourier integral operators as in~\eqref{e:bassumer}, we say that
$B,B'$ \emph{quantize~$\varkappa$ near $V_1\times V_2$}
if
\begin{equation}
  \label{e:fio-quantize}
\begin{aligned}
BB'&=I+\mathcal O(h^\infty)\quad\text{microlocally near }V_1,\\
B'B&=I+\mathcal O(h^\infty)\quad\text{microlocally near }V_2.
\end{aligned}
\end{equation}
If $\mathscr L_\varkappa$ is generated by a single phase
function $\Phi$ (in the sense of~\eqref{e:L-Phi}) then there exist $B,B'$ quantizing $\varkappa$
near $V_1\times V_2$. To show this, we choose $B$ in the form~\eqref{e:lag-dist}:
$$
Bf(x_1)=(2\pi h)^{-{n+m\over 2}}\int_{\mathbb R^{n+m}}e^{i\Phi(x_1,x_2,\theta)/h}
b(x_1,x_2,\theta)f(x_2)\,d\theta dx_2
$$
where $b\in \CIc(U)$ is chosen so that
$b(x_1,x_2,\theta)\neq 0$ for any $(x_1,x_2,\theta)\in\mathcal C_\Phi$
such that $(x_1,\partial_{x_1}\Phi(x_1,x_2,\theta))\in V_1$
(or equivalently $(x_2,-\partial_{x_2}\Phi(x_1,x_2,\theta))\in V_2$)
and $U$ is the domain of $\Phi$. We have $\sigma_h(BB^*)\neq 0$ on~$V_1$
and $\sigma_h(B^*B)\neq 0$ on~$V_2$, as can be proved using
stationary phase similarly to~\eqref{e:lola}.
Multiplying $B^*$ on the right by an elliptic parametrix of~$BB^*$ and multiplying it on the left by an elliptic parametrix of~$B^*B$ (see for instance~\cite[Proposition~E.32]{dizzy}),
we obtain two operators $B',B''\in I^{\comp}_h(\varkappa)$ such that
$$
\begin{aligned}
BB'&=I+\mathcal O(h^\infty)\quad\text{microlocally near }V_1,\\
B''B&=I+\mathcal O(h^\infty)\quad\text{microlocally near }V_2.
\end{aligned}
$$
We write
$$
I-B'B=(I-B''B)(I-B'B)+B''(I-BB')B.
$$
The wavefront set of the right-hand side does not intersect~$V_2$.
For the first term this is immediate since $\WFh(I-B''B)\cap V_2=\emptyset$.
For the second term this follows from the fact that $\WFh(I-BB')\cap V_1=\emptyset$,
computing the full symbol of $B''(I-BB')B$ similarly to~\eqref{e:lola}.
It follows that $B'B=I+\mathcal O(h^\infty)$ microlocally near $V_2$,
therefore $B,B'$ satisfy~\eqref{e:fio-quantize}.

%%%%%%%%%%%%%%%%%%%%%%%%%%%%%%%%%%%%%%%%%%%%%%%%%%%%%%%%%%%%%%%%%%%%%%%%%%%%%%%%
\subsubsection{Fourier localization}
  \label{s:fourloc-lag}

We finally prove a fine Fourier localization statement for a class of Lagrangian
distributions, used in the proof of Lemma~\ref{l:loca+} below. Its proof
is contained in Appendix~\ref{s:fourloc-proof}.
%%%%%%%%%%%%%%%%%%%%%%%%%%%%%%%%%%%%%%%%%%%%%%%%%%%%%%%%%%%%%%%%%%%%%%%%%%%%%%%%
\begin{prop}
  \label{l:fourloc-lag}
Assume that $h,h'\in (0,1]$ satisfy $h'\geq h^\tau$ for some $\tau<1$,
$U\subset\mathbb R^n$ is an open set, $K\subset U$ is compact, 
and we have for some constant $C_0>0$
\begin{equation}
  \label{e:fourloc-lag-1}
\vol(K)\leq C_0,\quad
d(K,\mathbb R^n\setminus U)\geq C_0^{-1}.
\end{equation}
Let $\Phi\in C^\infty(U;\mathbb R)$, $a\in\CIc(U;\mathbb C)$,
$\supp a\subset K$, and assume that
\begin{equation}
  \label{e:fourloc-lag-2}
\diam\Omega_\Phi\leq C_0h'\quad\text{where}\quad
\Omega_\Phi:=\{d\Phi(x)\mid x\in U\}\subset\mathbb R^n.
\end{equation}
Assume also that $\Phi$ and $a$ satisfy, for all $N\geq 1$ and some
constants $C_N$:
\begin{equation}
  \label{e:fourloc-lag-3}
\max_{0<|\alpha|\leq N}\sup_U|\partial^\alpha \Phi|\leq C_N,\quad
\max_{0\leq|\alpha|\leq N}\sup_U|\partial^\alpha a|\leq C_N.
\end{equation}
Define the Lagrangian state
\begin{equation}
  \label{e:fourloc-u-def}
u(x):=a(x)\,e^{i\Phi(x)/h}\in\CIc(U)\subset\CIc(\mathbb R^n).
\end{equation}
Denote $\Omega_\Phi(C_0^{-1}h'):=\Omega_\Phi+B(0,C_0^{-1}h')$.
Then we have for all $N\geq 1$
\begin{equation}
  \label{e:ibp-lag}
\|\indic_{\mathbb R^n\setminus \Omega_\Phi(C_0^{-1}h')}(hD_x)u\|_{L^2(\mathbb R^n)}\leq C'_Nh^N
\end{equation}
where the constant $C'_N$ depends only on $\tau,n,N,C_0,C_{N'}$ for
 $N':=\lceil{2N+n\over 1-\tau}\rceil+1$.
\end{prop}
%%%%%%%%%%%%%%%%%%%%%%%%%%%%%%%%%%%%%%%%%%%%%%%%%%%%%%%%%%%%%%%%%%%%%%%%%%%%%%%%
\Remarks
1. If $\Phi$, $a$ are fixed and $h$ goes to zero, then
the set $\Omega_\Phi$ is the projection of the Lagrangian
$\mathscr L_\Phi$ defined in~\eqref{e:basic-lagrangian}
onto the $\xi$ variables and the function $u$ defined in~\eqref{e:fourloc-u-def}
is a Lagrangian distribution in $I^{\comp}_h(\mathscr L_\Phi)$.
However, the condition~\eqref{e:fourloc-lag-2} with $h'\sim h^\tau$,
$\tau>0$, implies that, if the phase
$\Phi(x)$ is not constant (which would correspond to a ``horizontal''
Lagrangian), then it necessarily depends on~$h$. We may still view $u(h)$ as a family
of Lagrangian states, but associated to $h$-dependent Lagrangians
$\mathscr L_\Phi(h)$ which become more and more horizontal as $h\to
0$. The proposition shows that these Lagrangian states are
microlocalized in boxes which are microscopic in the momentum variables.

\noindent 2. For $\tau<{1\over 2}$ one can prove Proposition~\ref{l:fourloc-lag}
without the assumption~\eqref{e:fourloc-lag-3}
using~\cite[Theorem~7.7.1]{Hormander1}. On the other hand, if
$\tau=1-\epsilon$, the boxes of momentum diameter $h^\tau$ are almost
Planckian (they almost saturate the uncertainty principle).

%%%%%%%%%%%%%%%%%%%%%%%%%%%%%%%%%%%%%%%%%%%%%%%%%%%%%%%%%%%%%%%%%%%%%%%%%%%%%%%%
\subsection{Fractal uncertainty principle}
  \label{s:fups}

The fractal uncertainty principle of Bourgain--Dyatlov~\cite[Theorem~4]{fullgap}
is the central tool of our proof. (See~\cite[\S4]{fwlEDP}
for an expository account.)
In this section we prove a slightly
more general version, Proposition~\ref{l:fup-2}, which will be used
in~\S\ref{s:fup-applied} below. 

We recall the definition of a porous set~\cite[Definition~5.1]{meassupp}:
%%%%%%%%%%%%%%%%%%%%%%%%%%%%%%%%%%%%%%%%%%%%%%%%%%%%%%%%%%%%%%%%%%%%%%%%%%%%%%%%
\begin{defi}
  \label{d:porous}
Let $\nu\in (0,1)$ and $0<\alpha_0\leq\alpha_1$. We say that a subset
$\Omega\subset\mathbb R$ is \textbf{$\nu$-porous on scales $\alpha_0$ to $\alpha_1$}
if for every interval $I\subset\mathbb R$ of size
$|I|\in [\alpha_0,\alpha_1]$ there exists a subinterval $J\subset I$
of size $|J|=\nu|I|$ such that $J\cap\Omega=\emptyset$.
\end{defi}
%%%%%%%%%%%%%%%%%%%%%%%%%%%%%%%%%%%%%%%%%%%%%%%%%%%%%%%%%%%%%%%%%%%%%%%%%%%%%%%%
Define the unitary semiclassical Fourier transform $\mathcal F_h:L^2(\mathbb R)\to L^2(\mathbb R)$ by
\begin{equation}
  \label{e:F-h-def}
\mathcal F_h f(\xi)=(2\pi h)^{-1/2}\int_{\mathbb R}e^{-ix\xi/h}f(x)\,dx.
\end{equation}
For a set $\Omega\subset\mathbb R$, let $\indic_\Omega:L^2(\mathbb R)\to L^2(\mathbb R)$
be the multiplication operator by the indicator function of $\Omega$.

We first prove the following fractal uncertainty principle,
which is a version of~\cite[Theorem~4]{fullgap} adapted to unbounded $\nu$-porous sets
using almost orthogonality and tools from~\cite{meassupp}:
%%%%%%%%%%%%%%%%%%%%%%%%%%%%%%%%%%%%%%%%%%%%%%%%%%%%%%%%%%%%%%%%%%%%%%%%%%%%%%%%
\begin{prop}
  \label{l:fup-1}
For each $\nu\in (0,1)$ there exist $\beta=\beta(\nu)>0$ and $C=C(\nu)>0$ such that the following estimate
holds
\begin{equation}
  \label{e:fup-1}
\|\indic_{\Omega_-}\mathcal F_h\indic_{\Omega_+}\|_{L^2(\mathbb R)\to L^2(\mathbb R)}\leq Ch^\beta
\end{equation}
for all $0<h\leq 1$ and all sets $\Omega_\pm\subset \mathbb R$ which are $\nu$-porous on scales $h$ to~$1$.
\end{prop}
%%%%%%%%%%%%%%%%%%%%%%%%%%%%%%%%%%%%%%%%%%%%%%%%%%%%%%%%%%%%%%%%%%%%%%%%%%%%%%%%
\Remark
An explicit expression for the exponent $\beta$ (for the smaller class
of $\delta$-regular sets; see Step~4 of the proof below for an explanation
of why this gives a result for all $\nu$-porous sets) was obtained by Jin--Zhang~\cite{JinZhang}.
Using~\cite[Theorem~1.2]{JinZhang}, one can get~\eqref{e:fup-1} with
\begin{equation}
\beta(\nu)=\exp(-\exp(\exp(K/\nu^3)))
\end{equation}
where $K$ is a global constant.
%%%%%%%%%%%%%%%%%%%%%%%%%%%%%%%%%%%%%%%%%%%%%%%%%%%%%%%%%%%%%%%%%%%%%%%%%%%%%%%%
\begin{proof}
1. We first replace the indicator functions in~\eqref{e:fup-1} by their smoothed
out versions $\chi_\pm\in C^\infty(\mathbb R;[0,1])$.
The functions $\chi_\pm$ satisfy for all $N$,
\begin{gather}
\label{e:fup-1-cutoff-1}
  \supp\chi_\pm\subset \Omega_\pm(h),\quad
  \supp(1-\chi_\pm)\cap \Omega_\pm=\emptyset,\\
\label{e:fup-1-cutoff-2}
  \sup|\partial_x^N\chi_\pm|\leq C_Nh^{-N}. 
\end{gather}
Here $\Omega_\pm(h)=\Omega_\pm+[-h,h]$ denotes the $h$-neighborhood of $\Omega_\pm$
and the constant $C_N$ depends only on $N$.
The functions $\chi_\pm$ are constructed by convolving the indicator
function of $\Omega_\pm(h/2)$ with a smooth cutoff which is rescaled
to be supported in $(-h/2,h/2)$. See the proof of~\cite[Lemma~3.3]{hgap} for details.

The left-hand side of~\eqref{e:fup-1} is equal to
$$
\|\indic_{\Omega_-}\chi_-\mathcal F_h\chi_+\indic_{\Omega_+}\|_{L^2(\mathbb R)\to L^2(\mathbb R)}
\leq \|\chi_-\mathcal F_h\chi_+\|_{L^2(\mathbb R)\to L^2(\mathbb R)}.
$$

\noindent 2.
We next write the cutoffs $\chi_\pm$ as sums of functions $\chi^\pm_j$,
each supported in an interval of size~2. 
More precisely, fix $\chi\in \CIc(\mathbb R;[0,1])$ such that
$\supp\chi\subset (-1,1)$ and
$$
1=\sum_{j\in\mathbb Z}\chi_j\quad\text{where}\quad
\chi_j(x):=\chi(x-j).
$$
Put
\begin{equation}
  \label{e:fup-1-chi-j-pm}
\chi_j^\pm:=\chi_j\chi_\pm,\quad
\supp\chi_j^\pm\subset \Omega_\pm(h)\cap (j-1,j+1).
\end{equation}
Note that $\chi_j^\pm$ satisfy the derivative bounds~\eqref{e:fup-1-cutoff-2}
for some constants $C_N$ depending only on $N$. We have
(with convergence in strong operator topology)
$$
\chi_-\mathcal F_h\chi_+=\sum_{j,k\in\mathbb Z}A_{jk}\quad\text{where}\quad
A_{jk}:=\chi_j^-\mathcal F_h\chi_k^+.
$$
Therefore it suffices to show the estimate
\begin{equation}
  \label{e:fup-1-bd}
\Big\|\sum_{j,k}A_{jk}\Big\|_{L^2(\mathbb R)\to L^2(\mathbb R)}\leq Ch^\beta.
\end{equation}

\noindent 3.
To show~\eqref{e:fup-1-bd} we use almost orthogonality.
More precisely it suffices to prove the following bounds for all~$j,k,j',k',N$:
\begin{align}
  \label{e:fup-1-bd-1}
\|A_{jk}\|_{L^2(\mathbb R)\to L^2(\mathbb R)}&\leq Ch^{2\beta},\\
  \label{e:fup-1-bd-2}
\|A_{jk}A_{j'k'}^*\|_{L^2(\mathbb R)\to L^2(\mathbb R)}&\leq C_Nh^{-1}(1+|j-j'|+|k-k'|)^{-N},\\
  \label{e:fup-1-bd-3}
\|A_{j'k'}^*A_{jk}\|_{L^2(\mathbb R)\to L^2(\mathbb R)}&\leq C_Nh^{-1}(1+|j-j'|+|k-k'|)^{-N}.
\end{align}
for some $\beta>0,C>0$ depending only on $\nu$ and
some $C_N$ depending only on $N$. Indeed, these estimates imply
\begin{align}
  \label{e:fup-1-bdd-1}
\sup_{j,k}\sum_{j',k'}\|A_{jk}A_{j'k'}^*\|_{L^2(\mathbb R)\to L^2(\mathbb R)}^{1/2}&\leq Ch^\beta,\\
  \label{e:fup-1-bdd-2}
\sup_{j,k}\sum_{j',k'}\|A_{j'k'}^*A_{jk}\|_{L^2(\mathbb R)\to L^2(\mathbb R)}^{1/2}&\leq Ch^\beta.
\end{align}
Here we use~\eqref{e:fup-1-bd-1} for $|j-j'|+|k-k'|\leq h^{-\beta/2}$
and~\eqref{e:fup-1-bd-2}, \eqref{e:fup-1-bd-3} with $N:=\lceil 8+2/\beta\rceil$ 
for $|j-j'|+|k-k'|> h^{-\beta/2}$.
Now~\eqref{e:fup-1-bdd-1} and~\eqref{e:fup-1-bdd-2} imply~\eqref{e:fup-1-bd} by
the Cotlar--Stein Theorem~\cite[Theorem~C.5]{e-z}.

\noindent 4.
We first prove~\eqref{e:fup-1-bd-1} which
will follow from the fractal uncertainty principle~\cite[Theorem~4]{fullgap}.
However~\cite{fullgap} used a more restrictive class
of \emph{$\delta$-regular sets} rather than $\nu$-porous sets.
We recall from~\cite[Definition~1.1]{fullgap} that a nonempty closed set $\Omega\subset\mathbb R$
is called \emph{$\delta$-regular with constant $C_R$ on scales~0 to~1} if there exists
a Borel measure~$\mu$ supported on~$\Omega$ such that for each interval $I$ of size $|I|\in (0,1]$
we have the upper bound $\mu(I)\leq C_R|I|^\delta$, and if additionally $I$ is centered at a point in $\Omega$,
then we have the lower bound $\mu(I)\geq C_R^{-1}|I|^\delta$.

To address the difference between porous and regular sets we
argue similarly to the proof of~\cite[Proposition~5.5]{meassupp}.
Since $\Omega_\pm$ are $\nu$-porous on scales $h$ to 1,
by~\cite[Lemma~5.4]{meassupp} there exist sets $\widetilde\Omega_\pm\subset\mathbb R$
such that:
\begin{enumerate}
\item $\Omega_\pm\subset \widetilde\Omega_\pm(h)$;
\item $\widetilde\Omega_\pm\subset \mathbb R$ are $\delta$-regular
with constant $C_R$ on scales 0 to 1, for some $\delta\in (0,1)$ and $C_R\geq 1$
which depend only on $\nu$.
\end{enumerate}
Denote $\Omega_\pm^j:=\widetilde\Omega_\pm-j$; note that these sets are still $\delta$-regular
with constant $C_R$ on scales 0 to~1.
By~\eqref{e:fup-1-chi-j-pm}
and since the norm $\|\indic_X\mathcal F_h^*\indic_Y\|_{L^2(\mathbb R)\to L^2(\mathbb R)}$
does not change when shifting $X$ and/or $Y$, we have
\begin{equation}
  \label{e:fup-1-bd-1.1}
\begin{aligned}
\|A_{jk}\|_{L^2\to L^2}
&\leq \|\indic_{\Omega_+(h)\cap [k-1,k+1]}\mathcal F_h^*\indic_{\Omega_-(h)\cap [j-1,j+1]}\|_{L^2(\mathbb R)\to L^2(\mathbb R)}
\\&\leq \|\indic_{\Omega_+^k(2h)\cap [-1,1]}\mathcal F_h^*\indic_{\Omega_-^j(2h)\cap [-1,1]}\|_{L^2(\mathbb R)\to L^2(\mathbb R)}.
\end{aligned}
\end{equation}
By~\cite[Proposition~4.1]{fullgap}
(which is a corollary of~\cite[Theorem~4]{fullgap}) the right-hand side of~\eqref{e:fup-1-bd-1.1} is bounded by $Ch^{2\beta}$
for some $C,\beta>0$ depending only on~$\delta,C_R$ (which in turn only depend on $\nu$),
giving~\eqref{e:fup-1-bd-1}. Note that~\cite{fullgap} used a slightly different
normalization of $\mathcal F_h$, rescaled by a factor of $2\pi$,
which however makes no difference in the proof. (Alternatively
one can use the more general~\cite[Proposition~4.3]{fullgap}
with $\Phi(x,y):=xy$.) Similarly the fact that~\eqref{e:fup-1-bd-1.1} features
$\Omega_\pm^j(2h)$ instead of $\Omega_\pm^j(h)$ does not make a difference:
for instance we can write
$\Omega_\pm^j(2h)= (\Omega_\pm^j(h)+h)\cup (\Omega_\pm^j(h)-h)$
and use the triangle inequality.

\noindent 5.
It remains to show~\eqref{e:fup-1-bd-2} and~\eqref{e:fup-1-bd-3}. We only show
the former one, the latter proved similarly. We have
$$
A_{jk}A_{j'k'}^*=\chi_j^-\mathcal F_h\chi_k^+\chi_{k'}^+\mathcal F_h^*\chi_{j'}^-.
$$
If $|k-k'|>1$ then $\supp\chi_k^+\cap \supp\chi_{k'}^+=\emptyset$ and thus
$A_{jk}A_{j'k'}^*=0$. We henceforth assume that $|k-k'|\leq 1$. The integral kernel
of $A_{jk}A_{j'k'}^*$, which we denote $\mathcal K$, can be computed in terms of the Fourier
transform of $\chi_k^+\chi_{k'}^+$:
$$
\mathcal K(x,y)=(2\pi h)^{-1}\chi^-_j(x)\chi^-_{j'}(y)\int_{\mathbb R}e^{i(y-x)\xi/h}\chi_k^+(\xi)
\chi_{k'}^+(\xi)\,d\xi.
$$
We may assume that $|j-j'|>2$, then $|x-y|\geq {1\over 10}|j-j'|$ on $\supp\mathcal K$.
The function $\chi_k^+\chi_{k'}^+$ is supported inside an interval of size~2 and
satisfies the derivative bounds~\eqref{e:fup-1-cutoff-2}. Integrating by parts $N$ times in $\xi$,
we get
$$
\sup_{x,y}|\mathcal K(x,y)|\leq C_Nh^{-1}|j-j'|^{-N}.
$$
Since $\mathcal K(x,y)$ is supported in a square of size 2, this implies~\eqref{e:fup-1-bd-2}.
\end{proof}
%%%%%%%%%%%%%%%%%%%%%%%%%%%%%%%%%%%%%%%%%%%%%%%%%%%%%%%%%%%%%%%%%%%%%%%%%%%%%%%%
We now give a version of Proposition~\ref{l:fup-1} with relaxed assumptions
regarding the scales on which $\Omega_\pm$ are porous:
%%%%%%%%%%%%%%%%%%%%%%%%%%%%%%%%%%%%%%%%%%%%%%%%%%%%%%%%%%%%%%%%%%%%%%%%%%%%%%%%
\begin{prop}
  \label{l:fup-2}
Fix numbers $\gamma_j^\pm$, $j=0,1$, such that
$$
0\leq\gamma_1^\pm<\gamma_0^\pm\leq 1,\quad
\gamma_1^++\gamma_1^-<1<\gamma_0^++\gamma_0^-
$$
and define
\begin{equation}
  \label{e:fup-2-gamma}
\gamma:=\min(\gamma_0^+,1-\gamma_1^-)-\max(\gamma_1^+,1-\gamma_0^-)
=\big|[\gamma_1^+,\gamma_0^+]\cap [1-\gamma_0^-,1-\gamma_1^-]\big|>0.
\end{equation}
Then for each $\nu>0$ there exists $\beta=\beta(\nu)>0$ and
$C=C(\nu)>0$ such that the
estimate
\begin{equation}
  \label{e:fup-2}
\|\indic_{\Omega_-}\mathcal F_h\indic_{\Omega_+}\|_{L^2(\mathbb R)\to L^2(\mathbb R)}\leq Ch^{\gamma\beta}
\end{equation}
holds for all $0<h<1$ and all $\Omega_\pm\subset \mathbb R$ which are $\nu$-porous on scales $h^{\gamma_0^\pm}$
to $h^{\gamma_1^\pm}$.
\end{prop}
%%%%%%%%%%%%%%%%%%%%%%%%%%%%%%%%%%%%%%%%%%%%%%%%%%%%%%%%%%%%%%%%%%%%%%%%%%%
\Remark
The formula~\eqref{e:fup-2-gamma} is related to the fact that the proof
of the fractal uncertainty principle~\cite[Theorem~4]{fullgap} proceeds by
induction on scale and uses the structure of $\Omega_-$ on
scale $h^\mu$ together with the structure of $\Omega_+$ on the
dual scale $h^{1-\mu}$. In fact, it is likely that the proof in~\cite{fullgap}
can be adapted to yield Proposition~\ref{l:fup-2} directly.
%%%%%%%%%%%%%%%%%%%%%%%%%%%%%%%%%%%%%%%%%%%%%%%%%%%%%%%%%%%%%%%%%%%%%%%%%%%%%%%%
\begin{proof}
Define
$$
\gamma_0:=\min(\gamma_0^+,1-\gamma_1^-),\quad
\gamma_1:=\max(\gamma_1^+,1-\gamma_0^-),
$$
note that $\gamma_0-\gamma_1=\gamma>0$. The set $\Omega_+$ is $\nu$-porous
on scales $h^{\gamma_0}$ to $h^{\gamma_1}$,
and the set $\Omega_-$ is $\nu$-porous on scales $h^{1-\gamma_1}$ to $h^{1-\gamma_0}$.

Put
$$
\widehat\Omega_+ := h^{-\gamma_1}\Omega_+,\quad
\widehat\Omega_- := h^{\gamma_0-1}\Omega_-,\quad
\tilde h:=h^\gamma.
$$
Then the sets $\widehat\Omega_\pm$ are $\nu$-porous on scales $\tilde h$ to~1.
Consider the unitary rescaling operators
$$
T_\pm:L^2(\mathbb R)\to L^2(\mathbb R),\quad
T_+f(x)=h^{\gamma_1/2}f(h^{\gamma_1}x),\quad
T_-f(x)=h^{(1-\gamma_0)/2}f(h^{1-\gamma_0}x).
$$
We have
$$
T_\pm \indic_{\Omega_\pm}=\indic_{\widehat\Omega_\pm}T_\pm,\quad
T_-\mathcal F_h T_+^{-1}=\mathcal F_{\tilde h}.
$$
Therefore the left-hand side of~\eqref{e:fup-2} is equal to 
\begin{equation}
  \label{e:fup-2-int}
\|T_-\indic_{\Omega_-}\mathcal F_h\indic_{\Omega_+}T_+^{-1}\|_{L^2(\mathbb R)\to L^2(\mathbb R)}
=\|\indic_{\widehat\Omega_-}\mathcal F_{\tilde h}\indic_{\widehat\Omega_+}\|_{L^2(\mathbb R)\to L^2(\mathbb R)}.
\end{equation}
The right-hand side of~\eqref{e:fup-2-int} is bounded by $C\tilde h^\beta=Ch^{\gamma\beta}$ by Proposition~\ref{l:fup-1}.
\end{proof}
%%%%%%%%%%%%%%%%%%%%%%%%%%%%%%%%%%%%%%%%%%%%%%%%%%%%%%%%%%%%%%%%%%%%%%%%%%%%%%%%
We conclude this section with two simple lemmas used in~\S\S\ref{s:porosity-ultimate}--\ref{s:fup-applied} below:
%%%%%%%%%%%%%%%%%%%%%%%%%%%%%%%%%%%%%%%%%%%%%%%%%%%%%%%%%%%%%%%%%%%%%%%%%%%%%%%%
\begin{lemm}
  \label{l:porous-nbhd}
Let $\nu\in (0,1)$, $0<\alpha_0\leq \alpha_1$, and $0<\alpha_2\leq {\nu\over 3}\alpha_1$.
Assume that $\Omega\subset\mathbb R$ is $\nu$-porous on scales $\alpha_0$
to $\alpha_1$. Then the neighborhood
$\Omega(\alpha_2)=\Omega+[-\alpha_2,\alpha_2]$ is ${\nu\over 3}$-porous
on scales $\max(\alpha_0,{3\over\nu}\alpha_2)$ to $\alpha_1$.
\end{lemm}
%%%%%%%%%%%%%%%%%%%%%%%%%%%%%%%%%%%%%%%%%%%%%%%%%%%%%%%%%%%%%%%%%%%%%%%%%%%%%%%%
\begin{proof}
Take an interval $I\subset\mathbb R$ such that $\max(\alpha_0,{3\over\nu}\alpha_2)\leq|I|\leq \alpha_1$.
Since $\Omega$ is $\nu$-porous on scales $\alpha_0$ to $\alpha_1$,
there exists a subinterval $J\subset I$
with $|J|=\nu|I|\geq 3\alpha_2$ and $J\cap\Omega=\emptyset$.
Let $J'\subset J$ be the subinterval with the same center
and $|J'|={1\over 3}|J|={\nu\over 3}|I|$,
then $J'(\alpha_2)\subset J$ and thus
 $J'\cap \Omega(\alpha_2)=\emptyset$.
\end{proof}
%%%%%%%%%%%%%%%%%%%%%%%%%%%%%%%%%%%%%%%%%%%%%%%%%%%%%%%%%%%%%%%%%%%%%%%%%%%%%%%%
%
%%%%%%%%%%%%%%%%%%%%%%%%%%%%%%%%%%%%%%%%%%%%%%%%%%%%%%%%%%%%%%%%%%%%%%%%%%%%%%%%
\begin{lemm}
  \label{l:porous-map}
Let
$\psi:\mathbb R\to \mathbb R$ be a $C^2$ diffeomorphism such that for some $C_1\geq 1$
\begin{equation}
  \label{e:psi-derby}
\max(\sup|\psi'|,\sup|\psi'|^{-1},\sup|\psi''|)\leq C_1.
\end{equation}
Let also $\nu\in (0,1)$, $0<\alpha_0\leq \alpha_1$, and
$\alpha_0\leq \min(C_1^{-2}\alpha_1,{1\over 2}C_1^{-4})$.
Assume that $\Omega\subset \mathbb R$ is $\nu$-porous on scales $\alpha_0$
to~$\alpha_1$.
Then the image $\psi(\Omega)$
is $\nu\over 2$-porous on scales $C_1\alpha_0$ to~$\min(C_1^{-1}\alpha_1,{1\over 2}C_1^{-3})$.
\end{lemm}
%%%%%%%%%%%%%%%%%%%%%%%%%%%%%%%%%%%%%%%%%%%%%%%%%%%%%%%%%%%%%%%%%%%%%%%%%%%%%%%%
\begin{proof}
We have
$$
\sup\big|\partial_x \log|\psi'(x)|\big|
=\sup\Big|{\psi''\over\psi'}\Big|\leq C_1^2.
$$
Therefore for each interval $I'\subset\mathbb R$ we have
\begin{equation}
  \label{e:kolya-2}
\sup_{I'}|\psi'|\leq e^{C_1^2|I'|}\inf_{I'}|\psi'|.
\end{equation}
Let $I\subset\mathbb R$ be an interval such that
$|I|\leq {1\over 2}C_1^{-3}$.
Put $I':=\psi^{-1}(I)$, then $|I'|\leq {1\over 2}C_1^{-2}$,
thus by~\eqref{e:kolya-2}
\begin{equation}
  \label{e:kolya-3}
{|\psi(J')|\over |J'|}\geq {|I|\over 2|I'|}
\quad\text{for all intervals}\quad J'\subset I'.
\end{equation}
Now assume additionally that $C_1\alpha_0\leq |I|\leq C_1^{-1}\alpha_1$.
Then $\alpha_0\leq |I'|\leq \alpha_1$,
thus by porosity of $\Omega$ there exists an interval
$$
J'\subset I',\quad
|J'|=\nu|I'|,\quad
J'\cap\Omega=\emptyset.
$$
Put $J:=\psi(J')\subset I$, then $J\cap \psi(\Omega)=\emptyset$
and we estimate by~\eqref{e:kolya-3}
$$
|J|\geq {|I|\cdot |J'|\over 2|I'|}={\nu\over 2}|I|.\qedhere
$$
\end{proof}
%%%%%%%%%%%%%%%%%%%%%%%%%%%%%%%%%%%%%%%%%%%%%%%%%%%%%%%%%%%%%%%%%%%%%%%%%%%%%%%%

%%%%%%%%%%%%%%%%%%%%%%%%%%%%%%%%%%%%%%%%%%%%%%%%%%%%%%%%%%%%%%%%%%%%%%%%%%%%%%%%
\subsection{Dynamics and porosity}
\label{s:porosity-basic}

In this section we use the results of~\S\ref{s:hyperbolics}
to establish porosity of certain sets in the stable/unstable direction
(Lemma~\ref{l:porosity-basic}).
This property is used in~\S\ref{s:fup-endgame} below in combination with the fractal uncertainty principle.

Recall from~\S\ref{s:stun}
that for each $\rho\in S^*M$ the local stable/unstable manifolds
$W_s(\rho),W_u(\rho)$ are $C^\infty$ submanifolds of $S^*M$ tangent to $E_s$, $E_u$
(despite the fact that $E_s(\rho),E_u(\rho)$ do not in general
depend smoothly on~$\rho$, see~\S\ref{s:stun-regularity}).
We define the \emph{global stable/unstable manifolds}
$$
\widehat W_s(\rho):=\bigcup_{j\geq 0}\varphi_{-j}\big(W_s(\varphi_{j}(\rho))\big),\quad
\widehat W_u(\rho):=\bigcup_{j\geq 0}\varphi_j\big(W_u(\varphi_{-j}(\rho))\big)
$$
which are immersed one-dimensional $C^\infty$ submanifolds of $S^*M$
tangent to $E_s(\rho),E_u(\rho)$,
see for instance~\cite[(17.4.1)]{KaHa} and~\cite[\S4.7.3]{stunnote}.

We fix a Riemannian metric on $S^*M$. A proper parametrization of pieces of global stable/unstable manifolds yields \emph{stable/unstable intervals} as defined below:
%%%%%%%%%%%%%%%%%%%%%%%%%%%%%%%%%%%%%%%%%%%%%%%%%%%%%%%%%%%%%%%%%%%%%%%%%%%%%%%%
\begin{defi}
  \label{d:stun-interval}
Let $L>0$. An \textbf{unstable interval} of length~$L$
is a $C^\infty$ map
$\gamma:I\to S^*M$, where $I\subset\mathbb R$ is an interval of size $L$,
such that for each $s\in I$ the tangent vector $\dot\gamma(s)\in T_{\gamma(s)}S^*M$
is a unit length vector in $E_u(\gamma(s))$.
A \textbf{stable interval} of length~$L$ is defined similarly except
we require $\dot\gamma(s)\in E_s(\gamma(s))$.
In both cases we denote $|\gamma|:=L$.
\end{defi}
%%%%%%%%%%%%%%%%%%%%%%%%%%%%%%%%%%%%%%%%%%%%%%%%%%%%%%%%%%%%%%%%%%%%%%%%%%%%%%%%
We sometimes identify a stable/unstable interval $\gamma$ with its range
$\gamma(I)\subset S^*M$. For a set $\mathcal W\subset S^*M$ denote
\begin{equation}
  \label{e:interval-pullback}
\gamma^{-1}(\mathcal W):=\{s\in I\mid \gamma(s)\in\mathcal W\}.
\end{equation}
If $\gamma:I\to S^*M$ is an unstable interval and $t\in\mathbb R$,
then the map $\varphi_t\circ\gamma:I\to S^*M$ can be reparametrized to yield
another unstable interval, which we denote by
$\varphi_t(\gamma)$. Same is true for stable intervals.

Recalling the definitions~\eqref{e:stun-J-def} of stable/unstable Jacobians
$J^s_t,J^u_t$, we see that there exists
a constant $C$ depending only on $(M,g)$
and the choice of the metric on $S^*M$ such that
for each unstable interval $\gamma$ and all $t\in\mathbb R$
\begin{equation}
  \label{e:unin-exp}
C^{-1}\big(\inf_{\gamma}J^u_t\big)|\gamma|\leq|\varphi_t(\gamma)|\leq C\big(\sup_{\gamma}J^u_t\big)|\gamma|.
\end{equation}
Similarly if $\gamma$ is a stable interval then
\begin{equation}
  \label{e:stin-exp}
C^{-1}\big(\inf_{\gamma}J^s_t\big)|\gamma|\leq|\varphi_t(\gamma)|\leq C\big(\sup_{\gamma}J^s_t\big)|\gamma|.
\end{equation}
In particular by~\eqref{e:Lambda-0-1} we have
\begin{equation}
  \label{e:stun-interval-expand}
|\varphi_t(\gamma)|\leq Ce^{-\Lambda_0 |t|}|\gamma|
\end{equation}
for all $t\geq 0$ and stable intervals $\gamma$,
and for all $t\leq 0$ and unstable intervals $\gamma$.
Therefore each stable/unstable interval is contained in
some global stable/unstable manifold.

Since $M$ is connected and
$\varphi_t$ is not a constant time suspension of an Anosov diffeomorphism
(being a contact flow),
each global stable/unstable manifold $\widehat W_s(\rho),\widehat W_u(\rho)$
is dense in $S^*M$, see~\cite[p.29, Theorem~15]{Anosov}.
A quantitative version of this statement is given by
%%%%%%%%%%%%%%%%%%%%%%%%%%%%%%%%%%%%%%%%%%%%%%%%%%%%%%%%%%%%%%%%%%%%%%%%%%%%%%%%
\begin{lemm}
  \label{l:unier}
Let $\mathcal U\subset S^*M$ be a nonempty open set.
Then there exists $L_{\mathcal U}>0$ such that
every unstable interval of length $L_{\mathcal U}$ intersects~$\mathcal U$.
Same is true for stable intervals.
\end{lemm}
%%%%%%%%%%%%%%%%%%%%%%%%%%%%%%%%%%%%%%%%%%%%%%%%%%%%%%%%%%%%%%%%%%%%%%%%%%%%%%%%
\begin{proof}
We argue by contradiction, considering the case of unstable intervals;
the case of stable intervals is handled similarly.
If the statement of the lemma fails, then there exists a sequence
of unstable intervals
$$
\gamma_j:[-\ell_j,\ell_j]\to S^*M,\quad
\ell_j\to \infty,\quad
\gamma_j([-\ell_j,\ell_j])\cap \mathcal U=\emptyset.
$$
Passing to a subsequence, we may assume that $(\gamma_j(0),\dot\gamma_j(0))$
converges to some point $(\rho,\xi)\in T(S^*M)$.
Take the unstable interval $\gamma:\mathbb R\to S^*M$
such that $(\gamma(0),\dot\gamma(0))=(\rho,\xi)$.
Then $\gamma(\mathbb R)$ is the global
unstable manifold $\widehat W_u(\rho)$.
We have $\gamma_j(s)\to\gamma(s)$ locally uniformly
in $s\in\mathbb R$. Therefore
$\widehat W_u(\rho)\cap\mathcal U=\emptyset$,
giving a contradiction with the fact that $\widehat W_u(\rho)$ is dense in $S^*M$.
\end{proof}
%%%%%%%%%%%%%%%%%%%%%%%%%%%%%%%%%%%%%%%%%%%%%%%%%%%%%%%%%%%%%%%%%%%%%%%%%%%%%%%%
To state the main result of this section, Lemma~\ref{l:porosity-basic},
we introduce some notation formally similar to the symbolic formalism in dynamical systems
and motivated by~\S\ref{s:proofs-notation} below (see also Remark~2 following Proposition~\ref{l:longdec-0}).
We fix finitely many open conic sets
\begin{equation}
  \label{e:porosity-basic-sets}
\mathcal V_1,\dots,\mathcal V_m\subset T^*M\setminus 0
\end{equation}
and assume that $S^*M\setminus\mathcal V_k$ has nonempty interior for each $k$.
In our application in Lemma~\ref{l:porosity-intersected} we will take $m=2$ and use a slight fattening of the sets $\mathcal V_1,\mathcal V_\star$
constructed in~\S\ref{s:partition} below. The set $\mathcal V_1$ will be assumed to be ``small'', as a consequence $V_\star$ will necessarily be ``large''.

%%%%%%%%%%%%%%%%%%%%%%%%%%%%%%%%%%%%%%%%%%%%%%%%%%%%%%%%%%%%%%%%%%%%%%%%%%%%%%%%
\begin{figure}
\includegraphics[height=7cm]{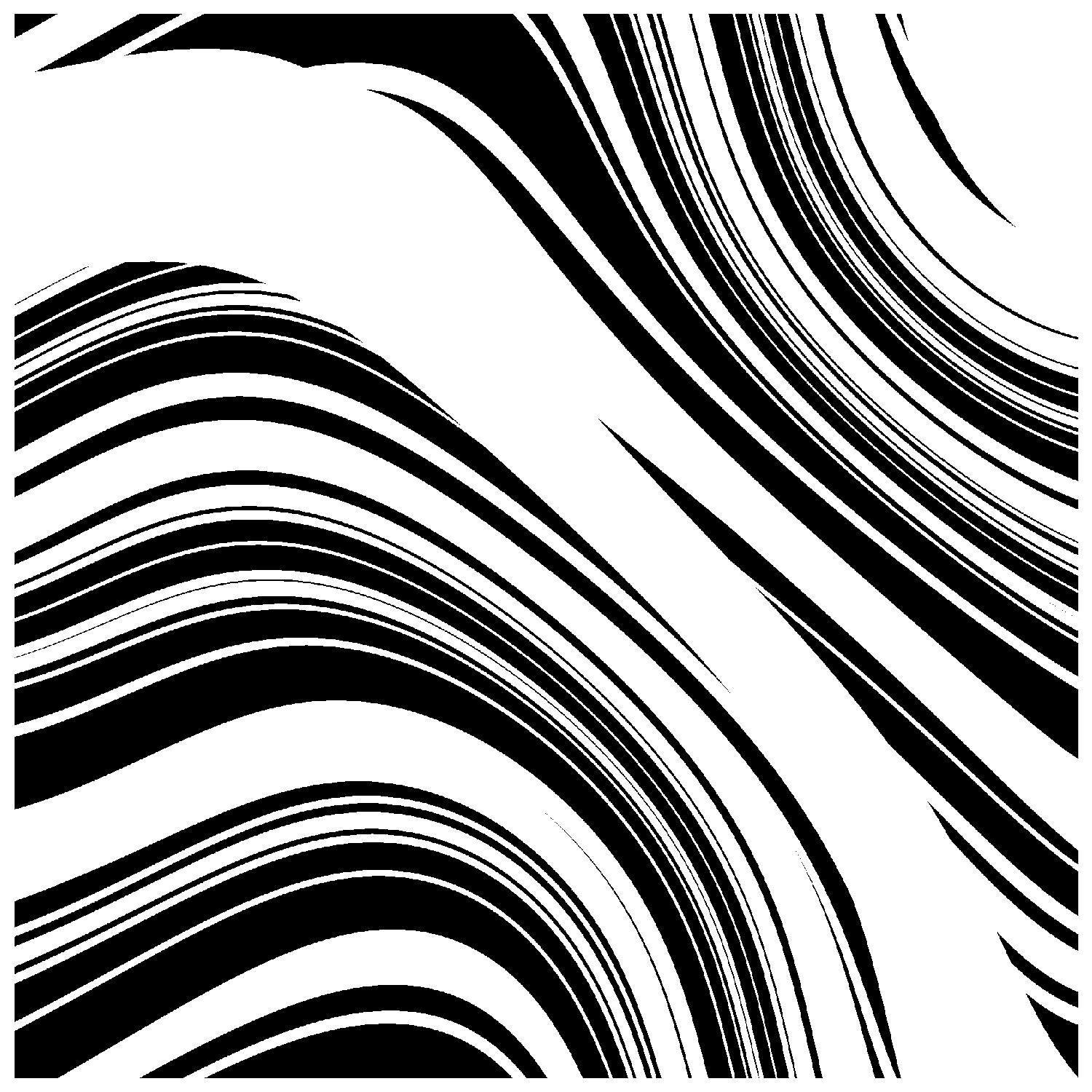}
\qquad
\includegraphics[height=7cm]{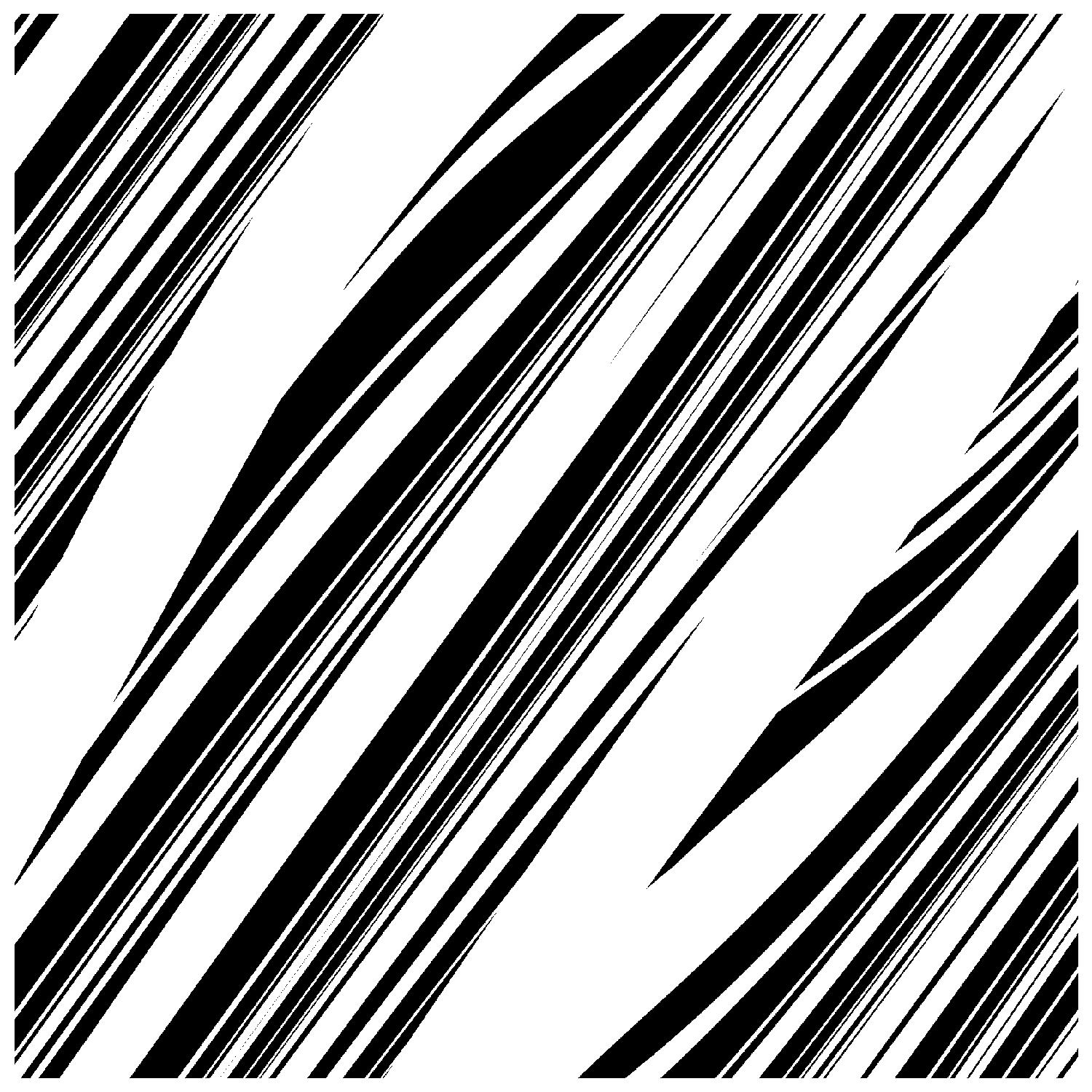}
\hbox to\hsize{\hss $\mathcal V^-_{\mathbf v}$ \hss\hss $\mathcal V^+_{\mathbf w}$\hss}
\caption{The sets $\mathcal V^-_{\mathbf v}$
and $\mathcal V^+_{\mathbf w}$ with the flow direction removed.
In this figure and in Figures~\ref{f:demo} and~\ref{f:moderate}
we use numerical simulations for a perturbed two-dimensional cat map
(which has similar properties to three-dimensional Anosov flows
studied here).}
\label{f:porosity-bas}
\end{figure}
%%%%%%%%%%%%%%%%%%%%%%%%%%%%%%%%%%%%%%%%%%%%%%%%%%%%%%%%%%%%%%%%%%%%%%%%%%%%%%%%

For words
$\mathbf v=v_0\dots v_{n-1}$, $\mathbf w=w_1\dots w_{n}$
where $v_j,w_j\in \{1,\dots,m\}$, define the open conic sets
(similarly to~\eqref{e:V+-} below)
\begin{equation}
  \label{e:V-def-porous}
\mathcal V^-_{\mathbf v}:=\bigcap_{j=0}^{n-1}\varphi_{-j}(\mathcal V_{v_j}),\quad
\mathcal V^+_{\mathbf w}:=\bigcap_{j=1}^n \varphi_j(\mathcal V_{w_j}).
\end{equation}
The following lemma shows the porosity of $\mathcal V^-_{\mathbf v}$
in the unstable direction and of $\mathcal V^+_{\mathbf w}$
in the stable direction, in the sense of Definition~\ref{d:porous}.
See Figure~\ref{f:porosity-bas}.
%%%%%%%%%%%%%%%%%%%%%%%%%%%%%%%%%%%%%%%%%%%%%%%%%%%%%%%%%%%%%%%%%%%%%%%%%%%%%%%%
\begin{lemm}
  \label{l:porosity-basic}
There exist $\nu>0$, $C_0>0$ depending only on $\mathcal V_1,\dots,\mathcal V_m$
such that
\begin{itemize}
\item for all words $\mathbf v=v_0\dots v_{n-1}$,
sets $\mathcal W^-\subset \mathcal V^-_{\mathbf v}\cap S^*M$,
and unstable intervals $\gamma:I_0\to S^*M$,
the set
$\gamma^{-1}(\mathcal W^-)$
is $\nu$-porous on scales $C_0(\inf_{\mathcal W^-} J^u_n)^{-1}$
to~1;
\item for all words $\mathbf w=w_1\dots w_n$,
sets $\mathcal W^+\subset\mathcal V^+_{\mathbf w}\cap S^*M$,
and stable intervals $\gamma:I_0\to S^*M$,
the set
$\gamma^{-1}(\mathcal W^+)
$ is $\nu$-porous on scales $C_0(\inf_{\mathcal W^+} J^s_{-n})^{-1}$ to~1.
\end{itemize}
Here the sets $\gamma^{-1}(\mathcal W^\pm)\subset I_0\subset\mathbb R$
are defined by~\eqref{e:interval-pullback}.
\end{lemm}
%%%%%%%%%%%%%%%%%%%%%%%%%%%%%%%%%%%%%%%%%%%%%%%%%%%%%%%%%%%%%%%%%%%%%%%%%%%%%%%%
\Remarks 1. In the situation where all the $\mathcal V_j$ are ``small conic balls'', the sets $\mathcal V^-_{\mathbf v}\cap S^*M$ have the shapes of ``deformed ellipses'' aligned along a small piece of weak stable manifold. Their width transversely to this manifold is bounded by $C_0 J^u_n(\rho)^{-1}$, for $\rho$ any point in $\mathcal V^-_{\mathbf v}\cap S^*M$, so $\gamma^{-1}(\mathcal V^-_{\mathbf v})$ will be contained in an interval of length $\leq C_0 J^u_n(\rho)^{-1}$. The Lemma shows that, in the general case where some $\mathcal V_j$ may be ``not small'', $\mathcal V^-_{\mathbf v}\cap S^*M$ may be a union of many such ``deformed ellipses'', arranged in a fractal (that is, porous) way along the unstable direction.

\noindent 2. By~\eqref{e:Lambda-0-1} we see in particular that
if $\gamma$ is an unstable interval, then
$\gamma^{-1}(\mathcal V^-_{\mathbf v})$ is $\nu$-porous on scales
$C_0e^{-\Lambda_0n}$ to~1. If $\gamma$ is instead a stable
interval, then $\gamma^{-1}(\mathcal V^+_{\mathbf w})$ is
$\nu$-porous on scales $C_0e^{-\Lambda_0n}$ to~1.

%%%%%%%%%%%%%%%%%%%%%%%%%%%%%%%%%%%%%%%%%%%%%%%%%%%%%%%%%%%%%%%%%%%%%%%%%%%%%%%%
\begin{proof}
1. We consider the case of unstable intervals, with stable intervals handled similarly.
Our proof is similar to~\cite[Lemma~5.10]{meassupp}. Throughout the proof
$C$ denotes constants depending only on $\mathcal V_1,\dots,\mathcal V_m$ whose precise value
might change from place to place.

Fix nonempty open sets $\mathcal U_1,\dots,\mathcal U_m\subset S^*M$
such that $\overline{\mathcal U_k}\cap\overline{\mathcal V_k}=\emptyset$;
this is possible since $S^*M\setminus\mathcal V_k$ have nonempty interior.
Fix $\varepsilon>0$ smaller than the distance between $\mathcal U_k$
and $\mathcal V_k$ for all~$k$.
Using Lemma~\ref{l:unier}, we fix
$L_0>0$ depending only on $\mathcal V_1,\dots,\mathcal V_m$ such that every unstable interval of length~$L_0$ intersects each
of the sets $\mathcal U_1,\dots,\mathcal U_m$.

\noindent 2. We fix $C_0>0$ large enough to be chosen later in Step~4 of the proof.
Take an arbitrary unstable interval $\gamma:I_0\to S^*M$ and
extend it to an unstable interval $\gamma:\mathbb R\to S^*M$.
Let $I\subset\mathbb R$ be an interval such that
$C_0(\inf_{\mathcal W^-} J^u_n)^{-1}\leq |I|\leq 1 $ and
$\gamma_I:=\gamma|_I$ be the corresponding unstable interval,
note that $|\gamma_I|=|I|$.
We may assume that $\gamma_I\cap \mathcal W^-\neq\emptyset$
as otherwise $\gamma^{-1}(\mathcal W^-)\cap I=\emptyset$
and we could take any $J\subset I$ in Definition~\ref{d:porous}.

Let $\varphi_j(\gamma_I)$, $j\geq 0$, be the images of $\gamma_I$
under $\varphi_j$.
By~\eqref{e:stun-interval-expand}
we have $|\varphi_j(\gamma_I)|\geq C^{-1}e^{\Lambda_0 j}|I|$.
Therefore there exists an integer $\ell\geq 0$ such that
$|\varphi_\ell(\gamma_I)|\geq L_0$. Take the minimal integer $\ell\geq 0$ with this property,
then there exists $C>L_0$ such that
\begin{equation}
  \label{e:pipa2}
L_0\leq |\varphi_\ell(\gamma_I)|\leq C.
\end{equation}

\noindent 3. The map $\varphi_\ell$ has a uniform expansion rate on $\gamma_I$, namely
\begin{equation}
  \label{e:steady}
\sup_{\gamma_I} J^u_\ell\leq C\inf_{\gamma_I}J^u_\ell.
\end{equation}
Indeed, by~\eqref{e:stun-interval-expand} and~\eqref{e:pipa2} there exists $t_0>0$
depending only on the constants in~\eqref{e:pipa2} (which in turn depend only on $\mathcal V_1,\dots,\mathcal V_m$) such that
$\varphi_{\ell-t_0}(\gamma_I)=\varphi_{-t_0}(\varphi_\ell(\gamma_I))$ is contained in a local unstable manifold,
more precisely
\begin{equation}
  \label{e:steady2}
\varphi_{\ell-t_0}(\tilde\rho)\in W_u(\varphi_{\ell-t_0}(\rho))\quad\text{for all}\quad
\rho,\tilde\rho\in \gamma_I.
\end{equation}
If $\ell\leq t_0$ then~\eqref{e:steady} is immediate since
$C^{-1}\leq J^u_\ell\leq C$. Assume now that $\ell>t_0$. Then we write
for all $\rho\in\gamma_I$
$$
J^u_\ell(\rho)={J^u_{t_0}(\varphi_{\ell-t_0}(\rho))
\over J^u_{t_0-\ell}(\varphi_{\ell-t_0}(\rho))}.
$$
By part~(4) of Lemma~\ref{l:stun-main}
and~\eqref{e:steady2} we have $J^u_{t_0-\ell}(\varphi_{\ell-t_0}(\rho))\leq CJ^u_{t_0-\ell}(\varphi_{\ell-t_0}(\tilde\rho))$
for all $\rho,\tilde\rho\in \gamma_I$.
Together with the bound
$C^{-1}\leq J^u_{t_0}\leq C$ this proves~\eqref{e:steady}.

%%%%%%%%%%%%%%%%%%%%%%%%%%%%%%%%%%%%%%%%%%%%%%%%%%%%%%%%%%%%%%%%%%%%%%%%%%%%%%%%
\begin{figure}
\includegraphics{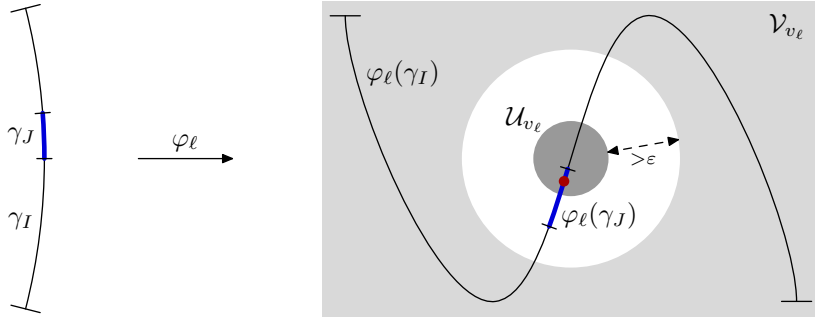}
\caption{An illustration of the proof of Lemma~\ref{l:porosity-basic}. The large lighter shaded region is $\mathcal V_{v_\ell}$
and the small darker shaded region is $\mathcal U_{v_\ell}$. The marked
point inside $\mathcal U_{v_\ell}\cap\varphi_\ell(\gamma_J)$ is $\varphi_\ell(\gamma(s))$.}
\label{f:porobas}
\end{figure}
%%%%%%%%%%%%%%%%%%%%%%%%%%%%%%%%%%%%%%%%%%%%%%%%%%%%%%%%%%%%%%%%%%%%%%%%%%%%%%%%

\noindent 4. By~\eqref{e:unin-exp}, \eqref{e:pipa2},
and~\eqref{e:steady} we relate the expansion rate of $\varphi_\ell$
on $\gamma_I$ to the length~$|I|$:
\begin{equation}
  \label{e:steady3}
C_1^{-1}\leq |I|\cdot \inf_{\gamma_I} J^u_\ell\leq |I|\cdot\sup_{\gamma_I}J^u_\ell\leq C_1
\end{equation}
where $C_1$ is some constant depending only on~$\mathcal V_1,\dots,\mathcal V_m$. Fix $C_0:=C_1+1$,
then the integer $\ell$ satisfies
$$
0\leq \ell \leq n-1,
$$
where we recall that $n=|\mathbf v|$.
Indeed, assume that $\ell\geq n$ instead.
Then $J^u_\ell(\rho)\geq J^u_n(\rho)$ for all $\rho$ by~\eqref{e:Lambda-0-1}.
Since $\gamma_I\cap \mathcal W^-\neq\emptyset$ and from our initial assumption on $|I|$, we have
\begin{equation}
  \label{e:steady4}
C_0\leq |I|\cdot \inf_{\mathcal W^-} J^u_n\leq
|I|\cdot\inf_{\mathcal W^-}J^u_\ell\leq
|I|\cdot\sup_{\gamma_I}J^u_\ell
\leq C_1
\end{equation}
giving a contradiction with our choice of~$C_0$.

\noindent 5. We finally construct an interval
$J\subset I$ such that $J\cap \gamma^{-1}(\mathcal W^-)=\emptyset$.
By~\eqref{e:pipa2} and the choice of $L_0$, the unstable interval
$\varphi_\ell(\gamma_I)$ intersects $\mathcal U_{v_\ell}$.
That is, there exists $s\in I$ such that $\varphi_\ell(\gamma(s))\in\mathcal U_{v_\ell}$.
Choose an interval $J\subset I$ such that
$s\in J$ and $|\varphi_\ell(\gamma_J)| =\varepsilon$
where $\gamma_J:=\gamma|_J$ is the corresponding unstable interval.
Since the distance between $\mathcal U_{v_\ell}$
and $\mathcal V_{v_\ell}$ is larger than $\varepsilon$,
the unstable interval $\varphi_\ell(\gamma_J)$ does not intersect~$\mathcal V_{v_\ell}$.
(See Figure~\ref{f:porobas}.)
By~\eqref{e:V-def-porous}, the unstable interval
$\gamma_J$ does not intersect $\mathcal V^-_{\mathbf v}\supset\mathcal W^-$,
so that $J\cap \gamma^{-1}(\mathcal W^-)=\emptyset$ as needed.

By~\eqref{e:unin-exp} and~\eqref{e:steady3} we obtain a lower bound on the size of $J$:
$$
|J|\geq{|\varphi_\ell(\gamma_J)|\over C\sup_{\gamma_I}J^u_\ell}
\geq {\varepsilon\over C^2}|I|.
$$
Thus $\gamma^{-1}(\mathcal W^-)$ is $\nu$-porous
on scales $C_0(\inf_{\mathcal W^-} J^u_n)^{-1}$ to~1 with
$\nu:=\varepsilon/C^2>0$.
\end{proof}
%%%%%%%%%%%%%%%%%%%%%%%%%%%%%%%%%%%%%%%%%%%%%%%%%%%%%%%%%%%%%%%%%%%%%%%%%%%%%%%%
We finally discuss the dependence of the constant $\nu$ on the sets
$\mathcal V_1,\dots,\mathcal V_m$ in Lemma~\ref{l:porosity-basic}, used in
Theorem~\ref{t:eig-quant}. We use the following
%%%%%%%%%%%%%%%%%%%%%%%%%%%%%%%%%%%%%%%%%%%%%%%%%%%%%%%%%%%%%%%%%%%%%%%%%%%%%%%%
\begin{defi}
  \label{d:stun-dense}
Let $\mathcal U\subset S^*M$ be a set and $0<L_1\leq 1\leq L_0$. We say
that $\mathcal U$ is \textbf{$(L_0,L_1)$-dense in the unstable direction} if
for each unstable interval
$\gamma:I\to S^*M$ of length $L_0$ there exists a subinterval $J\subset I$
of length $L_1$ such that $\gamma(J)\subset\mathcal U^\circ$, where $\mathcal U^\circ$
denotes the interior of $\mathcal U$.
We similarly define the notion of being dense in the stable direction.
\end{defi}
%%%%%%%%%%%%%%%%%%%%%%%%%%%%%%%%%%%%%%%%%%%%%%%%%%%%%%%%%%%%%%%%%%%%%%%%%%%%%%%%
Lemma~\ref{l:unier} implies (similarly to step~5 in the proof of Lemma~\ref{l:porosity-basic}) that if $\mathcal U$ has nonempty interior then
it is $(L_0,L_1)$-dense in both stable and unstable directions for some $L_0,L_1$. Following the proof of Lemma~\ref{l:porosity-basic}
(using density in the stable/unstable directions in step~5), we obtain
%%%%%%%%%%%%%%%%%%%%%%%%%%%%%%%%%%%%%%%%%%%%%%%%%%%%%%%%%%%%%%%%%%%%%%%%%%%%%%%%
\begin{lemm}
  \label{l:porosity-basic-quant}
In the notation of Lemma~\ref{l:porosity-basic}, assume that each of the complements
$S^*M\setminus \mathcal V_1,\dots,S^*M\setminus\mathcal V_m$ is $(L_0,L_1)$-dense in the unstable direction.
Then for all words $\mathbf v=v_0\dots v_{n-1}$, sets $\mathcal W^-\subset \mathcal V^-_{\mathbf v}\cap S^*M$,
and unstable intervals $\gamma:I_0\to S^*M$, the set $\gamma^{-1}(\mathcal W^-)$ is $\nu$-porous on scales $C_0(\inf_{\mathcal W^-}J^u_n)^{-1}$
to~1, where $\nu,C_0>0$ depend only on $(M,g),L_0,L_1$.
A similar statement holds for stable intervals under the assumption of $(L_0,L_1)$-density in the stable direction.
\end{lemm}
%%%%%%%%%%%%%%%%%%%%%%%%%%%%%%%%%%%%%%%%%%%%%%%%%%%%%%%%%%%%%%%%%%%%%%%%%%%%%%%%
We also record here a useful property of $(L_0,L_1)$-dense sets:
%%%%%%%%%%%%%%%%%%%%%%%%%%%%%%%%%%%%%%%%%%%%%%%%%%%%%%%%%%%%%%%%%%%%%%%%%%%%%%%%
\begin{lemm}
  \label{l:dense-useful}
Assume that $\mathcal U\subset S^*M$ is $(L_0,L_1)$-dense in the unstable direction.
Then there exists $\mathcal U^\sharp\subset S^*M$ which is $(L_0,L_1)$-dense in the unstable
direction and such that the closure of $\mathcal U^\sharp$ is contained in the interior of $\mathcal U$. The same is true for $(L_0,L_1)$-dense sets in the stable direction.
\end{lemm}
%%%%%%%%%%%%%%%%%%%%%%%%%%%%%%%%%%%%%%%%%%%%%%%%%%%%%%%%%%%%%%%%%%%%%%%%%%%%%%%%
\begin{proof}
Without loss of generality we assume that $\mathcal U$ is open. We exhaust $\mathcal U$
by open subsets
$$
\mathcal U=\bigcup_{j\geq 0} \mathcal U_j,\quad
\mathcal U_j\subset \mathcal U_{j+1},\quad
\overline{\mathcal U_j}\subset\mathcal U.
$$
For instance, we may take $\mathcal U_j$ to be the set of all points $\rho\in S^*M$
such that the closed ball $\overline B(\rho,{1\over j})$ is contained in~$\mathcal U$.

We argue by contradiction, assuming that neither of the sets $\mathcal U_j$ is $(L_0,L_1)$-dense
in the unstable direction. Then there exists a sequence of unstable intervals
$\gamma_j:[0,L_0]\to S^*M$ such that for each $j$ and each subinterval $J\subset [0,L_0]$
of length~$L_1$, we have $\gamma_j(J)\not\subset \mathcal U_j$. 
Passing to a subsequence,
we may assume that $\gamma_j$ converges uniformly to some unstable interval $\gamma:[0,L_0]\to S^*M$. 
Since $\mathcal U$ is $(L_0,L_1)$-dense in the unstable direction, there exists
a subinterval $J\subset I$ of length~$L_1$ such that
$\gamma(J)\subset\mathcal U$. 
Then for $j$ large enough, $\gamma_j(J)\subset \mathcal U_j$,
giving a contradiction.
\end{proof}
%%%%%%%%%%%%%%%%%%%%%%%%%%%%%%%%%%%%%%%%%%%%%%%%%%%%%%%%%%%%%%%%%%%%%%%%%%%%%%%%

%%%%%%%%%%%%%%%%%%%%%%%%%%%%%%%%%%%%%%%%%%%%%%%%%%%%%%%%%%%%%%%%%%%%%%%%%%%%%%%%
%%%%%%%%%%%%%%%%%%%%%%%%%%%%%%%%%%%%%%%%%%%%%%%%%%%%%%%%%%%%%%%%%%%%%%%%%%%%%%%%
\section{Proofs of the theorems}
  \label{s:proofs-of-theorems}

In this section we prove Theorems~\ref{t:eig} and~\ref{t:dwe}. We follow
the strategy used in~\cite{meassupp,JinDWE} in the case of constant curvature
(which in turn was partially inspired by~\cite{AnAnn}).
The main difference
is the proof of the key fractal uncertainty estimate (Proposition~\ref{l:longdec-0}).

In~\S\S\ref{s:proofs-notation}--\ref{s:key-estimate} we provide notation and statements used
in the proofs of both theorems. The proof of Theorem~\ref{t:eig} is presented
in~\S\ref{s:proofs-1}. In~\S\ref{s:proofs-2} we prove Theorem~\ref{t:dwe},
using some parts of~\S\ref{s:proofs-1} as well.

%%%%%%%%%%%%%%%%%%%%%%%%%%%%%%%%%%%%%%%%%%%%%%%%%%%%%%%%%%%%%%%%%%%%%%%%%%%%%%%%
\subsection{Notation}
  \label{s:proofs-notation}

We first introduce some notation used throughout the rest of the paper.
Let $M$ be a compact connected Anosov surface, see~\S\ref{s:hyperbolics}.
Fix a Riemannian metric on $S^*M$
inducing a distance function $d(\bullet,\bullet)$.
We assume that:
%%%%%%%%%%%%%%%%%%%%%%%%%%%%%%%%%%%%%%%%%%%%%%%%%%%%%%%%%%%%%%%%%%%%%%%%%%%%%%%%
\begin{enumerate}
\item we are given $h$-independent functions
$a_1,a_\star\in \CIc(T^*M\setminus 0)$ with%
\footnote{The choice of $1,\star$ for indices will become clear later
in~\S\ref{s:refined-partition} where we write $a_\star=a_2+a_3+\dots$}
$$
\supp a_1,\supp a_\star\subset \{\textstyle{1\over 4}<|\xi|_g<4\},\quad
a_1,a_\star\geq 0,\quad
a_1+a_\star\leq 1;
$$
\item $\supp a_1\subset \mathcal V_1$, $\supp a_\star\subset \mathcal V_\star$
where $\mathcal V_1,\mathcal V_\star\subset T^*M\setminus 0$ are some conic open sets;
\item the complements $T^*M\setminus \mathcal V_1,T^*M\setminus \mathcal V_\star$
have nonempty interiors;
\item the diameter of $\mathcal V_1\cap S^*M$ with respect to $d(\bullet,\bullet)$ is smaller than some
 constant $\varepsilon_0>0$ to be fixed later; as a consequence, $\mathcal V_\star\cap S^*M$ will cover a large part of $S^*M$
\item we are given $A_1,A_\star\in\Psi^{-\infty}_h(M)$
with $\sigma_h(A_w)=a_w$, $\WFh(A_w)\subset \mathcal V_w\cap \{{1\over 4}<|\xi|_g<4\}$, $w\in \{1,\star\}$.
\end{enumerate}
%%%%%%%%%%%%%%%%%%%%%%%%%%%%%%%%%%%%%%%%%%%%%%%%%%%%%%%%%%%%%%%%%%%%%%%%%%%%%%%%
The specific functions $a_1,a_\star$ used in the proof of Theorem~\ref{t:eig}
are fixed in~\S\ref{s:partition} below.
Roughly speaking, $a_1,a_\star$ will form a partition of unity on $S^*M$,
$a_1$ will be supported on the region $\{a\neq 0\}$,
where $a$ is the symbol featured in Theorem~\ref{t:eig}, and
$a_\star$ will be supported near the complement of this region.
The proof of Theorem~\ref{t:dwe}
uses a damped version of these functions, see~\S\ref{s:damped-partition}.
The fact that the complements $T^*M\setminus\mathcal V_1,T^*M\setminus\mathcal V_\star$
have nonempty interiors is used in~\S\ref{s:porosity-ultimate}.

We next introduce dynamically refined symbols corresponding to words,
using the geodesic flow $\varphi_t$ defined in~\eqref{e:phi-def}.
Define
$$
\mathscr A_\star:=\{1,\star\},\quad
\mathscr A_\star^\bullet:=\{\mathbf w=w_0\dots w_{n-1}\mid n\geq 0,\ w_0,\dots,w_{n-1}\in\mathscr A_\star\}.
$$
We call elements of $\mathscr A_\star^\bullet$ \emph{words}.
Denote by $\mathscr A_\star^n\subset \mathscr A_\star^\bullet$ the set of words of length $n$.
We write $|\mathbf v|:=n$ for $v\in\mathscr A_\star^n$.

For each word $\bv=v_0\dots v_{n-1}$, resp. $\mathbf w=w_1\dots w_{n}$, define the functions
\begin{equation}
  \label{e:a-pm-def}
a^-_{\bv}:=\prod_{j=0}^{n-1}(a_{v_j}\circ\varphi_j),\quad
a^+_{\bw}:=\prod_{j=1}^{n}(a_{w_j}\circ\varphi_{-j}).
\end{equation}
Note the different indexing for $\mathbf v$ and $\mathbf w$ which makes sure
that the product $a^-_{\mathbf v}a^+_{\mathbf w}$ has only one factor of the form
$a_w\circ\varphi_0$, $w\in \{1,\star\}$.
The supports of $a^-_{\mathbf v},a^+_{\mathbf w}$ are contained in
the open conic sets
\begin{equation}
\label{e:V+-}
\mathcal V^-_{\bv}:=\bigcap_{j=0}^{n-1}\varphi_{-j}(\mathcal V_{v_j}),\quad
\mathcal V^+_{\bw}:=\bigcap_{j=1}^{n}\varphi_{j}(\mathcal V_{w_j}).
\end{equation}
The operators corresponding to $a^-_{\mathbf v},a^+_{\mathbf w}$ are
defined using the notation
$A(t):=U(-t)AU(t)$ 
from~\eqref{e:A-t-def}:
\begin{equation}
  \label{e:A-pm-def}
\begin{aligned}
A^-_{\bv}&:=A_{v_{n-1}}(n-1)A_{v_{n-2}}(n-2)\cdots A_{v_1}(1)A_{v_0}(0),\\
A^+_{\bw}&:=A_{w_1}(-1)A_{w_2}(-2)\cdots A_{w_{n-1}}(-(n-1))A_{w_{n}}(-n).
\end{aligned}
\end{equation}
If $n$ is bounded independently of~$h$ then Egorov's Theorem~\eqref{e:egorov-basic} implies
\begin{equation}
  \label{e:cq-basic}
A^-_{\mathbf v}=\Op_h(a^-_{\mathbf v})+\mathcal O(h)_{L^2\to L^2},\quad
A^+_{\mathbf w}=\Op_h(a^+_{\mathbf w})+\mathcal O(h)_{L^2\to L^2}.
\end{equation}
This is a form of \emph{classical/quantum correspondence}.

For future use we record the following \emph{concatenation formulas}:
if $\mathbf v=v_1\dots v_{k}$, $\mathbf w=w_1\dots w_\ell$, then
\begin{equation}
  \label{e:word-concat}
A^+_{\mathbf v\mathbf w}=U(k)A^-_{\overline{\mathbf v}}A^+_{\mathbf w}U(-k),\quad
A^-_{\mathbf v\mathbf w}=U(-k)A^-_{\mathbf w}A^+_{\overline{\mathbf v}}U(k)
\end{equation}
where the reverse word $\overline{\mathbf v}$ is defined
by $\overline{\mathbf v}:= v_{k}\dots v_1$. Similarly we have
\begin{align}
  \label{e:set-concat}
\mathcal V^+_{\mathbf v\mathbf w}=\varphi_k(\mathcal V^-_{\overline{\mathbf v}}\cap \mathcal V^+_{\mathbf w}),&\quad
\mathcal V^-_{\mathbf v\mathbf w}=\varphi_{-k}(\mathcal V^-_{\mathbf w}\cap \mathcal V^+_{\overline{\mathbf v}}),
\\
  \label{e:symbol-concat}
a^+_{\mathbf v\mathbf w}=(a^-_{\overline{\mathbf v}}a^+_{\mathbf w})\circ\varphi_{-k},&
\quad
a^-_{\mathbf v\mathbf w}=(a^-_{\mathbf w}a^+_{\overline{\mathbf v}})\circ\varphi_k.
\end{align}
In the particular case $\mathbf w=\emptyset$ we get the \emph{reversal formulas}
\begin{equation}
  \label{e:reversing}
A^+_{\mathbf v}=U(k)A^-_{\overline{\mathbf v}}U(-k),\quad
\mathcal V^+_{\mathbf v}=\varphi_k(\mathcal V^-_{\overline{\mathbf v}}),\quad
a^+_{\mathbf v}=a^-_{\overline{\mathbf v}}\circ\varphi_{-k}.
\end{equation}
If $\mathcal E\subset\mathscr A_\star^\bullet$ is a finite set, then we define
\begin{equation}
  \label{e:A-E-def}
a^\pm_{\mathcal E}:=\sum_{\mathbf w\in\mathcal E} a^\pm_{\mathbf w},\quad
A^\pm_{\mathcal E}:=\sum_{\mathbf w\in\mathcal E} A^\pm_{\mathbf w},
\end{equation}
and if $F:\mathscr A_\star^\bullet\to\mathbb C$ is zero except at finitely many words, then we put
\begin{equation}
  \label{e:A-F-def}
a^\pm_F:=\sum_{\mathbf w\in\mathscr A_\star^\bullet} F(\mathbf w)a^\pm_{\mathbf w},\quad
A^\pm_F:=\sum_{\mathbf w\in\mathscr A_\star^\bullet} F(\mathbf w)A^\pm_{\mathbf w}.
\end{equation}
Note that if $\mathcal E\subset\mathscr A_\star^n$ for some~$n$, then
$0\leq a^\pm_{\mathcal E}\leq 1$.

In the remainder of~\S\ref{s:proofs-of-theorems}
we will only use the operators $A_{\mathbf w}^-$. (This is an arbitrary choice~--
one could instead only use the operators $A_{\mathbf w}^+$.) To simplify notation,
we denote
$$
a_{\mathbf w}:=a_{\mathbf w}^-,\quad
A_{\mathbf w}:=A_{\mathbf w}^-,
$$
and same for $a_{\mathcal E},A_{\mathcal E},a_F,A_F$.

%%%%%%%%%%%%%%%%%%%%%%%%%%%%%%%%%%%%%%%%%%%%%%%%%%%%%%%%%%%%%%%%%%%%%%%%%%%%%%%%
\subsection{Long propagation times and the key estimate}
\label{s:key-estimate}

Similarly to~\cite{meassupp,JinDWE} our argument uses words of length that grows like $\log(1/h)$.
More precisely, we define the following integer propagation times:
\begin{equation}
  \label{e:prop-times}
N_0:=\Big\lceil{\log(1/h)\over 6\Lambda_1}\Big\rceil,\quad
N:=(6\Lambda+1)N_0 > {\log(1/h)\over\Lambda_0}
\end{equation}
where the `minimal/maximal expansion rates' $0<\Lambda_0\leq\Lambda_1$ were defined in~\eqref{e:Lambda-0-1} and $\Lambda:=\lceil \Lambda_1/\Lambda_0\rceil$.
We call $N_0$ a \emph{short logarithmic time} and $N$ a \emph{long logarithmic time}.
Note that if $(M,g)$ had constant curvature~$-1$ as in~\cite{meassupp} then we could take
$\Lambda_0=\Lambda_1=1$ and $N\approx{7\over 6}\log(1/h)$.

%%%%%%%%%%%%%%%%%%%%%%%%%%%%%%%%%%%%%%%%%%%%%%%%%%%%%%%%%%%%%%%%%%%%%%%%%%%%%%%%
\subsubsection{Short logarithmic words}
\label{s:short-words}

We first study words of length $N_0$, for which a version 
of the classical/quantum correspondence~\eqref{e:cq-basic} still applies.
We use the mildly exotic symbol classes introduced in~\S\ref{s:mildly-exotic}.
%%%%%%%%%%%%%%%%%%%%%%%%%%%%%%%%%%%%%%%%%%%%%%%%%%%%%%%%%%%%%%%%%%%%%%%%%%%%%%%%
\begin{lemm}
  \label{l:cq-log}
For each $\mathbf w\in\mathscr A_\star^{N_0}$, we have
\begin{equation}
  \label{e:cq-log-1}
a_{\mathbf w}\in S_{1/6+}^{\comp}(T^*M),\quad
A_{\mathbf w}=\Op_h(a_{\mathbf w})+\mathcal{O}(h^{2/3-})_{L^2\to L^2}.
\end{equation}
Moreover, for each $F:\mathscr A_\star^{N_0}\to\mathbb{C}$ with $\sup|F|\leq1$, we have
(using the notation~\eqref{e:A-F-def})
\begin{equation}
  \label{e:cq-log-2}
a_F\in S_{1/6+}^{\comp}(T^*M),\quad
A_F=\Op_h(a_F)+\mathcal{O}(h^{1/2-})_{L^2\to L^2}
\end{equation}
with the constant in the remainder independent of the function~$F$.
\end{lemm}
%%%%%%%%%%%%%%%%%%%%%%%%%%%%%%%%%%%%%%%%%%%%%%%%%%%%%%%%%%%%%%%%%%%%%%%%%%%%%%%%
\Remarks
1. The choice of index $\delta:={1\over 6}$ (which corresponds
to the factor ${1\over 6}$ in the definition of~$N_0$) was guided by the proof of Proposition~\ref{l:longdec-0}, yet it is somewhat arbitrary ~---
in practice one could probably replace ${1\over 6}$ 
by any $\delta\in (0,{1\over 2})$.

\noindent 2. Later we will prove much finer statements regarding the
propagation up to the \emph{local Ehrenfest time}~--- see~\S\ref{s:local-Ehr}-\ref{s:local-ehrenfest}.
It is possible to avoid the precise derivative bounds for $a_F$
by increasing the value of~$\delta$, as in~\cite[Lemma~4.4]{meassupp},
however the proof of these bounds below can seen as a basic case
of the more complicated bounds of~\S\ref{s:ehr-sum}.

%%%%%%%%%%%%%%%%%%%%%%%%%%%%%%%%%%%%%%%%%%%%%%%%%%%%%%%%%%%%%%%%%%%%%%%%%%%%%%%%
\begin{proof}
We write $\mathbf w=w_0\dots w_{N_0-1}$. By Lemma~\ref{l:egorov-mild} with $\delta:={1\over 6}$
we have uniformly in $j=0,\dots,N_0-1$
\begin{equation}
  \label{e:cq-log-int-1}
a_{w_j}\circ\varphi_j\in S_{1/6+}^{\comp}(T^*M),\quad
A_{w_j}(j)=\Op_h(a_{w_j}\circ\varphi_j)+\mathcal O(h^{2/3-})_{L^2\to L^2}.
\end{equation}
Now~\eqref{e:cq-log-1} follows from Lemma~\ref{l:log-product} with $\delta:={1\over 6}+\epsilon$
and $\epsilon>0$ arbitrarily small.

To establish bounds on $a_F$, we first note that
$\sup |a_F|\leq 1$ since $\sup |F|\leq 1$ and $|a_1|+|a_\star|=a_1+a_\star\leq 1$.
To prove bounds on derivatives, take arbitrary vector fields $X_1,\dots,X_k$
on $T^*M$. 
For a set $I\subset \{1,\dots,k\}$
define the differential operator
$$
X_I:=X_{i_1}\cdots X_{i_r}\quad\text{where}\quad
I=\{i_1,\dots,i_r\},\quad
i_1<\dots<i_r.
$$
By the product rule we have for all $\mathbf w\in \mathscr A_\star^{N_0}$
$$
X_1\dots X_k a_{\mathbf w}=\sum_{L\in\mathscr L} \prod_{j=0}^{N_0-1}X_{\mathcal I(L,j)}(a_{w_j}\circ\varphi_j).
$$
where the sum is over the set of sequences (with each $\ell_i$ encoding which of the factors of the product
defining $a_{\mathbf w}$ the vector field $X_i$ was applied to)
$$
\mathscr L:=\big\{L=(\ell_1,\dots,\ell_k)\mid \ell_1,\dots,\ell_k\in \{0,\dots,N_0-1\}\big\}
$$
and for $L\in\mathscr L$ and $j\in \{0,\dots,N_0-1\}$ we put
$$
\mathcal I(L,j):=\big\{i\in \{1,\dots,k\}\mid \ell_i=j\big\}.
$$
It follows that (with $\mathbf w$ summed over $\mathscr A_\star^{N_0}$)
$$
\begin{gathered}
|X_1\dots X_k a_{F}|\leq \sum_{L\in\mathscr L}\sum_{\mathbf w} \prod_{j=0}^{N_0-1}|X_{\mathcal I(L,j)}(a_{w_j}\circ\varphi_j)|=
\sum_{L\in\mathscr L}\prod_{j=0}^{N_0-1}\mathcal N(L,j)\\\quad\text{where}\quad
\mathcal N(L,j):=
\sum_{w\in \{1,\star\}}|X_{\mathcal I(L,j)}(a_{w}\circ\varphi_j)|.
\end{gathered}
$$
Fix arbitrary $\epsilon>0$. By~\eqref{e:cq-log-int-1} and since $|a_1|+|a_\star|\leq 1$ we have
for some constant $C$ depending only on $X_1,\dots,X_k,\epsilon$
$$
\begin{aligned}
\mathcal N(L,j)&\leq 1,&\quad\text{if}\quad \mathcal I(L,j)=\emptyset;\\
\mathcal N(L,j)&\leq Ch^{-(1/6+\epsilon)\#(\mathcal I(L,j))},&\quad\text{if}\quad\mathcal I(L,j)\neq\emptyset.
\end{aligned}
$$
For each $L\in\mathscr L$, we have $\sum_{j=0}^{N_0-1}\#(\mathcal I(L,j))=k$. Moreover, the set $\mathscr L$ has $N_0^k=\mathcal O(h^{0-})$
elements. It follows that
$$
\sup|X_1\dots X_k a_F|\leq Ch^{-(1/6+2\epsilon)k}
$$
which implies that $a_F\in S^{\comp}_{1/6+}(T^*M\setminus 0)$.

Finally, to show that $A_F=\Op_h(a_F)+\mathcal O(h^{1/2-})_{L^2\to L^2}$
it suffices to sum the second parts of~\eqref{e:cq-log-1} over $\mathbf w$ with coefficients
$F(\mathbf w)$ and use the counting bound $\#(\mathscr A_\star^{N_0})=2^{N_0}=\mathcal O(h^{-1/6})$
which holds since $\Lambda_1\geq 1$.
\end{proof}
%%%%%%%%%%%%%%%%%%%%%%%%%%%%%%%%%%%%%%%%%%%%%%%%%%%%%%%%%%%%%%%%%%%%%%%%%%%%%%%%
Lemma~\ref{l:cq-log} together with~\eqref{e:precise-norm} give
the norm bound
\begin{equation}
  \label{e:A-E-norm}
\|A_F\|_{L^2\to L^2}\leq 1+\mathcal O(h^{1/3-})\quad\text{for all}\quad
F:\mathscr A_\star^{N_0}\to\mathbb{C},\
\sup|F|\leq 1
\end{equation}
where the constant in the remainder is independent of~$F$.
This bound in particular applies to operators of the form
$A_{\mathbf w}$, $\mathbf w\in \mathscr A_\star^{N_0}$,
and more generally of the form $A_{\mathcal E}$ where $\mathcal E\subset \mathscr A_\star^{N_0}$.

%%%%%%%%%%%%%%%%%%%%%%%%%%%%%%%%%%%%%%%%%%%%%%%%%%%%%%%%%%%%%%%%%%%%%%%%%%%%%%%%
\subsubsection{Long logarithmic words}
\label{s:long-words}

We now study operators associated to words of length~$N$. 
The following key estimate is proved in~\S\ref{s:long-word-fup} below
using the fractal uncertainty principle
and the fact that the complements $T^*M\setminus\mathcal V_1,T^*M\setminus\mathcal V_\star$
have nonempty interior.
It implies that each operator
$A_{\mathbf w}$, where $\mathbf w\in\mathscr A_\star^N$,
has norm decaying with $h$.
%%%%%%%%%%%%%%%%%%%%%%%%%%%%%%%%%%%%%%%%%%%%%%%%%%%%%%%%%%%%%%%%%%%%%%%%%%%%%%%%
\begin{prop}
  \label{l:longdec-0}
Let the assumptions~(1)--(5) of~\S\ref{s:proofs-notation} hold and $\varepsilon_0$ be small enough
depending only on~$M$.
Then there exists $\beta>0$ depending only on $\mathcal V_1,\mathcal V_\star$
and there exists $C>0$ depending only on $A_1,A_\star$ such that
for all $\mathbf w\in \mathscr A_\star^N$
\begin{equation}
  \label{e:longdec-0}
\|A_{\mathbf w}\|_{L^2\to L^2}\leq Ch^\beta.
\end{equation}
\end{prop}
%%%%%%%%%%%%%%%%%%%%%%%%%%%%%%%%%%%%%%%%%%%%%%%%%%%%%%%%%%%%%%%%%%%%%%%%%%%%%%%%
\Remarks 1. We note that $N$ is considerably larger than \emph{twice the maximal Ehrenfest time}
$\log(1/h)\over \Lambda_0$,
that is for all $\rho\in S^*M$ the norm $d\varphi_{N}(\rho)$ is much larger than $h^{-1}$.
Therefore the classical/quantum correspondence~\eqref{e:cq-basic} no longer applies to the operator $A_{\mathbf w}$, $\mathbf w\in\mathscr A_\star^N$. In fact
the norm bound~\eqref{e:longdec-0} contradicts this correspondence:
if $A_{\mathbf w}$ were a quantization of $a_{\mathbf w}$, then
we would expect the norm $\|A_{\mathbf w}\|$ to be close to $\sup |a_{\mathbf w}|$,
however in general we could have $\sup|a_{\mathbf w}|=1$ while~\eqref{e:longdec-0}
implies that $\|A_{\mathbf w}\|$ is small.

\noindent 2. In the constant curvature case a version of Proposition~\ref{l:longdec-0} is
proved in~\cite[Proposition~3.5]{meassupp}. We remark that~\cite{meassupp} considered
words of length $\approx 2\log(1/h)$, while here we study words of shorter length $N\approx {7\over 6}\log(1/h)$.
The factor ${7\over 6}$ was chosen for convenience in the proof of Proposition~\ref{l:longdec-0}; see~\S\ref{s:outline} below
and in particular~\eqref{e:almoster}, \eqref{e:almoster2}. We could probably have replaced this factor by any number in the interval $(1 , {3\over 2})$; yet we did not try to optimize the estimate in the proposition by varying this factor.

\noindent 3. Proposition~\ref{l:longdec-0} is formally similar to~\cite[Theorem~2.7]{AN07}
and~\cite[Theorem~1.3.3]{AnAnn}, as all these statements imply norm decay for operators
corresponding to words of long logarithmic length.
However~\cite{AN07,AnAnn} used a \emph{fine partition}
of $S^*M$, for which each symbol $a_{\mathbf w}$ in a thin neighbourhood of a single stable leaf
(see~\S\ref{s:refined-partition} below). On the contrary, the partition~\eqref{e:partitor} we use here is {\it not} fine,
in fact $\supp a_\star$ contains all of $S^*M$ except a small ball, and the supports of operators $a_{\mathbf w}$ typically have a complicated fractal structure.
As a result, the method of proof of Proposition~\ref{l:longdec-0} is very different
from those in~\cite{AN07,AnAnn}, it relies on the fractal uncertainty principle, which takes advantage of the ``fractality'' of $\supp a_{\mathbf w}$. A common point with the proofs in~\cite{AN07}, is that we will only use words of ``moderately long'' logarithmic length (e.g. in constant curvature words of length  $\sim {7\over 6}\log(1/h)$), instead of ``very long'' logarithmic length as in \cite{AnAnn}.

\noindent 4. Following the proof of Proposition~\ref{l:longdec-0} in~\S\ref{s:long-word-fup}
and using the remarks after Lemmas~\ref{l:poro+}--\ref{l:poro-}, we obtain the following statement:
if the complements $S^*M\setminus\mathcal V_1,S^*M\setminus \mathcal V_\star$ are $(L_0,L_1)$-dense
in both unstable and stable directions (in the sense of Definition~\ref{d:stun-dense})
then Proposition~\ref{l:longdec-0} holds for some $\beta$ depending only on~$(M,g),L_0,L_1$.

%%%%%%%%%%%%%%%%%%%%%%%%%%%%%%%%%%%%%%%%%%%%%%%%%%%%%%%%%%%%%%%%%%%%%%%%%%%%%%%%
\subsection{Proof of Theorem~\ref{t:eig}}
  \label{s:proofs-1}

We now prove Theorem~\ref{t:eig}, following the strategy of~\cite[\S\S3,4]{meassupp}.

%%%%%%%%%%%%%%%%%%%%%%%%%%%%%%%%%%%%%%%%%%%%%%%%%%%%%%%%%%%%%%%%%%%%%%%%%%%%%%%%
\subsubsection{Construction of the partition}
  \label{s:partition}

We first construct the functions $a_1,a_\star$ and the operators $A_1,A_\star$
satisfying the assumptions of~\S\ref{s:proofs-notation} and used in the proof of Theorem~\ref{t:eig}.

In addition to $A_1,A_\star$ we use an operator $A_0$
which cuts away from the cosphere bundle $S^*M$. More precisely we put
\begin{equation}
  \label{e:A-0-def}
\begin{gathered}
A_0:=\psi_0(-h^2\Delta)\quad\text{where $\psi_0\in C^\infty(\mathbb R;[0,1])$ satisfies}
\\
\supp\psi_0\cap [\textstyle{1\over 4},4]=\emptyset,\quad
\supp(1-\psi_0)\subset (\textstyle{1\over 16},16).
\end{gathered}
\end{equation}
By the functional calculus~\eqref{e:funcal} applied to $1-\psi_0$ we see that
\begin{equation}
  \label{e:A-0}
A_0\in \Psi^0_h(M),\quad
\sigma_h(A_0)=a_0:=\psi_0(|\xi|_g^2),\quad
\WFh(I-A_0)\subset \{\textstyle{1\over 4}<|\xi|_g<4\}.
\end{equation}
The functions $a_1,a_\star$ and the operators $A_1,A_\star$ are constructed in the following
lemma. Here we let $a$ be the function in the statement of
Theorem~\ref{t:eig} and $\varepsilon_0>0$ be small enough so that
Proposition~\ref{l:longdec-0} applies.
%%%%%%%%%%%%%%%%%%%%%%%%%%%%%%%%%%%%%%%%%%%%%%%%%%%%%%%%%%%%%%%%%%%%%%%%%%%%%%%%
\begin{lemm}
  \label{l:quantum-partition}
Let $a\in C^\infty(T^*M)$ satisfy $a|_{S^*M}\not\equiv 0$, and fix $\varepsilon_0>0$. Then there exist
$a_1,a_\star,A_1,A_\star$ such that conditions~(1)--(5) of~\S\ref{s:proofs-notation} hold
and moreover
\begin{enumerate}
\setcounter{enumi}{5}
\item $A_0,A_1,A_\star$ form a pseudodifferential partition of unity, namely $I=A_0+A_1+A_\star$. This in particular implies that $1=a_0+a_1+a_\star$;
\item if $\mathcal V_1\subset T^*M\setminus 0$ is the open conic set containing $\supp a_1$
introduced in~\S\ref{s:proofs-notation}, then $\mathcal V_1\cap S^*M\subset \{a\neq 0\}$.
\end{enumerate}
\end{lemm}
%%%%%%%%%%%%%%%%%%%%%%%%%%%%%%%%%%%%%%%%%%%%%%%%%%%%%%%%%%%%%%%%%%%%%%%%%%%%%%%%
\begin{proof}
We first choose a nonempty open conic set $\mathcal V_1\subset T^*M\setminus 0$ such that
$\mathcal V_1\cap S^*M\subset \{a\neq 0\}$, the diameter of $\mathcal V_1\cap S^*M$ is less
than $\varepsilon_0$, and the complement $T^*M\setminus \mathcal V_1$ has nonempty interior.
For instance, we can let $\mathcal V_1\cap S^*M$ be a small ball centered around
a point in $\{a\neq 0\}\cap S^*M$. We next choose another open conic set
$\mathcal V_\star\subset T^*M\setminus 0$ such that $T^*M\setminus \mathcal V_\star$ has nonempty interior
and   
\begin{equation}
  \label{e:partitor}
T^*M\setminus 0=\mathcal V_1\cup \mathcal V_\star.
\end{equation}
By~\eqref{e:A-0} we may write
$$
I-A_0=\Op_h(b)+R,\quad
R=\mathcal O(h^\infty)_{\Psi^{-\infty}}
$$
where the $h$-dependent symbol $b\in S^{-\infty}_h(T^*M)$ satisfies for some compact $h$-independent set~$K$
$$
\supp b\subset K\subset\{\textstyle{1\over 4}<|\xi|_g<4\},\quad
b=1-a_0+\mathcal O(h).
$$
By~\eqref{e:partitor} we see that $K\subset \widetilde{\mathcal V}_1\cup\widetilde{\mathcal V}_\star$
where $\widetilde{\mathcal V}_w:=\mathcal V_w\cap \{{1\over 4}<|\xi|_g<4\}$.
Take an $h$-independent partition of unity
$$
\chi_1\in \CIc(\widetilde{\mathcal V}_1;[0,1]),\quad
\chi_\star\in \CIc(\widetilde{\mathcal V}_\star;[0,1]),\quad
\chi_1+\chi_\star\equiv 1\text{ on }K
$$
and define
$$
A_1:=\Op_h(\chi_1 b)+R,\quad
A_\star:=\Op_h(\chi_\star b).
$$
Then the conditions~(1)--(7) hold, where the principal symbols $a_1,a_\star$
are given by
$a_1=\chi_1(1-a_0)$,
$a_\star=\chi_\star(1-a_0)$.
\end{proof}
%%%%%%%%%%%%%%%%%%%%%%%%%%%%%%%%%%%%%%%%%%%%%%%%%%%%%%%%%%%%%%%%%%%%%%%%%%%%%%%%
We now establish two corollaries of properties~(6)--(7) in Lemma~\ref{l:quantum-partition}.
First of all, since $A_1+A_\star=I-A_0$ commutes with $U(t)$, we see that (using the notation~\eqref{e:A-E-def})
\begin{equation}
\label{e:A-total}
A_{\mathscr A_\star^n}=(A_1+A_\star)^n=(I-A_0)^n\quad\text{for all }n\in\mathbb N.
\end{equation}
The proof of~\cite[Lemma~3.1]{meassupp} then implies that for all $n\in\mathbb N$ and $u\in H^2(M)$
\begin{equation}
\label{e:A-full-word}
\|u-A_{\mathscr A_\star^n}u\|_{L^2}\leq C\|(-h^2\Delta-I)u\|_{L^2}
\end{equation} 
where $C$ is a constant independent of $n$, $h$. In particular,
if $(-h^2\Delta-I)u=0$ then $u=A_{\mathscr A_\star^n}u$.

Secondly, since $\supp a_1\cap S^*M\subset \{a\neq 0\}$, the elliptic estimate~\cite[Lemma~4.1]{meassupp}
 implies that
for all $u\in H^2(M)$
\begin{equation}
  \label{e:control-1}
\|A_1u\|_{L^2}\leq C\|\Op_h(a)u\|_{L^2}+C\|(-h^2\Delta-I)u\|_{L^2}+Ch\|u\|_{L^2}.  
\end{equation}
In particular, $A_1u$ is \emph{controlled}, by which we mean that
it is bounded in terms of the right-hand side of~\eqref{e:eig}
and a remainder which goes to~0 as $h\to 0$. Later in Lemma~\ref{l:propagated-control}
we extend~\eqref{e:control-1} to the propagated operators $A_1(t)$.

We remark that if we additionally know that $\supp a_1\cap S^*M\subset \{|a|\geq 1\}$
then we may take the first constant $C$ on the right-hand side of~\eqref{e:control-1} to be equal to~2
(or in fact, any fixed number larger than~1).
This follows from the proof of~\cite[Lemma~4.1]{meassupp} together with the norm bound~\eqref{e:precise-norm}.

The rest of the proof consists of writing $u=A_{\mathcal X}u+A_{\mathcal Y}u$
(microlocally near $S^*M$, see~\eqref{e:decompost}),
with the operators $A_{\mathcal X},A_{\mathcal Y}$ defined in~\S\ref{s:con-long} below,
such that:
\begin{itemize}
\item $A_{\mathcal Y}u$ is controlled (the proof of this uses classical/quantum correspondence,
Lemma~\ref{l:cq-log}), and
\item $A_{\mathcal X}u$ is small (the proof of this uses the smallness of the norm
$\|A_{\mathcal X}\|_{L^2\to L^2}$ which follows from the key estimate, Lemma~\ref{l:longdec-0}).
\end{itemize}

%%%%%%%%%%%%%%%%%%%%%%%%%%%%%%%%%%%%%%%%%%%%%%%%%%%%%%%%%%%%%%%%%%%%%%%%%%%%%%%%
\subsubsection{Controlled short logarithmic words}
\label{s:controlled-short}

We now define the set of controlled words of length~$N_0$ (see~\eqref{e:prop-times}).
Following~\cite[\S3.2]{meassupp} we define the \emph{density function}
\begin{equation}
  \label{e:density-def}
F:\mathscr A_\star^{N_0}\to [0,1],\quad
F(w_0\dots w_{N_0-1})={\#\{j\in \{0,\dots,N_0-1\}\mid w_j=1\}\over N_0}.
\end{equation}
Fix small $\alpha\in(0,\frac{1}{2})$ to be chosen in \eqref{e:alpha-chosen} below, and define the controlled, resp. uncontrolled words in $\mathscr A_\star^{N_0}$:
\begin{equation}
  \label{e:Z-def}
\mathcal Z:=\{\bw\in \mathscr A_\star^{N_0}\mid F(\bw)\geq \alpha\},\quad
\mathcal Z^{\complement}=\{\bw\in \mathscr A_\star^{N_0}\mid F(\bw)<\alpha\}.
\end{equation}
Define the operator $A_{\mathcal Z}$ by~\eqref{e:A-E-def}.
Then $A_{\mathcal Z}u$ is estimated by the following
%%%%%%%%%%%%%%%%%%%%%%%%%%%%%%%%%%%%%%%%%%%%%%%%%%%%%%%%%%%%%%%%%%%%%%%%%%%%%%%%
\begin{lemm}
  \label{l:Z-control}
There exists a constant $C>0$ independent of $\alpha$ or $h$, such that for all $\alpha\in (0,\frac 12)$,
$h\in (0,1]$, and $u\in H^2(M)$ we have
\begin{equation}
  \label{e:AZ}
\|A_{\mathcal Z}u\|_{L^2}\leq
{C\over\alpha}\|\Op_h(a)u\|_{L^2}+{C\log(1/h)\over \alpha h}\|(-h^2\Delta-I)u\|_{L^2}
+\mathcal O(h^{1/4-})\|u\|_{L^2}
\end{equation}
where the constant in $\mathcal O(\bullet)$ depends on~$\alpha$ but not on $h,u$.
\end{lemm}
%%%%%%%%%%%%%%%%%%%%%%%%%%%%%%%%%%%%%%%%%%%%%%%%%%%%%%%%%%%%%%%%%%%%%%%%%%%%%%%%
To prove Lemma~\ref{l:Z-control} we use the following almost monotonicity
property:
%%%%%%%%%%%%%%%%%%%%%%%%%%%%%%%%%%%%%%%%%%%%%%%%%%%%%%%%%%%%%%%%%%%%%%%%%%%%%%%%
\begin{lemm}
  \label{l:almost-monotone}
Assume that the functions $F_1,F_2:\sA_\star^{N_0}\to \mathbb C$ satisfy
$$
|F_1(\mathbf w)|\leq F_2(\mathbf w)\leq 1\quad\text{for all}\quad \mathbf w\in \sA_\star^{N_0}.
$$
Then for all $u\in L^2(M)$ we have (using the notation~\eqref{e:A-F-def})
\begin{equation}
  \label{e:almost-monotone}
\|A_{F_1} u\|_{L^2}\leq \|A_{F_2} u\|_{L^2}+ Ch^{1/4-}\|u\|_{L^2}
\end{equation}
where the constant $C$ is independent of $F_1,F_2,h,u$.
\end{lemm}
%%%%%%%%%%%%%%%%%%%%%%%%%%%%%%%%%%%%%%%%%%%%%%%%%%%%%%%%%%%%%%%%%%%%%%%%%%%%%%%%
\begin{proof}
We have 
$$
\|A_{F_2} u\|^2- \|A_{F_1} u\|^2 = \langle Bu,u\rangle
\quad\text{where}\quad
B:=A_{F_2}^* A_{F_2} - A_{F_1}^* A_{F_1}.
$$
By Lemma~\ref{l:cq-log} the operator $B$ is pseudodifferential:
$$
B=\Op_h(b)+\mathcal O(h^{1/2-})_{L^2\to L^2}\quad\text{where}\quad
b:=|a_{F_2}|^2-|a_{F_1}|^2\in S^{\comp}_{1/6+}(T^*M).
$$
From the positivity of the symbols $a_{\mathbf w}$, we deduce that 
$$
\Big|\sum_{\bw}F_1(\bw) a_{\bw}\Big|\leq \sum_{\bw}|F_1(\bw)| a_{\bw}\leq \sum_{\bw}F_2(\bw) a_{\bw},
$$
or in short $|a_{F_1}|\leq a_{F_2}$, which implies that $b\geq 0$.
By the G\r arding inequality~\eqref{e:gaarding} we have for all $\epsilon>0$
$$
\langle Bu,u\rangle\geq -C_\epsilon h^{1/2-\epsilon}\|u\|_{L^2}^2
$$
which gives $\|A_{F_1}u\|_{L^2}^2\leq \|A_{F_2}u\|_{L^2}^2+C_\epsilon h^{1/2-\epsilon}\|u\|_{L^2}^2$,
implying~\eqref{e:almost-monotone}.
\end{proof}
%%%%%%%%%%%%%%%%%%%%%%%%%%%%%%%%%%%%%%%%%%%%%%%%%%%%%%%%%%%%%%%%%%%%%%%%%%%%%%%%
We also use the following control bound on $A_1(t)u$
which is obtained from~\eqref{e:control-1} using that $\|U(t)u-e^{-it/h}u\|_{L^2}\leq C{|t|\over h} \|(-h^2\Delta-I)u\|_{L^2}$ (see~\cite[Lemma 4.3]{meassupp} for details):
%%%%%%%%%%%%%%%%%%%%%%%%%%%%%%%%%%%%%%%%%%%%%%%%%%%%%%%%%%%%%%%%%%%%%%%%%%%%%%%%
\begin{lemm}
  \label{l:propagated-control}
For all $t\in \mathbb R$ and $u\in H^2(M)$, we have
\begin{equation}
  \label{e:propagatedc}
\|A_1(t)u\|_{L^2}\leq C\|\Op_h(a) u\|_{L^2}
+{C\langle t\rangle\over h}\|(-h^2\Delta-I)u\|_{L^2}
+Ch\|u\|_{L^2}
\end{equation}
where $\langle t\rangle :=\sqrt{1+t^2}$ and
the constant $C$ is independent of $t$ and $h$.
\end{lemm}
%%%%%%%%%%%%%%%%%%%%%%%%%%%%%%%%%%%%%%%%%%%%%%%%%%%%%%%%%%%%%%%%%%%%%%%%%%%%%%%%
\Remark Using the remark after~\eqref{e:control-1} and the proof of~\cite[Lemma~4.3]{meassupp},
we see that under the condition $\supp a_1\cap S^*M\subset \{|a|\geq 1\}$
we may take the first constant on the right-hand side of~\eqref{e:propagatedc} to
be equal to~2 (or in fact, any fixed number larger than~1).

We are now ready to finish
%%%%%%%%%%%%%%%%%%%%%%%%%%%%%%%%%%%%%%%%%%%%%%%%%%%%%%%%%%%%%%%%%%%%%%%%%%%%%%%%
\begin{proof}[Proof of Lemma~\ref{l:Z-control}]
By the definition~\eqref{e:Z-def}
of the set~$\mathcal Z$, the indicator function $\mathbf 1_{\mathcal Z}$ satisfies $0\leq \alpha\mathbf 1_{\mathcal Z}(\mathbf w)\leq F(\mathbf w)\leq 1$ for all $\mathbf w\in\mathscr A_\star^{N_0}$.
Thus by Lemma~\ref{l:almost-monotone}
\begin{equation}
  \label{e:AZ-AF}
\alpha\|A_{\mathcal Z}u\|_{L^2}\leq \|A_Fu\|_{L^2}
+\mathcal O(h^{1/4-})\|u\|_{L^2}.
\end{equation}
On the other hand, \eqref{e:density-def} together with~\eqref{e:A-total} gives the following formula for $A_F$:
$$
A_F={1\over N_0}\sum_{j=0}^{N_0-1}\sum_{\mathbf w\in \mathscr A_\star^{N_0},w_j=1}A_{\mathbf w}
={1\over N_0}\sum_{j=0}^{N_0-1}(A_1+A_\star)^{N_0-1-j}A_1(j)(A_1+A_\star)^j.
$$
Recall that $\|A_1+A_\star\|_{L^2\to L^2}\leq 1$ by Lemma~\ref{l:quantum-partition}. It follows that
$$
\|A_Fu\|_{L^2}\leq \max_{0\leq j<N_0}\|A_1(j)(A_1+A_\star)^ju\|_{L^2}.
$$
Since $\|A_1(j)\|_{L^2\to L^2}=\|A_1\|_{L^2\to L^2}\leq C$ and
$(A_1+A_\star)^j u-u$ can be estimated by~\eqref{e:A-full-word}, we get
$$
\|A_Fu\|_{L^2}\leq 
\max_{0\leq j\leq N_0}\|A_1(j)u\|_{L^2}
+C\|(-h^2\Delta-I)u\|_{L^2}.
$$
Estimating $A_1(j)u$ by Lemma~\ref{l:propagated-control} and using that $N_0=\mathcal O(\log(1/h))$, we get
\begin{equation}
  \label{e:AF}
\|A_Fu\|_{L^2}\leq C\|\Op_h(a)u\|_{L^2}+{C\log(1/h)\over h}\|(-h^2\Delta-I)u\|_{L^2}+Ch\|u\|_{L^2}.
\end{equation}
Combining~\eqref{e:AZ-AF} and~\eqref{e:AF}, we obtain~\eqref{e:AZ}.
\end{proof}
%%%%%%%%%%%%%%%%%%%%%%%%%%%%%%%%%%%%%%%%%%%%%%%%%%%%%%%%%%%%%%%%%%%%%%%%%%%%%%%%

%%%%%%%%%%%%%%%%%%%%%%%%%%%%%%%%%%%%%%%%%%%%%%%%%%%%%%%%%%%%%%%%%%%%%%%%%%%%%%%%
\subsubsection{Controlled long logarithmic words}
  \label{s:con-long}

The proof of Lemma~\ref{l:Z-control} used the monotonicity property, Lemma~\ref{l:almost-monotone},
which in turn relied on classical/quantum correspondence. Thus it only applied
to words of short logarithmic length $N_0$. On the other hand, Lemma~\ref{l:longdec-0}
only applies to words of long logarithmic length~$N=(6\Lambda+1)N_0$. To bridge the gap between the two, we
define the sets of uncontrolled, resp. controlled words of length~$N$ as follows:
\begin{equation}
  \label{e:XY-def}
\begin{aligned}
\mathscr A_\star^{N}&=\mathcal X\sqcup\mathcal Y,\\
\mathcal X&:=\{\mathbf w^{(1)}\dots \bw^{(6\Lambda+1)}\mid \bw^{(\ell)}\in\mathcal Z^{\complement}\quad\text{for all }\ell\},\\
\mathcal Y&:=\{\mathbf w^{(1)}\dots \bw^{(6\Lambda+1)}\mid \text{there exists $\ell$ such that }\bw^{(\ell)}\in\cZ\}
\end{aligned}
\end{equation}
where $\mathcal Z\subset \mathscr A_\star^{N_0}$ is defined in~\eqref{e:Z-def} and we view words in $\mathscr A_\star^{N}$ as concatenations
$\mathbf w^{(1)}\dots \mathbf w^{(6\Lambda+1)}$ with
$\mathbf w^{(1)},\dots,\mathbf w^{(6\Lambda+1)}\in \mathscr A_\star^{N_0}$.

Using previously established bound on controlled short logarithmic words, Lemma~\ref{l:Z-control},
we now estimate the contribution of controlled long logarithmic words:
%%%%%%%%%%%%%%%%%%%%%%%%%%%%%%%%%%%%%%%%%%%%%%%%%%%%%%%%%%%%%%%%%%%%%%%%%%%%%%%%
\begin{prop}
\label{l:Y-control} For all $u\in H^2(M)$
\begin{equation}
  \label{e:controlled-estimate}
\|A_{\mathcal Y}u\|_{L^2}\leq {C\over\alpha}\|\Op_h(a)u\|_{L^2}+{C\log(1/h)\over \alpha h}\|(-h^2\Delta-I)u\|_{L^2}
+\mathcal O(h^{1/4-})\|u\|_{L^2}
\end{equation}
where the constant $C$ does not depend on $\alpha,h,u$
and the constant in $\mathcal O(\bullet)$ depends on~$\alpha$ but not on $h,u$.
\end{prop}
%%%%%%%%%%%%%%%%%%%%%%%%%%%%%%%%%%%%%%%%%%%%%%%%%%%%%%%%%%%%%%%%%%%%%%%%%%%%%%%%
\begin{proof}
The set $\mathcal{Y}$ can naturally be split as follows:
$$
\mathcal Y=\bigsqcup_{\ell=1}^{6\Lambda+1}\mathcal Y_\ell,\quad
\mathcal Y_\ell := \{\mathbf w^{(1)}\dots \mathbf w^{(6\Lambda+1)}\mid
\mathbf w^{(\ell)}\in \mathcal Z,\quad
\mathbf w^{(\ell+1)},\dots,\mathbf w^{(6\Lambda+1)}\in\cZ^{\complement}\}.
$$
Accordingly, we may write (using~\eqref{e:A-total})
$$
A_{\mathcal{Y}}=\sum_{\ell=1}^{6\Lambda+1}A_{\mathcal Y_\ell},\quad A_{\mathcal Y_\ell}=
A_{\cZ^\complement}(6\Lambda N_0)\cdots A_{\mathcal{\cZ^\complement}}(\ell N_0)A_{\mathcal{Z}}((\ell-1)N_0)
(A_1+A_\star)^{(\ell-1)N_0}.
$$
We have $\|A_1+A_\star\|_{L^2\to L^2}\leq 1$ by Lemma~\ref{l:quantum-partition}
and $\|A_{\mathcal Z}\|,\|A_{\mathcal Z^{\complement}}\|_{L^2\to L^2}\leq C$ by~\eqref{e:A-E-norm}.
Moreover, $u-(A_1+A_\star)^{(\ell-1)N_0}u$ can be estimated by~\eqref{e:A-full-word}.
It follows that for all $\ell$
\begin{equation}
  \label{e:nailao-1}
\|A_{\mathcal Y_{\ell}}u\|_{L^2}\leq C\|A_{\mathcal Z}((\ell-1)N_0)u\|_{L^2}+C\|(-h^2\Delta-I)u\|_{L^2}.
\end{equation}
We now estimate 
\begin{equation}
  \label{e:nailao-2}
\begin{gathered}
\|A_{\mathcal Z}((\ell-1)N_0)u\|_{L^2} 
\leq \|A_{\mathcal Z}u\|_{L^2}+ {C\log(1/h)\over h}\|(-h^2\Delta-I)u\|_{L^2}\\
\leq
{C\over\alpha}\|\Op_h(a)u\|_{L^2}
+{C\log(1/h)\over\alpha h}\|(-h^2\Delta-I)u\|_{L^2}
+\mathcal O(h^{1/4-})\|u\|_{L^2}
\end{gathered}
\end{equation}
where the first inequality follows similarly to~\eqref{e:propagatedc} from~\cite[Lemma~4.2]{meassupp}
and the bound $N_0=\mathcal O(\log(1/h))$,
and the second inequality follows from Lemma~\ref{l:Z-control}.
Combining~\eqref{e:nailao-1} and~\eqref{e:nailao-2} we get the bound~\eqref{e:controlled-estimate}.
\end{proof}
%%%%%%%%%%%%%%%%%%%%%%%%%%%%%%%%%%%%%%%%%%%%%%%%%%%%%%%%%%%%%%%%%%%%%%%%%%%%%%%%
\Remarks
1. In passing from $\|A_{\mathcal Y_\ell}u\|_{L^2}$ to $\|A_{\mathcal Y}u\|_{L^2}$
we used the triangle inequality. Consequently the constant $C$ in~\eqref{e:controlled-estimate}
has a factor of $6\Lambda+1=N/N_0$. Thus it is important in our argument that
the ratio $N/N_0$, where $N_0$ is the time for which classical/quantum correspondence
applies and $N$ is the time for which fractal uncertainty principle gives decay
of $\|A_{\mathbf w}\|$, is bounded by an $h$-independent constant.

\noindent 2. Following the proofs of Lemma~\ref{l:Z-control} and Proposition~\ref{l:Y-control}
and using the remark after Lemma~\ref{l:propagated-control},
we see that under the condition $\supp a_1\cap S^*M\subset \{|a|\geq 1\}$
we may take the first constant $C$ on the right-hand side of~\eqref{e:controlled-estimate}
to be equal to~$4(6\Lambda+1)$. Here the extra factor of~2 comes from taking $C:=2$ in~\eqref{e:nailao-1};
in fact, we could take that factor to be any fixed number larger than~1.

%%%%%%%%%%%%%%%%%%%%%%%%%%%%%%%%%%%%%%%%%%%%%%%%%%%%%%%%%%%%%%%%%%%%%%%%%%%%%%%%
\subsubsection{Uncontrolled long words and end of the proof}
\label{s:uncon-long}

We can now finish the proof of Theorem~\ref{t:eig}. Take arbitrary $u\in H^2(M)$. We decompose
\begin{equation}
  \label{e:decompost}
u=(u-A_{\mathscr A_\star^N}u)+A_{\mathcal Y}u+A_{\mathcal X}u
\end{equation}
where $A_{\mathcal X},A_{\mathcal Y}$ are defined using the notation~\eqref{e:A-E-def}
and the decomposition~\eqref{e:XY-def}.

The first term can be estimated by \eqref{e:A-full-word} and the second term can be estimated by
Proposition~\ref{l:Y-control}, giving
\begin{equation}
  \label{e:dianhua-1}
\begin{aligned}
\|u\|_{L^2}&\leq {C\over\alpha}\|\Op_h(a)u\|_{L^2}+{C\log(1/h)\over\alpha h}\|(-h^2\Delta-I)u\|_{L^2}
\\&\quad\,+\|A_\mathcal X u\|_{L^2}+\mathcal O(h^{1/4-})\|u\|_{L^2}.
\end{aligned}
\end{equation}
To deal with the term $A_{\mathcal X}u$ we apply the key estimate,
Proposition~\ref{l:longdec-0}, to each individual $A_{\mathbf w}$ with $\mathbf w\in \mathcal X$
and use the triangle inequality. 
For that need the following counting lemma on the number of  elements in $\mathcal X$:
%%%%%%%%%%%%%%%%%%%%%%%%%%%%%%%%%%%%%%%%%%%%%%%%%%%%%%%%%%%%%%%%%%%%%%%%%%%%%%%%
\begin{lemm}
  \label{l:X-count}
There exists a constant $C>0$ depending on $\alpha,\Lambda_0,\Lambda_1$ but not on $h$, such that
\begin{equation}
\label{e:count-X}
\#(\mathcal X)\leq Ch^{-(\Lambda_0^{-1}+2)\alpha(1-\log\alpha)}.
\end{equation}
\end{lemm}
%%%%%%%%%%%%%%%%%%%%%%%%%%%%%%%%%%%%%%%%%%%%%%%%%%%%%%%%%%%%%%%%%%%%%%%%%%%%%%%%
\begin{proof}
By definition, elements of $\mathcal X$ are concatenations of $6\Lambda+1$ words in $\cZC$, thus $\#(\mathcal X)=\#(\cZ^\complement)^{6\Lambda+1}$. 
Since $\mathcal Z^{\complement}$ consists of words $\mathbf w\in\mathscr A_\star^{N_0}$
such that less than $\alpha N_0$ letters of $\mathbf w$ are equal to~1, we have
$$
\#(\cZ^\complement)\leq \sum_{k=0}^{\lfloor\alpha N_0\rfloor}\binom{N_0}{k}.
$$
Since $\alpha<1/2$, we have for $k=0,1,\ldots,\lfloor\alpha N_0\rfloor-1$
$$
\binom{N_0}{k}=\frac{k+1}{N_0-k}\binom{N_0}{k+1}\leq\frac{\alpha N_0}{N_0-\alpha N_0}\binom{N_0}{k+1}=\frac{\alpha}{1-\alpha}\binom{N_0}{k+1}
$$
and thus
$$
\binom{N_0}{k}\leq\left(\frac{\alpha}{1-\alpha}\right)^{\lfloor\alpha N_0\rfloor-k}\binom{N_0}{\lfloor\alpha N_0\rfloor}.
$$
In particular,
$$
\#(\cZ^\complement)\leq\frac{1-\alpha}{1-2\alpha}\binom{N_0}{\lfloor\alpha N_0\rfloor}.
$$
Using Stirling's formula, we have
$$
\binom{N_0}{\lfloor\alpha N_0\rfloor}
=\frac{N_0!}{\lfloor\alpha N_0\rfloor!(N_0-\lfloor\alpha N_0\rfloor)!}
\leq C\exp(-(\alpha\log\alpha+(1-\alpha)\log(1-\alpha))N_0).
$$
Using the elementary inequality
$$
-(\alpha\log\alpha+(1-\alpha)\log(1-\alpha))\leq\alpha(1-\log\alpha)
$$
we see that
$$
\#(\mathcal X)=\#(\cZ^\complement)^{6\Lambda+1}
\leq Ch^{-(\Lambda_0^{-1}+2)\alpha(1-\log\alpha)}.\qedhere
$$
\end{proof}
%%%%%%%%%%%%%%%%%%%%%%%%%%%%%%%%%%%%%%%%%%%%%%%%%%%%%%%%%%%%%%%%%%%%%%%%%%%%%%%%
We are now ready to finish the proof of Theorem~\ref{t:eig}. Let $\beta>0$
be the constant from Proposition~\ref{l:longdec-0}. Fix $\alpha>0$ small enough so that 
\begin{equation}
\label{e:alpha-chosen}
(\Lambda_0^{-1}+2)\alpha(1-\log\alpha)\leq\frac{\beta}{2}.
\end{equation}
Combining Proposition~\ref{l:longdec-0} and Lemma~\ref{l:X-count} we get
$$
\|A_{\mathcal X}\|_{L^2\to L^2}\leq \#(\mathcal X)\cdot Ch^{\beta}\leq Ch^{\beta/2}
$$
which (assuming without loss of generality that $\beta<{1\over 2}$) together with~\eqref{e:dianhua-1} implies
for some constant $C$ depending only on $a$
$$
\|u\|_{L^2}\leq C\|\Op_h(a)u\|_{L^2}+{C\log(1/h)\over h}\|(-h^2\Delta-I)u\|_{L^2}
+Ch^{\beta/2}\|u\|_{L^2}.
$$
Taking $h$ small enough, we can remove the last term on the right-hand side,
giving Theorem~\ref{t:eig}.

\Remark Using the remarks after Propositions~\ref{l:longdec-0} and~\ref{l:Y-control}
we obtain the following statement: if $\supp a_1\cap S^*M\subset \{|a|\geq 1\}$
and the complements $S^*M\setminus\mathcal V_1,S^*M\setminus\mathcal V_\star$
are $(L_0,L_1)$-dense in both unstable and stable directions (in the sense
of Definition~\ref{d:stun-dense}) then the first constant
on the right-hand side of~\eqref{e:eig} depends only on $(M,g),L_0,L_1$.
In fact, we can take this constant to be
\begin{equation}
\label{e:terminator}
C:={4(6\Lambda+1)\over\alpha}
\end{equation}
where $\alpha$ satisfies~\eqref{e:alpha-chosen} and thus depends on the fractal uncertainty exponent~$\beta$.
(The factors~4 and~6 above can be improved but this does not improve the result significantly since
the known bounds on~$\beta$ are very small.) In particular, as $\beta\to 0$ the constant~$C$ from~\eqref{e:terminator}
behaves like $\beta^{-1}\log(1/\beta)$ times a constant depending only on the minimal/maximal expansion rates~$\Lambda_0,\Lambda_1$.

This gives Theorem~\ref{t:eig-quant} as follows. Take an open set $U\subset S^*M$
which is $(L_0,L_1)$-dense in both unstable and stable directions and has diameter
smaller than the constant $\varepsilon_0$ from Proposition~\ref{l:longdec-0}.
Using Lemma~\ref{l:dense-useful}, fix $U^\sharp$ compactly contained in $U$
which is also $(L_0,L_1)$-dense in both unstable and stable directions.
Choose
$$
a\in\CIc(T^*M;[0,1]),\quad
\supp a\cap S^*M\subset U,\quad
\supp (1-a)\cap U^\sharp=\emptyset.
$$
We choose the sets $\mathcal V_1,\mathcal V_\star$ in the proof
of Lemma~\ref{l:quantum-partition} such that
$$
\overline{U^\sharp}\subset \mathcal V_1\cap S^*M\subset \{a=1\},\quad
\mathcal V_\star\cap S^*M=S^*M\setminus\overline{U^\sharp}.
$$
Then $\supp a_1\cap S^*M\subset \{|a|\geq 1\}$ and the complement
$S^*M\setminus \mathcal V_\star$ is $(L_0,L_1)$-dense in both unstable
and stable directions. Next, 
$S^*M\setminus \mathcal V_1$ contains the complement
of a set in $S^*M$ diameter~$\varepsilon_0$, and thus is $(1,{1\over 2})$-dense
in both unstable and stable directions for small enough $\varepsilon_0$.
Now if
$u_{j_k}$ is a sequence of Laplacian eigenfunctions converging to a measure~$\mu$
in the sense of~\eqref{e:semimes} then by~\eqref{e:eig} we have the estimate
$$
1=\|u_{j_k}\|_{L^2}\leq C\|\Op_{h_{j_k}}(a)u_{j_k}\|_{L^2}
\xrightarrow{k\to\infty}
C\int_{S^*M} |a|^2\,d\mu
\leq C\mu(U)
$$
where $C$ is the constant from~\eqref{e:terminator}, which depends only on $(M,g),L_0,L_1$.

%%%%%%%%%%%%%%%%%%%%%%%%%%%%%%%%%%%%%%%%%%%%%%%%%%%%%%%%%%%%%%%%%%%%%%%%%%%%%%%%
\subsection{Proof of Theorem~\ref{t:dwe}}
  \label{s:proofs-2}

We finally give the proof of Theorem~\ref{t:dwe}, following the strategy of~\cite{JinDWE}
and using some parts of the proof of Theorem~\ref{t:eig}. 

%%%%%%%%%%%%%%%%%%%%%%%%%%%%%%%%%%%%%%%%%%%%%%%%%%%%%%%%%%%%%%%%%%%%%%%%%%%%%%%%
\subsubsection{Reduction to decay for a microlocal damped propagator}
\label{s:dwe-reduction}

We first reduce Theorem~\ref{t:dwe} to a decay statement for a damped
propagator following~\cite[\S4]{JinDWE}. Let $b\in C^\infty(M)$ be the damping function,
with $b\geq 0$ and $b\not\equiv 0$. We replace $h\partial_t$ by $-iz$
in the semiclassically rescaled damped wave operator $h^2(\partial_t^2-\Delta+2b(x)\partial_t)$, to obtain the following
differential operator on $M$:
\begin{equation}
  \label{e:dawwy-1}
\mathcal P(z):=
-h^2\Delta-2izhb(x)-z^2,\quad
z\in\mathbb C.
\end{equation}
By a standard argument (see~\cite[\S3]{Schenk} or~\cite[Theorem~5.10]{e-z})
Theorem~\ref{t:dwe} follows from the following
high energy spectral gap:
%%%%%%%%%%%%%%%%%%%%%%%%%%%%%%%%%%%%%%%%%%%%%%%%%%%%%%%%%%%%%%%%%%%%%%%%%%%%%%%%
\begin{prop}
\label{l:gap}
There exist $C_0>0$, $\gamma_0>0$, and $h_0>0$ such that
\begin{equation}
  \label{e:gap}
\|\mathcal P(z)^{-1}\|_{L^2\to L^2}\leq Ch^{-1+C_0\min(0,\Im z/h)}\log(1/h),\quad
0<h\leq h_0,\quad
|z-1|\leq \gamma_0 h.
\end{equation}
\end{prop}
%%%%%%%%%%%%%%%%%%%%%%%%%%%%%%%%%%%%%%%%%%%%%%%%%%%%%%%%%%%%%%%%%%%%%%%%%%%%%%%%
Recall the operator $P=\psi_P(-h^2\Delta)$ defined in~\eqref{e:U-t-def}.
Fix a cutoff function
$$
\psi_1\in \CIc((0,\infty);[0,1]),\quad
\supp\psi_1\subset \{\psi_P\neq 0\},\quad
\psi_1=1\text{ on }[\textstyle{1\over 16},16].
$$
Then
$$
\mathcal P(z)=P^2-2izhb(x)\psi_1(-h^2\Delta)-z^2+\mathcal O(h^\infty)\quad\text{microlocally near }S^*M.
$$
We now write
\begin{equation}
  \label{e:dawwy-2}
P^2-2izhb(x)\psi_1(-h^2\Delta)=(P-ih A(z))^2+\mathcal O(h^\infty)
\end{equation}
where $A(z)\in \Psi^{-\infty}_h(T^*M)$ is some family of pseudodifferential
operators entire in $z$ and satisfying $\sigma_h(A(z))=za$ with
\begin{equation}
  \label{e:damp-b-a}
a(x,\xi):={b(x)\psi_1(|\xi|_g^2)\over p(x,\xi)}.
\end{equation} 
See~\cite[\S4.1]{JinDWE} for the construction of~$A(z)$ (denoted by~$Q(z)$ there).

Define the \emph{microlocal damped propagator}
\begin{equation}
\label{e:damp-propagator}
\widetilde{U}(t)=\widetilde U(t;z):=\exp\Big(-{it(P-ihA(z))\over h}\Big),\quad t\geq 0.
\end{equation}
We also take the following frequency cutoff operator:
$$
\Pi:=\chi(-h^2\Delta)\quad\text{where}\quad
\chi\in \CIc(\mathbb{R};[0,1]),\quad
\supp\chi\subset[\textstyle{1\over 4},4],\quad
1\notin\supp(1-\chi).
$$
Following~\cite[\S4.2]{JinDWE} we see that Proposition~\ref{l:gap} (and thus Theorem~\ref{t:dwe}) follows
from a decay statement on the propagator $\widetilde U(t)$:
%%%%%%%%%%%%%%%%%%%%%%%%%%%%%%%%%%%%%%%%%%%%%%%%%%%%%%%%%%%%%%%%%%%%%%%%%%%%%%%%
\begin{prop}
	\label{t:propagator-decay}
There exists $\beta_1>0$ depending only on $M$ and $b$ such that for all $h\in (0,1]$,
$z\in\mathbb C$ such that $|z-1|\leq h$,
and $N$ defined in~\eqref{e:prop-times} we have
\begin{equation}
	\label{e:propagator-decay}
\|\widetilde U(N;z)\Pi\|_{L^2(M)\to L^2(M)}\leq Ch^{\beta_1}.
\end{equation}
\end{prop}
%%%%%%%%%%%%%%%%%%%%%%%%%%%%%%%%%%%%%%%%%%%%%%%%%%%%%%%%%%%%%%%%%%%%%%%%%%%%%%%%
In the rest of~\S\ref{s:proofs-2} we prove Proposition~\ref{t:propagator-decay}.

%%%%%%%%%%%%%%%%%%%%%%%%%%%%%%%%%%%%%%%%%%%%%%%%%%%%%%%%%%%%%%%%%%%%%%%%%%%%%%%%
\subsubsection{Damped partition of unity}
  \label{s:damped-partition}

Let $A_0$ be given by~\eqref{e:A-0-def}
and $a_1,a_\star,A_1,A_\star$ be constructed in Lemma~\ref{l:quantum-partition},
with the function $a$ given by~\eqref{e:damp-b-a} and
$\varepsilon_0>0$ taken small enough so that Proposition~\ref{l:longdec-0} applies.
Define the damped operators
\begin{equation}
\label{e:damp-op-A}
\widetilde{A}_w:=U(-1)\widetilde{U}(1)A_w,\quad w\in\{1,\star\}.
\end{equation}
Here $U(t)=\exp(-itP/h)$ is the unitary propagator defined in~\eqref{e:U-t-def}
and $\widetilde U(t)$ is the damped propagator defined in~\eqref{e:damp-propagator}.
By~\cite[(2.24)]{JinDWE} we have $\widetilde A_w\in\Psi^{-\infty}_h(M)$,
$\WFh(\widetilde A_w)\subset\WFh(A_w)$, and
\begin{equation}
	\label{e:damp-a}
\sigma_h(\widetilde A_w)=\widetilde{a}_w:=a_w\exp\left(-\int_0^1a\circ\varphi_s\,ds\right),\quad
w\in \{1,\star\}.
\end{equation}
%%%%%%%%%%%%%%%%%%%%%%%%%%%%%%%%%%%%%%%%%%%%%%%%%%%%%%%%%%%%%%%%%%%%%%%%%%%%%%%%
\begin{lemm}
  \label{l:still-fine}
The operators $\widetilde A_1,\widetilde A_\star$ and the symbols
$\widetilde a_1,\widetilde a_\star$ satisfy conditions~(1)--(5) in~\S\ref{s:proofs-notation}.
Moreover, there exists a constant $\eta>0$ such that
\begin{equation}
	\label{e:damp-a-a}
0\leq\widetilde{a}_1\leq e^{-\eta}a_1,\quad 0\leq \widetilde{a}_\star\leq a_\star. 
\end{equation}
\end{lemm}
%%%%%%%%%%%%%%%%%%%%%%%%%%%%%%%%%%%%%%%%%%%%%%%%%%%%%%%%%%%%%%%%%%%%%%%%%%%%%%%%
\begin{proof}
Since $a\geq 0$, we have $0\leq \widetilde a_w\leq a_w$, and conditions~(1)--(5)
in~\S\ref{s:proofs-notation} follow immediately. It remains to show that $\widetilde a_1\leq e^{-\eta}a_1$.
As a consequence of the homogeneity of $a$ in $\{{1\over 4}\leq|\xi|_g\leq4\}$, we see that condition~(7)
in Lemma~\ref{l:quantum-partition} implies that
$$
\mathcal V_1\cap\{\textstyle{1\over 4}\leq|\xi|_g\leq4\}\subset\{a > 0\}.
$$
Since $\supp a_1\subset\mathcal V_1\cap\{{1\over 4}<|\xi|_g<4\}$, there exists $\eta>0$ such that
$$
\int_0^1a\circ\varphi_s(x,\xi)\,ds\geq \eta\quad\text{for all}\quad
(x,\xi)\in\supp a_1.
$$
This immediately implies that $\widetilde a_1\leq e^{-\eta}a_1$.
\end{proof}
%%%%%%%%%%%%%%%%%%%%%%%%%%%%%%%%%%%%%%%%%%%%%%%%%%%%%%%%%%%%%%%%%%%%%%%%%%%%%%%%
Using $\widetilde A_1,\widetilde A_\star,a_1,a_\star$, we define
$\widetilde A_{\mathbf w},\widetilde A_{\mathcal E},\widetilde A_F,\widetilde a_{\mathbf w},
\widetilde a_{\mathcal E},\widetilde a_F$ by~\eqref{e:A-pm-def}, \eqref{e:A-E-def},
\eqref{e:A-F-def}. (As before, we use the notation $\widetilde A_{\mathbf w}:=\widetilde A_{\mathbf w}^-$ etc.)
We also consider the cutoff damped propagators
\begin{equation}
	\label{e:damp-word}
\widetilde{U}_{\mathbf w}:=U(n)\widetilde{A}_{\mathbf w}=\widetilde{U}(1)A_{w_{n-1}}\widetilde{U}(1)A_{w_{n-2}}\cdots \widetilde{U}(1)A_{w_0},\quad
\mathbf w=w_0\dots w_{n-1}.
\end{equation}
We define the operators $\widetilde U_{\mathcal E}$, $\widetilde U_F$ using
$\widetilde U_{\mathbf w}$ similarly to~\eqref{e:A-E-def}, \eqref{e:A-F-def}.

Let the partition $\mathcal X\sqcup\mathcal Y\subset\mathscr A_\star^N$ be defined in~\eqref{e:XY-def},
where we fix $\alpha>0$ in~\S\ref{s:dwe-uncontrolled} below.
We prove Proposition~\ref{t:propagator-decay} by establishing
decay of $\widetilde U_{\mathcal X}$ and $\widetilde U_{\mathcal Y}$.

%%%%%%%%%%%%%%%%%%%%%%%%%%%%%%%%%%%%%%%%%%%%%%%%%%%%%%%%%%%%%%%%%%%%%%%%%%%%%%%%
\subsubsection{Controlled words}
  \label{s:dwe-controlled}

To bound the norm of $\widetilde U_{\mathcal Y}$, we first
use the inequalities~\eqref{e:damp-a-a} to estimate
$\widetilde U_{\mathcal Z}$, where $\mathcal Z\subset \mathscr A_\star^{N_0}$
is defined in~\eqref{e:Z-def}:
%%%%%%%%%%%%%%%%%%%%%%%%%%%%%%%%%%%%%%%%%%%%%%%%%%%%%%%%%%%%%%%%%%%%%%%%%%%%%%%%
\begin{lemm}
	\label{l:Z-damp}
We have
\begin{equation}
  \label{e:Z-damp}
\|\widetilde{U}_\mathcal{Z}\|_{L^2\to L^2}\leq h^{\alpha_1}+\mathcal O(h^{1/3-})\quad\text{where}\quad
\alpha_1:=\textstyle{\alpha\eta\over 6\Lambda_1}>0.
\end{equation}
\end{lemm}
%%%%%%%%%%%%%%%%%%%%%%%%%%%%%%%%%%%%%%%%%%%%%%%%%%%%%%%%%%%%%%%%%%%%%%%%%%%%%%%%
\begin{proof}
Since $U(N_0)$ is unitary, we have
$\|\widetilde{U}_\mathcal{Z}\|_{L^2\to L^2}=\|\widetilde{A}_{\mathcal Z}\|_{L^2\to L^2}$.
The symbol $\widetilde{a}_{\mathcal Z}$ is given by
\begin{equation*}
\widetilde{a}_{\mathcal{Z}}=\sum_{\mathbf{w}\in\mathcal{Z}}\widetilde{a}_{\mathbf{w}}=\sum_{\mathbf{w}\in\mathcal{Z}}\prod_{j=0}^{N_0-1}(\widetilde{a}_{w_j}\circ\varphi_j).
\end{equation*}
By the definition \eqref{e:Z-def}, each $\mathbf w\in \mathcal Z$ has at least $\alpha N_0$
letters equal to~1. Therefore by~\eqref{e:damp-a-a},
recalling the definition~\eqref{e:prop-times} of~$N_0$,
\begin{equation*}
\begin{split}
|\widetilde{a}_{\mathcal{Z}}|
\leq e^{-\eta\alpha N_0}\sum_{\mathbf{w}\in\mathcal{Z}}\prod_{j=0}^{N_0-1}(a_{w_j}\circ\varphi_j)
\leq &\; e^{-\eta\alpha N_0}\sum_{\mathbf{w}\in\mathscr A_\star^{N_0}}\prod_{j=0}^{N_0-1}(a_{w_j}\circ\varphi_j)
\\
=&\; e^{-\eta\alpha N_0}\prod_{j=0}^{N_0-1}(a_1+a_\star)\circ\varphi_j\leq h^{\alpha_1}.
\end{split}
\end{equation*}
By Lemma~\ref{l:cq-log} (which still applies by Lemma~\ref{l:still-fine}) we have
$\widetilde a_{\mathcal Z}\in S^{\comp}_{1/6+}(T^*M)$ and
$\widetilde A_{\mathcal Z}=\Op_h(\widetilde a_{\mathcal Z})+\mathcal O(h^{1/2-})_{L^2\to L^2}$.
Then by~\eqref{e:precise-norm}
we have $\|\widetilde A_{\mathcal Z}\|\leq h^{\alpha_1}+\mathcal O(h^{1/3-})$,
finishing the proof.
\end{proof}
%%%%%%%%%%%%%%%%%%%%%%%%%%%%%%%%%%%%%%%%%%%%%%%%%%%%%%%%%%%%%%%%%%%%%%%%%%%%%%%%
Armed with Lemma~\ref{l:Z-damp} we now estimate the norm of $\widetilde U_{\mathcal Y}$:
%%%%%%%%%%%%%%%%%%%%%%%%%%%%%%%%%%%%%%%%%%%%%%%%%%%%%%%%%%%%%%%%%%%%%%%%%%%%%%%%
\begin{prop}
	\label{l:Y-damp}
With $\alpha_1>0$ defined in~\eqref{e:Z-damp}, we have
\begin{equation}
\|\widetilde{U}_\mathcal{Y}\|_{L^2\to L^2}\leq \mathcal O(h^{\alpha_1})+\mathcal O(h^{1/3-}).
\end{equation}
\end{prop}
%%%%%%%%%%%%%%%%%%%%%%%%%%%%%%%%%%%%%%%%%%%%%%%%%%%%%%%%%%%%%%%%%%%%%%%%%%%%%%%%
\begin{proof}
From the definition~\eqref{e:XY-def} of~$\mathcal Y$ we have
$$
\widetilde{U}_{\mathcal{Y}}=\sum_{\ell=1}^{6\Lambda+1}\widetilde{U}_{\mathcal{Z}^{\complement}}^{6\Lambda+1-\ell}\widetilde{U}_{\mathcal{Z}}\widetilde{U}_{\mathscr A_\star^{N_0}}^{\ell-1}.
$$
By~\eqref{e:A-E-norm} we have
$$
\|\widetilde U_{\mathcal Z^{\complement}}\|_{L^2\to L^2}=\|\widetilde A_{\mathcal Z^{\complement}}\|
\leq 1+\mathcal O(h^{1/3-})
$$
and same is true for~$\widetilde U_{\mathscr A_\star^{N_0}}$. Using Lemma~\ref{l:Z-damp}
and the triangle inequality
we then have
$$
\|\widetilde U_{\mathcal Y}\|_{L^2\to L^2}
\leq (6\Lambda+1)h^{\alpha_1}+\mathcal O(h^{1/3-}),
$$
finishing the proof.
\end{proof}
%%%%%%%%%%%%%%%%%%%%%%%%%%%%%%%%%%%%%%%%%%%%%%%%%%%%%%%%%%%%%%%%%%%%%%%%%%%%%%%%

%%%%%%%%%%%%%%%%%%%%%%%%%%%%%%%%%%%%%%%%%%%%%%%%%%%%%%%%%%%%%%%%%%%%%%%%%%%%%%%%
\subsubsection{Uncontrolled words and end of the proof}
  \label{s:dwe-uncontrolled}
  
We now finish the proof of Proposition~\ref{t:propagator-decay}
and thus of Theorem~\ref{t:dwe}. Similarly to~\cite[\S3.5]{JinDWE},
using the identities $\widetilde U_{\mathscr A_{\star}^N}=(\widetilde U(1)(I-A_0))^N$
and $A_0\Pi=0$ we have
\begin{equation}
  \label{e:david-1}
\widetilde U(N)\Pi=\widetilde U_{\mathscr A_{\star}^N}\Pi+\mathcal O(h^{1-})_{L^2\to L^2},\quad
\widetilde U_{\mathscr A_{\star}^N}=\widetilde U_{\mathcal X}+\widetilde U_{\mathcal Y}.
\end{equation}
Let $\beta>0$ be the constant in Proposition~\ref{l:longdec-0}
for the operators $\widetilde A_{\mathbf w}$, $\mathbf w\in\mathscr A_\star^N$. Choose
$\alpha>0$ satisfying~\eqref{e:alpha-chosen}. Using the triangle inequality,
Proposition~\ref{l:longdec-0}, and Lemma~\ref{l:X-count}, we have
\begin{equation}
  \label{e:david-2}
\|\widetilde U_{\mathcal X}\|_{L^2\to L^2}=\|\widetilde A_{\mathcal X}\|_{L^2\to L^2}=\mathcal O(h^{\beta/2}).
\end{equation}
Combining~\eqref{e:david-1}, \eqref{e:david-2}, and Proposition~\ref{l:Y-damp},
we get Proposition~\ref{t:propagator-decay} with
$$
\beta_1:=\min(\textstyle{\beta\over 2},\alpha_1,\textstyle{1\over 4})>0.
$$

%%%%%%%%%%%%%%%%%%%%%%%%%%%%%%%%%%%%%%%%%%%%%%%%%%%%%%%%%%%%%%%%%%%%%%%%%%%%%%%%
%%%%%%%%%%%%%%%%%%%%%%%%%%%%%%%%%%%%%%%%%%%%%%%%%%%%%%%%%%%%%%%%%%%%%%%%%%%%%%%%
\section{Decay for long words}
\label{s:long-word-fup}

In this section we prove the Proposition~\ref{l:longdec-0}, relying on propagation results up to the local Ehrenfest time
(Propositions~\ref{l:ehrenfest-prop}, \ref{l:ehrenfest-prop-sum}) established in~\S\ref{s:ops-Aq} below and on the fractal uncertainty
principle (Proposition~\ref{l:fup-2}).

Recall from~\eqref{e:prop-times} the short and long logarithmic propagation times
$N_0$ and~$N$. Put
\begin{equation}
  \label{e:N-1-def}
N_1:=N-N_0=6\Lambda N_0\geq {\log(1/h)\over\Lambda_0}.
\end{equation}
We will prove the following equivalent version of Proposition~\ref{l:longdec-0}
in terms of products of two operators corresponding to propagation
forward and backwards in time (see~\eqref{e:A-pm-def}
for the definitions of $A^-_{\mathbf v}$, $A^+_{\mathbf w}$):
%%%%%%%%%%%%%%%%%%%%%%%%%%%%%%%%%%%%%%%%%%%%%%%%%%%%%%%%%%%%%%%%%%%%%%%%%%%%%%%%
\begin{prop}
  \label{l:longdec-1}
Let the assumptions~(1)--(5) of~\S\ref{s:proofs-notation} hold and $\varepsilon_0>0$
be small enough depending only on~$(M,g)$.
Then there exists $\beta>0$ depending only on $\mathcal V_1,\mathcal V_\star$
and there exists $C>0$ depending only on $A_1,A_\star$ such that
for all $\mathbf v\in \mathscr A_\star^{N_0}$, $\mathbf w\in \mathscr A_\star^{N_1}$
\begin{equation}
  \label{e:longdec-1}
\|A^-_{\mathbf v}A^+_{\mathbf w}\|_{L^2(M)\to L^2(M)}\leq Ch^\beta.
\end{equation}
\end{prop}
%%%%%%%%%%%%%%%%%%%%%%%%%%%%%%%%%%%%%%%%%%%%%%%%%%%%%%%%%%%%%%%%%%%%%%%%%%%%%%%%
\Remark The smallness of $\varepsilon_0$ is used in several places in the proof,
in particular at the beginning of~\S\ref{s:refined-partition},
in~\S\ref{s:longtime}, in Lemma~\ref{l:cluster},
in the beginning of~\S\ref{s:fup-normal},
and in Lemma~\ref{l:loca+}.
Roughly speaking, we need $\varepsilon_0$ to be much smaller than
the sizes of local stable/unstable leaves from~\S\ref{s:stun}
and the domains of the local coordinates constructed in Lemma~\ref{l:stun-straight}.

To show that Proposition~\ref{l:longdec-1} implies Proposition~\ref{l:longdec-0}
we note that each word in $\mathscr A_\star^N$ can be written as a concatenation
$\overline{\mathbf w}\mathbf v$ where $\mathbf v\in\mathscr A_\star^{N_0}$, $\mathbf w\in\mathscr A_\star^{N_1}$
and $\overline{\mathbf w}=w_{N_1}\dots w_2w_1$ is the reverse of $\mathbf w=w_1w_2\dots w_{N_1}$.
We have by~\eqref{e:word-concat}
$$
A_{\overline{\mathbf w}\mathbf v}=A^-_{\overline{\mathbf w}\mathbf v}=
U(-N_1) A^-_{\mathbf v}A^+_{\mathbf w}U(N_1).
$$
Since $U(N_1)$ is unitary, the bound~\eqref{e:longdec-1}
implies that $\|A_{\overline{\mathbf w}\mathbf v}\|_{L^2(M)\to L^2(M)}\leq Ch^\beta$
which gives Proposition~\ref{l:longdec-0}.

%%%%%%%%%%%%%%%%%%%%%%%%%%%%%%%%%%%%%%%%%%%%%%%%%%%%%%%%%%%%%%%%%%%%%%%%%%%%%%%%
\subsection{Outline of the proof}
  \label{s:outline}

We provide here an informal explanation of the proof of Proposition~\ref{l:longdec-1}.
For this we use a naive version of the classical/quantum correspondence,
thinking of $A^-_{\mathbf v},A^+_{\mathbf w}$ as quantizations of the symbols
$a^-_{\mathbf v}$, $a^+_{\mathbf w}$ defined in~\eqref{e:a-pm-def}
and restricting the analysis to the cosphere bundle $S^*M$. We also make the simplifying assumption
\begin{equation}
  \label{e:simplifier}
\mathbf v=\underbrace{\star\ldots\star}_{N_0\text{ times}},\quad
\mathbf w=\underbrace{\star\ldots\star}_{N_1\text{ times}}.
\end{equation}
Recall from~\eqref{e:V+-} that $a^-_{\mathbf v},a^+_{\mathbf w}$ are supported in the sets
$\mathcal V^-_{\mathbf v},\mathcal V^+_{\mathbf w}$ which under the assumption~\eqref{e:simplifier}
have the form
$$
\mathcal V^-_{\mathbf v}=\bigcap_{j=0}^{N_0-1}\varphi_{-j}(\mathcal V_\star),\quad
\mathcal V^+_{\mathbf w}=\bigcap_{j=1}^{N_1}\varphi_j(\mathcal V_\star).
$$
%%%%%%%%%%%%%%%%%%%%%%%%%%%%%%%%%%%%%%%%%%%%%%%%%%%%%%%%%%%%%%%%%%%%%%%%%%%%%%%%
\begin{figure}
\includegraphics[height=3.45cm]{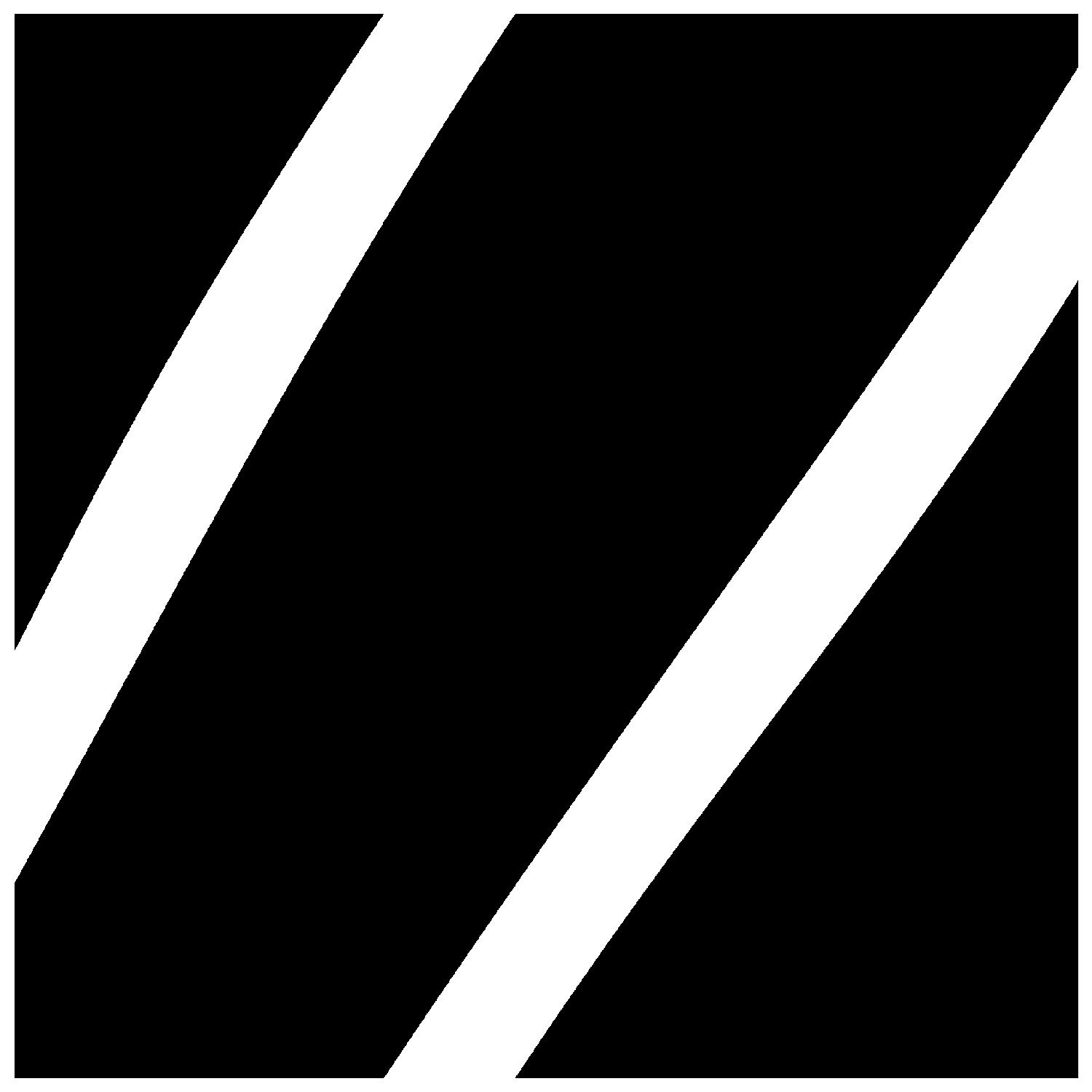}
\includegraphics[height=3.45cm]{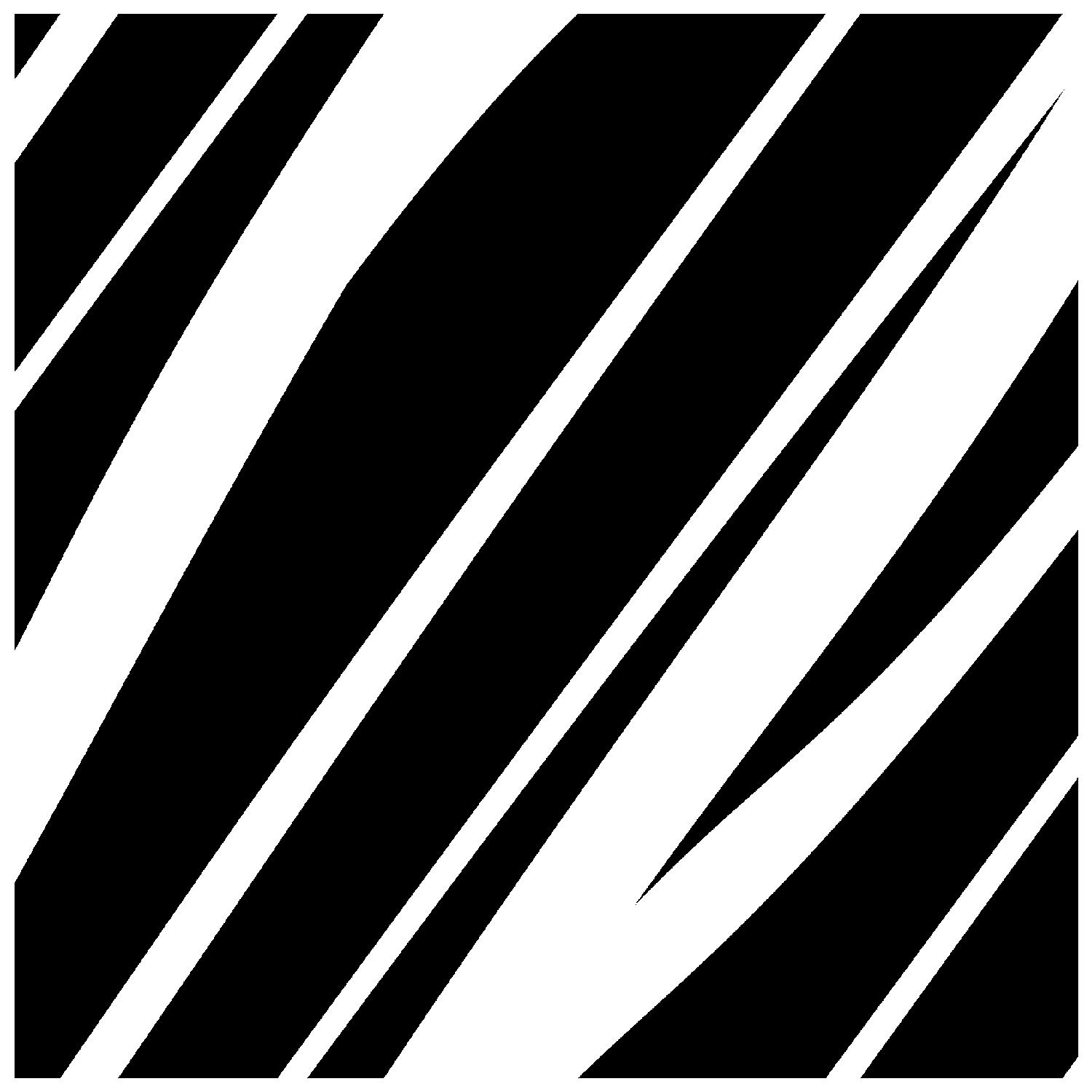}
\includegraphics[height=3.45cm]{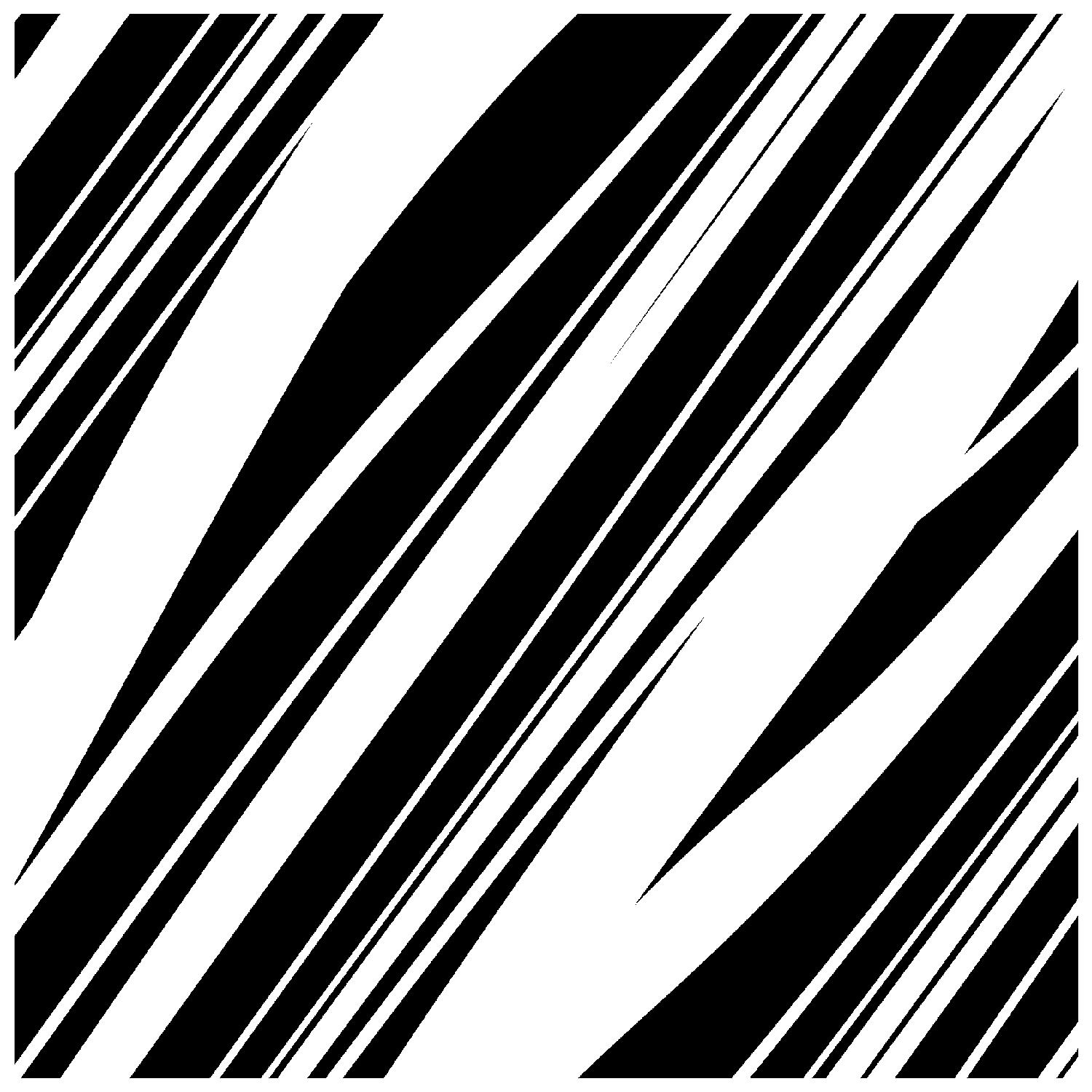}
\includegraphics[height=3.45cm]{vfmatfig/catfor4.jpg}
\caption{The sets $\bigcap_{j=1}^n \varphi_j(\mathcal V_\star)$
for $n=1,2,3,4$, pictured with the flow direction removed.
See also Figures~\ref{f:porosity-bas} (page~\pageref{f:porosity-bas})
and~\ref{f:moderate} (page~\pageref{f:moderate}).
}
\label{f:demo}
\end{figure}%
%%%%%%%%%%%%%%%%%%%%%%%%%%%%%%%%%%%%%%%%%%%%%%%%%%%%%%%%%%%%%%%%%%%%%%%%%%%%%%%%
We call the complement of $\mathcal V_\star$ (which has nonempty interior by assumption~(3) in~\S\ref{s:proofs-notation}) the \emph{hole}. 
Then $\rho\in \mathcal V^-_{\mathbf v}$
if the geodesic starting at $\rho$ does not enter the hole at least until the time~$N_0$
in the future, more precisely $\varphi_j(\rho)\in\mathcal V_\star$ for 
all integer $j\in [0,N_0-1]$. Similarly $\rho\in\mathcal V^+_{\mathbf w}$ if
that geodesic does not enter the hole up to the time $N_1$ in the past,
more precisely $\varphi_j(\rho)\in\mathcal V_\star$ for all integer $j\in [-N_1,-1]$.
See Figure~\ref{f:demo}.
Viewing $A^-_{\mathbf v},A^+_{\mathbf w}$ as operators which microlocalize
to $\mathcal V^-_{\mathbf v},\mathcal V^+_{\mathbf w}$, our goal is to use the
fractal uncertainty principle to show that microlocalizations to these two sets
are incompatible with each other, this incompatibility taking the form of the norm bound~\eqref{e:longdec-1}.

Recall from~\S\ref{s:stun} that $S^*M$ is foliated by (local) weak unstable leaves.
We use this foliation to partition $\mathcal V^+_{\mathbf w}$ into \emph{clusters}
$$
\mathcal V^+_{\mathbf w}=\bigsqcup_r \mathcal V^+_{\mathbf w,r}
$$
where each $\mathcal V^+_{\mathbf w,r}$ lies $\mathcal O(h^{2/3})$ close to a certain local weak unstable leaf (the construction of the partition uses the Lipschitz regularity of the unstable foliation).
On the operator side this gives the decomposition (see Lemma~\ref{l:cluster}
and~\eqref{e:clusterop})
\begin{equation}
  \label{e:clusterop-outline}
A^-_{\mathbf v}A^+_{\mathbf w}=\sum_r A^-_{\mathbf v}A^+_{\mathbf w,r}.
\end{equation}
If two clusters $\mathcal V^+_{\mathbf w,r},\mathcal V^+_{\mathbf w,r'}$
are ``sufficiently disjoint'', then the corresponding operators in~\eqref{e:clusterop-outline}
satisfy the almost orthogonality bounds
\begin{equation}
  \label{e:ao-outline}
(A^-_{\mathbf v}A^+_{\mathbf w,r})^* A^-_{\mathbf v}A^+_{\mathbf w,r'},\
A^-_{\mathbf v}A^+_{\mathbf w,r'}(A^-_{\mathbf v}A^+_{\mathbf w,r})^*
\ =\ \mathcal O(h^\infty)_{L^2\to L^2}.
\end{equation}
This follows from the classical/quantum correspondence and the fact that
\begin{equation}
  \label{e:almoster}
  h^{2/3}\cdot h^{1/6}\gg h
\end{equation}
where $h^{2/3}$ is the minimal distance between disjoint clusters (in the stable direction), while
$h^{1/6}$ is the minimal scale of oscillation of the symbol $a^+_{\mathbf v}$, along the unstable direction.
The almost orthogonality bounds are proved in Lemma~\ref{l:faror}, and the inequality \eqref{e:almoster} appears in~\eqref{e:tylen} in the proof. 
The remark following that Lemma gives an informal argument on how the inequality \eqref{e:almoster} leads to almost orthogonality. (Note that in~\S\ref{s:fup-endgame} the ``cluster objects'' $\mathcal V^+_{\mathbf w,r},A^+_{\mathbf w,r}$ are replaced by the more flexible objects $\mathcal V^+_{\mathcal Q},A^+_{\mathcal Q}$.)

Using~\eqref{e:clusterop-outline}, the Cotlar--Stein Theorem, and the fact that each cluster is disjoint from all but boundedly many other clusters, we reduce the estimate~\eqref{e:longdec-1} to a bound for every single cluster (see Proposition~\ref{l:longdec-3})
\begin{equation}
  \label{e:longdec-cluss}
\|A^-_{\mathbf v}A^+_{\mathbf w,r}\|_{L^2(M)\to L^2(M)}\leq Ch^\beta.
\end{equation}
We henceforth fix some cluster $\mathcal V^+_{\mathbf w,r}$, contained
in an $\mathcal O(h^{2/3})$ sized neighborhood of the weak unstable
leaf $W_{0u}(\rho_0)$ for some $\rho_0\in S^*M$. We use the symplectic coordinates
$\varkappa:(x,\xi)\mapsto (y,\eta)$ centered at $\rho_0$ which were
constructed in Lemma~\ref{l:stun-straight},
see~\eqref{e:conj-kappa}.
We conjugate $A^-_{\mathbf v},A^+_{\mathbf w,r}$ by Fourier integral operators
quantizing $\varkappa$ (see~\S\ref{s:micro-conjugation}).
This produces (still under our naive view of the classical/quantum correspondence)
pseudodifferential operators which microlocalize to the sets
$\varkappa(\mathcal V^-_{\mathbf v}),\varkappa(\mathcal V^+_{\mathbf w,r})$.
The latter are subsets of $T^*\mathbb R^2$ but we reduce them to subsets of $T^*\mathbb R$
by restricting to $\varkappa(S^*M)=\{\eta_2=1\}$ and projecting along the flow direction
$\partial_{y_2}$. Denote the resulting sets by
$\Theta^-,\Theta^+\subset T^*\mathbb R$. The informal argument above (see Lemma~\ref{l:2D-FUP} for more details
on reducing from $T^*\mathbb R^2$ to $T^*\mathbb R$ and Lemmas~\ref{l:loca+}--\ref{l:loca-}
for microlocalization of the conjugated operators)
reduces~\eqref{e:longdec-cluss} to the estimate
\begin{equation}
  \label{e:clustrofob}
\|\mathcal A^-\mathcal A^+\|_{L^2(\mathbb R)\to L^2(\mathbb R)}\leq Ch^\beta
\end{equation}
where $\mathcal A^\pm$ are operators on $L^2(\mathbb R)$ which microlocalize to
the sets $\Theta^\pm$ described above.

We next understand the structure of the sets $\Theta^\pm$.
The set $\mathcal V^+_{\mathbf w,r}$ is `smooth' along the flow and unstable directions:
if $\rho,\rho'$ lie on the same local weak unstable leaf then
the trajectories $\varphi_j(\rho),\varphi_j(\rho')$, $j\leq 0$, stay close to each other,
thus $\rho\in \mathcal V^+_{\mathbf w,r}$ if and only if $\rho'\in\mathcal V^+_{\mathbf w,r}$
unless the boundary of the hole was involved. This is easy to see on Figure~\ref{f:demo}
with the `strokes' along the unstable direction (corresponding to unstable rectangles
introduced below); see Lemma~\ref{l:blurry}
for a rigorous statement. We then embed $\mathcal V^+_{\mathbf w,r}$ into a union of many
\emph{`unstable rectangles'}, each of which is the $h^\tau$-neighborhood of a local weak unstable leaf,
with $\tau<1$, defined in~\eqref{e:tau-delta-def} below, chosen very close to~1.
This uses the inequality~\eqref{e:N-1-def} which ensures that the thickness of each
`stroke' is smaller than~$h$. On the operator side
unstable rectangles correspond to individual summands $A^+_{\mathbf q}$ in the operator
$A^+_{\mathcal Q'_n(\mathbf w,e)}$ introduced in~\S\ref{s:moderate-reduction}.
See also Figure~\ref{f:moderate} (page~\pageref{f:moderate}).

The specific unstable rectangles which are part of $\mathcal V^+_{\mathbf w,r}$ are distributed
in a porous way, which is where we use that the hole has nonempty interior (see Lemma~\ref{l:porosity-intersected}
which is an application of Lemma~\ref{l:porosity-basic}).
The set $\Theta^+$ is a union of components arising from the images of these rectangles under~$\varkappa$.
Using the fact that $\mathcal V^+_{\mathbf w,r}$ is within $\mathcal O(h^{2/3})$
of the leaf $W_{0u}(\rho_0)$ and the properties of~$\varkappa$ in Lemma~\ref{l:stun-straight}
(whose proof used the $C^{3/2}$ regularity of the unstable foliation),
we show that each component of $\Theta^+$ is contained in a `horizontal rectangle' 
of dimensions $1\times h^\tau$, stretched along the $y_1$ direction~-- see
Lemma~\ref{l:close-unstable} and Figure~\ref{f:porositors}. This gives
\begin{equation}
  \label{e:poro+over}
\Theta^+\subset \{(y_1,\eta_1)\mid \eta_1\in \Omega^+\}
\end{equation}
where $\Omega^+\subset\mathbb R$ is porous on scales $h^\tau$ to~1~--
see Lemma~\ref{l:poro+}.

As for the set $\mathcal V^-_{\mathbf v}$, it can be embedded into a union
of stable rectangles of thickness $h^{1/(6\Lambda)}$ each (here we use the definition of~$N_0$).
The corresponding components of~$\Theta^-$ look like rectangles of thickness $h^{1/(6\Lambda)}$
with the long axis aligned along the stable direction, thus transverse to the $\partial_{y_1}$ direction. Because the stable direction is usually not vertical, the projection of each of these rectangles onto the $y_1$ axis might be large (e.g. it could be an interval of a size~$1$). However, we only need to understand the intersection of $\Theta^-$ with a neighborhood of $\Theta^+$.
Since $\mathcal V^+_{\mathbf w,r}$ lies $\mathcal O(h^{2/3})$ close to the leaf $W_{0u}(\rho_0)$,
$\Theta^+$ lies $\mathcal O(h^{2/3})$ close to $\{\eta_1=0\}$, in particular
$\Theta^+\subset \{|\eta_1|\leq h^{1/6}\}$.
The intersection of each component of $\Theta^-$ with $\{|\eta_1|\leq h^{1/6}\}$
is a rectangle of thickness $h^{1/(6\Lambda)}$ and height $h^{1/6}\ll h^{1/(6\Lambda)}$,
thus its projection onto the $y_1$ variable is now contained in an $h^{1/(6\Lambda)}$ sized interval,
see Figure~\ref{f:porositors}.
This implies that
\begin{equation}
  \label{e:poro-over}
\Theta^-\subset \{(y_1,\eta_1)\mid y_1\in\Omega^-\}
\end{equation}
where $\Omega^-\subset\mathbb R$ is porous on scales $h^{1/(6\Lambda)}$ to~1~-- see Lemma~\ref{l:poro-}.

Together~\eqref{e:poro+over} and~\eqref{e:poro-over} show that in~\eqref{e:clustrofob},
we may replace $\mathcal A^+$ by the Fourier multiplier $\indic_{\Omega^+}(hD_{y_1})$ and
$\mathcal A^-$ by the multiplication operator $\indic_{\Omega^-}(y_1)$.
The resulting estimate follows by the fractal uncertainty principle, in the version given
by Proposition~\ref{l:fup-2}, see also Lemma~\ref{l:2D-FUP}.
Here we use that there is a nontrivial overlap in the porosity scales of $\Omega^+$ and $\Omega^-$, namely
\begin{equation}
  \label{e:almoster2}
h^\tau\cdot h^{1/(6\Lambda)}\ll h,
\end{equation}
see~\eqref{e:porosity-finally-meets}. This is where we use that $\tau$ is chosen very close to~1.

To make the above explanations into a rigorous proof, we in particular need to make precise
the classical/quantum correspondence naively used above. This is complicated since
to study $A^+_{\mathbf w}$ we need to go beyond the Ehrenfest time, that is the expansion rate of the geodesic flow for time $N_1$ is much larger than $h^{-1/2}$,
therefore $A^+_{\mathbf w}$ will not lie in the mildly exotic
pseudodifferential calculus $\Psi^{\comp}_\delta$ of~\S\ref{s:mildly-exotic}. To overcome this problem we use several ideas:
\begin{itemize}
\item We write $a_\star=a_2+\dots+a_Q$, $A_\star=A_2+\dots+A_Q$ where the supports of the
symbols $a_2,\dots,a_Q$ are small enough to form a dynamically fine partition
(\S\ref{s:refined-partition}). We next
write $A^+_{\mathbf w}$ as the sum of polynomially many in~$h$ terms
of the form $A^+_{\mathbf q}$ where $\mathbf q$ are words in the alphabet
$\{2,\dots,Q\}$. One advantage of this splitting is that each $\mathbf q$ has a well-defined
local expansion rate of the flow, see~\eqref{e:jacobians-same}.
\item If $\mathbf q$ has expansion rate no more than $h^{-2\tau}$ (i.e. the length of~$\mathbf q$
is below the local double Ehrenfest time) then we can conjugate
$A^+_{\mathbf q}$ by $U(t)$ for an appropriate choice of~$t$ to get a pseudodifferential
operator in the mildly exotic calculus $\Psi^{\comp}_{\delta+}$, $\delta:=\tau/2$.
Here we use Egorov's Theorem up to local Ehrenfest time and the fact that $\tau<1$,
see~\S\ref{s:local-ehrenfest}. This technique is used in the proof of the almost orthogonality
statements~\eqref{e:ao-outline} and also to show that the operators
$A^+_{\mathbf w,r}$ corresponding to individual clusters are bounded on~$L^2$
almost uniformly in~$h$
(see~\eqref{e:apn-2}). We also use mildly exotic symbol calculus to show microlocalization
of $A^-_{\mathbf v}$ in Lemma~\ref{l:loca-}.
\item For microlocalization of $A^+_{\mathbf w,r}$ (Lemma~\ref{l:loca+}) we again
write it as the sum of individual terms $A^+_{\mathbf q}$. We then study each of these
using the long logarithmic time hyperbolic parametrix of~\cite{AnAnn,AN07,NZ09}~-- see~\S\ref{s:longtime}.
\end{itemize}

%%%%%%%%%%%%%%%%%%%%%%%%%%%%%%%%%%%%%%%%%%%%%%%%%%%%%%%%%%%%%%%%%%%%%%%%%%%%%%%%
\subsection{A refined partition}
  \label{s:refined-partition}
  
For each $\mathbf w\in\mathscr A_\star^\bullet$ the supports of
$a^\pm_{\mathbf w}$ can be rather large, including many trajectories of the flow;
this is due to the fact that $\supp a_\star$ typically contains the entire
$S^*M$ minus a fixed small set.
It will be convenient to break the symbols~$a^\pm_{\mathbf w}$
and the operators~$A^\pm_{\mathbf w}$ into smaller pieces,
each of which is `dynamically simple'. To do this,
we let $\varepsilon_0>0$ be small enough so that Lemma~\ref{l:stun-main} holds and
write
\begin{equation}
  \label{e:finer-partition}
a_\star=a_2+\dots+a_Q,\quad
A_\star=A_2+\dots+A_Q
\end{equation}
where $Q$ is some $h$-independent number and:
\begin{enumerate}
\item $a_2,\dots,a_Q\in C_c^\infty(T^*M\setminus 0;[0,1])$ are $h$-independent;
\item $\supp a_q\subset \mathcal V_q\cap \{\textstyle{1\over 4}<|\xi|_g<4\}$
for all $q=2,\dots,Q$ where $\mathcal V_q\subset \mathcal V_\star$
are some conic open sets;
\item the diameter of each $\mathcal V_q\cap S^*M$ with respect to $d(\bullet,\bullet)$
is smaller than $\varepsilon_0$;
\item $A_2,\dots,A_Q\in\Psi^{-\infty}_h(M)$ satisfy for $q=2,\dots,Q$
\begin{equation}
  \label{e:A-q-prop}
\sigma_h(A_q)=a_q,\quad
\WFh(A_q)\subset \mathcal V_q\cap \{\textstyle{1\over 4}<|\xi|_g<4\}.
\end{equation}
\end{enumerate}
Following the proof of Lemma~\ref{l:quantum-partition} it is straightforward to see how to construct decompositions~\eqref{e:finer-partition}
with the above properties, given $a_\star,A_\star,\varepsilon_0$.

Denote
$$
\mathscr A:=\{1,\dots,Q\},
$$
then the properties~(1)--(4) above hold for all $q\in\mathscr A$ (indeed, for $q=1$
they follow from the assumptions of~\S\ref{s:proofs-notation}), except we do not have $\mathcal V_1\subset\mathcal V_\star$. We also note that
$$
a_1+a_2+\cdots+a_Q=a_1+a_\star\leq 1.
$$
Similarly to~\S\ref{s:proofs-notation} we define
the set of words $\mathscr A^\bullet$ over the alphabet $\mathscr A$. For $\bq\in\mathscr A^\bullet$
we define the symbols~$a^\pm_{\mathbf q}$, the conic sets~$\mathcal V^\pm_{\mathbf q}$,
and the operators~$A^\pm_{\mathbf q}$
following~\eqref{e:a-pm-def}, \eqref{e:V+-}, and~\eqref{e:A-pm-def}. We will also use the notation 
$A^\pm_{\mathcal E},\ A^\pm_F$ from~\eqref{e:A-E-def}, \eqref{e:A-F-def}, this time for $\cE$ which is a subset of $\mathscr A^\bullet$ (resp. $F$ which is a function on $\mathscr A^\bullet$).

Since $\sup |a_q|\leq 1$, we see from~\eqref{e:precise-norm} (with $\delta=0$)
that $\|A_q\|_{L^2\to L^2}\leq 1+Ch^{1/2}$. Therefore we have for any fixed constant $C_0$
and small enough~$h$ depending on $C_0$
\begin{equation}
  \label{e:A-q-bdd}
\|A^\pm_{\mathbf q}\|_{L^2\to L^2}\leq 2\quad\text{for all}\quad \mathbf q\in\mathscr A^n,\quad
n\leq C_0\log(1/h).
\end{equation}

%%%%%%%%%%%%%%%%%%%%%%%%%%%%%%%%%%%%%%%%%%%%%%%%%%%%%%%%%%%%%%%%%%%%%%%%%%%%%%%%
\subsubsection{Jacobians for the refined partition}
\label{s:refined-jacobians}

To each refined word $\bq \in \sA^n$ we associate the minimal Jacobians
\begin{equation}
  \label{e:word-J-def}
\cJ^-_{\bq}:= \inf_{\rho\in \cV^-_{\bq}} J^u_n(\rho),\quad
\cJ^+_{\bq}:= \inf_{\rho\in \cV^+_{\bq}} J^s_{-n}(\rho)
\end{equation}
where $J^u_n(\rho),J^s_{-n}(\rho)$ are defined in~\eqref{e:stun-J-def}.
Since the Jacobians $J^u$, $J^s$ are homogeneous of degree~0 on $T^*M\setminus 0$, one can replace
$\mathcal V^\pm_{\mathbf q}$ by $\mathcal V^\pm_{\mathbf q}\cap S^*M$
in~\eqref{e:word-J-def}. Note that the sets~$\mathcal V^\pm_{\mathbf q}$ might be empty
in which case we have $\mathcal J^\pm_{\mathbf q}=\infty$.

It follows from~\eqref{e:Lambda-0-1} that
the Jacobians $\mathcal J^\pm_{\mathbf q}$,
$\mathbf q\in\mathscr A^n$, grow exponentially
in~$n$:
\begin{equation}
  \label{e:jacobian-bounds}
\begin{aligned}
\mathcal V^-_{\mathbf q}\neq\emptyset
&\quad\Longrightarrow\quad
e^{\Lambda_0 n}\leq \mathcal J^-_{\mathbf q}\leq e^{\Lambda_1 n},\\
\mathcal V^+_{\mathbf q}\neq\emptyset
&\quad\Longrightarrow\quad
e^{\Lambda_0 n}\leq \mathcal J^+_{\mathbf q}\leq e^{\Lambda_1 n}.
\end{aligned}  
\end{equation}
Denote
\begin{equation}
  \label{e:q-prime}
\mathbf q':=q_1\dots q_{n-1}\quad\text{where}\quad
\mathbf q=q_1\dots q_n\in \mathscr A^n,\quad
n>0.
\end{equation}
Then we have for each $\mathbf q\in\mathscr A^n$, $n> 0$
\begin{equation}
  \label{e:jacobian-grows}
\mathcal J^\pm_{\mathbf q}\geq e^{\Lambda_0}\mathcal J^\pm_{\mathbf q'}.
\end{equation}
Indeed, for each $\rho\in \mathcal V^-_{\mathbf q}$ we have
$\rho\in \mathcal V^-_{\mathbf q'}$ and thus
$$
J^u_n(\rho)=J^u_1(\varphi_{n-1}(\rho))J^u_{n-1}(\rho)\geq e^{\Lambda_0} \mathcal J^-_{\mathbf q'}
$$
where the last inequality used~\eqref{e:Lambda-0-1}.
This proves~\eqref{e:jacobian-grows} for $\mathcal J^-$, with the case of $\mathcal J^+$ handled similarly.

Next, parts~(5)--(6) of~Lemma~\ref{l:stun-main} imply that
the quantities $\mathcal J^\pm_{\mathbf q}$ give the order of the
expansion rate of the flow $\varphi_{\mp n}$
at \emph{every} point in $\mathcal V^\pm_{\mathbf q}$:
\begin{equation}
  \label{e:jacobians-same}
\begin{aligned}
  J^u_n(\rho)\sim\mathcal J^-_{\mathbf q}&\quad\text{for all}\quad \rho\in\mathcal V^-_{\mathbf q},\\
  J^s_{-n}(\rho)\sim\mathcal J^+_{\mathbf q}&\quad\text{for all}\quad \rho\in\mathcal V^+_{\mathbf q} 
\end{aligned}
\end{equation}
where $A\sim B$ means that $C^{-1}A\leq B\leq CA$ for some constant~$C$
depending only on~$(M,g)$ (in particular, independent of $n$ and~$\mathbf q$).
More precisely, Lemma~\ref{l:stun-main} shows that
$J^u_{n-1}(\rho)\sim J^u_{n-1}(\tilde\rho)$ for all
$\rho,\tilde\rho\in \mathcal V^-_{\mathbf q}$;
using that $J^u_n(\rho)\sim J^u_{n-1}(\rho)$ we obtain
the first statement in~\eqref{e:jacobians-same}.
The second statement is obtained similarly using that
$J^s_{-n}(\rho)\sim J^s_{1-n}(\varphi_{-1}(\rho))$.
Note that~\eqref{e:jacobians-same} uses that the diameter of each $\mathcal V_q\cap S^*M$ is smaller than $\varepsilon_0$, in particular it is typically false for the sets $\mathcal V^\pm_{\mathbf v}$
corresponding to the unrefined partition defined in~\eqref{e:V+-}.

From~\eqref{e:jacobians-same} and~\eqref{e:exp-rate}
we derive the following bounds:
\begin{align}
  \label{e:global-expansion-rate-1}
\sup_{\rho\in\mathcal V^-_{\mathbf q}\cap \{{1\over 4}\leq |\xi|_g\leq 4\}}\|d\varphi_n(\rho)\|&\leq C\mathcal J^-_{\mathbf q},\\
  \label{e:global-expansion-rate-2}
\sup_{\rho\in\mathcal V^+_{\mathbf q}\cap \{{1\over 4}\leq |\xi|_g\leq 4\}}\|d\varphi_{-n}(\rho)\|&\leq C\mathcal J^+_{\mathbf q}.
\end{align}
It also follows from parts~(5) and~(6) of Lemma~\ref{l:stun-main}
that there exists $C$ depending only on $(M,g)$ such that
\begin{align}
  \label{e:strec}
d(\tilde\rho,W_{0s}(\rho))\leq {C\over \mathcal J^-_{\mathbf q}}
&\quad\text{for all}\quad
\rho,\tilde\rho\in \mathcal V^-_{\mathbf q}\cap S^*M,\\
  \label{e:unrec}
d(\tilde\rho,W_{0u}(\rho))\leq {C\over \mathcal J^+_{\mathbf q}}
&\quad\text{for all}\quad
\rho,\tilde\rho\in \mathcal V^+_{\mathbf q}\cap S^*M.
\end{align}
(Strictly speaking, for the proof of~\eqref{e:unrec} we should
strengthen the assumption on the sets $\mathcal V_1,\dots,\mathcal V_Q$,
requiring additionally that the diameter of each $\varphi_1(\mathcal V_q)\cap S^*M$
is smaller than $\varepsilon_0$.)
In other words, $\mathcal V^-_{\mathbf q}$ lies in a small neighborhood of a weak
stable leaf and $\mathcal V^+_{\mathbf q}$ lies in a small neighborhood of a weak
unstable leaf, with the sizes of the neighborhoods given by
the reciprocals of $\mathcal J^-_{\mathbf q}$, $\mathcal J^+_{\mathbf q}$.
See also Corollary~\ref{l:stun-main-cor} and Figure~\ref{f:stun-cor}.

From~\eqref{e:jacobians-same} we immediately derive the following
statement for every pair of words $\mathbf q,\tilde{\mathbf q}$ of the same length:
\begin{equation}
  \label{e:jacobian-pair}
\begin{aligned}
\mathcal V^+_{\mathbf q}\cap \mathcal V^+_{\tilde{\mathbf q}}\neq\emptyset
&\quad\Longrightarrow\quad
\mathcal J^+_{\mathbf q}\sim \mathcal J^+_{\tilde{\mathbf q}},\\
\mathcal V^-_{\mathbf q}\cap \mathcal V^-_{\tilde{\mathbf q}}\neq\emptyset
&\quad\Longrightarrow\quad
\mathcal J^-_{\mathbf q}\sim \mathcal J^-_{\tilde{\mathbf q}}.
\end{aligned}
\end{equation}

If we write a word $\mathbf q\in\mathscr A^n$ as a concatenation $\mathbf q=\mathbf q^1\mathbf q^2$ where $\mathbf q^j\in \mathscr A^{n_j}$, $n_1+n_2=n$, then
\begin{equation}
  \label{e:jacobians-concat}
\begin{aligned}
\mathcal V^-_{\mathbf q}\neq\emptyset
&\quad\Longrightarrow\quad
\mathcal J^-_{\mathbf q}\sim \mathcal J^-_{\mathbf q^1}\,\mathcal J^-_{\mathbf q^2},\\
\mathcal V^+_{\mathbf q}\neq\emptyset
&\quad\Longrightarrow\quad
\mathcal J^+_{\mathbf q}\sim \mathcal J^+_{\mathbf q^1}\,\mathcal J^+_{\mathbf q^2}.
\end{aligned}
\end{equation}
Indeed, for each $\rho\in\mathcal V^-_{\mathbf q}$ we have
$\rho\in \mathcal V^-_{\mathbf q^1}$,
$\varphi_{n_1}(\rho)\in\mathcal V^-_{\mathbf q^2}$,
and $J^u_n(\rho)=J^u_{n_1}(\rho)J^u_{n_2}(\varphi_{n_1}(\rho))$;
using~\eqref{e:jacobians-same} this gives the first statement in~\eqref{e:jacobians-concat}.
The second statement is proved similarly.

Finally, if $\mathbf q=q_1\dots q_n$ and $\overline{\mathbf q}=q_n\dots q_1$
is the reverse word, then
\begin{equation}
  \label{e:jacobians-reverse}
\mathcal J^-_{\mathbf q}\sim \mathcal J^+_{\overline{\mathbf q}}.
\end{equation}
Indeed, $\mathcal V^+_{\overline{\mathbf q}}=\varphi_n(\mathcal V^-_{\mathbf q})$
by~\eqref{e:reversing}.
It now suffices to use that for each $\rho\in T^*M$ we have
$J^u_n(\rho)=J^u_{-n}(\varphi_n(\rho))^{-1}\sim J^s_{-n}(\varphi_n(\rho))$ by~\eqref{e:J-su}.

%%%%%%%%%%%%%%%%%%%%%%%%%%%%%%%%%%%%%%%%%%%%%%%%%%%%%%%%%%%%%%%%%%%%%%%%%%%%%%%%
\subsection{Propagation results for refined words}

In this section we state several propagation results concerning
the operators $A^\pm_{\mathbf q}$, which will be used in the proof
of Proposition~\ref{l:longdec-1}. Some of these results will use the Jacobians
$\mathcal J^\pm_{\mathbf q}$ defined in~\eqref{e:word-J-def} above.
We recall that $a^\pm_{\mathbf q},\mathcal V^\pm_{\mathbf q},A^\pm_{\mathbf q}$
are defined using~\eqref{e:a-pm-def}, \eqref{e:V+-}, \eqref{e:A-pm-def}.

%%%%%%%%%%%%%%%%%%%%%%%%%%%%%%%%%%%%%%%%%%%%%%%%%%%%%%%%%%%%%%%%%%%%%%%%%%%%%%%%
\subsubsection{Local Ehrenfest times}\label{s:local-Ehr}

We have already encountered two global Ehrenfest times, a minimal one $T_{\min}= \big\lfloor\frac{\log(1/h)}{2\Lambda_1}\big\rfloor$, usually called \emph{the} Ehrenfest time, and a maximal one $T_{\max}=\big\lceil\frac{\log(1/h)}{2\Lambda_0}\big\rceil$.  
We will now attach \emph{(future or past) local Ehrenfest times} to each word~$\bq\in\mathscr A^\bullet$, describing the time it takes for the (future, resp. past) flow to expand by a factor $h^{-1/2}$, starting from points $\rho\in \mathcal V^{\mp}_{\bq}$.
We will not use these directly, but discuss them briefly here 
to motivate the constructions below.

Let us first define the \emph{future} local Ehrenfest time~$T^-_{\mathbf q}$, related to the values of $\cJ^{-}_{\bq}$. If $\cV^-_\bq=\emptyset$, we set $T^-_{\bq}=\infty$. Otherwise, let us assume that $h^{-1/2}\leq\cJ^{-}_{\bq}<\infty$ (this is for instance the case if $\cV^-_\bq\neq \emptyset$ and $|\bq|\geq T_{\max}$). Then there exists a unique integer $m\leq |\bq|$ such that, splitting $\bq$ into $\bq = \bq^1 q_{m}\bq^2$, where $\bq^1=q_{1}\ldots q_{m-1}$, we have
\begin{equation}\label{e:loc-Ehr}
\cJ^-_{\bq^1} < h^{-1/2} \leq \cJ^-_{\bq^1q_{m}}.
\end{equation}
We then call 
$$
T^-_{\bq}:=m\quad \text{the local future Ehrenfest time of the word $\bq$.} 
$$
In the case $\cJ^-_{\bq}< h^{-1/2}$, we consider the extensions $\bq\bp$ of $\bq$ with all possible words $\bp$ of length $T_{\max}$. For any such extension $\cJ^{-}_{\bq\bp}\geq h^{-1/2}$, so the corresponding times $T^-_{\bq\bp}$ can be defined as above. We then take
$$
T^-_{\bq}:=\min_{|\bp|=T_{\max}} T^-_{\bq\bp},\quad\text{a value which is necessarily finite.}
$$
For all $\bq$ such that $\cV^-_{\bq}\neq \emptyset$, the local Ehrenfest time satisfies 
$T_{\min}\leq T^{-}_{\bq}\leq T_{\max}$.

We similarly define the local \emph{past}  Ehrenfest time $T^+_{\bq}$ associated to the words $\bq$ such that $\cV^+_{\bq}\neq \emptyset$, depending on the values of the Jacobians $\cJ^+_{\bq}$.  

We also  define, similarly to the above, a \emph{local double Ehrenfest time} $\widetilde{T}^{\pm}_{\bq}$, by replacing $h^{-1/2}$ by $h^{-1}$ in the threshold property~\eqref{e:loc-Ehr}. Notice that if $\cV^{\pm}_{\bq}\neq\emptyset$, the double Ehrenfest times satisfy $ 2T_{\min}\leq \widetilde{T}^{\pm}_{\bq}\leq 2 T_{\max}$, but in general $\widetilde{T}^{\pm}_{\bq}\neq 2 T^{\pm}_{\bq}$. 

In the proofs below the thresholds $h^{-1/2}$ and $h^{-1}$ will be reduced to~$h^{-\delta}$ and~$h^{-2\delta}$ for some fixed $\delta\in (0,{1\over 2})$.
%%%%%%%%%%%%%%%%%%%%%%%%%%%%%%%%%%%%%%%%%%%%%%%%%%%%%%%%%%%%%%%
\subsubsection{Propagation up to local Ehrenfest time}
  \label{s:local-ehrenfest}

We first consider words $\mathbf q$ which are shorter than their local Ehfenfest times $T^{\pm}(\bq)$.
For these words the operators $A^\pm_{\mathbf q}$ lie in the mildly exotic
calculus introduced in~\S\ref{s:mildly-exotic}:
%%%%%%%%%%%%%%%%%%%%%%%%%%%%%%%%%%%%%%%%%%%%%%%%%%%%%%%%%%%%%%%%%%%%%%%%%%%%%%%%
\begin{prop}
  \label{l:ehrenfest-prop}
Fix $\delta\in [0,{1\over 2})$, $C_0>0$, and
let $\mathbf q\in\mathscr A^\bullet$.
 
1. Assume that $\mathcal J^-_{\mathbf q}\leq C_0h^{-\delta}$. Then we have
\begin{equation}
  \label{e:ehrenfest-prop}
A^-_{\mathbf q}=\Op_h(a_{\mathbf q}^{\flat-})+\mathcal O(h^\infty)_{L^2\to L^2}
\end{equation}
for some symbol $a_{\mathbf q}^{\flat-}\in S^{\comp}_{\delta+}(T^*M)$
such that
\begin{equation}
  \label{e:ehrenfest-prop-1}
a_{\mathbf q}^{\flat-}=a^-_{\mathbf q}+\mathcal O(h^{1-2\delta-})_{S^{\comp}_{\delta}},\quad
\supp a_{\mathbf q}^{\flat-}\subset \mathcal V^-_{\mathbf q}\cap \{\textstyle{1\over 4}\leq |\xi|_g\leq 4\}.
\end{equation}
The constants in $\mathcal O(\bullet)$ are independent of $h$ and $\mathbf q$.

2. The same is true for the operator $A^+_{\mathbf q}$ and some symbol
$a_{\mathbf q}^{\flat+}=a^+_{\mathbf q}+\mathcal O(h^{1-2\delta-})_{S^{\comp}_{\delta}}$,
$\supp a_{\mathbf q}^{\flat+}\subset \mathcal V^+_{\mathbf q}\cap \{{1\over 4}\leq |\xi|_g\leq 4\}$,
under the assumption $\mathcal J^+_{\mathbf q}\leq C_0h^{-\delta}$.
\end{prop}
%%%%%%%%%%%%%%%%%%%%%%%%%%%%%%%%%%%%%%%%%%%%%%%%%%%%%%%%%%%%%%%%%%%%%%%%%%%%%%%%
\Remarks
1. The assumption of part~1 of Proposition~\ref{l:ehrenfest-prop} does not hold
when $\mathcal V^-_{\mathbf q}=\emptyset$, as in that case $\mathcal J^-_{\mathbf q}=\infty$.
Yet, the statement~\eqref{e:ehrenfest-prop}, which in this case is $A^-_{\mathbf q}=\mathcal O(h^\infty)_{L^2\to L^2}$,
still holds (at least when $|\mathbf q|=\mathcal O(\log(1/h))$) but to prove it in the case of long logarithmic words $\mathbf q$  one would need
to employ the techniques of~\S\ref{s:longtime} below. In the present section we will only use a special
case of this rapid decay statement,
see Lemma~\ref{l:ehrenfest-prop-none} below.
The same remark applies to part~2.

\noindent 2. In the special case $\mathbf q\in\mathscr A^{N_0}$
the assumptions of Proposition~\ref{l:ehrenfest-prop} are satisfied
for $\delta={1\over 6}$ (assuming $\mathcal V^\pm_{\mathbf q}\neq\emptyset$)
as follows from~\eqref{e:jacobian-bounds} and the definition~\eqref{e:prop-times} of~$N_0$.
In this case a weaker version of~\eqref{e:ehrenfest-prop} (with $\mathcal O(h^{1-2\delta-})_{L^2\to L^2}$ remainder) follows from Lemma~\ref{l:cq-log} (more precisely, its version for the refined partition of~\S\ref{s:refined-partition}). The latter relies on Egorov's Theorem up to the (minimal) Ehrenfest time, Lemma~\ref{l:egorov-mild}.

Proposition~\ref{l:ehrenfest-prop} is proved in~\S\ref{s:ehr-step}. The argument is morally
similar to the proof of the first part of Lemma~\ref{l:cq-log}, but much more complicated
because of two reasons:
\begin{itemize}
\item We establish the classical/quantum correspondence up to the
\emph{local Ehrenfest times} associated with the particular words.
While the global expansion rates of $\varphi_{\pm n}$, where $n$ is the length of~$\mathbf q$, might be very large, the expansion rates of $\varphi_{\pm n}$ restricted to $\supp a^\pm_{\mathbf q}$ are still smaller than $h^{-\delta}\ll h^{-1/2}$.
\item We obtain asymptotic expansions of the full symbols of $A^\pm_{\mathbf q}$,
which give the $\mathcal O(h^\infty)$ remainder in~\eqref{e:ehrenfest-prop},
similarly to~\eqref{e:egorov-basic-more}.
\end{itemize}
As a corollary of Proposition~\ref{l:ehrenfest-prop} we obtain the following
rapid decay results for operators $A^\pm_{\mathbf q}$ and their products under
assumptions of empty or nonintersecting supports:
%%%%%%%%%%%%%%%%%%%%%%%%%%%%%%%%%%%%%%%%%%%%%%%%%%%%%%%%%%%%%%%%%%%%%%%%%%%%%%%%
\begin{lemm}
  \label{l:ehrenfest-prop-none}
Fix $\delta\in [0,{1\over 2})$ and $C_0>0$.  
 
1. Assume that $\mathbf p,\mathbf q\in\mathscr A^\bullet$.
Then
\begin{equation}
  \label{e:epn-1}
\max(\mathcal J^-_{\mathbf p},\mathcal J^+_{\mathbf q})\leq C_0h^{-\delta},\quad
\mathcal V^-_{\mathbf p}\cap \mathcal V^+_{\mathbf q}=\emptyset
\quad\Longrightarrow\quad
\|A^-_{\mathbf p}A^+_{\mathbf q}\|_{L^2\to L^2}=\mathcal O(h^\infty).
\end{equation}
 
2. Assume that $\mathbf q=q_1\dots q_n\in\mathscr A^\bullet$, $n\leq C_0\log(1/h)$,
satisfies $\mathcal V^+_{\mathbf q}=\emptyset$.
Take the largest~$m$ such that $\mathcal V^+_{q_1\dots q_m}\neq\emptyset$
and assume that $\mathcal J^+_{q_1\dots q_m}\leq C_0h^{-2\delta}$. Then
$$
\|A^+_{\mathbf q}\|_{L^2\to L^2}=\mathcal O(h^\infty).
$$
The same holds for $A^-_{\mathbf q}$ (under an assumption
on $\mathcal J^-_{q_1\dots q_m}$), and also if we consider subwords of the form
$q_{n-m+1}\dots q_n$ instead.

3. Assume that $\mathbf q,\tilde{\mathbf q}\in\mathscr A^\bullet$ have the same length
and $\max(\mathcal J^+_{\mathbf q},\mathcal J^+_{\tilde{\mathbf q}})\leq C_0h^{-2\delta}$,
$\mathcal V^+_{\mathbf q}\cap \mathcal V^+_{\tilde{\mathbf q}}=\emptyset$.
Then
$$
\|(A^+_{\mathbf q})^*A^+_{\tilde{\mathbf q}}\|_{L^2\to L^2}=\mathcal O(h^\infty),\quad
\|A^+_{\tilde{\mathbf q}}(A^+_{\mathbf q})^*\|_{L^2\to L^2}=\mathcal O(h^\infty).
$$
The same is true for the operators $A^-$ if we make assumptions on $\mathcal J^-,\mathcal V^-$ instead.

In all these statements the constants in $\mathcal O(\bullet)$ do not depend on $h$
and on the choice of the words.
\end{lemm}
%%%%%%%%%%%%%%%%%%%%%%%%%%%%%%%%%%%%%%%%%%%%%%%%%%%%%%%%%%%%%%%%%%%%%%%%%%%%%%%%
\Remark Note that the Jacobians in parts~2 and~3 above are required to be bounded by $C_0h^{-2\delta}$~--
that is Lemma~\ref{l:ehrenfest-prop-none} essentially applies up to the \emph{local double Ehrenfest time}.
We are able to do this by writing a word with Jacobian $\mathcal O(h^{-2\delta})$
as a concatenation of two words with Jacobians $\mathcal O(h^{-\delta})$ and using~\eqref{e:word-concat}. If $M$ had constant curvature, we could instead use pseudodifferential calculi
adapted to the stable/unstable foliations as in~\cite{meassupp}.
%%%%%%%%%%%%%%%%%%%%%%%%%%%%%%%%%%%%%%%%%%%%%%%%%%%%%%%%%%%%%%%%%%%%%%%%%%%%%%%%
\begin{proof}
1. Using Proposition~\ref{l:ehrenfest-prop} we write
$$
A^-_{\mathbf p}=\Op_h(a^{\flat-}_{\mathbf p})+\mathcal O(h^\infty)_{L^2\to L^2},\quad
A^+_{\mathbf q}=\Op_h(a^{\flat+}_{\mathbf q})+\mathcal O(h^\infty)_{L^2\to L^2}.
$$
Here $\supp a^{\flat-}_{\mathbf p}\subset\mathcal V^-_{\mathbf p}$
and $\supp a^{\flat+}_{\mathbf q}\subset \mathcal V^+_{\mathbf q}$,
therefore $\supp a^{\flat-}_{\mathbf p}\cap \supp a^{\flat+}_{\mathbf q}=\emptyset$.
It then follows from the product formula in the $S^{\comp}_{\delta+}$ calculus
(see for instance~\cite[Theorem~4.18]{e-z}) that $\Op_h(a^{\flat-}_{\mathbf p})\Op_h(a^{\flat+}_{\mathbf q})
=\mathcal O(h^\infty)_{L^2\to L^2}$.

2. We assume that $\mathcal V^+_{\mathbf q}=\emptyset$, with
the case of $\mathcal V^-_{\mathbf q},A^-_{\mathbf q}$ following from here using~\eqref{e:reversing}
and~\eqref{e:jacobians-reverse}.
We also assume that there exists $m<n$ such that
$\mathcal V^+_{q_1\dots q_{m+1}}=\emptyset$ and $\mathcal J^+_{q_1\dots q_m}\leq C_0h^{-2\delta}$;
the other case (when there exists $m<n$ such that
$\mathcal V^+_{q_{n-m}\dots q_n}=\emptyset$ and $\mathcal J^+_{q_{n-m+1}\dots q_n}\leq C_0h^{-2\delta}$)
is handled similarly.

We first show that $\mathbf q$ can be written as a concatenation
(where $C_1$ denotes a constant depending on $C_0$ whose exact value might
differ from place to place)
\begin{equation}
  \label{e:art-deco}
\mathbf q=\mathbf q^1\mathbf p\mathbf r\mathbf q^2\quad\text{where}\quad
\max(\mathcal J^+_{\mathbf p},\mathcal J^+_{\mathbf r})\leq C_1h^{-\delta},\quad
\mathcal V^+_{\mathbf p\mathbf r}=\emptyset.
\end{equation}
To do this we first put $\mathbf q^2:=q_{m+2}\dots q_n$.
Next, choose maximal
$\ell\leq m$ such that $\mathcal J^+_{q_1\dots q_\ell}\leq h^{-\delta}$.
We claim that
\begin{equation}
  \label{e:rooney}
\mathcal J^+_{q_{\ell+1}\dots q_{m}}\leq C_1 h^{-\delta}.
\end{equation}
Indeed, we may assume that $\ell<m$ since otherwise~\eqref{e:rooney} holds automatically.
Since $\ell$ was chosen maximal, we have
$\mathcal J^+_{q_1\dots q_{\ell+1}}> h^{-\delta}$, which by~\eqref{e:jacobians-concat}
implies that $\mathcal J^+_{q_1\dots q_\ell}\geq C_1^{-1}h^{-\delta}$.
Now~\eqref{e:rooney} follows from~\eqref{e:jacobians-concat} and the bound
$\mathcal J^+_{q_1\dots q_m}\leq C_0h^{-2\delta}$.

Now the decomposition~\eqref{e:art-deco} is obtained by considering two cases:
\begin{enumerate}
\item $\mathcal V^+_{q_{\ell+1}\dots q_{m+1}}=\emptyset$: put
$\mathbf q^1:=q_1\dots q_\ell$,
$\mathbf p:=q_{\ell+1}\dots q_m$,
$\mathbf r:=q_{m+1}$.
\item $\mathcal V^+_{q_{\ell+1}\dots q_{m+1}}\neq\emptyset$: put
$\mathbf q^1:=\emptyset$, $\mathbf p:=q_1\dots q_\ell$,
$\mathbf r:=q_{\ell+1}\dots q_{m+1}$. We have
$\mathcal J^+_{\mathbf r}\leq C_1h^{-\delta}$ by~\eqref{e:jacobians-concat}
and~\eqref{e:rooney}.
\end{enumerate}
Having established~\eqref{e:art-deco} we write by~\eqref{e:word-concat}
and~\eqref{e:reversing}
(where $\overline{\mathbf p}$ is the reverse of $\mathbf p$)
\begin{equation}
  \label{e:isaac}
\begin{aligned}
A^+_{\mathbf q}
&=U(|\mathbf q^1\mathbf p|) A^-_{\overline{\mathbf p}\,\overline{\mathbf q}^1}A^+_{\mathbf r\mathbf q_2}
U(-|\mathbf q^1\mathbf p|)\\
&=U(|\mathbf q^1|)A^-_{\overline{\mathbf q}^1}U(|\mathbf p|)
A^-_{\overline{\mathbf p}}A^+_{\mathbf r}U(|\mathbf r|)A^+_{\mathbf q^2}U(-|\mathbf q^1\mathbf p\mathbf r|).
\end{aligned}
\end{equation}
Recall that $\mathcal J^+_{\mathbf r}\leq C_1h^{-\delta}$.
We moreover have $\mathcal J^-_{\overline{\mathbf p}}\sim \mathcal J^+_{\mathbf p}\leq C_1h^{-\delta}$ by~\eqref{e:jacobians-reverse}.
Also $\mathcal V^-_{\overline{\mathbf p}}\cap \mathcal V^+_{\mathbf r}=\emptyset$
by~\eqref{e:set-concat} since $\mathcal V^+_{\mathbf p\mathbf r}=\emptyset$.
Finally $\|A^-_{\overline{\mathbf q}^1}\|_{L^2\to L^2}$
and $\|A^+_{\mathbf q^2}\|_{L^2\to L^2}$ are bounded by~\eqref{e:A-q-bdd}.
Therefore by~\eqref{e:epn-1} we have
\begin{equation}
  \label{e:wayne}
\|A^+_{\mathbf q}\|_{L^2\to L^2}\leq 
C \|A^-_{\overline{\mathbf p}}A^+_{\mathbf r}\|_{L^2\to L^2}
=\mathcal O(h^\infty).
\end{equation}

3. 
We consider the operators $A^+$, with the case of $A^-$ following from here using~\eqref{e:reversing} and~\eqref{e:jacobians-reverse}.
We first show that $\|(A^+_{\mathbf q})^*A^+_{\tilde{\mathbf q}}\|_{L^2\to L^2}=\mathcal O(h^\infty)$.
We write $\mathbf q=q_1\dots q_n$ and $\tilde{\mathbf q}=\tilde q_1\dots \tilde q_{n}$ and take maximal $\ell\leq n$ such that
\begin{equation}
  \label{e:decoto}
\max(\mathcal J^+_{q_1\dots q_\ell},\mathcal J^+_{\tilde q_1\dots \tilde q_\ell})\leq h^{-\delta}.
\end{equation}
We have the following two cases:
%%%%%%%%%%
\begin{enumerate}
\item $\mathcal V^+_{q_1\dots q_\ell}\cap \mathcal V^+_{\tilde q_1\dots \tilde q_\ell}=\emptyset$.
Arguing similarly to part~1 of this lemma and using~\eqref{e:decoto}, we see that
\begin{equation}
  \label{e:malibu}
\|(A^+_{q_1\dots q_\ell})^*A^+_{\tilde q_1\dots\tilde q_\ell}\|_{L^2\to L^2}=\mathcal O(h^\infty).
\end{equation}
By~\eqref{e:word-concat} and~\eqref{e:reversing} we have
$$
(A^+_{\mathbf q})^*A^+_{\tilde{\mathbf q}}=
U(\ell)(A^+_{q_{\ell+1}\dots q_n})^*U(-\ell)
(A^+_{q_1\dots q_\ell})^*A^+_{\tilde q_1\dots\tilde q_\ell}
U(\ell)A^+_{\tilde q_{\ell+1}\dots\tilde q_n}U(-\ell).
$$
Using~\eqref{e:malibu} and the norm bound~\eqref{e:A-q-bdd}
we get $\|(A^+_{\mathbf q})^*A^+_{\tilde{\mathbf q}}\|_{L^2\to L^2}=\mathcal O(h^\infty)$.
%%%%%%%%%%
\item $\mathcal V^+_{q_1\dots q_\ell}\cap \mathcal V^+_{\tilde q_1\dots \tilde q_\ell}\neq \emptyset$.
We claim that
\begin{equation}
  \label{e:godiva}
\max(\mathcal J^+_{q_{\ell+1}\dots q_n},\mathcal J^+_{\tilde q_{\ell+1}\dots \tilde q_n})
\leq C_1h^{-\delta}.
\end{equation}
Indeed, we may assume that $\ell<n$ since otherwise~\eqref{e:godiva} is immediate. Since $\ell$
was chosen maximal we have
$$
\max(\mathcal J^+_{q_1\dots q_{\ell+1}},\mathcal J^+_{\tilde q_1\dots \tilde q_{\ell+1}})> h^{-\delta}.
$$
Without loss of generality we may assume that 
$\mathcal J^+_{q_1\dots q_{\ell+1}}>h^{-\delta}$. Then by~\eqref{e:jacobians-concat}
we have $\mathcal J^+_{q_1\dots q_\ell}\geq C_1^{-1} h^{-\delta}$.
Since $\mathcal J^+_{q_1\dots q_\ell}\sim \mathcal J^+_{\tilde q_1\dots\tilde q_\ell}$
by~\eqref{e:jacobian-pair}, we have
$\mathcal J^+_{\tilde q_1\dots \tilde q_\ell}\geq C_1^{-1}h^{-\delta}$ as well.
Now~\eqref{e:godiva} follows from~\eqref{e:jacobians-concat} and the bound
$\max(\mathcal J^+_{\mathbf q},\mathcal J^+_{\tilde{\mathbf q}})\leq C_0h^{-2\delta}$.

Since $\mathcal V^+_{\mathbf q}\cap \mathcal V^+_{\tilde{\mathbf q}}=\emptyset$,
by~\eqref{e:set-concat} we have
$$
\mathcal V^-_{q_\ell\dots q_1}\cap \mathcal V^+_{q_{\ell+1}\dots q_n}
\cap \mathcal V^-_{\tilde q_\ell\dots \tilde q_1}\cap \mathcal V^+_{\tilde q_{\ell+1}\dots\tilde q_n}
=\emptyset.
$$
Arguing similarly to part~1 of this lemma and using~\eqref{e:decoto}, \eqref{e:godiva}, and~\eqref{e:jacobians-reverse} we get
$$
\|(A^-_{q_\ell\dots q_1}A^+_{q_{\ell+1}\dots q_n})^*
A^-_{\tilde q_\ell\dots \tilde q_1}A^+_{\tilde q_{\ell+1}\dots \tilde q_n}\|_{L^2\to L^2}=\mathcal O(h^\infty).
$$
Now by~\eqref{e:word-concat} we have
\begin{equation}
  \label{e:huntington}
(A^+_{\mathbf q})^*A^+_{\tilde{\mathbf q}}=U(\ell)(A^-_{q_\ell\dots q_1}A^+_{q_{\ell+1}\dots q_n})^*
A^-_{\tilde q_\ell\dots \tilde q_1}A^+_{\tilde q_{\ell+1}\dots \tilde q_n}U(-\ell)
\end{equation}
which gives $\|(A^+_{\mathbf q})^*A^+_{\tilde{\mathbf q}}\|_{L^2\to L^2}=\mathcal O(h^\infty)$.
\end{enumerate}
%%%%%%%%%%
To prove that $\|A^+_{\tilde{\mathbf q}}(A^+_{\mathbf q})^*\|_{L^2\to L^2}=\mathcal O(h^\infty)$ we argue
similarly. More precisely, take minimal $\ell\geq 1$ such that
$$
\max(\mathcal J^+_{q_\ell\dots q_n},\mathcal J^+_{\tilde q_\ell\dots\tilde q_n})\leq h^{-\delta}.
$$
Assume first that $\mathcal V^+_{q_\ell\dots q_n}\cap \mathcal V^+_{\tilde q_\ell\dots \tilde q_n}=\emptyset$.
Arguing similarly to part~1 of this lemma we get
\begin{equation}
  \label{e:malibu2}
\|A^+_{\tilde q_\ell\dots\tilde q_n}(A^+_{q_\ell\dots q_n})^*\|_{L^2\to L^2}=\mathcal O(h^\infty).
\end{equation}
By~\eqref{e:word-concat} we have
$$
A^+_{\tilde{\mathbf q}}(A^+_{\mathbf q})^*=
U(\ell-1)A^-_{\tilde q_{\ell-1}\dots \tilde q_1}A^+_{\tilde q_\ell\dots\tilde q_n}
(A^+_{q_\ell\dots q_n})^*(A^-_{q_{\ell-1}\dots q_1})^*U(1-\ell)
$$
and the right-hand side is $\mathcal O(h^\infty)_{L^2\to L^2}$ by~\eqref{e:malibu2}
and~\eqref{e:A-q-bdd}.

Assume now that $\mathcal V^+_{q_\ell\dots q_n}\cap \mathcal V^+_{\tilde q_\ell\dots\tilde q_n}\neq\emptyset$.
Then similarly to~\eqref{e:godiva} we get
$$
\max(\mathcal J^+_{q_1\dots q_{\ell-1}},\mathcal J^+_{\tilde q_1\dots \tilde q_{\ell-1}})\leq C_1h^{-\delta}.
$$
The bound $\|A^+_{\tilde{\mathbf q}}(A^+_{\mathbf q})^*\|_{L^2\to L^2}=\mathcal O(h^\infty)$
is now proved similarly to the case~(2) above, with~\eqref{e:huntington} replaced by the following
corollary of~\eqref{e:word-concat}:
$$
A^+_{\tilde{\mathbf q}}(A^+_{\mathbf q})^*=U(\ell-1)
A^-_{\tilde q_{\ell-1}\dots \tilde q_1}A^+_{\tilde q_\ell\dots \tilde q_n}
(A^-_{q_{\ell-1}\dots q_1}A^+_{q_\ell\dots q_n})^*
U(1-\ell).
$$
\end{proof}
%%%%%%%%%%%%%%%%%%%%%%%%%%%%%%%%%%%%%%%%%%%%%%%%%%%%%%%%%%%%%%%%%%%%%%%%%%%%%%%%
In addition to Proposition~\ref{l:ehrenfest-prop}
we will also need the following statement regarding sums of operators
of the form $A^-_{\mathbf p}A^+_{\mathbf r}$:
%%%%%%%%%%%%%%%%%%%%%%%%%%%%%%%%%%%%%%%%%%%%%%%%%%%%%%%%%%%%%%%%%%%%%%%%%%%%%%%%
\begin{prop}
  \label{l:ehrenfest-prop-sum}
Fix $\delta\in [0,{1\over 2})$, $C_0>0$.
Assume that $F:\mathscr A^\bullet\times \mathscr A^\bullet\to\mathbb C$ is a function
such that:
\begin{enumerate}
\item for each $(\mathbf p,\mathbf r)$ with $F(\mathbf p,\mathbf r)\neq 0$,
we have $\max(\mathcal J^-_{\mathbf p},\mathcal J^+_{\mathbf r})\leq C_0h^{-\delta}$;
\item $\sup |F|\leq 1$.
\end{enumerate}
Then we have for some constant $C$ independent of $h$ and~$F$
\begin{equation}
  \label{e:eps}
\|A_F\|_{L^2\to L^2}\leq C\log^2(1/h)\quad\text{where}\quad
A_F:=\sum_{\mathbf p,\mathbf r}F(\mathbf p,\mathbf r)A^-_{\mathbf p}A^+_{\mathbf r}.
\end{equation}
\end{prop}
%%%%%%%%%%%%%%%%%%%%%%%%%%%%%%%%%%%%%%%%%%%%%%%%%%%%%%%%%%%%%%%%%%%%%%%%%%%%%%%%
\Remarks 1. It is easy to see that $\sup |a_F|\leq C\log^2(1/h)$
where $a_F=\sum_{(\mathbf p,\mathbf r)}F(\mathbf p,\mathbf r)a^-_{\mathbf p}a^+_{\mathbf r}$
is the symbol corresponding to $A_F$,
grouping terms in the sum by the lengths $|\mathbf p|,|\mathbf r|$. However
the statement~\eqref{e:eps} does not follow by summing Proposition~\ref{l:ehrenfest-prop}
over $(\mathbf p,\mathbf r)$, since the number of terms in this sum grows polynomially with~$h$.
(We got around this problem in Lemma~\ref{l:cq-log}
by taking $\delta:={1\over 6}$ small enough so that the individual remainder still dominates
the growth of the number of terms, however in this section we will need to take
$\delta$ very close to~$1\over 2$.) Instead the proof of Proposition~\ref{l:ehrenfest-prop-sum},
given in~\S\ref{s:ehr-sum} below,
uses fine estimates on the full symbols of $A^-_{\mathbf p}$, $A^+_{\mathbf r}$.

\noindent 2. The proof of Proposition~\ref{l:ehrenfest-prop-sum} shows
that $A_F$ is a pseudodifferential operator, similarly to Proposition~\ref{l:ehrenfest-prop}.
However, we will only need a norm bound on $A_F$.

Similarly to Lemma~\ref{l:ehrenfest-prop-none} we deduce from Proposition~\ref{l:ehrenfest-prop-sum}
a statement up to the local double Ehrenfest time
which is used to establish the norm bound~\eqref{e:apn-2} below:
%%%%%%%%%%%%%%%%%%%%%%%%%%%%%%%%%%%%%%%%%%%%%%%%%%%%%%%%%%%%%%%%%%%%%%%%%%%%%%%%
\begin{lemm}
  \label{l:ehrenfest-summary}
Fix $\delta\in [0,{1\over 2})$, $C_0>0$. Assume that $F:\mathscr A^\bullet\to\mathbb C$ and
\begin{enumerate}
\item for each $\mathbf q$ with $F(\mathbf q)\neq 0$, we have $\mathcal J^+_{\mathbf q}\leq C_0h^{-2\delta}$;
\item $\sup |F|\leq 1$.
\end{enumerate}
Then we have for some constant $C$ independent of $h$ and~$F$
\begin{equation}
  \label{e:eps-snake}
\|A^+_F\|_{L^2\to L^2}\leq C\log^3(1/h)\quad\text{where}\quad
A^+_F:=\sum_{\mathbf q}F(\mathbf q)A^+_{\mathbf q}.
\end{equation}
Same is true for $A^-_F$ if we make an assumption on $\mathcal J^-_{\mathbf q}$ instead.
\end{lemm}
%%%%%%%%%%%%%%%%%%%%%%%%%%%%%%%%%%%%%%%%%%%%%%%%%%%%%%%%%%%%%%%%%%%%%%%%%%%%%%%%
\Remark We make no attempt to optimize the power of $\log(1/h)$ in~\eqref{e:eps-snake}~--
for our purposes all that matters is that $\|A^+_F\|_{L^2\to L^2}=\mathcal O(h^{0-})$.
%%%%%%%%%%%%%%%%%%%%%%%%%%%%%%%%%%%%%%%%%%%%%%%%%%%%%%%%%%%%%%%%%%%%%%%%%%%%%%%%
\begin{proof}
We prove a bound on $A^+_F$, with the case of $A^-_F$ handled similarly.

For each $\mathbf q$
with $\mathcal J^+_{\mathbf q}\leq C_0h^{-2\delta}$ there exists an integer $\ell=\ell(\mathbf q)\in [0,n]$
such that
\begin{equation}
  \label{e:snakemid0}
\max(\mathcal J^+_{q_1\dots q_\ell},\mathcal J^+_{q_{\ell+1}\dots q_n})\leq C_1 h^{-\delta}
\end{equation}
where $C_1$ is a large constant depending on $C_0$. Indeed,
we choose maximal $\ell\leq n$ such that 
$\mathcal J^+_{q_1\dots q_\ell}\leq h^{-\delta}$.
If $\ell=n$ then $\mathcal J^+_{q_{\ell+1}\dots q_n}=1$.
If $\ell<n$ then $\mathcal J^+_{q_1\dots q_{\ell+1}}>h^{-\delta}$,
which by~\eqref{e:jacobians-concat} implies that
$\mathcal J^+_{q_1\dots q_\ell}\geq C^{-1}h^{-\delta}$ and thus by another
application of~\eqref{e:jacobians-concat},
$\mathcal J^+_{q_{\ell+1}\dots q_n}\leq C_1h^{-\delta}$.

We may take $C_1$ large enough so that $\mathcal J^+_{\mathbf q}\leq C_0h^{-2\delta}$
implies that $|\mathbf q|\leq C_1\log(1/h)$. Then we decompose
\begin{equation}
  \label{e:snaketot}
A^+_F=\sum_{0\leq\ell\leq C_1\log(1/h)}A^+_{F_\ell},\quad
F_\ell(\mathbf q):=\begin{cases}
F(\mathbf q),&\text{if }\ell(\mathbf q)=\ell,\\
0,&\text{otherwise}.
\end{cases}
\end{equation}
We have by~\eqref{e:word-concat}
$$
A^+_{F_\ell}=U(\ell)A_{G_\ell} U(-\ell)\quad\text{where}\quad
A_{G_\ell}:=\sum_{(\mathbf p,\mathbf r)} G_\ell(\mathbf p,\mathbf r)A^-_{\mathbf p}A^+_{\mathbf r}
$$
and the function $G_\ell:\mathscr A^\bullet\times \mathscr A^\bullet\to\mathbb C$ is defined as follows:
$$
G_\ell(\mathbf p,\mathbf r):=\begin{cases}
F_\ell(\overline{\mathbf p}\mathbf r),&\text{if }|\mathbf p|=\ell,\\
0,&\text{otherwise.}\end{cases}
$$
For each $(\mathbf p,\mathbf r)$ with $G_\ell(\mathbf p,\mathbf r)\neq 0$
we have
$\max(\mathcal J^-_{\mathbf p},\mathcal J^+_{\mathbf r})\leq Ch^{-\delta}$
by~\eqref{e:snakemid0} and~\eqref{e:jacobians-reverse}.
Therefore by Proposition~\ref{l:ehrenfest-prop-sum}
\begin{equation}
  \label{e:snakebd}
\|A^+_{F_\ell}\|_{L^2\to L^2}=
\|A_{G_\ell}\|_{L^2\to L^2}\leq C\log^2(1/h).
\end{equation}
Using the triangle inequality in~\eqref{e:snaketot}
and the norm bound~\eqref{e:snakebd} we get~\eqref{e:eps-snake}.
\end{proof}
%%%%%%%%%%%%%%%%%%%%%%%%%%%%%%%%%%%%%%%%%%%%%%%%%%%%%%%%%%%%%%%%%%%%%%%%%%%%%%%%

%%%%%%%%%%%%%%%%%%%%%%%%%%%%%%%%%%%%%%%%%%%%%%%%%%%%%%%%%%%%%%%%%%%%%%%%%%%%%%%%
\subsubsection{Propagation beyond Ehrenfest time}
  \label{s:longtime}

We now study microlocalization of the operators $A^+_{\mathbf q}$
for words $\mathbf q$ of length no more than $C\log(1/h)$, where $C$ is any fixed constant.
The resulting Proposition~\ref{l:longtime-prop} is applied in the proof of Lemma~\ref{l:loca+} in~\S\ref{s:micro-conjugation} below
to words $\mathbf q$ with $\mathcal J^+_{\mathbf q}\sim h^{-\tau}$, where $\tau\in ({1\over 2},1)$
is defined in~\eqref{e:tau-delta-def}. Analogous statements hold for
the operators $A^-_{\mathbf q}$, but we will not make or use them here.

When $\mathcal J^+_{\mathbf q}\gg h^{-1/2}$ (as in the proof of Lemma~\ref{l:loca+}) the
symbol $a^+_{\mathbf q}$ oscillates too strongly to belong to the symbol class $S^{\comp}_{\delta}$ for any $\delta<{1\over 2}$. In the case when $M$ has constant curvature, it was shown in~\cite{hgap,meassupp} that for $\mathcal J^+_{\mathbf q}\ll h^{-1}$ the operator $A^+_{\mathbf q}$ belongs to a certain anisotropic class of pseudodifferential operators ``aligned'' with the unstable foliation, see~\cite[Lemma~3.2]{meassupp}. The construction of this anisotropic class strongly relied on the smoothness of the unstable foliation, see~\cite[\S3.3]{hgap}. However in the case
of variable curvature considered here, the unstable foliation is no longer smooth
and it is not clear how to define the corresponding anisotropic pseudodifferential class.

We will therefore take a different strategy to study the microlocalization of $A^+_{\bq}$, which uses methods developed in~\cite{AnAnn,AN07,NZ09}. Given an arbitrary function $f\in L^2(M)$ (possibly depending on $h$), we will study the microlocalization of the function $A^+_{\bq}f$. This gives less information than $A^+_{\mathbf q}$ being pseudodifferential but it suffices for the application in~\S\ref{s:micro-conjugation}.

Since $f$ is chosen arbitrary and the microlocal wave propagator
$U(t)$ defined in~\eqref{e:U-t-def} is unitary, it suffices to study microlocalization
of $U^+_{\mathbf q}f$ where the operator $U^+_{\mathbf q}:L^2(M)\to L^2(M)$ is defined similarly to~\eqref{e:damp-word} (recalling the definition~\eqref{e:A-pm-def} of $A^+_{\mathbf q}$):
\begin{equation}
  \label{e:U+def}
U^+_{\mathbf q}:=A^+_{\mathbf q}U(n)=U(1)A_{q_1}U(1)A_{q_2}\cdots U(1)A_{q_n},\quad
\mathbf q=q_1\dots q_n\in\mathscr A^\bullet.
\end{equation}
Using the Fourier inversion formula we will decompose $f$ into a superposition of Lagrangian distributions (see~\S\ref{s:intro-lagr}) associated to a family of Lagrangian submanifolds ${\mathscr L}_{q_n,\theta}\subset T^*M$, $\theta\in\mathbb R^2$.
Roughly speaking, the main result of the present subsection,
Proposition~\ref{l:longtime-prop}, shows that
\begin{equation}
  \label{e:rough-lag}
f\in I^{\comp}_h({\mathscr L}_{q_n,\theta})
\quad\Longrightarrow\quad
U(-1)U^+_{\mathbf q}f\in I^{\comp}_h({\mathscr L}_{\mathbf q,\theta})
\end{equation}
where ${\mathscr L}_{\mathbf q,\theta}$ is the propagated
Lagrangian manifold (see Definition~\ref{d:compost} below).
The key point, exploited in the proof of Lemma~\ref{l:loca+}, is that for long $\mathbf q$ the manifold ${\mathscr L}_{\mathbf q,\theta}$ depends little on~$\theta$, so that the full state $A^+_{\bq}f$ (written as an integral of propagated Lagrangian distributions over~$\theta$) is microlocalized in a very small neighborhood of a single unstable leaf.

The propagator $U(1)$ is a Fourier integral operator (see~\S\ref{s:prelim-fio-s}) associated to the time-one map
of the geodesic flow $\varphi_1$, microlocally in~$\{{1\over 4}<|\xi|_g<4\}$:
\begin{equation}
  \label{e:U-1-fio}
U(1)A,AU(1)\in I^{\comp}_h(\varphi_1)\quad\text{for all}\quad
A\in\Psi^0_h(M),\
\WFh(A)\subset \{\textstyle{1\over 4}<|\xi|_g<4\}.
\end{equation}
This follows from the definition~\eqref{e:U-t-def} and the
standard hyperbolic parametrix construction, see e.g.
\cite[Theorem~10.4]{e-z} or~\cite[Lemma~4.2]{NZ09}.

Using~\eqref{e:U-1-fio} we can prove~\eqref{e:rough-lag}
for $\mathbf q$ of bounded length
using standard properties of Lagrangian distributions (more specifically, property~(3) in~\S\ref{s:prelim-fio-s}).
However, since the length of $\mathbf q$ grows with $h$, the argument becomes more complicated.
In fact, we cannot even use the general definition of the class $I^{\comp}_h(\mathscr L)$
in~\S\ref{s:intro-lagr} since it applies to an $h$-dependent family of distributions
with $h$-independent $\mathscr L$.
We will rely on the results of~\cite{NZ09},
featuring a detailed analysis of the behavior
of the propagated Lagrangian manifolds and the oscillatory integral representations~\eqref{e:lag-dist}
for $U(-1)U^+_{\mathbf q}f$ as the length of~$\mathbf q$ grows.
For this analysis it will be important that the initial Lagrangians ${\mathscr L}_{q,\theta}$
are chosen close to weak unstable leaves, and thus transverse to stable leaves.

To fix the parametrization of propagated Lagrangian manifolds and distributions,
it is convenient to introduce adapted symplectic coordinates. For each $\rho_0\in S^*M$
let
\begin{equation}
  \label{e:kappa-rho-0}
\varkappa_{\rho_0}:U_{\rho_0}\to V_{\rho_0},\quad
U_{\rho_0}\subset T^*M\setminus 0,\quad
V_{\rho_0}\subset T^*\mathbb R^2\setminus 0
\end{equation}
be the symplectomorphism constructed in Lemma~\ref{l:stun-straight}
(in fact we will only use properties~(1)--(4) of Lemma~\ref{l:stun-straight} here). Since $\varkappa_{\rho_0}$ is homogeneous
we may shrink $U_{\rho_0}$ so that the flipped graph $\mathscr L_{\varkappa_{\rho_0}}$ is generated by a single phase function, see~\S\ref{s:lagr-mflds}.

Let $\varepsilon_0>0$ be the constant from~\S\ref{s:refined-partition};
recall that the diameter of each $\mathcal V_q\cap S^*M$ is smaller than $\varepsilon_0$.
We will assume in several places in this subsection
that $\varepsilon_0$ is small depending only on $(M,g)$. For
each $q\in \mathscr A$ fix an arbitrary point $\rho_q\in\mathcal V_q\cap S^*M$ and put
\begin{equation}
  \label{e:kappa-q}
\varkappa_q:=\varkappa_{\rho_q}:\mathcal V^\sharp_q\to \mathcal W_q,\quad
\mathcal V^\sharp_q:=U_{\rho_q},\quad
\mathcal W_q:=V_{\rho_q}.
\end{equation}
We denote elements of $T^*M$ by $\rho=(x,\xi)$
and elements of $T^*\mathbb R^2$ by $(y,\eta)$. We assume that
$\varepsilon_0$ is small enough so that
$\overline{\mathcal V}_q\subset\mathcal V^\sharp_q$ where
the closure is taken in $T^*M\setminus 0$.

We are now ready to define the Lagrangian submanifolds
${\mathscr L}_{\mathbf q,\theta}$:
%%%%%%%%%%%%%%%%%%%%%%%%%%%%%%%%%%%%%%%%%%%%%%%%%%%%%%%%%%%%%%%%%%%%%%%%%%%%%%%%
\begin{figure}
\includegraphics{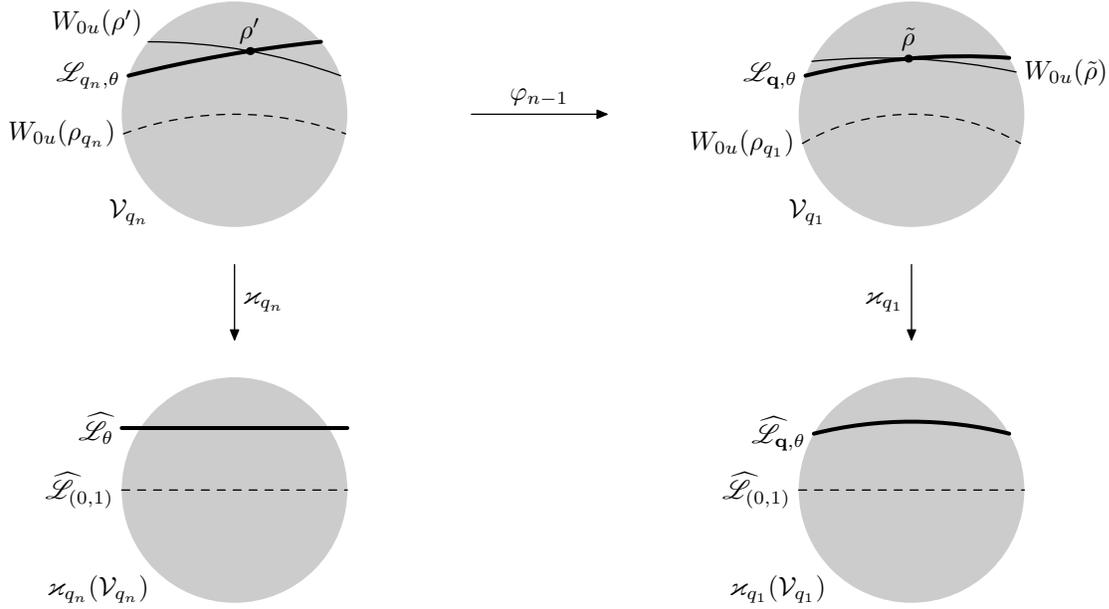}
\caption{An illustration of Definition~\ref{d:compost}
and Lemma~\ref{l:inclinator}, fixing $\tilde\rho=\varphi_{n-1}(\rho')\in \mathscr L_{\mathbf q,\theta}\cap S^*M$. We restrict to $S^*M=\{\eta_2=1\}$
and remove the flow direction $\partial_{y_2}$.
In the bottom figures the horizontal direction is $y_1$ and the vertical one
is $\eta_1$.
The original Lagrangian $\mathscr L_{q_n,\theta}$ is $\mathcal O(\varepsilon_0)$
close to the weak unstable leaf $W_{0u}(\rho')$ as a $C^\infty$ submanifold,
thus the propagated Lagrangian $\mathscr L_{\mathbf q,\theta}$ is $\mathcal O(\varepsilon_0)$
close to the weak unstable leaf $W_{0u}(\tilde\rho)$ (in fact, it is $\mathcal O(e^{-\gamma n}\varepsilon_0)$ close for some $\gamma>0$).
A word of caution: in general $\mathcal V_{q_n},W_{0u}(\rho_{q_n})$
are not mapped by $\varphi_{n-1}$ to
$\mathcal V_{q_1},W_{0u}(\rho_{q_1})$.}
\label{f:incline}
\end{figure}
%%%%%%%%%%%%%%%%%%%%%%%%%%%%%%%%%%%%%%%%%%%%%%%%%%%%%%%%%%%%%%%%%%%%%%%%%%%%%%%%
%
%%%%%%%%%%%%%%%%%%%%%%%%%%%%%%%%%%%%%%%%%%%%%%%%%%%%%%%%%%%%%%%%%%%%%%%%%%%%%%%%
\begin{defi}
\label{d:compost}
Consider the family of `horizontal' Lagrangian submanifolds
$$
\widehat{\mathscr L}_\theta:=\{(y,\theta)\mid y\in\mathbb R^2\}\subset T^*\mathbb R^2,\quad
\theta\in\mathbb R^2.
$$
For $\mathbf q=q_1\dots q_n\in\mathscr A^\bullet$ and $\theta\in\mathbb R^2$, define
\begin{equation}
  \label{e:compost-1}
\begin{aligned}
\mathscr L_{\mathbf q,\theta}&:=\varphi_{n-1}(\varkappa_{q_n}^{-1}(\widehat{\mathscr L}_\theta))\cap \varphi_{-1}(\mathcal V^+_{\mathbf q})
\ \subset\ \mathcal V_{q_1}\ \subset\ T^*M\setminus 0,\\
\widehat{\mathscr L}_{\mathbf q,\theta}&:=\varkappa_{q_1}(\mathscr L_{\mathbf q,\theta})\ \subset\ \mathcal W_{q_1}\ \subset\ T^*\mathbb R^2\setminus 0.
\end{aligned}
\end{equation}
We call $\mathscr L_{q,\theta}:=\varkappa_q^{-1}(\widehat{\mathscr L}_\theta)\cap \mathcal V_q$,
$q\in\mathscr A$, the \textbf{original Lagrangian} corresponding to~$q,\theta$,
and $\mathscr L_{\mathbf q,\theta}$, $\mathbf q\in\mathscr A^\bullet$, the \textbf{propagated Lagrangian}
corresponding to $\mathbf q,\theta$. See Figure~\ref{f:incline}.
\end{defi}
%%%%%%%%%%%%%%%%%%%%%%%%%%%%%%%%%%%%%%%%%%%%%%%%%%%%%%%%%%%%%%%%%%%%%%%%%%%%%%%%
\Remarks
1. The set $\mathscr L_{\mathbf q,\theta}$ may be empty. This happens in particular
if $\mathcal V^+_{\mathbf q}=\emptyset$, if $\theta_2\leq 0$,
or if $|\theta_1/\theta_2|\geq C\varepsilon_0$ for some large fixed~$C$.

\noindent 2. We see from the definition~\eqref{e:compost-1} and the properties
of $\varkappa_q$ in Lemma~\ref{l:stun-straight} that
$\mathscr L_{\mathbf q,\theta}$ is a Lagrangian submanifold of $p^{-1}(\theta_2)\subset T^*M\setminus 0$
and the flow lines of $\varphi_t$ are tangent to~$\mathscr L_{\mathbf q,\theta}$.
Therefore $\widehat{\mathscr L}_{\mathbf q,\theta}$ is a Lagrangian submanifold
of $\{(y,\eta)\mid \eta_2=\theta_2\}\subset T^*\mathbb R^2\setminus 0$
and $\partial_{y_2}$ is tangent to this manifold. 

\noindent 3. Recalling the definition~\eqref{e:V+-} of $\mathcal V^+_{\mathbf q}$,
we see that $\mathscr L_{\mathbf q,\theta}$ is obtained starting
from the original Lagrangian
$\mathscr L_{q_n,\theta}=\varkappa_{q_n}^{-1}(\widehat{\mathscr L}_{\theta})\cap \mathcal V_{q_n}$
by iteratively applying the map $\varphi_1$ and intersecting with $\mathcal V_{q_{n-1}},\dots,
\mathcal V_{q_1}$:
\begin{equation}
  \label{e:compost-iter}
\mathscr L_{q_j\dots q_n,\theta}
=\varphi_1(\mathscr L_{q_{j+1}\dots q_n,\theta})\cap \mathcal V_{q_j},\quad
1\leq j<n.
\end{equation}

\smallskip

By~\eqref{e:unrec} the submanifold $\mathscr L_{\mathbf q,\theta}$
is contained in a $C/\mathcal J^+_{\mathbf q}$ neighborhood of the weak unstable leaf
$W_{0u}(\tilde\rho)$, for any $\tilde\rho\in \mathscr L_{\mathbf q,\theta}$.
The next statement, which is a weak version of the \emph{Inclination Lemma}, shows
in particular that $\mathscr L_{\mathbf q,\theta}$ is controlled as a $C^\infty$ submanifold
uniformly in $\mathbf q,\theta$, regardless of the length of $\mathbf q$.
(A stronger version is that
$\mathscr L_{\mathbf q,\theta}$ is exponentially close in $C^\infty$ to $W_{0u}(\tilde\rho)$
when $|\mathbf q|$ is large.)
To make the statement precise it is convenient to write the image $\widehat{\mathscr L}_{\mathbf q,\theta}$ of $\mathscr L_{\mathbf q,\theta}$ under $\varkappa_{q_1}$ as a graph in the~$y$ variables.
%%%%%%%%%%%%%%%%%%%%%%%%%%%%%%%%%%%%%%%%%%%%%%%%%%%%%%%%%%%%%%%%%%%%%%%%%%%%%%%%
\begin{lemm}
  \label{l:inclinator}
If $\varepsilon_0>0$ is small enough depending only on $(M,g)$ then the following holds.
Let $\mathbf q\in\mathscr A^\bullet$, $\theta\in\mathbb R^2$, and 
assume that $\mathscr L_{\mathbf q,\theta}\neq \emptyset$. Then
\begin{equation}
  \label{e:inclinator}
\widehat{\mathscr L}_{\mathbf q,\theta}=\{(y,\eta)\mid y\in\mathscr U_{\mathbf q,\theta},\
\eta_1=\theta_2 G_{\mathbf q,\theta}(y_1),\ \eta_2=\theta_2\}
\end{equation}
where $\mathscr U_{\mathbf q,\theta}\subset\mathbb R^2$ is an open set and
$G_{\mathbf q,\theta}$ is a function on an open subset of $\mathbb R$ which
satisfies the following derivative bounds:
\begin{enumerate}
\item $\|G_{\mathbf q,\theta}\|_{C^1}\leq C\varepsilon_0$ for some constant
$C$ depending only on~$(M,g)$;
\item $\|G_{\mathbf q,\theta}\|_{C^{\mathbf N}}\leq C_{\mathbf N}$ for all $\mathbf N$,%
\footnote{Here and in Proposition~\ref{l:longtime-prop} below we use boldface $\mathbf N$ to distinguish it from
the propagation time defined in~\eqref{e:prop-times}.}
where the constant
$C_{\mathbf N}$ depends only on~$(M,g)$ and~$\mathbf N$.
\end{enumerate}
Moreover, if $F_{\mathbf q,\theta}:\mathscr U_{\mathbf q,\theta}\to \mathbb R^2$
is defined by
\begin{equation}
  \label{e:F-q-theta}
\varphi_{n-1}(\varkappa_{q_n}^{-1}(F_{\mathbf q,\theta}(y),\theta))
=\varkappa_{q_1}^{-1}(y,\theta_2G_{\mathbf q,\theta}(y_1),\theta_2),\quad
y\in\mathscr U_{\mathbf q,\theta}
\end{equation}
then we have the weakly contracting property for some $C$ depending only on $(M,g)$
\begin{equation}
  \label{e:F-q-theta-der}
\|dF_{\mathbf q,\theta}(y)\|\leq C\quad\text{for all}\quad
y\in \mathscr U_{\mathbf q,\theta}.
\end{equation}
\end{lemm}
%%%%%%%%%%%%%%%%%%%%%%%%%%%%%%%%%%%%%%%%%%%%%%%%%%%%%%%%%%%%%%%%%%%%%%%%%%%%%%%%
\Remark
The set $\mathscr U_{\mathbf q,\theta}$ (the domain of the function $G_{\mathbf q,\theta}$)
depends on $\mathbf q$ but it has macroscopic size (of the same scale as the sets $\mathcal V_q$) even for long words $\mathbf q$.

We omit the proof of Lemma~\ref{l:inclinator} here, referring the reader to~\cite[Proposition~5.1]{NZ09}, \cite[Proposition~6.2.23]{KaHa}, and the first version of this article~\cite[Lemma~4.7]{v1}.

%%%%%%%%%%%%%%%%%%%%%%%%%%%%%%%%%%%%%%%%%%%%%%%%%%%%%%%%%%%%%%%%%%%%%%%%%%%%%%%%
We now quantize the symplectomorphisms $\varkappa_q$. As explained following~\eqref{e:kappa-rho-0}
the flipped graph of each~$\varkappa_q$ is generated by a single phase function.
Then (see~\S\ref{s:prelim-fio-s}) there exist Fourier integral operators
\begin{equation}
  \label{e:B-q-def}
\begin{aligned}
B_q:L^2(M)\to L^2(\mathbb R^2),&\quad
B_q\in I^{\comp}_h(\varkappa_q),\\
B'_q:L^2(\mathbb R^2)\to L^2(M),&\quad
B'_q\in I^{\comp}_h(\varkappa_q^{-1})
\end{aligned}
\end{equation}
quantizing $\varkappa_q$ near $\varkappa_q(\overline{\mathcal V}_q\cap \{{1\over 4}\leq |\xi|_g\leq 4\})\times
(\overline{\mathcal V}_q\cap \{{1\over 4}\leq|\xi|_g\leq 4\})$ in the sense of~\eqref{e:fio-quantize}.

Using the operators $B_q$ we give a precise definition of the classes
$I^{\comp}_h(\mathscr L_{q_n,\theta})$ and $I^{\comp}_h(\mathscr L_{\mathbf q,\theta})$
featured in~\eqref{e:rough-lag}. We have $\mathscr L_{q_n,\theta}= \varkappa_{q_n}^{-1}(\widehat{\mathscr L}_\theta)\cap\mathcal V_{q_n}$
where $\widehat{\mathscr L}_\theta$ is generated in the sense of~\eqref{e:basic-lagrangian}
by the function
\begin{equation}
  \label{e:Phi-theta-def}
\Phi_\theta\in C^\infty(\mathbb R^2;\mathbb R),\quad
\Phi_\theta(y)=\langle y,\theta\rangle.
\end{equation}
Thus by~\eqref{e:lag-dist-basic} the elements of $I^{\comp}_h(\mathscr L_{q_n,\theta})$ which
are microlocalized in $\mathcal \{{1\over 4}<|\xi|_g<4\}$ have the form
$B'_{q_n} (e^{i\Phi_\theta/h}a)$ for some $a\in C^\infty(\mathbb R^2)$. We will
in fact take $a\equiv 1$.

Next, by Lemma~\ref{l:inclinator} the Lagrangian manifold
$\widehat{\mathscr L}_{\mathbf q,\theta}=\varkappa_{q_1}(\mathscr L_{\mathbf q,\theta})$
is generated in the sense of~\eqref{e:basic-lagrangian} by a function
$$
\Phi_{\mathbf q,\theta}\in C^\infty(\mathscr U_{\mathbf q,\theta};\mathbb R),\quad
\partial_{y_1}\Phi_{\mathbf q,\theta}=\theta_2 G_{\mathbf q,\theta}(y_1),\quad
\partial_{y_2}\Phi_{\mathbf q,\theta}=\theta_2.
$$
Here $\Phi_{\mathbf q,\theta}$ is defined uniquely up to a locally constant function.
We fix this freedom by recalling that the functions induced on~$\widehat{\mathscr L}_\theta,
\widehat{\mathscr L}_{\mathbf q,\theta}$ by $\Phi_\theta,\Phi_{\mathbf q,\theta}$
are antiderivatives on these Lagrangian submanifolds (see~\eqref{e:basic-lagrangian}).
The antiderivative
on $\widehat{\mathscr L}_{\mathbf q,\theta}$ can be computed by applying~\eqref{e:anti-composition}
to the definition~\eqref{e:compost-1},
where the symplectomorphisms $\varkappa_{q_1},\varphi_{n-1},\varkappa_{q_n}^{-1}$ are homogeneous
and thus have zero antiderivative (see~\S\ref{s:prelim-fio-s}). 
Thus we may put
\begin{equation}
  \label{e:Phi-q-theta-def}
\Phi_{\mathbf q,\theta}(y):=\Phi_\theta(F_{\mathbf q,\theta}(y)),\quad
y\in\mathscr U_{\mathbf q,\theta},
\end{equation}
where $F_{\mathbf q,\theta}$ is defined in~\eqref{e:F-q-theta}.
Then by~\eqref{e:lag-dist-basic} the elements of $I^{\comp}_h(\mathscr L_{\mathbf q,\theta})$
which are microlocalized in $\{{1\over 4}<|\xi|_g<4\}$ have the form
$B'_{q_1}(e^{i\Phi_{\mathbf q,\theta}/h}a)$ for some $a\in \CIc(\mathscr U_{\mathbf q,\theta})$.

Building on the above discussion we now give the main statement of this subsection,
which is a precise version of~\eqref{e:rough-lag}. We again omit the proof, referring to~\cite[Proposition~4.1 and~\S7.2]{NZ09}
and to the first version of this article~\cite[Proposition~4.8]{v1}. See also~\cite[\S3]{AnantharamanShort} for a simplified proof in a model case.
%%%%%%%%%%%%%%%%%%%%%%%%%%%%%%%%%%%%%%%%%%%%%%%%%%%%%%%%%%%%%%%%%%%%%%%%%%%%%%%%
\begin{prop}
  \label{l:longtime-prop}
Assume that $\varepsilon_0$ is small enough depending only on $(M,g)$.
Let $\mathbf q=q_1\dots q_n\in\mathscr A^\bullet$, $\theta\in\mathbb R^2$,
and assume that $n\leq C_0\log(1/h)$, $|\theta_1|\leq C_0$, ${1\over 4}\leq \theta_2\leq 4$ for some constant~$C_0$.
Define $\Phi_\theta,\Phi_{\mathbf q,\theta}$ using~\eqref{e:Phi-theta-def},
\eqref{e:Phi-q-theta-def}. Let $U^+_{\mathbf q}$
be defined in~\eqref{e:U+def} and fix $\mathbf N>0$. Then we have uniformly in $\mathbf q,\theta$
\begin{equation}
  \label{e:longtime-prop}
U^+_{\mathbf q} B'_{q_n}(e^{i\Phi_\theta/h})=U(1)B'_{q_1}(e^{i\Phi_{\mathbf q,\theta}/h}a_{\mathbf q,\theta,\mathbf N})+\mathcal O(h^{\mathbf N})_{L^2(M)}
\end{equation}
for some $a_{\mathbf q,\theta,\mathbf N}(y;h)\in \CIc(\mathscr U_{\mathbf q,\theta})$ such that:
\begin{enumerate}
\item the distance between $\supp a_{\mathbf q,\theta,\mathbf N}$ and the complement
of $\mathscr U_{\mathbf q,\theta}$ is larger than $C^{-1}$
for some constant $C>0$ depending only on the choice of $A_q,\mathcal V_q,\varkappa_q$,
$q\in\mathscr A$;
\item for any multiindex $\alpha$ there exists $C_{\mathbf N,\alpha}>0$ such that
\begin{equation}
  \label{e:longprop-derbs}
\sup_y|\partial^\alpha_y a_{\mathbf q,\theta,\mathbf N}(y) |\leq C_{\mathbf N,\alpha}.
\end{equation}
Here $C_{\mathbf N,\alpha}$ depends only on the choices of $A_q,B_q,B'_q$, and~$C_0$.
\end{enumerate}
\end{prop}
%%%%%%%%%%%%%%%%%%%%%%%%%%%%%%%%%%%%%%%%%%%%%%%%%%%%%%%%%%%%%%%%%%%%%%%%%%%%%%%%
\Remarks
1. If $\mathscr L_{\mathbf q,\theta}=\emptyset$ then we have $a_{\mathbf q,\theta,\mathbf N}=0$ and Proposition~\ref{l:longtime-prop} states that the left-hand side of~\eqref{e:longtime-prop}
is $\mathcal O(h^\infty)_{L^2(M)}$.

\noindent 2. \cite{AN07,NZ09} show that the symbols $a_{\mathbf q,\theta,\mathbf N}$ satisfy stronger bounds,
in fact they decay exponentially with $|\mathbf q|$,
see~\cite[Lemma~3.5]{AN07} and~\cite[(7.11)]{NZ09}. We state the weaker bound~\eqref{e:longprop-derbs} since it suffices for our application in~\S\ref{s:micro-conjugation}.

%%%%%%%%%%%%%%%%%%%%%%%%%%%%%%%%%%%%%%%%%%%%%%%%%%%%%%%%%%%%%%%%%%%%%%%%%%%%%%%%
\subsection{Reduction to words of moderate length}
\label{s:moderate-reduction}

We now return to the proof of Proposition~\ref{l:longdec-1}. Henceforth we fix two words
$$
\mathbf v\in \mathscr A_\star^{N_0},\quad
\mathbf w\in \mathscr A_\star^{N_1}.
$$
We first write a decomposition~\eqref{e:Aw-dec} of $A^+_{\mathbf w}$ into a sum of terms of the form $A^+_{\mathbf q}$ where
$\mathbf q$ are words over the refined alphabet $\mathscr A=\{1,\dots,Q\}$ (see~\S\ref{s:refined-partition}).
For that we use the following
%%%%%%%%%%%%%%%%%%%%%%%%%%%%%%%%%%%%%%%%%%%%%%%%%%%%%%%%%%%%%%%%%%%%%%%%%%%%%%%%
\begin{defi}
  \label{d:prec-word}
For $q\in\mathscr A$ and $w\in\mathscr A_\star$, we write
$q\lesssim w$ if one of the following holds:
\begin{itemize}
\item $w=1$ and $q=1$, or
\item $w=\star$ and $q\in \{2,\dots,Q\}$.
\end{itemize}
If $\mathbf q=q_1\dots q_{n}\in \mathscr A^\bullet$
and $\mathbf w=w_1\dots w_{m}\in \mathscr A_\star^\bullet$, then
we say that $\mathbf q\lesssim \mathbf w$ if
$n\leq m$ and $q_j\lesssim w_j$ for all $j=1,\dots,n$.
\end{defi}
%%%%%%%%%%%%%%%%%%%%%%%%%%%%%%%%%%%%%%%%%%%%%%%%%%%%%%%%%%%%%%%%%%%%%%%%%%%%%%%%
Since $A_\star=A_2+\dots +A_Q$, we have
\begin{equation}
  \label{e:Aw-dec}
A^+_{\mathbf w}=\sum_{\mathbf q\in \mathscr A^{N_1},\ \mathbf q\lesssim \mathbf w} A^+_{\mathbf q}.
\end{equation}

Since $N_1$ is larger than the maximal Ehrenfest time $T_{\max}$ (see~\eqref{e:N-1-def}), for all words $\mathbf q\in\mathscr A^{N_1}$ we have $\cJ^{+}_{\bq}>h^{-1}$, so the symbol $a^+_{\mathbf q}$
is very irregular.
To fix this problem, we will rewrite~\eqref{e:Aw-dec} in terms of an expression with involves words
with length bounded by the local double Ehrenfest time~-- see~\eqref{e:furdec}
and Figure~\ref{f:moderate}.
%%%%%%%%%%%%%%%%%%%%%%%%%%%%%%%%%%%%%%%%%%%%%%%%%%%%%%%%%%%%%%%%%%%%%%%%%%%%%%%%
\begin{figure}
\includegraphics[height=4.75cm]{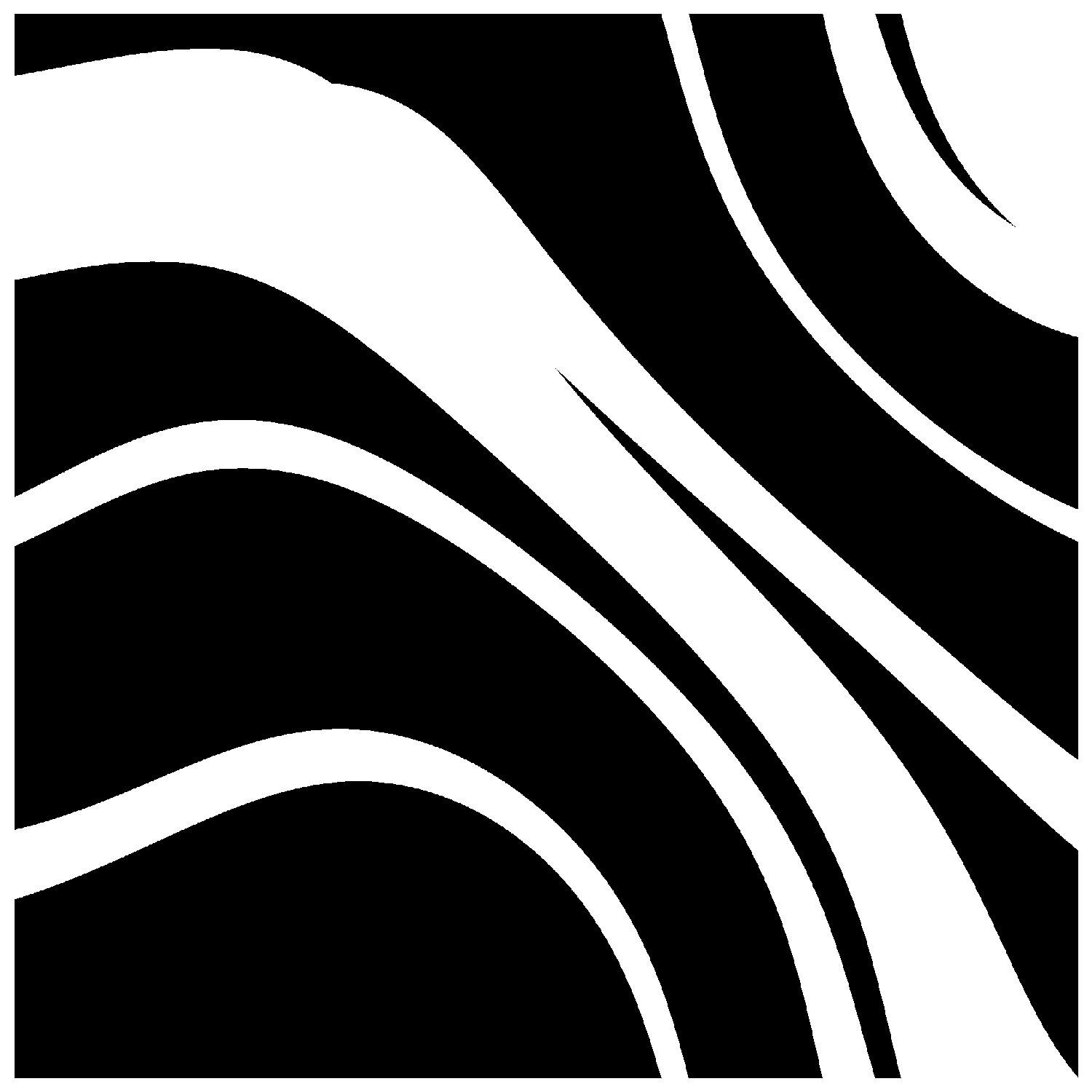}\quad
\includegraphics[height=4.75cm]{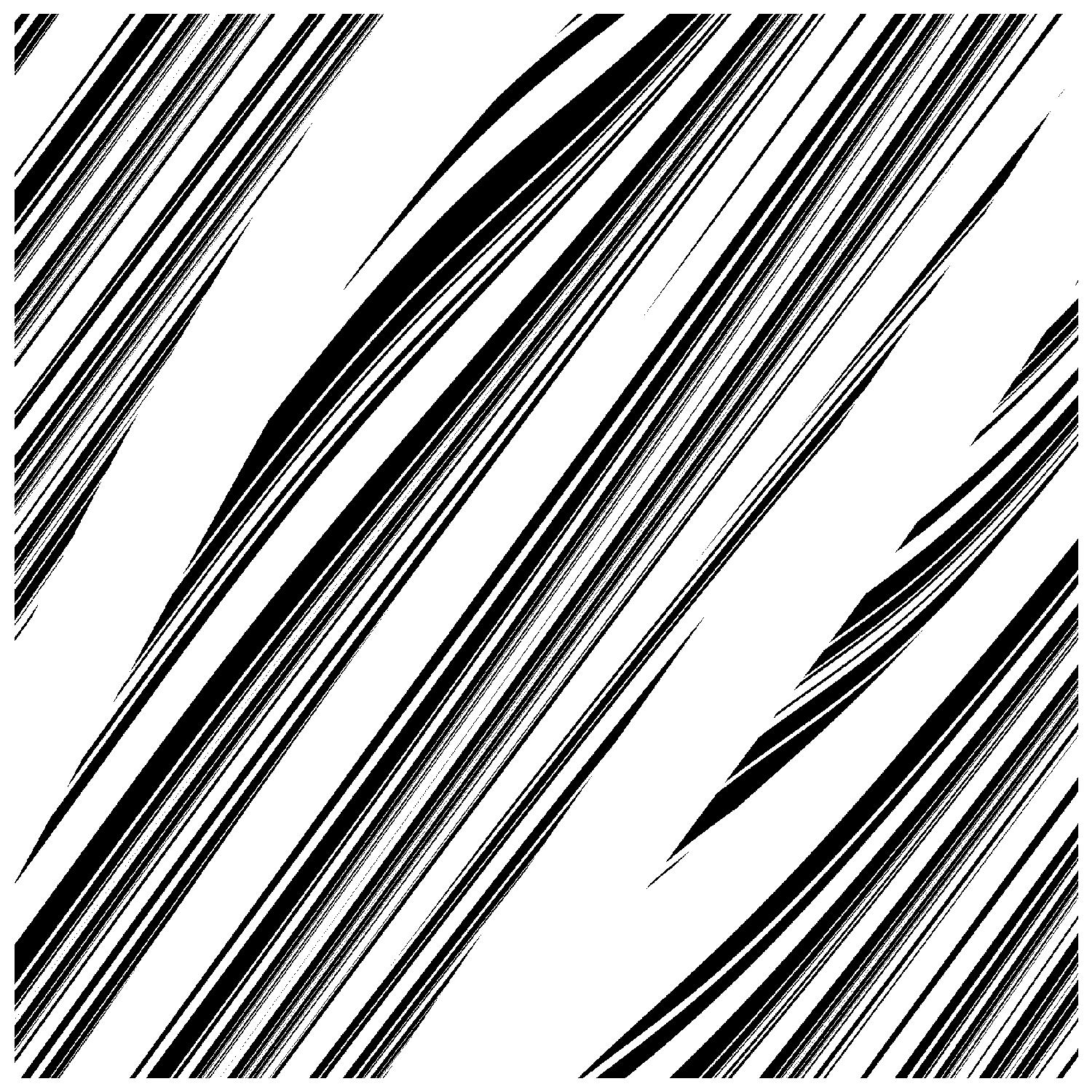}\quad
\includegraphics[height=4.75cm]{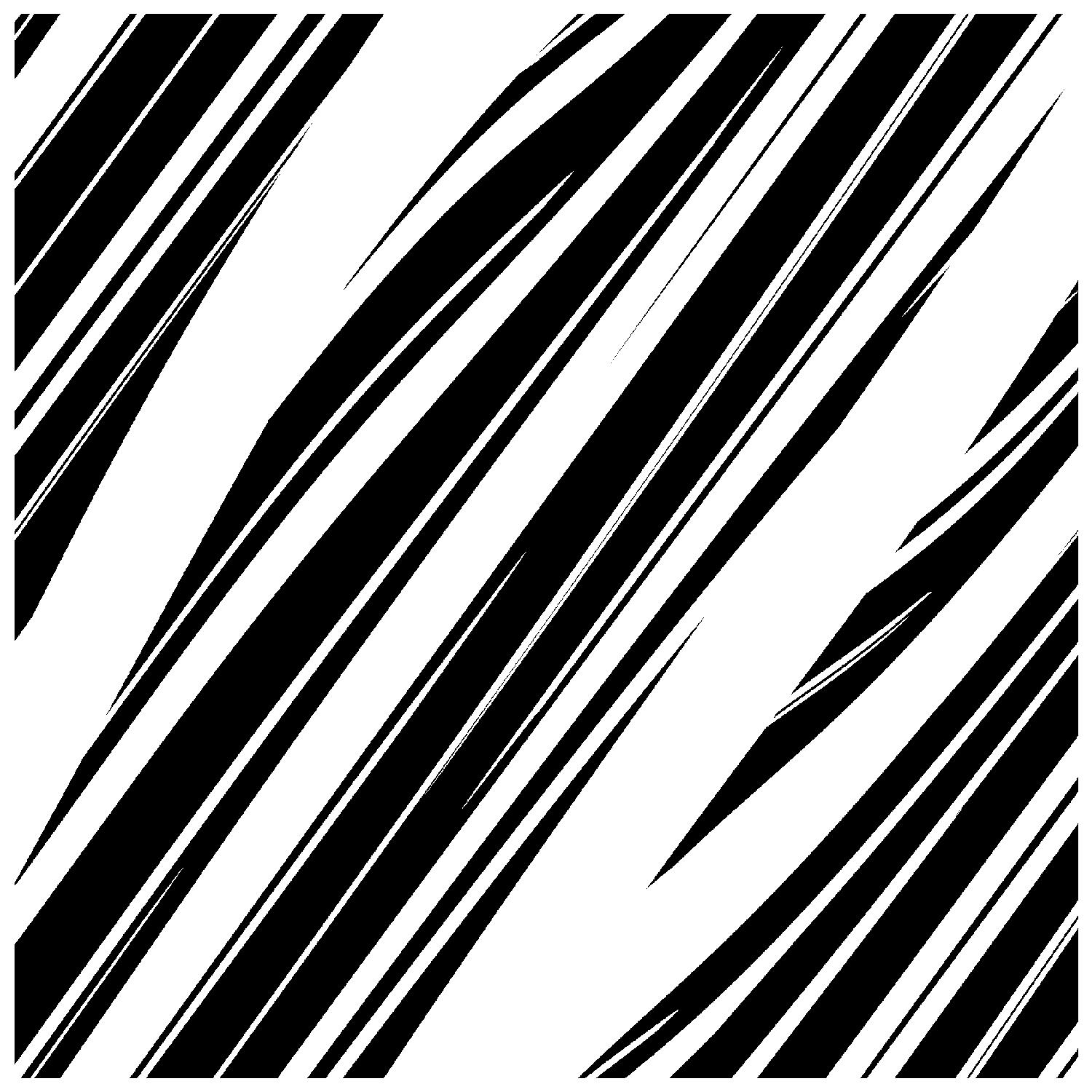}
\hbox to\hsize{\hss $\supp a^-_{\mathbf v}$\hss\hss \hskip.2in$\supp a^+_{\mathbf w}$\hskip-.2in\hss\hss 
$\supp\sum_{n,e} a^+_{\mathcal Q_n(\mathbf w,e)}$\hskip-.4in \hss}
\caption{Supports of the symbols
$a^-_{\mathbf v}$, $a^+_{\mathbf w}$, and $\sum_{n,e} a^+_{\mathcal Q_n(\mathbf w,e)}$,
corresponding to the operators $A^-_{\mathbf v}$, $A^+_{\mathbf w}$,
and $\sum_{n,e} A^+_{\mathcal Q_n(\mathbf w,e)}$. We restrict to some hypersurface
in $S^*M$ transversal to the flow direction. By~\eqref{e:prop-times} and~\eqref{e:N-1-def}
the thickness of the strokes in $\supp a^-_{\mathbf v}$
(corresponding to the Jacobian $(J^u_{N_0})^{-1}$) is at least $h^{1/6}$,
while in $\supp a^+_{\mathbf w}$ it is at most $h$.
Both of these have strokes of very different thicknesses because
the Jacobians vary from point to point.
The set $\supp\sum_{n,e} a^+_{\mathcal Q_n(\mathbf w,e)}$
contains $\supp a^+_{\mathbf w}$ and has strokes of uniform thickness approximately $h^{-\tau}=h^{-2\delta}$
(roughly speaking, each stroke corresponds to one term $a^+_{\mathbf q}$), so that classical/quantum correspondence still applies.}
\label{f:moderate}
\end{figure}
%%%%%%%%%%%%%%%%%%%%%%%%%%%%%%%%%%%%%%%%%%%%%%%%%%%%%%%%%%%%%%%%%%%%%%%%%%%%%%%%

Recall the `minimal/maximal expansion rates' $0<\Lambda_0\leq\Lambda_1$ defined in~\eqref{e:Lambda-0-1};
as before we put $\Lambda:=\lceil \Lambda_1/\Lambda_0\rceil$.
We fix constants
\begin{equation}
  \label{e:tau-delta-def}
\tau:=1-{1\over 10\Lambda},\quad
\delta := {\tau\over 2}<{1\over 2}.
\end{equation}
Note that $\tau$ is very close to~1; this will be used in~\eqref{e:porosity-finally-meets} below.
(In~\cite{meassupp} the parameter $\tau$ was denoted by $\rho$.)

For $n=1,\dots,N_1$ and $e\in\mathscr A$
let us define sets of refined words starting with the letter~$e$ and controlled by their local Jacobians:
\begin{equation}
  \label{e:Q-n-def}
\begin{aligned}
\mathcal Q_n(\mathbf w,e)&:=\{\mathbf q=q_1\dots q_n\in \mathscr A^n\mid q_1=e,\ \mathbf q\lesssim\mathbf w,\ \mathcal J^+_{\mathbf q}\geq h^{-\tau}> \mathcal J^+_{\mathbf q'}\},\\
\mathcal Q'_n(\mathbf w,e)&:=\{\mathbf q\in\mathcal Q_n(\mathbf w,e)\mid \mathcal V^+_{\mathbf q}\neq\emptyset\},\\
\mathcal Q''_n(\mathbf w,e)&:=\{\mathbf q\in\mathcal Q_n(\mathbf w,e)\mid \mathcal V^+_{\mathbf q}=\emptyset\}
\end{aligned}
\end{equation}
where we recall that for any $\bq=q_1\cdots q_{n}$, we denote $\bq':=q_1\cdots q_{n-1}$. 
By~\eqref{e:jacobians-concat} we have for some constant $C$ depending only on $(M,g)$
\begin{equation}
  \label{e:go-deep}
h^{-\tau}\leq \mathcal J^+_{\mathbf q}\leq Ch^{-\tau}=Ch^{-2\delta}\quad\text{for all}\quad
\mathbf q\in\mathcal Q'_n(\mathbf w,e).
\end{equation}
That is,  words $\bq\in \cQ_n'(\bw,e)$ correspond to sets $\cV^+_\bq$ on which the backwards stable
Jacobian $J^s_{-n}(\rho)$ is approximately equal to $h^{-\tau}$. These words are such that their local double Ehrenfest time $\widetilde T^+_{\bq}$ is approximately equal to their length $n$ (they would be equal if we had taken $\tau=1$).

For each $\mathbf q=q_1\dots q_{N_1}\in\mathscr A^{N_1}$ with $\mathbf q\lesssim \mathbf w$ 
we have $\mathcal J^+_{\mathbf q}\geq e^{\Lambda_0 N_1}\geq h^{-1}\geq h^{-\tau}$ by~\eqref{e:N-1-def} and~\eqref{e:jacobian-bounds}.
Using~\eqref{e:jacobian-grows} we see that for each such $\mathbf q$ there exists unique
$n\in\{1,\dots,N_1\}$ such that the prefix
$q_1\dots q_{n}$ lies in $\mathcal Q_n(\mathbf w,q_1)$.
We also have $\mathcal Q_n(\mathbf w,e)=\mathcal Q'_n(\mathbf w,e)\sqcup \mathcal Q''_n(\mathbf w,e)$.
Therefore the decomposition~\eqref{e:Aw-dec} can be written as
\begin{equation}
  \label{e:furdec}
A^+_{\mathbf w}=\sum_{n=1}^{N_1}\sum_{e\in\mathscr A}A^+_{\mathcal Q_n(\mathbf w,e)}Z_{n,\mathbf w}
=\sum_{n=1}^{N_1}\sum_{e\in\mathscr A}(A^+_{\mathcal Q'_n(\mathbf w,e)}
+A^+_{\mathcal Q''_n(\mathbf w,e)})Z_{n,\mathbf w}
\end{equation}
where $A^+_{\mathcal Q_n(\mathbf w,e)}$ is defined by~\eqref{e:A-E-def} and
$$
Z_{n,\mathbf w}:=A_{w_{n+1}}(-n-1)\cdots A_{w_{N_1}}(-N_1) = U(n+1) A^+_{w_{n+1}\cdots w_{N_1}}U(-n-1).
$$
We have $\|Z_{n,\mathbf w}\|_{L^2\to L^2}\leq 2$ similarly to~\eqref{e:A-q-bdd}.
Moreover, since the number of elements of $\mathcal Q''_n(\mathbf w,e)$ is bounded
by some negative power of~$h$,
by part~2 of Lemma~\ref{l:ehrenfest-prop-none} we get
$$
\|A^+_{\mathcal Q''_n(\mathbf w,e)}\|_{L^2\to L^2}=\mathcal O(h^\infty).
$$
We then estimate
$$
\|A^-_{\mathbf v}A^+_{\mathbf w}\|_{L^2\to L^2}
\leq 2\sum_{n=1}^{N_1}\sum_{e\in\mathscr A}\|A^-_{\mathbf v}A^+_{\mathcal Q'_n(\mathbf w,e)}\|_{L^2\to L^2}
+\mathcal O(h^\infty).
$$
Since $N_1=\mathcal O(\log(1/h))$,
Proposition~\ref{l:longdec-1} is proved once we establish
its analogue with $A^+_{\mathbf w}$ replaced by $A^+_{\mathcal Q'_n(\mathbf w,e)}$, that is the sum of $A^+_{\mathbf q}$
over the refined words $\mathbf q$ with length~$n$, initial letter~$e$, and local Jacobians $\cJ^+_{\bq}\sim h^{-\tau}$ (that is, their local double Ehrenfest time is approximately equal to $n$):
%%%%%%%%%%%%%%%%%%%%%%%%%%%%%%%%%%%%%%%%%%%%%%%%%%%%%%%%%%%%%%%%%%%%%%%%%%%%%%%%
\begin{prop}
\label{l:longdec-2}
Assume that $\mathbf v\in \mathscr A_\star^{N_0}$, $\mathbf w\in \mathscr A_\star^{N_1}$,
$1\leq n\leq N_1$, and $e\in\mathscr A$. Then there
exists $\beta>0$ depending only on $\mathcal V_1,\mathcal V_\star$
and there exists $C>0$ depending only on~$A_1,A_\star$ such that
$$
\|A^-_{\mathbf v}A^+_{\mathcal Q'_n(\mathbf w,e)}\|_{L^2\to L^2}\leq Ch^\beta.
$$
\end{prop}
%%%%%%%%%%%%%%%%%%%%%%%%%%%%%%%%%%%%%%%%%%%%%%%%%%%%%%%%%%%%%%%%%%%%%%%%%%%%%%%%
\Remark
The value of~$\beta$ in Proposition~\ref{l:longdec-1}
can be taken to be any number smaller than the value of~$\beta$ in Proposition~\ref{l:longdec-2}.
Since we do not give a precise formula for $\beta$ we call both by the same letter
to simplify notation.

%%%%%%%%%%%%%%%%%%%%%%%%%%%%%%%%%%%%%%%%%%%%%%%%%%%%%%%%%%%%%%%%%%%%%%%%%%%%%%%%
\subsection{Partition into clusters}
  \label{s:clusters}

We fix $\mathbf v\in\mathscr A_\star^{N_0}$,
$\mathbf w\in\mathscr A_\star^{N_1}$,
$n\in \{1,\dots,N_1\}$, $e\in\mathscr A$,
and define $\mathcal Q'_n(\mathbf w,e)\subset \mathscr A^n$
by~\eqref{e:Q-n-def}.
We make the following
%%%%%%%%%%%%%%%%%%%%%%%%%%%%%%%%%%%%%%%%%%%%%%%%%%%%%%%%%%%%%%%%%%%%%%%%%%%%%%%%
\begin{defi}
  \label{d:close-far}
Let $\mathbf q,\tilde{\mathbf q}\in \mathcal Q'_n(\mathbf w,e)$. We say
$\mathbf q,\tilde{\mathbf q}$ are \textbf{close} to each other
if $\mathcal V^+_{\mathbf q}\cup\mathcal V^+_{\tilde{\mathbf q}}$
lies in the $h^{2/3}$-sized conic neighborhood of some weak unstable leaf,
more precisely there exists $\rho\in \mathcal V^+_{e}\cap S^*M$ such that
$$
d(\tilde\rho,W_{0u}(\rho))\leq h^{2/3}\quad\text{for all}\quad
\tilde\rho\in (\mathcal V^+_{\mathbf q}\cup \mathcal V^+_{\tilde{\mathbf q}})\cap S^*M.
$$
If $\mathbf q,\tilde{\mathbf q}$ are not close to each other, we say
they are \textbf{far} from each other.
\end{defi}
%%%%%%%%%%%%%%%%%%%%%%%%%%%%%%%%%%%%%%%%%%%%%%%%%%%%%%%%%%%%%%%%%%%%%%%%%%%%%%%%
\Remark If $\mathbf q$, $\tilde{\mathbf q}$ are far from each other, then $\mathcal V^+_{\mathbf q}\cap \mathcal V^+_{\tilde{\mathbf q}}=\emptyset$. The proof of Lemma~\ref{l:faror} below in fact gives a stronger statement, see~\eqref{e:pimptor}.

For words which are far from each other, we have the following almost orthogonality statement:
%%%%%%%%%%%%%%%%%%%%%%%%%%%%%%%%%%%%%%%%%%%%%%%%%%%%%%%%%%%%%%%%%%%%%%%%%%%%%%%%
\begin{lemm}
  \label{l:faror}
Assume that $\mathbf q,\tilde{\mathbf q}\in\mathcal Q'_n(\mathbf w,e)$ are
far from each other. Then
\begin{align}
  \label{e:faror-1}
\|(A^-_{\mathbf v}A^+_{\mathbf q})^* A^-_{\mathbf v}A^+_{\tilde{\mathbf q}}\|_{L^2\to L^2}
&=\mathcal O(h^\infty),\\
  \label{e:faror-2}
\| A^-_{\mathbf v}A^+_{\tilde{\mathbf q}}(A^-_{\mathbf v}A^+_{\mathbf q})^*\|_{L^2\to L^2}
&=\mathcal O(h^\infty)
\end{align}
with the constants in $\mathcal O(h^\infty)$ independent of $h,n,\mathbf v,\mathbf w,\mathbf q,\tilde{\mathbf q}$.
\end{lemm}
%%%%%%%%%%%%%%%%%%%%%%%%%%%%%%%%%%%%%%%%%%%%%%%%%%%%%%%%%%%%%%%%%%%%%%%%%%%%%%%%
\Remark 
Lemma~\ref{l:faror} has the following informal interpretation (which is different from the formal proof below). 
Imagine that we remove the flow and dilation directions from $T^*M$
and conjugate by a Fourier integral operator whose canonical transformation maps stable leaves
into horizontal lines $\{\eta=\const\}$ and unstable leaves into vertical lines
$\{y=\const\}$ on $T^*\mathbb R_y\simeq \mathbb R^2_{y,\eta}$.
(This is not possible to do globally but the argument
in~\S\ref{s:fup-endgame} below uses a localized version of such conjugation
with the roles of $y,\eta$ switched.) 
Then $A^-_{\mathbf v}$ is replaced by a Fourier multiplier
$\chi_-(hD_y)$ where $\sup_\eta|\partial^k_\eta \chi_-(\eta;h)|=\mathcal O(h^{-k/6-})$
(corresponding to the fact that $a^-_{\mathbf v}\in S^{\comp}_{1/6+}$ which follows
from Lemma~\ref{l:cq-log}).
Next, $A^+_{\mathbf q},A^+_{\tilde{\mathbf q}}$ are replaced by multiplication operators
$\chi_+(y)$, $\widetilde\chi_+(y)$ where $\chi_+,\widetilde\chi_+$
have supports of size $\sim h^\tau$.
The condition that $\mathbf q,\tilde{\mathbf q}$ are far from each other
implies that the supports of $\chi_+$, $\widetilde\chi_+$ are
at least $h^{2/3}$ apart. Then~\eqref{e:faror-1}
turns into the estimate (assuming $\chi_-,\chi_+,\widetilde\chi_+$ are real valued)
$$
\|\chi_+(y)\chi_-^2(hD_y)\widetilde\chi_+(y)\|_{L^2(\mathbb R)\to L^2(\mathbb R)}
=\mathcal O(h^\infty)
$$
which can be proved using repeated integration by parts to establish
rapid decay of the integral kernel: at each integration we gain a factor $h\cdot h^{-2/3}\cdot h^{-1/6}=h^{1/6}$. Notice that the size of the supports of $\chi_+$ and $\tilde\chi_+$ does not matter, it is the distance between the two supports which is responsible for the factor $h^{-2/3}$.
In turn, the analogue of~\eqref{e:faror-2} trivially follows
from the fact that $\supp \chi_+\cap\supp\widetilde\chi_+=\emptyset$.
In this interpretation~\eqref{e:faror-1}, \eqref{e:faror-2}
are analogous to the bounds~\cite[(4.26),(4.25)]{fullgap}
and the decomposition into clusters below
to the one used in the proof of~\cite[Proposition~4.3]{fullgap}.
%%%%%%%%%%%%%%%%%%%%%%%%%%%%%%%%%%%%%%%%%%%%%%%%%%%%%%%%%%%%%%%%%%%%%%%%%%%%%%%%
\begin{proof}
1. Denote $\mathbf q=q_1\dots q_n$, $\tilde{\mathbf q}=\tilde q_1\dots \tilde q_n$.
Take maximal $m\leq n$ such that
$$
\mathcal V^+_{q_1\dots q_m}\cap \mathcal V^+_{\tilde q_1\dots \tilde q_m}\neq \emptyset.
$$
If $\mathcal V^+_{q_1}\cap\mathcal V^+_{\tilde q_1}=\emptyset$ then we put $m:=0$.

By~\eqref{e:jacobian-pair} we have $\mathcal J^+_{q_1\dots q_m}\sim\mathcal J^+_{\tilde q_1\dots \tilde q_m}$. We claim that
\begin{equation}
  \label{e:grizzly}
\max(\mathcal J^+_{q_1\dots q_m},\mathcal J^+_{\tilde q_1\dots \tilde q_m})\leq Ch^{-2/3}.
\end{equation}
The case $m=0$ is trivial, so we assume that $m>0$. Take
$\rho\in \mathcal V^+_{q_1\dots q_m}\cap \mathcal V^+_{\tilde q_1\dots \tilde q_m}\cap S^*M$.
Note that since $q_1=\tilde q_1=e$ we have $\rho\in \mathcal V^+_e$.
By~\eqref{e:unrec} we have for every
$\tilde\rho\in(\mathcal V^+_{\mathbf q}\cup\mathcal V^+_{\tilde{\mathbf q}})\cap S^*M
\subset (\mathcal V^+_{q_1\dots q_m}\cup\mathcal V^+_{\tilde q_1\dots \tilde q_m})\cap S^*M$
\begin{equation}
  \label{e:sunset}
d(\tilde\rho,W_{0u}(\rho))\leq {C'\over \min(\mathcal J^+_{q_1\dots q_m},\mathcal J^+_{\tilde q_1\dots \tilde q_m})}
\leq {C\over \max(\mathcal J^+_{q_1\dots q_m},\mathcal J^+_{\tilde q_1\dots \tilde q_m})}. 
\end{equation}
Since $\mathbf q,\tilde{\mathbf q}$ are far from each other, the right-hand side
of~\eqref{e:sunset} has to be greater than~$h^{2/3}$, which gives~\eqref{e:grizzly}.

By~\eqref{e:go-deep} we have $\mathcal J^+_{\mathbf q}\geq h^{-\tau}\gg h^{-2/3}$,
so from~\eqref{e:grizzly} we obtain $m<n$.
Denote
$$
\mathbf p:=q_1\dots q_{m+1},\quad
\tilde{\mathbf p}:=\tilde q_1\dots\tilde q_{m+1}.
$$
Since $m$ was chosen maximal we have
\begin{equation}
  \label{e:pimptor}
\mathcal V^+_{\mathbf p}\cap\mathcal V^+_{\tilde{\mathbf p}}=\emptyset.
\end{equation}
Moreover by~\eqref{e:grizzly} and~\eqref{e:jacobians-concat}
and since $\mathcal V^+_{\mathbf q},\mathcal V^+_{\tilde{\mathbf q}}\neq\emptyset$
and thus $\mathcal V^+_{\mathbf p},\mathcal V^+_{\tilde{\mathbf p}}\neq\emptyset$ we get
\begin{equation}
  \label{e:jabby}
\max(\mathcal J^+_{\mathbf p},\mathcal J^+_{\tilde{\mathbf p}})\leq Ch^{-2/3}.
\end{equation}

\noindent 2. We now prove~\eqref{e:faror-1}.
We have by~\eqref{e:word-concat} and~\eqref{e:reversing}
$$
\begin{aligned}
(A^-_{\mathbf v}A^+_{\mathbf q})^*A^-_{\mathbf v}A^+_{\tilde{\mathbf q}}
=\,&U(m+1)(A^+_{q_{m+2}\dots q_n})^*U(-m-1)
\\&\cdot(A^-_{\mathbf v}A^+_{\mathbf p})^* A^-_{\mathbf v}A^+_{\tilde{\mathbf p}}
U(m+1)A^+_{\tilde q_{m+2}\dots \tilde q_n}U(-m-1).
\end{aligned}
$$
Thus by~\eqref{e:A-q-bdd} it suffices to prove that
$\|(A^-_{\mathbf v}A^+_{\mathbf p})^* A^-_{\mathbf v}A^+_{\tilde{\mathbf p}}\|_{L^2\to L^2}=\mathcal O(h^\infty)$.
Similarly to~\eqref{e:Aw-dec} we write
$$
A^-_{\mathbf v}=\sum_{\mathbf s\in\mathscr A^{N_0},\
\mathbf s\lesssim \mathbf v}A^-_{\mathbf s}.
$$
Then by~\eqref{e:word-concat}
$$
(A^-_{\mathbf v}A^+_{\mathbf p})^*
A^-_{\mathbf v}A^+_{\tilde{\mathbf p}}
=\sum_{\mathbf s,\tilde{\mathbf s}\in\mathscr A^{N_0},\
\mathbf s,\tilde{\mathbf s}\lesssim \overline{\mathbf v}}
U(-N_0)(A^+_{\mathbf s\mathbf p})^*A^+_{\tilde{\mathbf s}\tilde{\mathbf p}}U(N_0).
$$
Since the number of terms in the sum above is bounded polynomially in~$h$, it suffices to show that
\begin{equation}
  \label{e:faraday}
\|(A^+_{\mathbf s\mathbf p})^*A^+_{\tilde{\mathbf s}\tilde{\mathbf p}}\|_{L^2\to L^2}
=\mathcal O(h^\infty)
\quad\text{for all}\quad
\mathbf s,\tilde{\mathbf s}\in\mathscr A^{N_0}.
\end{equation}
By~\eqref{e:prop-times}, \eqref{e:jacobian-bounds},
\eqref{e:jacobians-concat}, and~\eqref{e:jabby}
for each word $\mathbf t$ of length no more than $N_0$
we have
\begin{equation}
  \label{e:tylen}
\mathcal V^+_{\mathbf t\mathbf p}\neq\emptyset
\quad\Longrightarrow\quad
\mathcal J^+_{\mathbf t\mathbf p}\leq 
C\mathcal J^+_{\mathbf t}\mathcal J^+_{\mathbf p}\leq
Ce^{\Lambda_1N_0}\cdot h^{-2/3}\leq Ch^{-5/6}\leq Ch^{-2\delta}.
\end{equation}
Then by part~2 of Lemma~\ref{l:ehrenfest-prop-none},
if $\mathcal V^+_{\mathbf s\mathbf p}=\emptyset$ then
$\|A^+_{\mathbf s\mathbf p}\|_{L^2\to L^2}=\mathcal O(h^\infty)$
which immediately implies~\eqref{e:faraday}.
A similar argument applies to $A^+_{\tilde{\mathbf s}\tilde{\mathbf p}}$.

We may now assume that $\mathcal V^+_{\mathbf s\mathbf p}\neq\emptyset$,
$\mathcal V^+_{\tilde{\mathbf s}\tilde{\mathbf p}}\neq \emptyset$.
Then by~\eqref{e:tylen} we have
$\max(\mathcal J^+_{\mathbf s\mathbf p},\mathcal J^+_{\tilde{\mathbf s}\tilde{\mathbf p}})\leq
Ch^{-2\delta}$. Moreover $\mathcal V^+_{\mathbf s\mathbf p}\cap \mathcal V^+_{\tilde{\mathbf s}\tilde{\mathbf p}}
\subset\varphi_{N_0}(\mathcal V^+_{\mathbf p}\cap\mathcal V^+_{\tilde{\mathbf p}})=\emptyset$ by~\eqref{e:set-concat} and~\eqref{e:pimptor}.
Then~\eqref{e:faraday} follows from part~3 of Lemma~\ref{l:ehrenfest-prop-none}.

\noindent 3. To show~\eqref{e:faror-2},
we first write
$$
A^-_{\mathbf v}A^+_{\tilde{\mathbf q}}(A^-_{\mathbf v}A^+_{\mathbf q})^*
=A^-_{\mathbf v}A^+_{\tilde{\mathbf q}}(A^+_{\mathbf q})^*(A^-_{\mathbf v})^*.
$$
Thus it suffices to prove that
$$
\|A^+_{\tilde{\mathbf q}}(A^+_{\mathbf q})^*\|_{L^2\to L^2}=\mathcal O(h^\infty).
$$
This follows from part~3 of Lemma~\ref{l:ehrenfest-prop-none}.
Indeed, we have $\max(\mathcal J^+_{\mathbf q},\mathcal J^+_{\tilde{\mathbf q}})\leq Ch^{-2\delta}$
by~\eqref{e:go-deep}
and $\mathcal V^+_{\mathbf q}\cap\mathcal V^+_{\tilde{\mathbf q}}
\subset\mathcal V^+_{\mathbf p}\cap\mathcal V^+_{\tilde{\mathbf p}}=\emptyset$
by~\eqref{e:pimptor}.
\end{proof}
%%%%%%%%%%%%%%%%%%%%%%%%%%%%%%%%%%%%%%%%%%%%%%%%%%%%%%%%%%%%%%%%%%%%%%%%%%%%%%%%
We will decompose $A^-_{\mathbf v}A^+_{\mathcal Q'_n(\mathbf w,e)}$
into a sum of operators, each of which corresponds to a \emph{cluster}
of words $\mathbf q\in\mathcal Q'_n(\mathbf w,e)$~-- see~\eqref{e:clusterop} below.
Each cluster has the property that
the sets $\mathcal V^+_{\mathbf q}$ lie in an $\mathcal O(h^{2/3})$ sized conic neighborhood
of some weak unstable leaf. Moreover,
most clusters lie far from each other
in the sense of Definition~\ref{d:close-far}, which will
let us decouple different clusters using the Cotlar--Stein Theorem
and Lemma~\ref{l:faror}.
The clusters are constructed in the following
%%%%%%%%%%%%%%%%%%%%%%%%%%%%%%%%%%%%%%%%%%%%%%%%%%%%%%%%%%%%%%%%%%%%%%%%%%%%%%%%
\begin{lemm}
  \label{l:cluster}
If the constant $\varepsilon_0$ in~\S\ref{s:refined-partition} is chosen
small enough depending on $(M,g)$ then there exists a partition into clusters
$$
\mathcal Q'_n(\mathbf w,e)=\bigsqcup_{r=1}^{R_n(\mathbf w,e)}
\mathcal Q_n(\mathbf w,e,r)
$$
such that for some constant $C$ depending only on $(M,g)$ we have:
\begin{enumerate}
\item for each $r$ there exists $\rho(r)\in \mathcal V^+_e\cap S^*M$ such that
the $r$-th cluster is contained in a $Ch^{2/3}$ sized conic neighborhood
of the weak unstable leaf $W_{0u}(\rho(r))$, that is
\begin{equation}
  \label{e:cluster-1}
d\big(\tilde\rho,W_{0u}(\rho(r))\big)\leq Ch^{2/3}\quad\text{for all}\quad
\tilde\rho\in \bigcup_{\mathbf q\in \mathcal Q_n(\mathbf w,e,r)}(\mathcal V^+_{\mathbf q}\cap S^*M);
\end{equation}
\item let us call the clusters $r,\tilde r$ \textbf{disjoint} when
each pair of words $\mathbf q\in\mathcal Q_n(\mathbf w,e,r),
\tilde{\mathbf q}\in\mathcal Q_n(\mathbf w,e,\tilde r)$
is far from each other in the sense of Definition~\ref{d:close-far}.
Then for each~$r$, the number of clusters $\tilde r$ which are \textbf{not}
disjoint from $r$ is bounded by~$C$.
\end{enumerate}
\end{lemm}
%%%%%%%%%%%%%%%%%%%%%%%%%%%%%%%%%%%%%%%%%%%%%%%%%%%%%%%%%%%%%%%%%%%%%%%%%%%%%%%%
\begin{proof}
In this proof $C$ denotes constants depending only on $(M,g)$ whose precise value might change from place to place.

Since the weak unstable leaves $W_{0u}(\rho)$, $\rho\in \mathcal V^+_e\cap S^*M$, foliate $\mathcal V^+_e\cap S^*M$,
and depend Lipschitz continuously on $\rho$,
if the diameter of $\mathcal V^+_e\cap S^*M$ is less than $\varepsilon_0$
and $\varepsilon_0$ is small enough, there exists a Lipschitz continuous function
(with Lipschitz constant~$C$)
$$
Z:\mathcal V^+_e\cap S^*M\to\mathbb R
$$
which is constant on each weak unstable leaf $W_{0u}(\rho)\cap \mathcal V^+_e$, $\rho\in\mathcal V^+_e\cap S^*M$ and
\begin{equation}
  \label{e:labrador}
d(\tilde\rho,W_{0u}(\rho))\leq C|Z(\rho)-Z(\tilde\rho)|\quad\text{for all}\quad
\rho,\tilde\rho\in \mathcal V^+_e\cap S^*M.
\end{equation}
For instance, one could take as $Z(\rho)$ the function constructed in Lemma~\ref{l:stun-straight}.

For each $\mathbf q\in\mathcal Q'_n(\mathbf w,e)$, define the set
$$
I_{\mathbf q}:=Z(\mathcal V^+_{\mathbf q}\cap S^*M)\subset\mathbb R.
$$
Fix an arbitrary point $z_{\mathbf q}\in I_{\mathbf q}$.
We choose a maximal subset
$$
\{z_1,\dots,z_R\}\subset \{z_{\mathbf q}\mid \mathbf q\in\mathcal Q'_n(\mathbf w,e)\}
$$
which is $h^{2/3}$ separated, that is
$|z_r-z_{\tilde r}|\geq h^{2/3}$ for each $r\neq\tilde r$.
Put $R_n(\mathbf w,e):=R$.

Since the set $\{z_1,\dots,z_R\}$ was chosen maximal,
for each $\mathbf q\in\mathcal Q'_n(\mathbf w,e)$ there exists $r$
such that $|z_{\mathbf q}-z_r|\leq h^{2/3}$.
We can thus define a partition into clusters
$$
\mathcal Q'_n(\mathbf w,e)=\bigsqcup_{r=1}^{R}
\mathcal Q_n(\mathbf w,e,r)\quad\text{where}\quad
|z_{\mathbf q}-z_r|\leq h^{2/3}\quad\text{for all}\quad
\mathbf q\in \mathcal Q_n(\mathbf w,e,r).
$$
By~\eqref{e:unrec} and~\eqref{e:go-deep}, each $\mathcal V^+_{\mathbf q}\cap S^*M$ is contained
in a $Ch^{\tau}$ sized neighborhood of some weak unstable leaf, therefore (since the map $Z$ is Lipschitz continuous)
$I_{\mathbf q}\subset [z_{\mathbf q}-Ch^\tau,z_{\mathbf q}+Ch^\tau]$.
Since $h^{\tau}\ll h^{2/3}$ we see that for each
$\mathbf q\in\mathcal Q_n(\mathbf w,e,r)$ we have
$I_{\mathbf q}\subset [z_r-Ch^{2/3},z_r+Ch^{2/3}]$.
Take $\rho(r)\in\mathcal V^+_e\cap S^*M$ such that $Z(\rho(r))=z_r$, then
by~\eqref{e:labrador} for each $\mathbf q\in\mathcal Q_n(\mathbf w,e,r)$ and $\tilde\rho\in \mathcal V^+_{\mathbf q}\cap S^*M$ we have 
$d(\tilde \rho,W_{0u}(\rho(r)))\leq Ch^{2/3}$. This gives property~(1).

Finally, if $\mathbf q,\tilde{\mathbf q}\in \mathcal Q'_n(\mathbf w,e)$
are close in the sense of Definition~\ref{d:close-far}, then
$|z_{\mathbf q}-z_{\tilde{\mathbf q}}|\leq Ch^{2/3}$. Therefore, if the clusters
$r,\tilde r$ are not disjoint then
$|z_r-z_{\tilde r}|\leq Ch^{2/3}$.
Since $\{z_1,\dots,z_R\}$ is $h^{2/3}$ separated, we see that
for each $r$ the number of clusters $\tilde r$ not disjoint from $r$ is bounded
by some constant $C$. This gives the property~(2).
\end{proof}
%%%%%%%%%%%%%%%%%%%%%%%%%%%%%%%%%%%%%%%%%%%%%%%%%%%%%%%%%%%%%%%%%%%%%%%%%%%%%%%%
Armed with Lemma~\ref{l:cluster} we now decompose
\begin{equation}
  \label{e:clusterop}
A^-_{\mathbf v}A^+_{\mathcal Q'_n(\mathbf w,e)}
=\sum_{r=1}^{R_n(\mathbf w,e)}B_r,\quad
B_r:=A^-_{\mathbf v}A^+_{\mathcal Q_n(\mathbf w,e,r)}
=\sum_{\mathbf q\in\mathcal Q_n(\mathbf w,e,r)}A^-_{\mathbf v}A^+_{\mathbf q}.
\end{equation}
We claim that, with the constant $C$ appearing in Lemma~\ref{l:cluster},
\begin{equation}
  \label{e:clusteres}
\max_r \sum_{\tilde r}\|B_r^* B_{\tilde r}\|_{L^2\to L^2}^{1/2},
\max_r \sum_{\tilde r}\|B_{\tilde r} B_r^*\|_{L^2\to L^2}^{1/2}
\leq C\max_r \|B_r\|_{L^2\to L^2}+\mathcal O(h^\infty).
\end{equation}
Indeed, the sum over clusters $\tilde r$ not disjoint from $r$ is estimated
by $C\max_r \|B_r\|_{L^2\to L^2}$.
The sum over clusters disjoint from $r$ is $\mathcal O(h^\infty)$
by Lemma~\ref{l:faror},
using that the number of elements in $\mathcal Q'_n(\mathbf w,e)$
and thus the number $R_n(\mathbf w,e)$ of clusters are~$\mathcal O(h^{-C})$
for some constant~$C$.

Applying the Cotlar--Stein Theorem~\cite[Theorem~C.5]{e-z}, we see
that 
$$
\|A^-_{\mathbf v}A^+_{\mathcal Q'_n(\mathbf w,e)}\|_{L^2\to L^2}
\leq C\max_r \|B_r\|_{L^2\to L^2}+\mathcal O(h^\infty).
$$
Therefore Proposition~\ref{l:longdec-2} follows from the bound
$$
\max_r \|A^-_{\mathbf v}A^+_{\mathcal Q_n(\mathbf w,e,r)}\|_{L^2\to L^2}\leq Ch^\beta
$$
which in turn is implied by the following
%%%%%%%%%%%%%%%%%%%%%%%%%%%%%%%%%%%%%%%%%%%%%%%%%%%%%%%%%%%%%%%%%%%%%%%%%%%%%%%%
\begin{prop}
  \label{l:longdec-3}
Assume that $\mathbf v\in\mathscr A_\star^{N_0}$,
$\mathbf w\in\mathscr A_\star^{N_1}$,
$1\leq n\leq N_1$,
$e\in\mathscr A$,
$\rho_0\in \mathcal V^+_e\cap S^*M$,
and $\mathcal Q\subset \mathcal Q_n'(\mathbf w,e)$ lies
in an $\mathcal O(h^{2/3})$ sized conic neighborhood of
the weak unstable leaf $W_{0u}(\rho_0)$, namely
for some constant $C_0$
\begin{equation}
  \label{e:getting-close}
d(\tilde\rho,W_{0u}(\rho_0))\leq C_0h^{2/3}
\quad\text{for all}\quad
\tilde\rho\in \bigcup_{\mathbf q\in\mathcal Q}(\mathcal V^+_{\mathbf q}\cap S^*M).
\end{equation}
Then there exists $\beta>0$ depending only on $\mathcal V_1,\mathcal V_\star$
and there exists $C>0$ depending only on~$A_1,A_\star,C_0$ such that
\begin{equation}
  \label{e:longdec-3}
\|A^-_{\mathbf v}A^+_{\mathcal Q}\|_{L^2\to L^2}\leq Ch^\beta.
\end{equation}
\end{prop}
%%%%%%%%%%%%%%%%%%%%%%%%%%%%%%%%%%%%%%%%%%%%%%%%%%%%%%%%%%%%%%%%%%%%%%%%%%%%%%%%
In the above expression, $A^+_{\mathcal Q}$ is a sum of many refined words operators $A^+_{\bq}$ with $\bq$ having Jacobians $\cJ^+_{\bq}\sim h^{-\tau}$; in turn, $A^-_{\mathbf v}$ can also be split into the sum of many word operators $A^-_{\tilde\bq}$ with words $|\tilde\bq |=N_0$. The \emph{hyperbolic dispersion estimates} of \cite{AN07} show that all the individual terms $A^-_{\tilde\bq}A^+_{\bq}$ are small (their norms are bounded by some $h^\alpha$), yet to cope with the sum of many such terms, we will have to use another ingredient, namely a fractal uncertainty principle.

%%%%%%%%%%%%%%%%%%%%%%%%%%%%%%%%%%%%%%%%%%%%%%%%%%%%%%%%%%%%%%%%%%%%%%%%%%%%%%%%
\subsection{Fractal uncertainty principle and decay for a single cluster}
  \label{s:fup-endgame}

In this section we prove Proposition~\ref{l:longdec-3}; as shown earlier in~\S\ref{s:long-word-fup} this implies
Proposition~\ref{l:longdec-0}. We fix
\begin{equation}
  \label{e:fixer}
\mathbf v\in\mathscr A_\star^{N_0},\quad
\mathbf w\in\mathscr A_\star^{N_1},\quad
n\in \{1,\dots,N_1\},\quad
e\in\mathscr A,\quad
\rho_0\in \mathcal V^+_e\cap S^*M,
\end{equation}
and
$\mathcal Q\subset\mathcal Q_n'(\mathbf w,e)$ which lies
in an $\mathcal O(h^{2/3})$ sized conic neighborhood of $W_{0u}(\rho_0)$ in the sense of~\eqref{e:getting-close}.

Throughout this section $C$ denotes constants depending only on
$A_1,\dots, A_Q$, and~$C_0$,
whose meaning might change from place to place,
unless noted otherwise.

The strategy of the proof is to conjugate the operators $A^-_{\mathbf v}$,
$A^+_{\mathcal Q}$ by Fourier integral operators
to obtain a situation to which the fractal uncertainty principle
of Proposition~\ref{l:fup-2} can be applied. The proof of Proposition~\ref{l:longdec-3}
is given in~\S\ref{s:micro-conjugation} below, using components described in the rest of this section.

%%%%%%%%%%%%%%%%%%%%%%%%%%%%%%%%%%%%%%%%%%%%%%%%%%%%%%%%%%%%%%%%%%%%%%%%%%%%%%%%
\subsubsection{Normal form}
  \label{s:fup-normal}
  
We first study the symbols $a^-_{\mathbf v}$, $a^+_{\mathcal Q}$.
We use the symplectomorphism constructed in Lemma~\ref{l:stun-straight},
which approximately straightens out the weak unstable leaves close
to $W_{0u}(\rho_0)$.

By the assumptions on~$\mathcal V_1,\dots,\mathcal V_Q$
in~\S\ref{s:refined-partition}, the diameter
of $\mathcal V^+_e\cap S^*M=\varphi_1(\mathcal V_e\cap S^*M)$ is bounded above
by $C\varepsilon_0$ for some $C$ depending only on $(M,g)$.
Therefore, if we fix $\varepsilon_0>0$ small enough then
by Lemma~\ref{l:stun-straight} there exists 
a symplectomorphism
\begin{equation}
  \label{e:conj-kappa}
\varkappa=\varkappa_{\rho_0}:U_{\rho_0}\to V_{\rho_0},\quad
U_{\rho_0}\subset T^*M\setminus 0,\quad
V_{\rho_0}\subset T^*\mathbb R^2\setminus 0
\end{equation}
which satisfies conditions~(1)--(7) of Lemma~\ref{l:stun-straight}
and $\overline{\mathcal V^+_e}\subset U_{\rho_0}$.
(Here the closure of $\mathcal V^+_e$ is taken in $T^*M\setminus 0$.)
We denote elements of $T^*M$ and $T^*\mathbb R^2$
by $(x,\xi)$ and $(y,\eta)=(y_1,y_2,\eta_1,\eta_2)$ respectively.

Since $\varkappa$ is homogeneous, the flipped graph $\mathscr L_\varkappa$
defined in~\eqref{e:L-varkappa} is conic. Therefore, shrinking $U_{\rho_0}$ (and reducing
$\varepsilon_0$) we may assume that $\mathscr L_\varkappa$
is generated by a single phase function, see~\S\ref{s:lagr-mflds}.

We will analyze the images of the supports $\supp a^-_{\mathbf v}$,
$\supp a^+_{\mathcal Q}$ under the map~$\varkappa$.
The goal is to relate these to localization to porous sets
in $y_1$ and $\eta_1/\eta_2$ respectively,
see~\eqref{e:Omega+con},\eqref{e:Omega-con} below.

We start with $\supp a^+_{\mathcal Q}$ which is contained in the open conic set
\begin{equation}
  \label{e:V-Q-def}
\mathcal V^+_{\mathcal Q}:=\bigcup_{\mathbf q\in\mathcal Q}\mathcal V^+_{\mathbf q}
\ \subset\ \mathcal V^+_e\ \Subset\ U_{\rho_0}.
\end{equation}
The following lemma is a key point in the argument where $C^{3/2}$ regularity
of the unstable foliation (used in Lemma~\ref{l:stun-straight}) is combined with the fact that
$\mathcal Q$ lies $\mathcal O(h^{2/3})$ close to the weak
unstable leaf $W_{0u}(\rho_0)$ (the latter was made possible
by the cluster decomposition of~\S\ref{s:clusters}).
It states that the projection of each weak unstable leaf
$\varkappa(W_{0u}(\tilde\rho))$, $\tilde\rho\in \mathcal V^+_{\mathcal Q}\cap S^*M$,
onto the $\eta_1$ coordinate lies in an interval of size $\mathcal O(h)$.
Since by~\eqref{e:unrec} and~\eqref{e:go-deep}
each $\mathcal V^+_{\mathbf q}\cap S^*M$, $\mathbf q\in\mathcal Q$, lies in an $\mathcal O(h^\tau)$ neighborhood
of some weak unstable leaf, we see that
the projection of $\varkappa(\mathcal V^+_{\mathbf q}\cap S^*M)$ 
onto the $\eta_1$ coordinate lies in an interval of size $\mathcal O(h^\tau)$.
%%%%%%%%%%%%%%%%%%%%%%%%%%%%%%%%%%%%%%%%%%%%%%%%%%%%%%%%%%%%%%%%%%%%%%%%%%%%%%%%
\begin{lemm}
  \label{l:close-unstable}
Let $\tilde\rho\in\mathcal V^+_{\mathcal Q}\cap S^*M$.
Then
\begin{equation}
  \label{e:cun-1}
|\eta_1(\varkappa(\rho))-\eta_1(\varkappa(\tilde\rho))|\leq Ch\quad\text{for all}\quad
\rho\in W_{0u}(\tilde\rho)\cap U_{\rho_0}.
\end{equation}
\end{lemm}
%%%%%%%%%%%%%%%%%%%%%%%%%%%%%%%%%%%%%%%%%%%%%%%%%%%%%%%%%%%%%%%%%%%%%%%%%%%%%%%%
\begin{proof}
We recall the straightening of the unstable foliation described in Lemma~\ref{l:stun-straight}. 
By~\eqref{e:stun-straight} we have
\begin{equation}
  \label{e:fgame-2}
\varkappa(W_{0u}(\tilde \rho)\cap U_{\rho_0})=
\big\{\big(y_1,y_2,F(y_1,\tilde\zeta),1\big)\mid (y_1,\tilde\zeta)\in\Omega,\
y_2\in\mathbb R \big\}\cap V_{\rho_0}
\end{equation}
where $\tilde\zeta:=Z(\tilde\rho)$ and the functions
$F\in C^{3/2}(\Omega;\mathbb R)$, $Z\in C^{3/2}(U_{\rho_0}\cap S^*M;\mathbb R)$
are defined in Lemma~\ref{l:stun-straight}.
Moreover, by~\eqref{e:getting-close} we have
$d(\tilde\rho,W_{0u}(\rho_0))\leq C_0h^{2/3}$, which
by parts~(5)--(6) of Lemma~\ref{l:stun-straight} implies
\begin{equation}
  \label{e:egame-1}
|\tilde\zeta|\leq Ch^{2/3}.
\end{equation}
Combining this estimate with the point (8) of Lemma~\ref{l:stun-straight},
we obtain
\begin{equation}
  \label{e:egame-3}
\sup_{y_1}|F(y_1,\tilde\zeta)-\tilde\zeta|\leq Ch
\end{equation}
which together with~\eqref{e:fgame-2} gives~\eqref{e:cun-1}.
\end{proof}
%%%%%%%%%%%%%%%%%%%%%%%%%%%%%%%%%%%%%%%%%%%%%%%%%%%%%%%%%%%%%%%%%%%%%%%%%%%%%%%%
In~\S\ref{s:porosity-ultimate} below we use Lemma~\ref{l:close-unstable}
and the results of~\S\ref{s:porosity-basic} to show the following porosity statement
(see Definition~\ref{d:porous}):
%%%%%%%%%%%%%%%%%%%%%%%%%%%%%%%%%%%%%%%%%%%%%%%%%%%%%%%%%%%%%%%%%%%%%%%%%%%%%%%%
\begin{lemm}
  \label{l:poro+}
Define the set
\begin{equation}
  \label{e:Omega+}
\Omega^+:=\eta_1(\varkappa(\mathcal V^+_{\mathcal Q}\cap S^*M))\subset\mathbb R.
\end{equation}
Then there exist $R$ and $\nu>0$ depending only on $\mathcal V_1,\mathcal V_\star$
such that
$\Omega^+\subset \Omega^+_1\cup\dots\cup \Omega^+_R$ where
each $\Omega^+_k$ is $\nu$-porous on scales $Ch^\tau$ to $C^{-1}$.
\end{lemm}
%%%%%%%%%%%%%%%%%%%%%%%%%%%%%%%%%%%%%%%%%%%%%%%%%%%%%%%%%%%%%%%%%%%%%%%%%%%%%%%%
%
%%%%%%%%%%%%%%%%%%%%%%%%%%%%%%%%%%%%%%%%%%%%%%%%%%%%%%%%%%%%%%%%%%%%%%%%%%%%%%%%
\begin{figure}
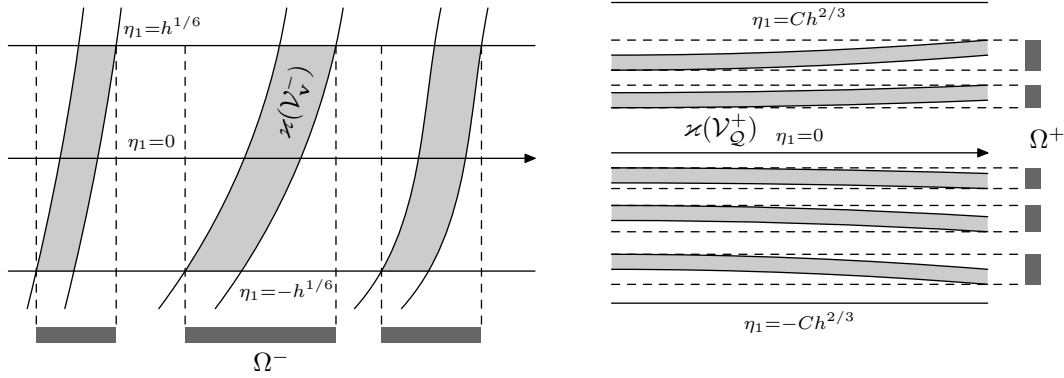

\includegraphics{varfup.5}
\qquad
\includegraphics{varfup.6}
\caption{The sets $\varkappa(\mathcal V^-_{\mathbf v}\cap \mathcal V^+_e\cap S^*M)\cap \{|\eta_1|\leq h^{1/6}\}$ and~$\varkappa(\mathcal V^+_{\mathcal Q}\cap S^*M)$ (lighter shaded).
Here $y_1$ is the horizontal coordinate (with the width of the figure having $h$-independent scale) and $\eta_1$ is the vertical coordinate.
The darker shaded sets are $\Omega^-$ and~$\Omega^+$, defined in~\eqref{e:Omega-}
and~\eqref{e:Omega+}.}
\label{f:porositors}
\end{figure}
%%%%%%%%%%%%%%%%%%%%%%%%%%%%%%%%%%%%%%%%%%%%%%%%%%%%%%%%%%%%%%%%%%%%%%%%%%%%%%%%
\Remarks 1.
Since $\varkappa(\mathcal V^+_{\mathcal Q}\cap S^*M)$ is contained in an $\mathcal O(h^{2/3})$ sized
neighborhood of $\{\eta_1=0\}$
by~\eqref{e:getting-close} and parts~(5)--(6) of Lemma~\ref{l:stun-straight}, we have
\begin{equation}
  \label{e:Omega+small}
\Omega^+\subset [-Ch^{2/3},Ch^{2/3}].
\end{equation}
In particular, it is easy to see that $\Omega^+$ is $1\over 3$-porous on scales above $Ch^{2/3}$
for $C$ large enough. Lemma~\ref{l:poro+} shows that each $\Omega^+_k$ is in fact $\nu$-porous
on scales above $Ch^\tau$ (where $\tau$ is very close to~1) for some $\nu>0$.

\noindent 2. Using Lemmas~\ref{l:porosity-basic-quant}--\ref{l:dense-useful} and following the proof of Lemma~\ref{l:poro+},
we get the following statement: if the complements $S^*M\setminus\mathcal V_1,
S^*M\setminus\mathcal V_\star$ are $(L_0,L_1)$-dense in the stable direction
(in the sense of Definition~\ref{d:stun-dense}) then Lemma~\ref{l:poro+} holds for
some $\nu$ depending only on $(M,g),L_0,L_1$.

We next study $\supp a^-_{\mathbf v}$, which is contained in $\mathcal V^-_{\mathbf v}$.
By~\eqref{e:Omega+small} and since $\supp a^+_{\mathcal Q}\subset \mathcal V^+_e$
it would be enough to study the intersection
of $\varkappa(\mathcal V^-_{\mathbf v}\cap \mathcal V^+_e\cap S^*M)$
with the set $\{|\eta_1|\leq Ch^{2/3}\}$.
However, for the purpose of microlocalization
of the operator $A^-_{\mathbf v}$ it is convenient to choose
a larger, $h^{1/6}$-sized, neighborhood of $\{\eta_1=0\}$. We thus define
\begin{equation}
  \label{e:Omega-}
\Omega^-:=y_1\big(\varkappa(\mathcal V^-_{\mathbf v}\cap \mathcal V^+_e\cap S^*M)\cap \{|\eta_1|\leq h^{1/6}\}\big)
\subset\mathbb R.
\end{equation}
The next lemma, proved in~\S\ref{s:porosity-ultimate} below, establishes
porosity of $\Omega^-$:
%%%%%%%%%%%%%%%%%%%%%%%%%%%%%%%%%%%%%%%%%%%%%%%%%%%%%%%%%%%%%%%%%%%%%%%%%%%%%%%%
\begin{lemm}
  \label{l:poro-}
Let $\Lambda:=\lceil \Lambda_1/\Lambda_0\rceil$ be defined in~\eqref{e:Lambda-def}.
Then there exist $R$ and $\nu>0$ depending only on $\mathcal V_1,\mathcal V_\star$
such that $\Omega^-\subset \Omega^-_1\cup\dots\cup\Omega^-_R$ where
each $\Omega^-_k$ is $\nu$-porous on scales $Ch^{1/(6\Lambda)}$ to $C^{-1}$.
\end{lemm}
%%%%%%%%%%%%%%%%%%%%%%%%%%%%%%%%%%%%%%%%%%%%%%%%%%%%%%%%%%%%%%%%%%%%%%%%%%%%%%%%
\Remark Using Lemmas~\ref{l:porosity-basic-quant}--\ref{l:dense-useful} and following the proof of Lemma~\ref{l:poro-},
we get the following statement: if the complements $S^*M\setminus\mathcal V_1,
S^*M\setminus\mathcal V_\star$ are $(L_0,L_1)$-dense in the unstable direction
(in the sense of Definition~\ref{d:stun-dense}) then Lemma~\ref{l:poro-} holds for
some $\nu$ depending only on $(M,g),L_0,L_1$.

For future use we record the following corollaries of the definitions~\eqref{e:Omega+},
\eqref{e:Omega-} of~$\Omega^\pm$ and the homogeneity of~$\varkappa$:
\begin{align}
  \label{e:Omega+con}
\varkappa\Big(\mathcal V^+_{\mathcal Q}\cap \Big\{{1\over 4}\leq |\xi|_g\leq 4\Big\}\Big)\ &\subset\ \Big\{{\eta_1\over\eta_2}\in \Omega^+\Big\}
\cap \Big\{{1\over 4}\leq \eta_2\leq 4\Big\},\\
  \label{e:Omega-con}
\varkappa(\mathcal V^-_{\mathbf v}\cap \mathcal V^+_e)\cap \Big\{\Big|{\eta_1\over\eta_2}\Big|\leq h^{1/6}\Big\}
\ &\subset\ \{y_1\in \Omega^-\}.
\end{align}
See Figure~\ref{f:porositors}. For~\eqref{e:Omega+con} we additionally used
part~(4) of Lemma~\ref{l:stun-straight}.

%%%%%%%%%%%%%%%%%%%%%%%%%%%%%%%%%%%%%%%%%%%%%%%%%%%%%%%%%%%%%%%%%%%%%%%%%%%%%%%%
\subsubsection{Proof of porosity}
\label{s:porosity-ultimate}

We now prove Lemmas~\ref{l:poro+} and~\ref{l:poro-}.
We start by defining fattened versions of the sets
$\mathcal V^+_{\mathcal Q}$, $\mathcal V^-_{\mathbf v}$.
Fix two conic open sets
$$
\mathcal V_1^\sharp,
\mathcal V_\star^\sharp\subset T^*M\setminus 0
$$
such that:
\begin{itemize}
\item $\overline{\mathcal V_w}\subset \mathcal V_w^\sharp$
for $w\in\mathscr A_\star=\{1,\star\}$ where the closure is taken in $T^*M\setminus 0$;
\item the complements $T^*M\setminus \mathcal V_w^\sharp$
have nonempty interior.
\end{itemize}
This is possible since $T^*M\setminus \mathcal V_1,T^*M\setminus \mathcal V_\star$
have nonempty interior, see~\S\ref{s:proofs-notation}.

Since $\mathcal V_q\subset\mathcal V_\star$ for $q=2,\dots,Q$ (see~\S\ref{s:refined-partition}),
we can also fix conic open sets
$$
\mathcal V_q^\sharp\subset\mathcal V_\star^\sharp,\quad
\overline{\mathcal V_q}\subset \mathcal V_q^\sharp,\quad
q=2,\dots,Q.
$$
Moreover, since the diameters of $\mathcal V_q\cap S^*M$,
$q\in \mathscr A:=\{1,\dots,Q\}$,
are less than $\varepsilon_0$, we can make
the diameters of $\mathcal V_q^\sharp\cap S^*M$ less than $\varepsilon_0$ as well.
We may also assume that $\overline{\mathcal V_e^{\sharp+}}\subset U_{\rho_0}$
where $U_{\rho_0}$ is the domain of the map~$\varkappa$, see~\eqref{e:conj-kappa}.

Let $\mathbf v=v_0\dots v_{N_0-1}\in\mathscr A_\star^{N_0}$
be the word in the statement of Proposition~\ref{l:longdec-3}
and 
$\mathbf q=q_1\dots q_n\in\mathscr A^n$ be arbitrary.
Similarly to~\eqref{e:V+-}
define the open conic sets
\begin{equation}
  \label{e:sharp-words}
\mathcal V^{\sharp-}_{\mathbf v}:=\bigcap_{j=0}^{N_0-1} \varphi_{-j}(\mathcal V^\sharp_{v_j}),\quad
\mathcal V^{\sharp+}_{\mathbf q}:=\bigcap_{j=1}^n \varphi_j(\mathcal V^\sharp_{q_j}).
\end{equation}
Clearly $\mathcal V^-_{\mathbf v}\subset \mathcal V^{\sharp-}_{\mathbf v}$,
$\mathcal V^+_{\mathbf q}\subset \mathcal V^{\sharp+}_{\mathbf q}$.
Following~\eqref{e:V-Q-def} define also
\begin{equation}
  \label{e:sharp-words-2}
\mathcal V^{\sharp+}_{\mathcal Q}:=\bigcup_{\mathbf q\in\mathcal Q}
\mathcal V^{\sharp+}_{\mathbf q}
\ \supset\ \mathcal V^+_{\mathcal Q}.
\end{equation}
We use the results of~\S\ref{s:porosity-basic}
and the fact that $T^*M\setminus\mathcal V_1^{\sharp}$,
$T^*M\setminus\mathcal V_\star^\sharp$ have nonempty interiors
to establish the porosity of the intersections of $\mathcal V^{\sharp-}_{\mathbf v}$,
$\mathcal V^{\sharp+}_{\mathcal Q}$ with unstable/stable intervals:
%%%%%%%%%%%%%%%%%%%%%%%%%%%%%%%%%%%%%%%%%%%%%%%%%%%%%%%%%%%%%%%%%%%%%%%%%%%%%%%%
\begin{lemm}
  \label{l:porosity-intersected}
There exists $\nu>0$ depending only on~$\mathcal V_1,\mathcal V_\star$ such that:
\begin{enumerate}
\item for every unstable interval $\gamma:I_0\to S^*M$ 
(see Definition~\ref{d:stun-interval}),
the preimage $\gamma^{-1}(\mathcal V^{\sharp-}_{\mathbf v})\subset\mathbb R$
is $\nu$-porous on scales $Ch^{1/(6\Lambda)}$ to~1;
\item for every stable interval $\gamma:I_0\to S^*M$,
the set $\gamma^{-1}(\mathcal V^{\sharp+}_{\mathcal Q})$
is $\nu$-porous on scales $Ch^\tau$ to~1.
\end{enumerate}
\end{lemm}
%%%%%%%%%%%%%%%%%%%%%%%%%%%%%%%%%%%%%%%%%%%%%%%%%%%%%%%%%%%%%%%%%%%%%%%%%%%%%%%%
\begin{proof}
Recall that $\mathcal Q$ is contained in the set
$\mathcal Q'_n(\mathbf w,e)$ defined by~\eqref{e:Q-n-def}.
Therefore, each $\mathbf q=q_1\dots q_n\in\mathcal Q$ satisfies
$\mathbf q\lesssim\mathbf w$
(where $\mathbf w\in\mathscr A_\star^{N_1}$ is fixed in the statement of Proposition~\ref{l:longdec-3}),
which (recalling Definition~\ref{d:prec-word}) implies that
$\mathcal V_{q_j}^\sharp\subset \mathcal V_{w_j}^\sharp$
for all $j=1,\dots,n$. It follows that
$$
\mathcal V^{\sharp+}_{\mathcal Q}\subset \mathcal V^{\sharp+}_{w_1\dots w_n}
:=\bigcap_{j=1}^n \varphi_j(\mathcal V^\sharp_{w_j}).
$$
Thus the required porosity statements follow from Lemma~\ref{l:porosity-basic}
(taking the sets $\mathcal V_1^\sharp$, $\mathcal V_\star^\sharp$ in~\eqref{e:porosity-basic-sets})
once we establish the Jacobian bounds
\begin{align}
  \label{e:pinn-1}
\inf_{\mathcal V^{\sharp-}_{\mathbf v}\cap S^*M} J^u_{N_0}&\geq h^{-1/(6\Lambda)},\\
  \label{e:pinn-2}
\inf_{\mathcal V^{\sharp+}_{\mathcal Q}\cap S^*M} J^s_{-n}&\geq C^{-1}h^{-\tau}.
\end{align}
The estimate~\eqref{e:pinn-1} follows immediately
from~\eqref{e:Lambda-0-1} and the definitions~\eqref{e:prop-times}
of~$N_0$ and~\eqref{e:Lambda-def} of $\Lambda$.

To show~\eqref{e:pinn-2}, take arbitrary
$\rho\in\mathcal V^{\sharp+}_{\mathcal Q}\cap S^*M$,
then $\rho\in\mathcal V^{\sharp+}_{\mathbf q}\cap S^*M$
for some $\mathbf q\in\mathcal Q\subset\mathcal Q'_n(\mathbf w,e)$.
Take some $\tilde\rho\in \mathcal V^+_{\mathbf q}\cap S^*M\subset\mathcal V^{\sharp+}_{\mathbf q}\cap S^*M$.
We have
$$
J^s_{-n}(\rho)\geq C^{-1}J^s_{-n}(\tilde\rho)\geq C^{-1}\mathcal J^+_{\mathbf q}\geq C^{-1}h^{-\tau}
$$
where the first inequality is proved similarly to~\eqref{e:jacobians-same} (using
that the diameter of each $\mathcal V_q^\sharp\cap S^*M$, $q\in\mathscr A$,
is less than $\varepsilon_0$), the second one follows from the definition~\eqref{e:word-J-def}
of~$\mathcal J^+_{\mathbf q}$,
and the third one follows from~\eqref{e:go-deep}.
\end{proof}
%%%%%%%%%%%%%%%%%%%%%%%%%%%%%%%%%%%%%%%%%%%%%%%%%%%%%%%%%%%%%%%%%%%%%%%%%%%%%%%%
The next lemma shows that
each sufficiently short weak stable leaf centered at a point in $\mathcal V^-_{\mathbf v}$
is contained in the slightly larger set $\mathcal V^{\sharp-}_{\mathbf v}$,
and same is true for weak unstable leaves and the sets
$\mathcal V^+_{\mathcal Q},\mathcal V^{\sharp+}_{\mathcal Q}$.
It will be useful in approximating $\Omega^\pm$ by the sets
studied in Lemma~\ref{l:porosity-intersected},
see~\eqref{e:poro+in}, \eqref{e:poro-in} below.
As in Lemma~\ref{l:stun-main} we fix a distance function $d(\bullet,\bullet)$
on $S^*M$.
%%%%%%%%%%%%%%%%%%%%%%%%%%%%%%%%%%%%%%%%%%%%%%%%%%%%%%%%%%%%%%%%%%%%%%%%%%%%%%%%
\begin{lemm}
  \label{l:blurry}
There exists $\varepsilon_1>0$ depending only on $\mathcal V_1,\mathcal V_\star$
such that for
all $\rho,\tilde\rho\in S^*M$
we have
\begin{align}
  \label{e:blurry-1}
d(\rho,\tilde\rho)\leq \varepsilon_1,\quad \tilde\rho\in W_{0s}(\rho),\quad
\rho\in \mathcal V^-_{\mathbf v}
&\quad\Longrightarrow\quad
\tilde\rho\in\mathcal V^{\sharp-}_{\mathbf v},\\
  \label{e:blurry-2}
d(\rho,\tilde\rho)\leq \varepsilon_1,\quad \tilde\rho\in W_{0u}(\rho),\quad
\rho\in \mathcal V^+_{\mathcal Q}
&\quad\Longrightarrow\quad
\tilde \rho\in\mathcal V^{\sharp+}_{\mathcal Q}.
\end{align}
\end{lemm}
%%%%%%%%%%%%%%%%%%%%%%%%%%%%%%%%%%%%%%%%%%%%%%%%%%%%%%%%%%%%%%%%%%%%%%%%%%%%%%%%
\begin{proof}
It suffices to show that
there exists a constant $C$ depending only on $(M,g)$ such that
for all $\varepsilon_1>0$ and $\rho,\tilde\rho\in S^*M$
\begin{align}
  \label{e:blurry-1.1}
d(\rho,\tilde\rho)\leq \varepsilon_1,\quad
\tilde\rho\in W_{0s}(\rho)
&\quad\Longrightarrow\quad
d(\varphi_t(\rho),\varphi_t(\tilde\rho))\leq C\varepsilon_1\quad\text{for all}\quad t\geq 0;\\
  \label{e:blurry-1.2}
d(\rho,\tilde\rho)\leq \varepsilon_1,\quad
\tilde\rho\in W_{0u}(\rho)
&\quad\Longrightarrow\quad
d(\varphi_t(\rho),\varphi_t(\tilde\rho))\leq C\varepsilon_1\quad\text{for all}\quad t\leq 0.
\end{align}
Indeed, to show~\eqref{e:blurry-1} and~\eqref{e:blurry-2} it suffices
to take $\varepsilon_1$ small enough so that the distance
between $\mathcal V_q\cap S^*M$ and $S^*M\setminus\mathcal V^\sharp_q$
is larger than $C\varepsilon_1$ for all $q\in \{1,2,\dots,Q,\star\}$
(which is possible since $\overline{\mathcal V_q}\subset \mathcal V^\sharp_q$).
Then $\varphi_t(\rho)\in\mathcal V_q\cap S^*M$ and $d(\varphi_t(\rho),\varphi_t(\tilde\rho))\leq C\varepsilon_1$
together imply that $\varphi_t(\tilde\rho)\in \mathcal V_q^{\sharp}$ and it remains to use
the definitions~\eqref{e:V+-}, \eqref{e:V-Q-def}, \eqref{e:sharp-words}, \eqref{e:sharp-words-2}.

We show~\eqref{e:blurry-1.1}, with~\eqref{e:blurry-1.2} proved similarly.
By the definition~\eqref{e:weak-leaves} of $W_{0s}(\rho)$
we have $\tilde\rho=\varphi_r(\rho')$
for some
$\rho'\in W_s(\rho)$ and $r\in [-\tilde\varepsilon,\tilde\varepsilon]$.
Since stable leaves are transversal to the flow lines of $\varphi_t$,
we have
$$
d(\rho',\rho)+|r|\leq C\varepsilon_1.
$$
By~\eqref{e:hyprop-1} there exists $\theta>0$ such that for all $t\geq 0$
\begin{equation}
  \label{e:bear-1}
d(\varphi_t(\rho),\varphi_t(\rho'))\leq Ce^{-\theta t}d(\rho,\rho')\leq C\varepsilon_1.
\end{equation}
On the other hand since $\varphi_t(\tilde\rho)=\varphi_r(\varphi_t(\rho'))$ we have
\begin{equation}
  \label{e:bear-2}
d(\varphi_t(\rho'),\varphi_t(\tilde\rho))\leq C|r|\leq C\varepsilon_1.
\end{equation}
Combining~\eqref{e:bear-1}--\eqref{e:bear-2} we get~\eqref{e:blurry-1.1}.
\end{proof}
%%%%%%%%%%%%%%%%%%%%%%%%%%%%%%%%%%%%%%%%%%%%%%%%%%%%%%%%%%%%%%%%%%%%%%%%%%%%%%%%
Since the stable leaves, the unstable leaves, and the flow trajectories
are transversal to each other, if $\rho,\tilde\rho\in S^*M$
are sufficiently close to each other then
the weak stable leaf $W_{0s}(\rho)$ intersects the unstable leaf
$W_u(\tilde\rho)$, and same is true for the stable leaf $W_s(\rho)$ and the weak unstable leaf $W_{0u}(\tilde\rho)$~-- see~\eqref{e:sma-3}.
This immediately gives
%%%%%%%%%%%%%%%%%%%%%%%%%%%%%%%%%%%%%%%%%%%%%%%%%%%%%%%%%%%%%%%%%%%%%%%%%%%%%%%%
\begin{lemm}
  \label{l:productor}
There exist $C_2\geq 1$, $\varepsilon_2>0$ depending only on $(M,g)$ such that
for each $\rho,\tilde\rho\in S^*M$ with $d(\rho,\tilde\rho)\leq\varepsilon_2$
there exist
\begin{gather}
  \label{e:productor}
\rho'\in W_s(\rho),\
\rho''\in W_u(\tilde\rho),\
r\in \mathbb R\quad\text{such that}\quad
\rho'=\varphi_r(\rho'');\\
  \label{e:productor2}
\max\big\{d(\rho_1,\rho_2)\mid \rho_1,\rho_2\in \{\rho,\tilde\rho,\rho',\rho''\}\big\}
+|r|\leq C_2d(\rho,\tilde\rho).
\end{gather}
\end{lemm}
%%%%%%%%%%%%%%%%%%%%%%%%%%%%%%%%%%%%%%%%%%%%%%%%%%%%%%%%%%%%%%%%%%%%%%%%%%%%%%%%
We now define the sets $\Omega^\pm_k$ from Lemmas~\ref{l:poro+}--\ref{l:poro-}.
Let $\varepsilon_1,\varepsilon_2,C_2$ be the constants from Lemmas~\ref{l:blurry}
and~\ref{l:productor}.
Without loss of generality we may assume that $\varepsilon_1\leq \varepsilon_2$.
We will also assume that $\varepsilon_2$ is small enough depending only on $(M,g)$
in the beginning of the proofs of Lemmas~\ref{l:poro+help}
and~\ref{l:poro-help} below.
Fix finitely many points
$$
\rho_1,\dots,\rho_R\in W_{0u}(\rho_0),
$$
with $R$ depending only on
$(M,g)$ and~$\varepsilon_1$, such that each point in $W_{0u}(\rho_0)$
is $\varepsilon_1\over 2C_2$ close to at least one of the points $\rho_1,\dots,\rho_R$.
%%%%%%%%%%%%%%%%%%%%%%%%%%%%%%%%%%%%%%%%%%%%%%%%%%%%%%%%%%%%%%%%%%%%%%%%%%%%%%%%
\begin{lemm}
  \label{l:Omega-cont}
We have $\Omega^\pm\subset \Omega^\pm_1\cup\dots\cup \Omega^\pm_R$
where for $k=1,\dots, R$
$$
\Omega^+_k:=\eta_1(\varkappa(\Sigma^+_k)),\quad
\Omega^-_k:=y_1(\varkappa(\Sigma^-_k))
$$
and the sets $\Sigma^\pm_k\subset \mathcal V^+_e\cap S^*M$ are defined by
$$
\begin{aligned}
\Sigma^+_k\,&:=\{\rho\in\mathcal V^+_{\mathcal Q}\cap S^*M\mid
d(\rho,\rho_k)\leq \textstyle{\varepsilon_1\over C_2}\},\\
\Sigma^-_k\,&:=\big\{\rho\in \mathcal V^-_{\mathbf v}\cap \mathcal V^+_e\cap S^*M\mid
d(\rho,W_{0u}(\rho_0))\leq C_3h^{1/6},\
d(\rho,\rho_k)\leq \textstyle{\varepsilon_1\over C_2}
\big\}
\end{aligned}
$$
where $C_3$ is a sufficiently large constant depending only on $\mathcal V_1,\mathcal V_\star,C_0$.
\end{lemm}
%%%%%%%%%%%%%%%%%%%%%%%%%%%%%%%%%%%%%%%%%%%%%%%%%%%%%%%%%%%%%%%%%%%%%%%%%%%%%%%%
\begin{proof}
Recalling the definitions~\eqref{e:Omega+},\eqref{e:Omega-}
of $\Omega^\pm$ we see that it suffices to show the inclusions
\begin{align}
  \label{e:omc-1}
\mathcal V^+_{\mathcal Q}\cap S^*M\ &\subset\ \Sigma^+_1\cup\dots\cup \Sigma^+_R,\\
  \label{e:omc-2}
\mathcal V^-_{\mathbf v}\cap \mathcal V^+_e\cap S^*M\cap \varkappa^{-1}(\{|\eta_1|\leq h^{1/6}\})\ &\subset\
\Sigma^-_1\cup\dots\cup\Sigma^-_R.
\end{align}
We first take arbitrary $\rho\in \mathcal V^+_{\mathcal Q}\cap S^*M$.
By~\eqref{e:getting-close} we have
$d(\rho,W_{0u}(\rho_0))\leq C_0h^{2/3}\leq {\varepsilon_1\over 2C_2}$.
Therefore there exists $k\in \{1,\dots,R\}$ such that
$d(\rho,\rho_k)\leq {\varepsilon_1\over C_2}$.
It follows that $\rho\in\Sigma^+_k$ which gives~\eqref{e:omc-1}.

We next take arbitrary $\rho\in \mathcal V^-_{\mathbf v}\cap \mathcal V^+_e\cap S^*M$
such that $|\eta_1(\varkappa(\rho))|\leq h^{1/6}$.
Since $\varkappa(W_{0u}(\rho_0)\cap U_{\rho_0})=\{\eta_1=0,\ \eta_2=1\}\cap V_{\rho_0}$,
we have
$d(\rho,W_{0u}(\rho_0))\leq C_3h^{1/6}$ for some constant $C_3$.
In particular $d(\rho,W_{0u}(\rho_0))\leq {\varepsilon_1\over 2C_2}$,
so there exists $k\in \{1,\dots,R\}$ such that
$d(\rho,\rho_k)\leq {\varepsilon_1\over C_2}$.
It follows that $\rho\in \Sigma^-_k$ which gives~\eqref{e:omc-2}.
\end{proof}
%%%%%%%%%%%%%%%%%%%%%%%%%%%%%%%%%%%%%%%%%%%%%%%%%%%%%%%%%%%%%%%%%%%%%%%%%%%%%%%%
We are now ready to finish the proofs of Lemmas~\ref{l:poro+}--\ref{l:poro-}.
Using Lemma~\ref{l:Omega-cont} we see that Lemma~\ref{l:poro+} follows from
%%%%%%%%%%%%%%%%%%%%%%%%%%%%%%%%%%%%%%%%%%%%%%%%%%%%%%%%%%%%%%%%%%%%%%%%%%%%%%%%
\begin{lemm}
  \label{l:poro+help}
Let $\nu>0$ be fixed in Lemma~\ref{l:porosity-intersected}.
Then for each $k\in \{1,\dots,R\}$ the set $\Omega^+_k$
is $\nu\over 6$-porous on scales $Ch^\tau$ to~$C^{-1}$.
\end{lemm}
%%%%%%%%%%%%%%%%%%%%%%%%%%%%%%%%%%%%%%%%%%%%%%%%%%%%%%%%%%%%%%%%%%%%%%%%%%%%%%%%
\begin{proof}
Without loss of generality we may assume that $\Sigma^+_k\neq\emptyset$.
Then $\rho_k$ lies in the ${\varepsilon_1\over C_2}\leq\varepsilon_2$ sized neighborhood
of $\mathcal V^+_e\cap S^*M\Subset U_{\rho_0}$.
Let $\gamma_k^s:[-C\varepsilon_2,C\varepsilon_2]\to S^*M$ be a stable interval
(see Definition~\ref{d:stun-interval}) such that
$\gamma_k^s(0)=\rho_k$. Here $C$ is chosen large enough (depending only on~$(M,g)$) so that
every point $\rho'\in W_s(\rho_k)$ with $d(\rho_k,\rho')\leq \varepsilon_2$
lies in~$\gamma_k^s$.
We may choose $\varepsilon_2$ small enough
so that
$\gamma_k^s\subset U_{\rho_0}$.

Since $E_s(\rho_k)\subset T_{\rho_k}(S^*M)$ is transversal to $T_{\rho_k}W_{0u}(\rho_0)$
and
(recalling that $\varkappa$ maps $S^*M$ to $\{\eta_2=1\}$
and $W_{0u}(\rho_0)$ to $\{\eta_1=0,\ \eta_2=1\}$)
$$
d\varkappa(\rho_k)(T_{\rho_k}(S^*M))=\{d\eta_2=0\},\quad
d\varkappa(\rho_k)(T_{\rho_k}W_{0u}(\rho_0))=\{d\eta_1=d\eta_2=0\}
$$
we have $d(\eta_1\circ\varkappa)(\rho_k)\dot \gamma_k^s(0)\neq 0$.
Therefore if $\varepsilon_2$ is small enough depending only on $(M,g)$
then the map
$$
\psi_k^s:=\eta_1\circ\varkappa\circ\gamma_k^s:[-C\varepsilon_2,C\varepsilon_2]\to\mathbb R
$$
is a diffeomorphism onto its image. We extend $\psi_k^s$ to a global diffeomorphism
$\mathbb R\to\mathbb R$ so that it satisfies the derivative bounds~\eqref{e:psi-derby}
with some constant $C_1$ depending only on $(M,g)$. Define
$$
\widetilde\Omega^+_k:=\psi_k^s\big((\gamma_k^s)^{-1}(\mathcal V^{\sharp+}_{\mathcal Q})\big)
=\eta_1(\varkappa(\gamma_k^s\cap \mathcal V^{\sharp+}_{\mathcal Q}))\subset\mathbb R.
$$
Then by Lemmas~\ref{l:porosity-intersected} and~\ref{l:porous-map}
the set $\widetilde\Omega^+_k$ is $\nu\over 2$-porous on scales
$Ch^\tau$ to~$C^{-1}$.

%%%%%%%%%%%%%%%%%%%%%%%%%%%%%%%%%%%%%%%%%%%%%%%%%%%%%%%%%%%%%%%%%%%%%%%%%%%%%%%%
\begin{figure}
\includegraphics{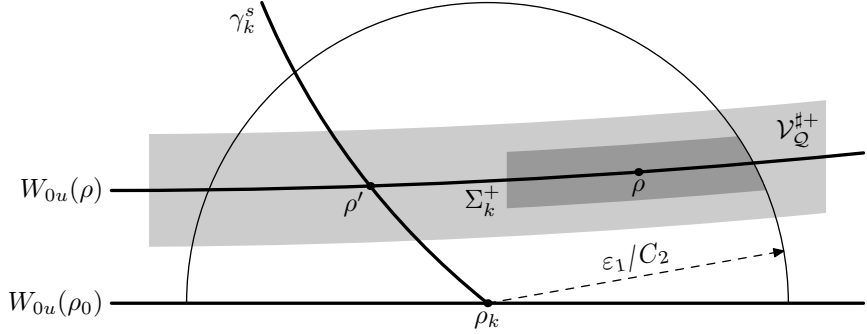}
\caption{An illustration of the proof of Lemma~\ref{l:poro+help}. We use the
coordinates provided by the diffeomorphism $\varkappa$, with $y_1$ the horizontal
coordinate and $\eta_1$ the vertical one; we restrict to $S^*M=\{\eta_2=1\}$
and suppress the flow direction $\partial_{y_2}$ (thus $\rho',\rho''$ are mapped to the same point).
The darker shaded set is $\Sigma^+_k$ and the lighter shaded set is $\mathcal V^{\sharp+}_{\mathcal Q}$.}
\label{f:poro+help}
\end{figure}
%%%%%%%%%%%%%%%%%%%%%%%%%%%%%%%%%%%%%%%%%%%%%%%%%%%%%%%%%%%%%%%%%%%%%%%%%%%%%%%%
We now claim that
\begin{equation}
  \label{e:poro+in}
\Omega^+_k\subset \widetilde\Omega^+_k+[-Ch,Ch].
\end{equation}
Indeed, take arbitrary $\rho\in \Sigma^+_k$.
Then $d(\rho,\rho_k)\leq {\varepsilon_1\over C_2}\leq \varepsilon_2$,
so by Lemma~\ref{l:productor} there exist
$$
\rho'\in W_s(\rho_k),\ \rho''\in W_u(\rho),\ r\in[-\varepsilon_1,\varepsilon_1]\quad\text{such that}\quad
\rho'=\varphi_r(\rho'').
$$
(See Figure~\ref{f:poro+help}.)
By~\eqref{e:productor2} we have
$d(\rho_k,\rho')\leq \varepsilon_1\leq\varepsilon_2$, thus $\rho'\in\gamma_k^s$.
We also have $d(\rho,\rho')\leq\varepsilon_1$,
$\rho'\in W_{0u}(\rho)$, and $\rho\in \mathcal V^+_{\mathcal Q}\cap S^*M$,
which by Lemma~\ref{l:blurry} imply that $\rho'\in \mathcal V^{\sharp+}_{\mathcal Q}$.
Therefore
\begin{equation}
  \label{e:poro+nt}
\eta_1(\varkappa(\rho'))\in \widetilde\Omega^+_k.
\end{equation}
On the other hand by Lemma~\ref{l:close-unstable}
we have
$$
|\eta_1(\varkappa(\rho))-\eta_1(\varkappa(\rho'))|\leq Ch.
$$
Since $\Omega^+_k=\eta_1(\varkappa(\Sigma^+_k))$, together with~\eqref{e:poro+nt} this gives~\eqref{e:poro+in}.

To show that $\Omega^+_k$ is $\nu\over 6$-porous on scales $Ch^\tau$ to~$C^{-1}$
it now remains to use~\eqref{e:poro+in}, Lemma~\ref{l:porous-nbhd},
and the previously established porosity of $\widetilde\Omega^+_k$.
\end{proof}
%%%%%%%%%%%%%%%%%%%%%%%%%%%%%%%%%%%%%%%%%%%%%%%%%%%%%%%%%%%%%%%%%%%%%%%%%%%%%%%%
Finally, using Lemma~\ref{l:Omega-cont} we see that Lemma~\ref{l:poro-} follows from
%%%%%%%%%%%%%%%%%%%%%%%%%%%%%%%%%%%%%%%%%%%%%%%%%%%%%%%%%%%%%%%%%%%%%%%%%%%%%%%%
\begin{lemm}
  \label{l:poro-help}
Let $\nu>0$ be fixed in Lemma~\ref{l:porosity-intersected}.
Then for each $k\in \{1,\dots,R\}$ the set $\Omega^-_k$
is $\nu\over 6$-porous on scales $Ch^{1/(6\Lambda)}$ to~$C^{-1}$.
\end{lemm}
%%%%%%%%%%%%%%%%%%%%%%%%%%%%%%%%%%%%%%%%%%%%%%%%%%%%%%%%%%%%%%%%%%%%%%%%%%%%%%%%
\begin{proof}
Without loss of generality we may assume that $\Sigma^-_k\neq\emptyset$.
Then $\rho_k$ lies in the ${\varepsilon_1\over C_2}\leq\varepsilon_2$ sized neighborhood
of $\mathcal V^+_e\cap S^*M\Subset U_{\rho_0}$.
Let $\gamma_k^u:[-C\varepsilon_2,C\varepsilon_2]\to S^*M$ be an unstable interval
(see Definition~\ref{d:stun-interval}) such that
$\gamma_k^u(0)=\rho_k$. Here $C$ is chosen large enough
(depending only on~$(M,g)$) so that
every point $\rho''\in W_u(\rho_k)$ with $d(\rho_k,\rho'')\leq \varepsilon_2$
lies in~$\gamma_k^u$.
We may choose $\varepsilon_2$ small enough
so that
$\gamma_k^u\subset U_{\rho_0}$.

%%%%%%%%%%%%%%%%%%%%%%%%%%%%%%%%%%%%%%%%%%%%%%%%%%%%%%%%%%%%%%%%%%%%%%%%%%%%%%%%
\begin{figure}
\includegraphics{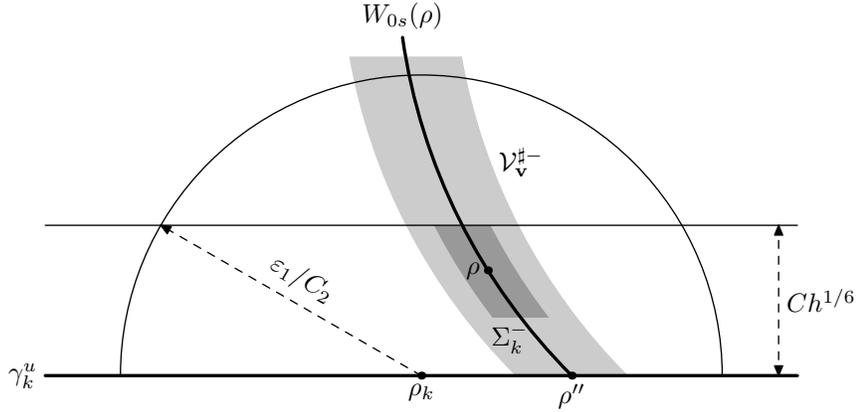}
\caption{An illustration of the proof of Lemma~\ref{l:poro-help},
following the same convention as Figure~\ref{f:poro+help}.
The darker shaded set is $\Sigma^-_k$
and the lighter shaded set is $\mathcal V^{\sharp-}_{\mathbf v}$.}
\label{f:poro-help}
\end{figure}
%%%%%%%%%%%%%%%%%%%%%%%%%%%%%%%%%%%%%%%%%%%%%%%%%%%%%%%%%%%%%%%%%%%%%%%%%%%%%%%%

Since $\varkappa$ is a symplectomorphism and $p=\eta_2\circ\varkappa$
by part~(4) of Lemma~\ref{l:stun-straight},
$\varkappa$ maps the Hamiltonian field $H_p$ into $\partial_{y_2}$.
Since $E_u(\rho_k)$ is transversal to $H_p$
and tangent to $W_{0u}(\rho_0)$, which is mapped by $\varkappa$
to $\{\eta_1=0,\ \eta_2=1\}$, we have
$d(y_1\circ\varkappa)(\rho_k)\dot\gamma^u_k(0)\neq 0$. Therefore
if $\varepsilon_2$ is small enough depending only on $(M,g)$ then
the map
$$
\psi^u_k:=y_1\circ\varkappa\circ \gamma^u_k:[-C\varepsilon_2,C\varepsilon_2]\to\mathbb R
$$
is a diffeomorphism onto its image. We extend $\psi^u_k$ to a global diffeomorphism
similarly to the proof of Lemma~\ref{l:poro+help} and define
$$
\widetilde\Omega^-_k:=\psi^u_k\big((\gamma^u_k)^{-1}(\mathcal V^{\sharp-}_{\mathbf v})\big)
=y_1(\varkappa(\gamma^u_k\cap \mathcal V^{\sharp-}_{\mathbf v}))\subset\mathbb R.
$$
Then by Lemmas~\ref{l:porosity-intersected} and~\ref{l:porous-map}
the set $\widetilde\Omega^-_k$ is $\nu\over 2$-porous on scales $Ch^{1/(6\Lambda)}$ to $C^{-1}$.

We now claim that
\begin{equation}
  \label{e:poro-in}
\Omega^-_k\subset \widetilde\Omega^-_k+[-Ch^{1/6},Ch^{1/6}].  
\end{equation}
Indeed, take arbitrary $\rho\in \Sigma^-_k$. Then $d(\rho,\rho_k)\leq {\varepsilon_1\over C_2}\leq\varepsilon_2$, so by Lemma~\ref{l:productor} there exist
$$
\rho'\in W_s(\rho),\
\rho''\in W_u(\rho_k),\
r\in [-\varepsilon_1,\varepsilon_1]\quad\text{such that}\quad
\rho'=\varphi_r(\rho'').
$$
(See Figure~\ref{f:poro-help}.)
By~\eqref{e:productor2} we have $d(\rho_k,\rho'')\leq\varepsilon_1\leq\varepsilon_2$,
thus $\rho''\in \gamma^u_k$. We also have
$d(\rho,\rho'')\leq\varepsilon_1$,
$\rho''\in W_{0s}(\rho)$, and
$\rho\in \mathcal V^-_{\mathbf v}\cap S^*M$, which by Lemma~\ref{l:blurry} imply
that $\rho''\in \mathcal V^{\sharp-}_{\mathbf v}$. Therefore
\begin{equation}
  \label{e:poro-nt}
y_1(\varkappa(\rho''))\in \widetilde \Omega^-_k.
\end{equation}
Since $d(\rho,W_{0u}(\rho_0))\leq C_3h^{1/6}$
and $\rho'\in W_{0u}(\rho_0)\cap W_s(\rho)$, we have
$d(\rho,\rho')\leq Ch^{1/6}$. We also have
$y_1(\varkappa(\rho'))=y_1(\varkappa(\rho''))$. It follows that
$$
|y_1(\varkappa(\rho))-y_1(\varkappa(\rho''))|\leq Ch^{1/6}.
$$
Since $\Omega^-_k=y_1(\varkappa(\Sigma^-_k))$, together with~\eqref{e:poro-nt} this gives~\eqref{e:poro-in}.

To show that $\Omega^-_k$ is $\nu\over 6$-porous on scales $Ch^{1/(6\Lambda)}$ to~$C^{-1}$
it remains to use~\eqref{e:poro-in}, Lemma~\ref{l:porous-nbhd},
and the previously established porosity of $\widetilde\Omega^-_k$.
\end{proof}
%%%%%%%%%%%%%%%%%%%%%%%%%%%%%%%%%%%%%%%%%%%%%%%%%%%%%%%%%%%%%%%%%%%%%%%%%%%%%%%%

%%%%%%%%%%%%%%%%%%%%%%%%%%%%%%%%%%%%%%%%%%%%%%%%%%%%%%%%%%%%%%%%%%%%%%%%%%%%%%%%
\subsubsection{Application of the fractal uncertainty principle}
\label{s:fup-applied}

We now use the fractal uncertainty principle (in the form given by Proposition~\ref{l:fup-2})
and the porosity statements proved in Lemmas~\ref{l:poro+}--\ref{l:poro-}
to establish an uncertainty principle for neighborhoods
of the right-hand sides of~\eqref{e:Omega+con}--\eqref{e:Omega-con}.
Recall the sets $\Omega^\pm\subset\mathbb R$ from~\eqref{e:Omega+},\eqref{e:Omega-}.
As before, denote by $\Omega^\pm(\alpha):=\Omega^\pm+[-\alpha,\alpha]$ the
$\alpha$-neighborhood of~$\Omega^\pm$. 
%%%%%%%%%%%%%%%%%%%%%%%%%%%%%%%%%%%%%%%%%%%%%%%%%%%%%%%%%%%%%%%%%%%%%%%%%%%%%%%%
\begin{lemm}
\label{l:2D-FUP}
Define
the following subsets of $\mathbb R^2$: 
\begin{align}
  \label{e:Upsilon+def}
\Upsilon^+:=\Big\{(\eta_1,\eta_2)\,\Big|\, {1\over 4}\leq \eta_2\leq 4,\
{\eta_1\over \eta_2}\in\Omega^+(h^\tau) \Big\},\\
  \label{e:Upsilon-def}
\Upsilon^-:=\big\{(y_1,y_2)\mid y_1\in\Omega^-(h^{1/6})\big\}.
\end{align}
Then there exists
$\beta>0$ depending only on~$\mathcal V_1,\mathcal V_\star$ such that
\begin{equation}
\label{e:2D-FUP}
\big\| \indic_{\Upsilon^-}(y) \indic_{\Upsilon^+}(hD_y) \big\|_{L^2(\mathbb R^2)\to L^2(\mathbb R^2)}
\leq Ch^{\beta}.
\end{equation}
\end{lemm}
%%%%%%%%%%%%%%%%%%%%%%%%%%%%%%%%%%%%%%%%%%%%%%%%%%%%%%%%%%%%%%%%%%%%%%%%%%%%%%%%
\begin{proof}
1. Put
$\widehat\Omega^-:=\Omega^-(h^{1/6})$,
$\widehat\Omega^+:=\Omega^+(h^\tau)$.
We first show that
\begin{equation}
  \label{e:2DF-1}
\big\| \indic_{\Upsilon^-}(y) \indic_{\Upsilon^+}(hD_y) \big\|_{L^2(\mathbb R^2)\to L^2(\mathbb R^2)}
\leq \sup_{\eta_2\in [{1\over 4},4]}
\|\indic_{\widehat\Omega^-}(hD_{\eta_1})\indic_{-\eta_2\widehat\Omega^+}(\eta_1)\|_{L^2(\mathbb R)\to L^2(\mathbb R)}.
\end{equation}
Indeed, conjugating by the semiclassical Fourier transform we see that
$$
\big\| \indic_{\Upsilon^-}(y) \indic_{\Upsilon^+}(hD_y) \big\|_{L^2(\mathbb R^2)\to L^2(\mathbb R^2)}
=
\big\| \indic_{\Upsilon^-}(hD_\eta)\indic_{\Upsilon^+}(-\eta)  \big\|_{L^2(\mathbb R^2)\to L^2(\mathbb R^2)}.
$$
Now take
$$
f\in \CIc(\IR^2),\quad
g:=\indic_{\Upsilon^-}(hD_\eta)\indic_{\Upsilon^+}(-\eta) f.
$$
For each $\eta_2\in\mathbb R$ define the functions $f_{\eta_2},g_{\eta_2}\in L^2(\mathbb R)$ by
$f_{\eta_2}(\eta_1):=f(\eta_1,-\eta_2)$,
$g_{\eta_2}(\eta_1):=g(\eta_1,-\eta_2)$.
Then
$$
g_{\eta_2}=\begin{cases}
\indic_{\widehat\Omega^-}(hD_{\eta_1})\indic_{-\eta_2\widehat\Omega^+}(\eta_1) f_{\eta_2},& \eta_2\in [{1\over 4},4];\\
0,&\text{otherwise.}
\end{cases}
$$
Writing $\|f\|_{L^2(\mathbb R^2)}^2$ as the integral of $\|f_{\eta_2}\|_{L^2(\mathbb R)}^2$
over $\eta_2$, and same for the norm of~$g$, we obtain~\eqref{e:2DF-1}.

\noindent 2. Fix $\eta_2\in [{1\over 4},4]$. Denoting by $\mathcal F_h$ the one-dimensional unitary semiclassical Fourier transform (see~\eqref{e:F-h-def}), we have  
\begin{equation}
  \label{e:2DF-2}
\|\indic_{\widehat\Omega^-}(hD_{\eta_1})\indic_{-\eta_2\widehat\Omega^+}(\eta_1)\|_{L^2(\mathbb R)\to L^2(\mathbb R)}=
\|\indic_{\widehat\Omega^-}\mathcal F_h\indic_{-\eta_2\widehat\Omega^+}\|_{L^2(\mathbb R)\to L^2(\mathbb R)}.
\end{equation}
Let $\Omega^-_k$, $\Omega^+_\ell$ be the sets defined in Lemmas~\ref{l:poro+}--\ref{l:poro-};
here $|k|,|\ell|\leq C$.
We put
$$
\widehat\Omega^-_k:=\Omega^-_k(h^{1/6}),\quad
\widehat\Omega^+_\ell:=\Omega^+_\ell(h^\tau).
$$
By Lemma~\ref{l:poro-} we have $\widehat\Omega^-\subset \bigcup_k\widehat\Omega^-_k$,
which means that $\indic_{\widehat\Omega^-}=\sum_k b_- \indic_{\widehat\Omega^-_k}$
for some $b_-\in L^\infty(\mathbb R)$, $0\leq b_-\leq 1$.
Similarly by Lemma~\ref{l:poro+} we may write
$\indic_{-\eta_2\widehat\Omega^+}=\sum_\ell\indic_{-\eta_2\widehat\Omega^+_\ell}b_+$
where $0\leq b_+\leq 1$. This gives
\begin{equation}
  \label{e:2DF-3}
\|\indic_{\widehat\Omega^-}\mathcal F_h\indic_{-\eta_2\widehat\Omega^+}\|_{L^2(\mathbb R)\to L^2(\mathbb R)}\leq\sum_{k,\ell}\|\indic_{\widehat\Omega^-_k}\mathcal F_h\indic_{-\eta_2\widehat\Omega^+_\ell}\|_{L^2(\mathbb R)\to L^2(\mathbb R)}.
\end{equation}
By Lemma~\ref{l:poro+} each set $\Omega^+_\ell$ is $\nu$-porous on scales $Ch^\tau$
to $C^{-1}$, where $\nu>0$ depends only on $\mathcal V_1,\mathcal V_\star$.
By Lemma~\ref{l:porous-nbhd} the set
$\widehat\Omega^+_\ell$ is then $\nu\over 3$-porous on scales $Ch^\tau$
to $C^{-1}$. It follows from Definition~\ref{d:porous} that
$-\eta_2\widehat\Omega^+_\ell$ is $\nu\over 3$-porous on scales $4Ch^\tau$
to $(4C)^{-1}$. Similarly, by Lemmas~\ref{l:poro-} and~\ref{l:porous-nbhd},
each set $\widehat\Omega^-_k$ is $\nu\over 3$-porous on scales $Ch^{1/(6\Lambda)}$
to~$C^{-1}$. 

We now apply Proposition~\ref{l:fup-2} to the sets $\widehat\Omega^-_k$,
$-\eta_2\widehat\Omega^+_\ell$. By the discussion in the previous paragraph,
for $h$ small enough these sets are $\nu\over 3$-porous on scales $h^{\gamma_0^-}$ to $h^{\gamma_1^-}$ and $h^{\gamma_0^+}$ to $h^{\gamma_1^+}$ respectively, where
$$
\gamma_0^-={1\over 6\Lambda}-\epsilon,\quad
\gamma_0^+=\tau-\epsilon,\quad
\gamma_1^-=\gamma_1^+=\epsilon:={1\over 60\Lambda}.
$$
Recalling from~\eqref{e:tau-delta-def} that $\tau=1-{1\over 10\Lambda}$, we compute
\begin{equation}
  \label{e:porosity-finally-meets}
\gamma:=\min(\gamma_0^+,1-\gamma_1^-)-\max(\gamma_1^+,1-\gamma_0^-)
={1\over 30\Lambda}>0.
\end{equation}
If $\beta_0>0$ is the constant from Proposition~\ref{l:fup-2} with
$\nu$ replaced by $\nu\over 3$, then~\eqref{e:fup-2} gives
\begin{equation}
  \label{e:2DF-4}
\|\indic_{\widehat\Omega^-_k}\mathcal F_h\indic_{-\eta_2\widehat\Omega^+_\ell}\|_{L^2(\mathbb R)\to L^2(\mathbb R)}
\leq Ch^{\beta},\quad
\beta:=\gamma\beta_0>0.
\end{equation}
Together~\eqref{e:2DF-1}--\eqref{e:2DF-3} and~\eqref{e:2DF-4} imply~\eqref{e:2D-FUP}.
\end{proof}
%%%%%%%%%%%%%%%%%%%%%%%%%%%%%%%%%%%%%%%%%%%%%%%%%%%%%%%%%%%%%%%%%%%%%%%%%%%%%%%%

%%%%%%%%%%%%%%%%%%%%%%%%%%%%%%%%%%%%%%%%%%%%%%%%%%%%%%%%%%%%%%%%%%%%%%%%%%%%%%%%
\subsubsection{Microlocal conjugation and the proof of Proposition~\ref{l:longdec-3}}
\label{s:micro-conjugation}
      
We now conjugate the operators $A^-_{\mathbf v}$, $A^+_{\mathcal Q}$ by Fourier
integral operators and give the proof of Proposition~\ref{l:longdec-3}
using Lemma~\ref{l:2D-FUP}.

Let $\varkappa$ be the symplectomorphism defined in~\eqref{e:conj-kappa}. As
explained in~\S\ref{s:fup-normal} we may assume that $\mathscr L_\varkappa$
is generated by a single phase function. Then (see~\S\ref{s:prelim-fio-s})
there exist Fourier integral operators
$$
\begin{aligned}
\mathcal B=\mathcal B(h):L^2(M)\to L^2(\mathbb R^2),&\quad
\mathcal B\in I^{\comp}_h(\varkappa),
\\
\mathcal B'=\mathcal B'(h):L^2(\mathbb R^2)\to L^2(M),&\quad
\mathcal B'\in I^{\comp}_h(\varkappa^{-1})
\end{aligned}
$$
which quantize $\varkappa$ near $\varkappa(\overline{\mathcal V^+_e}\cap \{{1\over 4}\leq |\xi|_g\leq 4\})\times (\overline{\mathcal V^+_e}\cap \{{1\over 4}\leq |\xi|_g\leq 4\})$ in the sense of~\eqref{e:fio-quantize}.
In particular
\begin{equation}
  \label{e:zee}
\cB'\cB=I+\mathcal O(h^\infty)\quad\text{microlocally near}\quad
\overline{\mathcal V^+_e}\cap \{\textstyle{1\over 4}\leq |\xi|_g\leq 4\}.
\end{equation}
By Lemma~\ref{l:stun-straight} all derivatives of $\varkappa$ are bounded independently
of the choice of the base point~$\rho_0$ fixed in~\eqref{e:fixer}.
Thus we may choose $\mathcal B,\mathcal B'$ which are bounded uniformly
in~$h,\rho_0$; that is, all derivatives of the corresponding phase functions and amplitudes
in the oscillatory integral representations~\eqref{e:lag-dist} are bounded.

By Egorov's Theorem~\eqref{e:egorov-basic-more}
and since $\WFh(A_e)\subset \mathcal V_e\cap \{{1\over 4}<|\xi|_g<4\}$
by~\eqref{e:A-q-prop} and
$\mathcal V^+_e=\varphi_1(\mathcal V_e)$ by~\eqref{e:V+-}, we have
$$
\WFh(A_e(-1))\subset \mathcal V^+_e\cap \{\textstyle{1\over 4}<|\xi|_g <4\}.
$$
Fix a pseudodifferential cutoff $Z_e\in \Psi^0_h(M)$ such that
\begin{equation}
  \label{e:zee2}
\WFh(Z_e)\subset  \mathcal V^+_e\cap\{\textstyle{1\over 4}< |\xi|_g< 4\},\quad
\WFh(I-Z_e)\cap\WFh(A_e(-1))=\emptyset.
\end{equation}
Since $A^+_{\mathcal Q}$ is the sum of polynomially many in $h$ terms of the form
$A^+_{\mathbf q}$ (see~\eqref{e:A-E-def}) with the words $\mathbf q\in\mathcal Q'_n(\mathbf w,e)$ starting with the letter $e$ (see~\eqref{e:Q-n-def}), we see
from the definition~\eqref{e:A-pm-def} of $A^+_{\mathbf q}$ that
\begin{equation}
  \label{e:zebra}
A^+_{\mathcal Q}=Z_eA^+_{\mathcal Q}+\mathcal O(h^\infty)_{L^2(M)\to L^2(M)}.
\end{equation}
Since $\WFh(Z_e)\cap \WFh(I-\mathcal B'\mathcal B)=\emptyset$ by~\eqref{e:zee}--\eqref{e:zee2}, we then have
\begin{equation}
  \label{e:starter}
A^-_{\mathbf v}A^+_{\mathcal Q}=
A^-_{\mathbf v}Z_e\mathcal B'\mathcal B A^+_{\mathcal Q}+\mathcal O(h^\infty)_{L^2(M)\to L^2(M)}.
\end{equation}
We also have norm bounds
\begin{align}
\label{e:apn-1}
\|A^-_{\mathbf v}\|_{L^2(M)\to L^2(M)}&\leq 2,\\
\label{e:apn-2}
\|A^+_{\mathcal Q}\|_{L^2(M)\to L^2(M)}&\leq C\log^3(1/h).
\end{align}
Here~\eqref{e:apn-1} follows from~\eqref{e:A-E-norm} and~\eqref{e:apn-2} follows from Lemma~\ref{l:ehrenfest-summary} and~\eqref{e:go-deep}.

By the equivariance of pseudodifferential operators
under conjugation by Fourier integral operators (see~\eqref{e:egorov-gen})
the conjugated operators $\cB A^-_{\mathbf v}Z_e\cB'$
and $\cB A^+_{\mathcal Q}\cB'$ formally correspond to the symbols
$$
(a^-_{\mathbf v}\sigma_h(Z_e))\circ\varkappa^{-1},\quad
a^+_{\mathcal Q}\circ\varkappa^{-1}.
$$
By~\eqref{e:Omega+con}--\eqref{e:Omega-con} the supports of the above symbols satisfy
\begin{align}
  \label{e:Omega+con2}
\varkappa(\supp a^+_{\mathcal Q})&\ \subset\
\Big\{{\eta_1\over \eta_2}\in\Omega^+\Big\}\cap\Big\{{1\over 4}\leq \eta_2\leq 4\Big\}, 
  \\
  \label{e:Omega-con2}
\varkappa\big(\supp(a^-_{\mathbf v}\sigma_h(Z_e))\big)
\cap\Big\{\Big|{\eta_1\over\eta_2}\Big|\leq h^{1/6}\Big\} &\ \subset\
\{y_1\in\Omega^-\}
\end{align}
where the sets $\Omega^\pm\subset\mathbb R$ are defined in~\eqref{e:Omega+},\eqref{e:Omega-}.
Here we denote points in~$T^*\mathbb R^2$ by~$(y,\eta)$ where $y,\eta\in\mathbb R^2$.

We now make two microlocalization statements which
quantize the above containments.
The first statement, proved using the results of~\S\ref{s:longtime}
and~\S\ref{s:fourloc-lag},
quantizes~\eqref{e:Omega+con2}:
%%%%%%%%%%%%%%%%%%%%%%%%%%%%%%%%%%%%%%%%%%%%%%%%%%%%%%%%%%%%%%%%%%%%%%%%%%%%%%%%
\begin{lemm}
  \label{l:loca+}
Assume that the constant $\varepsilon_0$ in~\S\ref{s:refined-partition} is chosen small enough depending only on $(M,g)$.
Let $\Upsilon^+\subset\mathbb R^2$ be defined in~\eqref{e:Upsilon+def}. Then
\begin{equation}
  \label{e:loca+}
\mathcal BA^+_{\mathcal Q}=\indic_{\Upsilon^+}(hD_y)\mathcal BA^+_{\mathcal Q}+\mathcal O(h^\infty)_{L^2(M)\to L^2(\mathbb R^2)}.
\end{equation} 
\end{lemm}
%%%%%%%%%%%%%%%%%%%%%%%%%%%%%%%%%%%%%%%%%%%%%%%%%%%%%%%%%%%%%%%%%%%%%%%%%%%%%%%%
\begin{proof}
1. By~\eqref{e:zebra} it suffices to prove that
$\indic_{\mathbb R^2\setminus\Upsilon^+}(hD_y)\mathcal BZ_eA^+_{\mathcal Q}=\mathcal O(h^\infty)_{L^2(M)\to L^2(\mathbb R^2)}$.
Since $\mathcal Q$ has polynomially many in $h$ elements, recalling
the definition~\eqref{e:A-E-def} of $A^+_{\mathcal Q}$ it suffices to show that
uniformly in $\mathbf q\in\mathcal Q$
\begin{equation}
  \label{e:cola-1}
\indic_{\mathbb R^2\setminus \Upsilon^+}(hD_y)\mathcal BZ_eA^+_{\mathbf q}=
\mathcal O(h^\infty)_{L^2(M)\to L^2(\mathbb R^2)}.
\end{equation}
We henceforth fix $\mathbf q\in\mathcal Q$. Recalling the definitions~\eqref{e:Omega+} and~\eqref{e:V-Q-def} of~$\Omega^+$ and $\mathcal V^+_{\mathcal Q}$ we see
that $\Omega^+_{\mathbf q}\subset\Omega^+$ where
\begin{equation}
  \label{e:Omega+q}
\Omega^+_{\mathbf q}:=\eta_1(\varkappa(\mathcal V^+_{\mathbf q}\cap S^*M))\ \subset\ \mathbb R.
\end{equation}
Recalling the definition~\eqref{e:Upsilon+def} of $\Upsilon^+$ we then have $\Upsilon^+_{\mathbf q}\subset\Upsilon^+$ where
\begin{equation}
  \label{e:Upsilon+q}
\Upsilon^+_{\mathbf q}:=\Big\{(\eta_1,\eta_2)\,\Big|\, {1\over 4}\leq \eta_2\leq 4,\ {\eta_1\over\eta_2}\in \Omega^+_{\mathbf q}(h^\tau)\Big\}.
\end{equation}
Moreover, we have $A^+_{\mathbf q}=U^+_{\mathbf q}U(-n)$ where the cutoff propagator $U^+_{\mathbf q}$
is defined in~\eqref{e:U+def}. Since $U(-n)$ is unitary, \eqref{e:cola-1} follows from the bound
\begin{equation}
  \label{e:cola-2}
\indic_{\mathbb R^2\setminus\Upsilon^+_{\mathbf q}}(hD_y)\mathcal BZ_eU
^+_{\mathbf q}
=\mathcal O(h^\infty)_{L^2(M)\to L^2(\mathbb R^{2})}.
\end{equation}

\noindent 2. Let $B_q,B'_q$, $q\in\mathscr A$, be the Fourier integral operators defined in~\eqref{e:B-q-def}. They quantize the symplectomorphisms $\varkappa_q$ defined in~\eqref{e:kappa-q}.
Since $\WFh(A_q)\subset \mathcal V_q\cap \{{1\over 4}<|\xi|_g<4\}$ we have
\begin{equation}
  \label{e:coca-cola}
A_{q}=B'_qB_qA_q+\mathcal O(h^\infty)_{L^2(M)\to L^2(M)}=B'_{q}B_{q}A_{q}B'_{q}B_{q}+\mathcal O(h^\infty)_{L^2(M)\to L^2(M)}.
\end{equation}
Put $\widehat A_{q}:=B_{q}A_{q}B'_{q}$. By~\eqref{e:egorov-gen} 
and part~(4) of Lemma~\ref{l:stun-straight} we have
$$
\widehat A_{q}\in\Psi^0_h(\mathbb R^2),\quad
\WFh(\widehat A_{q})\subset \varkappa_{q}(\WFh(A_{q}))\Subset \{\textstyle{1\over 4}<\eta_2<4\}.
$$
Thus there exists an $h$-independent function $\chi\in\CIc(\mathbb R^2)$ such that
for all $q\in\mathscr A$
$$
\supp\chi\subset \{\textstyle{1\over 4}<\eta_2<4\},\quad
\widehat A_{q}=\widehat A_{q}\chi(hD_y)+\mathcal O(h^\infty)_{L^2(\mathbb R^2)\to L^2(\mathbb R^2)}.
$$
Together with~\eqref{e:coca-cola} this implies
$$
A_q=A_q B'_q \chi(hD_y)B_q+\mathcal O(h^\infty)_{L^2(M)\to L^2(M)}.
$$
We write $\mathbf q=q_1\dots q_n$, where $q_1=e$ (see~\eqref{e:Q-n-def}).
Recalling~\eqref{e:U+def}, we have
$$
U^+_{\mathbf q}=U^+_{\mathbf q}B'_{q_n}\chi(hD_y)B_{q_n}+\mathcal O(h^\infty)_{L^2(M)\to L^2(M)}.
$$
Thus~\eqref{e:cola-2} follows from the estimate
\begin{equation}
  \label{e:cola-3}
\indic_{\mathbb R^2\setminus\Upsilon^+_{\mathbf q}}(hD_y)\mathcal BZ_eU
^+_{\mathbf q}B'_{q_n}\chi(hD_y)
=\mathcal O(h^\infty)_{L^2(\mathbb R^2)\to L^2(\mathbb R^2)}
\end{equation}
and the $L^2$-boundedness of Fourier integral operators.

Now, take arbitrary $f\in L^2(\mathbb R^2)$ such that $\|f\|_{L^2}=1$.
Following~\eqref{e:Phi-theta-def} define $\Phi_\theta(y)=\langle y,\theta\rangle$, $y,\theta\in\mathbb R^2$. Using the Fourier inversion formula we write
\begin{equation}
  \label{e:fourier-decomposed}
\chi(hD_y)f(y)=(2\pi h)^{-1}\int_{\mathbb R^2} \chi(\theta)\mathcal F_hf(\theta)e^{i\Phi_\theta(y)/h}\,d\theta
\end{equation}
where $\mathcal F_hf(\theta)=(2\pi h)^{-1}\hat f(\theta/h)$ is the semiclassical Fourier transform of~$f$, satisfying $\|\mathcal F_hf\|_{L^2(\mathbb R^2)}=1$.
Using H\"older's inequality we bound
$$
\begin{gathered}
\|\indic_{\mathbb R^2\setminus\Upsilon^+_{\mathbf q}}(hD_y)\mathcal BZ_eU^+_{\mathbf q}B'_{q_n}\chi(hD_y)f\|_{L^2(\mathbb R^2)}
\\\leq Ch^{-1}\sup_{\theta\in\supp\chi}\|\indic_{\mathbb R^2\setminus\Upsilon^+_{\mathbf q}}(hD_y)\mathcal BZ_eU^+_{\mathbf q}B'_{q_n}(e^{i\Phi_\theta/h})\|_{L^2(\mathbb R^2)}.
\end{gathered}
$$
Thus to prove~\eqref{e:cola-3} it is enough to show the following estimate on the propagated Lagrangian distributions $U^+_{\mathbf q}B'_{q_n}(e^{i\Phi_\theta/h})$:
\begin{equation}
  \label{e:cola-4}
\sup_{\theta\in\supp\chi}\|\indic_{\mathbb R^2\setminus\Upsilon^+_{\mathbf q}}(hD_y)\mathcal BZ_eU^+_{\mathbf q}B'_{q_n}(e^{i\Phi_\theta/h})\|_{L^2(\mathbb R^2)}=\mathcal O(h^\infty).
\end{equation}

\noindent 3.
Henceforth we fix $\theta\in\supp \chi$.
In particular, ${1\over 4}+\epsilon\leq \theta_2\leq 4-\epsilon$ for some fixed $\epsilon>0$.
Let $\mathbf N>0$. Using Proposition~\ref{l:longtime-prop} we write (recalling that $q_1=e$)
\begin{equation}
  \label{e:toro-0}
U^+_{\mathbf q}B'_{q_n}(e^{i\Phi_\theta/h})=U(1)B'_{e}(e^{i\Phi_{\mathbf q,\theta}/h}a_{\mathbf q,\theta,\mathbf N})+\mathcal O(h^{\mathbf N})_{L^2(M)}.
\end{equation}
Here $\Phi_{\mathbf q,\theta}$ is a generating function (in the sense of~\eqref{e:basic-lagrangian}) of the propagated Lagrangian
$\widehat{\mathscr L}_{\mathbf q,\theta}=\varkappa_e(\mathscr L_{\mathbf q,\theta})$
defined in~\eqref{e:compost-1}.

We now analyze the function $\mathcal BZ_eU(1)B'_e(e^{i\Phi_{\mathbf q,\theta}/h}a_{\mathbf q,\theta,\mathbf N})$.
By~\eqref{e:U-1-fio}, the composition property~(4) in~\eqref{s:prelim-fio-s},
and the condition~\eqref{e:zee2} on $\WFh(Z_e)$ we have
$$
\mathcal BZ_eU(1)B'_e\in I^{\comp}_h(\widetilde\varkappa),\quad
\widetilde\varkappa:=
\varkappa\circ \varphi_1\circ\varkappa_e^{-1}|_{\varkappa_e(\mathcal V_e)}.
$$
Recall from~\eqref{e:kappa-q} and~\eqref{e:conj-kappa} that
$\varkappa_e=\varkappa_{\rho_e}$, $\varkappa=\varkappa_{\rho_0}$
are homogeneous symplectomorphisms constructed using Lemma~\ref{l:stun-straight} and
$\rho_0\in\varphi_1(\mathcal V_e\cap S^*M)$ (as assumed in Proposition~\ref{l:longdec-3}),
$\rho_e\in \mathcal V_e\cap S^*M$,
with the diameter of $\mathcal V_e\cap S^*M$ bounded above by~$\varepsilon_0$.
In particular, $d\varkappa_e(\rho_e)$ maps the flow/stable/unstable spaces
$E_0(\rho_e),E_s(\rho_e),E_u(\rho_e)$ to $\mathbb R\partial_{y_2},\mathbb R\partial_{\eta_1},\mathbb R\partial_{y_1}$ and a similar 
statement is true for $d\varkappa(\rho_0)$. Thus for $\varepsilon_0$ small enough,
the differential $d\widetilde\varkappa(0,0,0,1)$ maps the vertical subspace $\ker dy$ to
an almost vertical subspace. It follows that
$\widetilde\varkappa$ has a generating function in the sense of~\eqref{e:ct-std-par},
and thus $\mathcal BZ_eU(1)B'_e$ can be written in the oscillatory integral form~\eqref{e:fio-std-par}.
(See the proof of~\cite[Lemma~4.4]{NZ09} for details.)
Moreover, by Lemma~\ref{l:inclinator} the Lagrangian $\widehat{\mathscr L}_{\mathbf q,\theta}$
is a graph in the $y$ variables and its tangent planes are $\mathcal O(\varepsilon_0)$
close to horizontal. Thus for $\varepsilon_0$ small enough the Lagrangian submanifold
$$
\widetilde{\mathscr L}:=\widetilde\varkappa(\widehat{\mathscr L}_{\mathbf q,\theta})
=\varkappa\big(\varphi_n(\varkappa_{q_n}^{-1}(\widehat{\mathscr L}_{\theta}))\cap\mathcal V^+_{\mathbf q}\big) \subset T^*\mathbb R^2
$$
is also a graph in the $y$ variables, and thus can be written in the form~\eqref{e:basic-lagrangian}:
$$
\widetilde{\mathscr L}=\{(y,d\widetilde\Phi(y))\mid y\in\widetilde{\mathscr U}\}.
$$
From the properties of $\widehat{\mathscr L}_{\mathbf q,\theta}$ in Lemma~\ref{l:inclinator}
we see that for every $\alpha$
\begin{equation}
  \label{e:toro-1}
\sup_{\widetilde{\mathscr U}} |\partial^\alpha\widetilde\Phi|\leq C_\alpha
\end{equation}
where the constant $C_\alpha$ depends only on $(M,g)$ and~$\alpha$.

We now apply the method of stationary phase using~\eqref{e:stator}, \eqref{e:lag-repas} and get
\begin{equation}
  \label{e:toro-1.5}
\mathcal BZ_eU(1)B'_e(e^{i\Phi_{\mathbf q,\theta}/h}a_{\mathbf q,\theta,\mathbf N})=e^{i\widetilde\Phi/h}\tilde a+\mathcal O(h^{\mathbf N})_{L^2(\mathbb R^2)}.
\end{equation}
Here $\tilde a$ is given by the stationary phase expansion and depends on the symbol $a_{\mathbf q,\theta,\mathbf N}$;
see~\cite[Lemma~4.1]{NZ09} for details. From the properties of the symbol $a_{\mathbf q,\theta,\mathbf N}$ in
Proposition~\ref{l:longtime-prop} we see that $\tilde a\in\CIc(\widetilde{\mathscr U})$ and
for all $\alpha$
\begin{equation}
  \label{e:toro-2}
d(\supp\tilde a,\mathbb R^2\setminus \widetilde{\mathscr U})\geq C^{-1},\quad
\sup |\partial^\alpha\tilde a|\leq C_{\mathbf N,\alpha}.
\end{equation}

\noindent 4.  Together~\eqref{e:toro-0} and~\eqref{e:toro-1.5} give
$$
\mathcal B Z_eU^+_{\mathbf q}B'_{q_n}(e^{i\Phi_\theta/h})
=e^{i\widetilde\Phi/h}\tilde a+\mathcal O(h^{\mathbf N})_{L^2(\mathbb R^2)}.
$$
Since $\mathbf N$ is chosen arbitrary, to prove~\eqref{e:cola-4} it suffices to show that
\begin{equation}
  \label{e:cola-5}
\|\indic_{\mathbb R^2\setminus\Upsilon^+_{\mathbf q}}(hD_y)(e^{i\widetilde\Phi/h}\tilde a)\|_{L^2(\mathbb R^2)}=\mathcal O(h^{\mathbf N}).
\end{equation}
To do that we use Proposition~\ref{l:fourloc-lag} (which is a Fourier localization statement
for Lagrangian distributions) with $h':=h^\tau$,
$U:=\widetilde{\mathscr U}$, $\Phi:=\widetilde\Phi$, $K:=\supp \tilde a$,
and $a:=\tilde a$. 
The assumptions~\eqref{e:fourloc-lag-1} and~\eqref{e:fourloc-lag-3} of that proposition
are satisfied due to~\eqref{e:toro-1} and~\eqref{e:toro-2}.
Next, define
$$
\widetilde\Omega:=\{d\widetilde\Phi(y)\mid y\in\widetilde{\mathscr U}\}\subset\mathbb R^2.
$$
Then $\widetilde\Omega$ is the projection of $\widetilde{\mathscr L}$ onto the $\eta$ variables.
Since $\widetilde{\mathscr L}\subset \varkappa(\mathcal V^+_{\mathbf q}\cap p^{-1}(\theta_2))$,
recalling the definition~\eqref{e:Omega+q} of $\Omega^+_{\mathbf q}$ we have
$$
\widetilde\Omega\subset (\theta_2\Omega^+_{\mathbf q})\times \{\theta_2\}.
$$
As explained in the paragraph preceding Lemma~\ref{l:close-unstable},
the diameter of $\Omega^+_{\mathbf q}$ is bounded above by $Ch^\tau$.
Then $\diam \widetilde\Omega\leq Ch^\tau$ as well, giving the assumption~\eqref{e:fourloc-lag-2}. Thus Proposition~\ref{l:fourloc-lag} applies, giving
$$
\|\indic_{\mathbb R^2\setminus\widetilde\Omega({1\over 8}h^\tau)}(hD_y)(e^{i\widetilde\Phi/h}\tilde a)\|_{L^2(\mathbb R^2)}\leq C_{\mathbf N}h^{\mathbf N}.
$$
Since the neighborhood $\widetilde\Omega({1\over 8}h^{\tau})$ lies inside~$\Upsilon^+_{\mathbf q}$ by~\eqref{e:Upsilon+q} and~\eqref{e:Omega+small}, this gives~\eqref{e:cola-5}, finishing the proof.
\end{proof}
%%%%%%%%%%%%%%%%%%%%%%%%%%%%%%%%%%%%%%%%%%%%%%%%%%%%%%%%%%%%%%%%%%%%%%%%%%%%%%%%
Our second microlocalization statement quantizes~\eqref{e:Omega-con2}:
%%%%%%%%%%%%%%%%%%%%%%%%%%%%%%%%%%%%%%%%%%%%%%%%%%%%%%%%%%%%%%%%%%%%%%%%%%%%%%%%
\begin{lemm}
  \label{l:loca-}
Let $\Upsilon^-\subset\mathbb R^2$ be defined in~\eqref{e:Upsilon-def}. Then
there exists $\chi_-\in\CIc(\mathbb R^2;[0,1])$ such that
$\supp\chi_-\subset\Upsilon^-$ and
\begin{equation}
  \label{e:loca-}
A^-_{\mathbf v}Z_e\mathcal B'\indic_{\Upsilon^+}(hD_y)=A^-_{\mathbf v}Z_e\mathcal B'\chi_-(y)\indic_{\Upsilon^+}(hD_y)+\mathcal O(h^{2/3-})_{L^2(\mathbb R^2)\to L^2(M)}.
\end{equation}
\end{lemm}
%%%%%%%%%%%%%%%%%%%%%%%%%%%%%%%%%%%%%%%%%%%%%%%%%%%%%%%%%%%%%%%%%%%%%%%%%%%%%%%%
\begin{proof}
By Lemma~\ref{l:cq-log} (recalling that we suppressed the `$-$' sign in the notation there)
and the product formula in the $\Psi_{1/6+}^{\comp}$ calculus
we have
$$
a^-_{\mathbf v}\in S^{\comp}_{1/6+}(T^*M),\quad
A^-_{\mathbf v}Z_e=\Op_h\big(a^-_{\mathbf v}\sigma_h(Z_e)\big)+\mathcal O(h^{2/3-})_{L^2(M)\to L^2(M)}.
$$
Then by~\eqref{e:zee}--\eqref{e:zee2} we get
$$
A^-_{\mathbf v}Z_e=\mathcal B'\mathcal B\Op_h\big(a^-_{\mathbf v}\sigma_h(Z_e)\big)+\mathcal O(h^{2/3-})_{L^2(M)\to L^2(M)}.
$$
Thus it suffices to show that there exists
$\chi_+\in\CIc(\mathbb R^2;[0,1])$ such that
$\chi_+=1$ on $\Upsilon^+$ and
$$
\big\|\mathcal B\Op_h\big(a^-_{\mathbf v}\sigma_h(Z_e)\big)\mathcal B'(1-\chi_-(y))\chi_+(hD_y)\big\|_{L^2(\mathbb R^2)\to L^2(\mathbb R^2)}=\mathcal O(h^{2/3-}).
$$
By~\eqref{e:egorov-gen} and since $\sigma_h(\mathcal B\mathcal B')=1$ on $\varkappa(\WFh(Z_e))$ we have
$$
\mathcal B\Op_h\big(a^-_{\mathbf v}\sigma_h(Z_e)\big)\mathcal B'
=\Op_h\big((a^-_{\mathbf v}\sigma_h(Z_e))\circ\varkappa^{-1}\big)
+\mathcal O(h^{2/3-})_{L^2(\mathbb R^2)\to L^2(\mathbb R^2)}.
$$
Thus is is enough to show the bound
\begin{equation}
  \label{e:loca-1}
\big\|\Op_h\big((a^-_{\mathbf v}\sigma_h(Z_e))\circ\varkappa^{-1}\big)(1-\chi_-(y))\chi_+(hD_y)\big\|_{L^2(\mathbb R^2)\to L^2(\mathbb R^2)}=\mathcal O(h^{2/3-}).
\end{equation}
We now define the cutoff functions $\chi_\pm$, in a way that they lie in the symbol class
$S^{\comp}_{1/6}(\mathbb R^2)$.
By~\eqref{e:Omega+small} and~\eqref{e:Upsilon+def} we have
$$
\Upsilon^+\Big({1\over 10}h^{1/6}\Big)\subset \Big\{\Big|{\eta_1\over\eta_2}\Big|\leq h^{1/6}\Big\}
$$
where $\Upsilon^+(\alpha):=\Upsilon^++B(0,\alpha)$ denotes the $\alpha$-neighborhood of $\Upsilon^+$.
By~\cite[Lemma~3.3]{hgap} there exists
$\chi_+\in S^{\comp}_{1/6}(\mathbb R^2;[0,1])$ such that
$$
\supp\chi_+\subset \Big\{\Big|{\eta_1\over\eta_2}\Big|\leq h^{1/6}\Big\},\quad
\supp(1-\chi_+)\cap\Upsilon^+=\emptyset.
$$
Next, by~\eqref{e:Omega-con2} and~\eqref{e:Upsilon-def} we have
$$
\widetilde\Upsilon^-(h^{1/6})\subset\Upsilon^-\quad\text{where}\quad
\widetilde\Upsilon^-:=y\big(\varkappa(\supp(a^-_{\mathbf v}\sigma_h(Z_e)))\cap\{\eta\in \supp\chi_+\}\big).
$$
Thus by another application of~\cite[Lemma~3.3]{hgap}
there exists $\chi_-\in S^{\comp}_{1/6}(\mathbb R^2;[0,1])$ such that
$$
\supp\chi_-\subset \Upsilon^-,\quad
\supp(1-\chi_-)\cap \widetilde\Upsilon^-=\emptyset.
$$
To prove~\eqref{e:loca-1} it remains to
use the product formula in the $\Psi_{1/6+}(\mathbb R^2)$ calculus
(see e.g.~\cite[Theorems~4.18 and~4.23]{e-z})
and the identity
$$
\big((a^-_{\mathbf v}\sigma_h(Z_e))\circ\varkappa^{-1}\big)(1-\chi_-(y))\chi_+(\eta)\equiv 0
$$
which follows from the fact that $\supp(1-\chi_-)\cap \widetilde\Upsilon^-=\emptyset$.
\end{proof}
%%%%%%%%%%%%%%%%%%%%%%%%%%%%%%%%%%%%%%%%%%%%%%%%%%%%%%%%%%%%%%%%%%%%%%%%%%%%%%%%
Armed with Lemmas~\ref{l:loca+}--\ref{l:loca-} we are finally ready to give
%%%%%%%%%%%%%%%%%%%%%%%%%%%%%%%%%%%%%%%%%%%%%%%%%%%%%%%%%%%%%%%%%%%%%%%%%%%%%%%%
\begin{proof}[Proof of Proposition~\ref{l:longdec-3}]
We have
$$
\begin{aligned}
A^-_{\mathbf v}A^+_{\mathcal Q}
&=
A^-_{\mathbf v}Z_e\mathcal B' \indic_{\Upsilon^+}(hD_y)\mathcal BA^+_{\mathcal Q}
+\mathcal O(h^\infty)_{L^2(M)\to L^2(M)}
\\
&=A^-_{\mathbf v}Z_e\mathcal B'\chi_-(y)\indic_{\Upsilon^+}(hD_y)\mathcal BA^+_{\mathcal Q}
+\mathcal O(h^{2/3-})_{L^2(M)\to L^2(M)}
\end{aligned}
$$
where the first line follows from~\eqref{e:starter}, Lemma~\ref{l:loca+}, and~\eqref{e:apn-1};
the second line follows from Lemma~\ref{l:loca-} and~\eqref{e:apn-2}.

Using the norm bounds~\eqref{e:apn-1}--\eqref{e:apn-2} and the fact that
$Z_e,\mathcal B',\mathcal B$ are bounded in $L^2\to L^2$ norm uniformly in~$h$, we get
$$
\|A^-_{\mathbf v}A^+_{\mathcal Q}\|_{L^2(M)\to L^2(M)}
\leq C\log^3(1/h)\|\indic_{\Upsilon^-}(y)\indic_{\Upsilon^+}(hD_y)\|_{L^2(\mathbb R^2)\to L^2(\mathbb R^2)}
+\mathcal O(h^{2/3-}).
$$
Using the uncertainty principle given by Lemma~\ref{l:2D-FUP} we then have
$$
\|A^-_{\mathbf v}A^+_{\mathcal Q}\|_{L^2(M)\to L^2(M)}
\leq Ch^{\beta}\log^3(1/h)+\mathcal O(h^{2/3-}).
$$
This gives~\eqref{e:longdec-3} (with a smaller value of $\beta$),
finishing the proof.
\end{proof}
%%%%%%%%%%%%%%%%%%%%%%%%%%%%%%%%%%%%%%%%%%%%%%%%%%%%%%%%%%%%%%%%%%%%%%%%%%%%%%%%

%%%%%%%%%%%%%%%%%%%%%%%%%%%%%%%%%%%%%%%%%%%%%%%%%%%%%%%%%%%%%%%%%%%%%%%%%%%%%%%%
%%%%%%%%%%%%%%%%%%%%%%%%%%%%%%%%%%%%%%%%%%%%%%%%%%%%%%%%%%%%%%%%%%%%%%%%%%%%%%%%
\section{Propagation of observables up to local Ehrenfest time}
\label{s:ops-Aq}

In this section we prove Propositions~\ref{l:ehrenfest-prop} and~\ref{l:ehrenfest-prop-sum}
on the structure of the operators $A^\pm_{\mathbf q}$ when $\mathcal J^\pm_{\mathbf q}\leq Ch^{-\delta}$. We will focus on the operators $A^-_{\mathbf q}$, with $A^+_{\mathbf q}$ handled the same way (reversing the direction of propagation).
Recall from~\eqref{e:A-pm-def} that
$$
A^-_{\mathbf q}=A_{q_{n-1}}(n-1)\cdots A_{q_0}(0),\quad
\mathbf q=q_0\dots q_{n-1}
$$
where the operators $A_q\in\Psi^{-\infty}_h(M)$, $q\in \mathscr A=\{1,\dots,Q\}$,
are defined in~\S\ref{s:refined-partition}.
Here we use the notation~\eqref{e:A-t-def}:
$$
A(t)=U(-t)AU(t),\quad
U(t)=e^{-itP/h}
$$
where $P\in\Psi^{-\infty}_h(M)$ is defined in~\eqref{e:U-t-def}.

To analyze $A^-_{\mathbf q}$ we write it as a result of an iterative process,
where at each step we conjugate by $U(1)$ and multiply by an operator $A_q$, see~\S\ref{s:ehr-step} below. We carefully estimate the resulting symbols
and the remainders at each step of the iteration, using quantitative
semiclassical expansions established in Appendix~\ref{s:semi-detail}.
This largely follows~\cite[Section 7]{Riv10}; the estimates on the symbol of $A^-_{\mathbf q}$
there are similar in spirit to those in~\cite[Section 3.4]{AN07}.
Compared to~\cite{Riv10} we will obtain more precise information on the propagated symbols
in order to control the sums over many operators $A^-_{\mathbf q}$
which is needed in the proof of Proposition~\ref{l:ehrenfest-prop-sum}.

%%%%%%%%%%%%%%%%%%%%%%%%%%%%%%%%%%%%%%%%%%%%%%%%%%%%%%%%%%%%%%%%%%%%%%%%%%%%%%%%
\subsection{Iterative construction of the operators}
\label{s:ehr-step}

Let $\mathbf q=q_0\dots q_{n-1}\in\mathscr A^\bullet$
and assume that $n\leq C_0\log(1/h)$ for some constant $C_0$.
Define
\begin{equation}
  \label{e:A-r-def}
\widehat A_{\mathbf q,r}:=A^-_{q_{n-r}\dots q_{n-1}},\quad
r=1,\dots,n.
\end{equation}
Then $\widehat A_{\mathbf q,1}=A_{q_{n-1}}$, $A^-_{\mathbf q}=\widehat A_{\mathbf q,n}$, and
we have the iterative formula
\begin{equation}
  \label{e:A-r-iter}
\widehat A_{\mathbf q,r}=U(-1)\widehat A_{\mathbf q,r-1} U(1)A_{q_{n-r}},\quad
r=2,\dots,n.
\end{equation}
The next statement gives the dependence of the full symbol
of the operator $\widehat A_{\mathbf q,r}$ on that of the operator $\widehat A_{\mathbf q,r-1}$,
with explicit remainders. We use the quantization procedure $\Op_h$
on~$M$ defined in~\eqref{e:Op-h-M}.
%%%%%%%%%%%%%%%%%%%%%%%%%%%%%%%%%%%%%%%%%%%%%%%%%%%%%%%%%%%%%%%%%%%%%%%%%%%%%%%%
\begin{lemm}\label{l:one-step}
Assume that $a\in\CIc(T^*M)$, $\supp a\subset\{{1\over 4}\leq |\xi|_g\leq 4\}$,
and $q\in \mathscr A$.
Then for each%
\footnote{We use boldface $\mathbf N$ here to avoid confusion with the propagation
time defined in~\eqref{e:prop-times}.}
$\mathbf N\in\mathbb N$ we have
\begin{equation}
  \label{e:one-step}
U(-1)\Op_h(a)U(1)A_q=\Op_h\bigg(\sum_{j=0}^{\mathbf N-1} h^j L_{j,q}(a\circ\varphi_1) \bigg)
+\mathcal O(\|a\|_{C^{2\mathbf N+17}}h^{\mathbf N})_{L^2\to L^2}.
\end{equation}
Here each $L_{j,q}$ is a differential operator of order~$2j$ on $T^*M$.
We have $L_{0,q}=a_q$. Moreover, each $L_{j,q}$ is supported in $\mathcal V_q\cap \{{1\over 4}<|\xi|_g<4\}$.

In addition to $\mathbf N$, the constant in $\mathcal O(\bullet)$ depends
only on $(M,g)$, the choice
of the coordinate charts and cutoffs in~\eqref{e:Op-h-M}, and the choice
of the operators $A_1,\dots,A_Q$. The operators $L_{j,q}$ depend only on the above data as well
as on $j,q$.
\end{lemm}
%%%%%%%%%%%%%%%%%%%%%%%%%%%%%%%%%%%%%%%%%%%%%%%%%%%%%%%%%%%%%%%%%%%%%%%%%%%%%%%%
\begin{proof}
From the construction of $A_q$ in~\S\ref{s:refined-partition} we have
for all~$\mathbf N$
\begin{equation}
  \label{e:A-q-full}
A_q=\Op_h\bigg(\sum_{j=0}^{\mathbf N-1}h^j a_{q,j}\bigg)+\mathcal O(h^{\mathbf N})_{L^2\to L^2}
\end{equation}
for some $h$-independent $a_{q,j}\in\CIc(T^*M)$ such that
$\supp a_{q,j}\subset \mathcal V_q\cap \{{1\over 4}<|\xi|_g<4\}$
and $a_{q,0}=a_q$.
Now~\eqref{e:one-step} follows by combining the precise versions of Egorov's Theorem,
Lemma~\ref{l:egorov-precise}, and of the product formula, \eqref{e:prod-mfld}.
\end{proof}
%%%%%%%%%%%%%%%%%%%%%%%%%%%%%%%%%%%%%%%%%%%%%%%%%%%%%%%%%%%%%%%%%%%%%%%%%%%%%%%%
Now, arguing by induction on $r$ with~\eqref{e:A-q-full} as the base and~\eqref{e:A-r-iter},
\eqref{e:one-step} as the inductive step,
we write for each $\mathbf N\in\mathbb N$
\begin{equation}
  \label{e:iterated-operator}
\widehat A_{\mathbf q,r}=\Op_h\bigg(\sum_{k=0}^{\mathbf N-1} h^k a_{\mathbf q,r}^{(k)}\bigg)+ R_{\mathbf q,r}^{(\mathbf N)},\quad
r=1,\dots,n
\end{equation}
where:
\begin{itemize}
\item $a_{\mathbf q,1}^{(k)}=a_{q_{n-1},k}$ where the latter function is defined in~\eqref{e:A-q-full};
\item for $r\geq 2$, we have
\begin{equation}
  \label{e:iterated-symbol}
a_{\mathbf q,r}^{(k)}=\sum_{j=0}^k L_{j,q_{n-r}}(a_{\mathbf q,r-1}^{(k-j)}\circ\varphi_1)
\end{equation}
where $L_{j,q}$ are the operators from~\eqref{e:one-step};
\item the remainder $R_{\mathbf q,r}^{(\mathbf N)}$ satisfies the norm bound
\begin{equation}
  \label{e:iterated-remainder}
\|R_{\mathbf q,r}^{(\mathbf N)}\|_{L^2\to L^2}\leq C_{\mathbf N} h^{\mathbf N} \bigg(1+\sum_{\ell=1}^{r-1}
\sum_{k=0}^{\mathbf N-1} \|a_{\mathbf q,\ell}^{(k)}\|_{C^{2(\mathbf N-k)+17}}\bigg)
\end{equation}
for some constant $C_{\mathbf N}$ independent of $\mathbf q,r$.
\end{itemize}
Here the bound~\eqref{e:iterated-remainder} is obtained from the iterative remainder
bound
$$
\|R_{\mathbf q,r}^{(\mathbf N)}\|_{L^2\to L^2}\leq \|R_{\mathbf q,r-1}^{(\mathbf N)}\|_{L^2\to L^2}\cdot \|A_{q_{n-r}}\|_{L^2\to L^2}
+C'_{\mathbf N}h^{\mathbf N}\sum_{k=0}^{\mathbf N-1}\|a^{(k)}_{\mathbf q,r-1}\|_{C^{2(\mathbf N-k)+17}}
$$
using that $\|A_q\|_{L^2\to L^2}\leq 1+Ch^{1/2}$ similarly to~\eqref{e:A-q-bdd}.

Here are some basic properties
of the symbols $a^{(k)}_{\mathbf q,r}$ which follow immediately from their construction,
using the notation~\eqref{e:a-pm-def}, \eqref{e:V+-}:
\begin{itemize}
\item $a_{\mathbf q,r}^{(k)}\in\CIc(T^*M)$ and
\begin{equation}
  \label{e:supp-symbol}
\supp a_{\mathbf q,r}^{(k)}\ \subset\ \mathcal V^-_{q_{n-r}\dots q_{n-1}}\cap \{\textstyle{1\over 4}<|\xi|_g<4\};
\end{equation}
\item $a_{\mathbf q,r}^{(0)}=a^-_{q_{n-r}\dots q_{n-1}}$, in particular
$a_{\mathbf q,n}^{(0)}=a^-_{\mathbf q}$.
\end{itemize}
The following is a key estimate on the
symbols $a^{(k)}_{\mathbf q,r}$ and their derivatives, proved in~\S\ref{s:iterated-symbol-estimates} below.
Recall that for a word $\mathbf q\in\mathscr A^\bullet$ its Jacobian
$\mathcal J^-_{\mathbf q}$ was defined in~\eqref{e:word-J-def}.
%%%%%%%%%%%%%%%%%%%%%%%%%%%%%%%%%%%%%%%%%%%%%%%%%%%%%%%%%%%%%%%%%%%%%%%%%%%%%%%%
\begin{lemm}
  \label{l:symbols-bounds}
Assume that $\mathcal V^-_{\mathbf q}\neq\emptyset$.
Then we have the following bounds for all $r,k,m$:
\begin{equation}
  \label{e:symbols-bounds}
\|a_{\mathbf q,r}^{(k)}\|_{C^m}\leq C_{km} r^{4k+2m} (\mathcal J^-_{q_{n-r}\dots q_{n-1}})^{2k+m}
\end{equation}
where the constant $C_{km}$ depends on $k,m$ but not on $r,\mathbf q$.
\end{lemm}
%%%%%%%%%%%%%%%%%%%%%%%%%%%%%%%%%%%%%%%%%%%%%%%%%%%%%%%%%%%%%%%%%%%%%%%%%%%%%%%%
\Remark We allow the factor $r^{4k+2m}$ in~\eqref{e:symbols-bounds} to simplify the proof;
it does not matter for Proposition~\ref{l:ehrenfest-prop} since $r=\mathcal O(\log(1/h))$.
It is quite possible that more careful analysis can remove this factor.

Using Lemma~\ref{l:symbols-bounds} we now give
%%%%%%%%%%%%%%%%%%%%%%%%%%%%%%%%%%%%%%%%%%%%%%%%%%%%%%%%%%%%%%%%%%%%%%%%%%%%%%%%
\begin{proof}[Proof of Proposition~\ref{l:ehrenfest-prop}]
We consider the case of $A^-_{\mathbf q}$, with $A^+_{\mathbf q}$ handled similarly.
By~\eqref{e:symbols-bounds}, recalling that
$\mathcal J^-_{\mathbf q}\leq C_0h^{-\delta}$ and $n\leq C_0\log(1/h)$, we have for all $k,m$
\begin{equation}
  \label{e:symbols-bounds-used}
\max_{1\leq r\leq n}\|a_{\mathbf q,r}^{(k)}\|_{C^m}\leq C'_{km}h^{-(2k+m)\delta}(\log(1/h))^{4k+2m}. 
\end{equation}
This implies that 
$h^k a_{\mathbf q,n}^{(k)}=\mathcal O(h^{(1-2\delta)k-})_{S^{\comp}_{\delta}}$.
Using additionally that $\sup|a_{\mathbf q,n}^{(0)}|\leq 1$ we see that
$a_{\mathbf q,n}^{(0)}=a^-_{\mathbf q}=\mathcal O(1)_{S^{\comp}_{\delta+}}$.

By Borel's Theorem~\cite[Theorem~4.15]{e-z} there exists a symbol
$a^{\flat-}_{\mathbf q}\in S^{\comp}_{\delta+}(T^*M)$ such that
$a^{\flat-}_{\mathbf q}\sim \sum_{k\geq 0}h^ka_{\mathbf q,n}^{(k)}$ in the following sense:
$$
a^{\flat-}_{\mathbf q}=\sum_{k=0}^{\mathbf N-1}h^ka_{\mathbf q,n}^{(k)}+\mathcal O_{\mathbf N}(h^{(1-2\delta)\mathbf N-})_{S^{\comp}_{\delta}}\quad\text{for all }\mathbf N\in\mathbb N.
$$
From the basic properties of the symbols $a_{\mathbf q,n}^{(k)}$ listed above we see that
$$
a^{\flat-}_{\mathbf q}=a^-_{\mathbf q}+\mathcal O(h^{1-2\delta-})_{S^{\comp}_{\delta}},\quad
\supp a^{\flat-}_{\mathbf q}\subset \mathcal V^-_{\mathbf q}\cap \{\textstyle{1\over 4}\leq |\xi|_g\leq 4\}.
$$
By~\eqref{e:iterated-operator} and the $L^2$ boundedness of operators
with symbols in $S^{\comp}_{\delta}$ we have for all~$\mathbf N$
\begin{equation}
  \label{e:iterated-operator-used}
A^-_{\mathbf q}=\widehat A_{\mathbf q,n}=\Op_h(a^{\flat-}_{\mathbf q})+R^{(\mathbf N)}_{\mathbf q,n}+\mathcal O(h^{(1-2\delta)\mathbf N-})_{L^2\to L^2}.
\end{equation}
The remainder $R^{(\mathbf N)}_{\mathbf q,n}$ is estimated using~\eqref{e:iterated-remainder} and~\eqref{e:symbols-bounds-used}:
\begin{equation}
  \label{e:iterated-remainder-estimated}
\|R^{(\mathbf N)}_{\mathbf q,n}\|_{L^2\to L^2}\leq C_{\mathbf N}h^{\mathbf N-(2\mathbf N+17)\delta}(\log(1/h))^{4\mathbf N+35}.
\end{equation}
Since $\mathbf N$ can be chosen arbitrarily large and $\delta<{1\over 2}$, together~\eqref{e:iterated-operator-used}
and~\eqref{e:iterated-remainder-estimated} imply that
$A^-_{\mathbf q}=\Op_h(a^{\flat-}_{\mathbf q})+\mathcal O(h^\infty)_{L^2\to L^2}$,
finishing the proof.
\end{proof}
%%%%%%%%%%%%%%%%%%%%%%%%%%%%%%%%%%%%%%%%%%%%%%%%%%%%%%%%%%%%%%%%%%%%%%%%%%%%%%%%

%%%%%%%%%%%%%%%%%%%%%%%%%%%%%%%%%%%%%%%%%%%%%%%%%%%%%%%%%%%%%%%%%%%%%%%%%%%%%%%%
\subsection{Estimating the iterated symbols}
\label{s:iterated-symbol-estimates}

In this section we prove Lemma~\ref{l:symbols-bounds}. To do this
we differentiate the inductive formulas~\eqref{e:iterated-symbol}
and represent the terms in the resulting expressions by the edges of
a directed graph $\mathscr G$. We then iterate~\eqref{e:iterated-symbol}
to write each derivative of $a^{(k)}_{\mathbf q,r}$ as the sum of many terms, each corresponding
to a path of length $r-1$ in $\mathscr G$~-- see~\eqref{e:itesym-2}. The reduced graph
$\widetilde{\mathscr G}$, obtained by removing the loops
from $\mathscr G$, is acyclic, which implies that the number of paths
of length $r-1$ in~$\mathscr G$ is bounded polynomially in~$r$. 
We finally analyze the term corresponding to each path, bounding it in terms of the
Jacobian $\mathcal J^-_{q_{n-r}\dots q_{n-1}}$.

%%%%%%%%%%%%%%%%%%%%%%%%%%%%%%%%%%%%%%%%%%%%%%%%%%%%%%%%%%%%%%%%%%%%%%%%%%%%%%%%
\subsubsection{Graph formalism}

We first introduce some notation to keep track of the derivatives of the symbols.
We fix some affine connection
$\nabla$ on $T^*M$. For each function $a\in C^\infty(T^*M)$ and $m\in\mathbb N_0$, let
$\nabla^m a$ be the $m$-th covariant derivative of~$a$, which
is a section of $\otimes^m T^*(T^*M)$, the $m$-th tensor power of the cotangent
bundle of $T^*M$. 
We fix an inner product on the fibers of $T^*(T^*M)$
which naturally induces a norm on each $\otimes^mT^*(T^*M)$.
When $\supp a\subset \{{1\over 4}\leq |\xi|_g\leq 4\}$ we have for some
constant $C$
\begin{equation}
  \label{e:C-m-norm}
C^{-1}\|a\|_{C^m}\leq \max_{j\leq m}\sup_{\rho\in T^*M} \|\nabla^j a(\rho)\|
\leq C\|a\|_{C^m}.
\end{equation}
Fix $\mathbf N_0\in\mathbb N_0$. The objects below will depend on $\mathbf N_0$ but
for the sake of brevity we will suppress it in the notation.
Denote
\begin{equation}
  \label{e:V-graph}
\mathscr V:=\{(k,m)\mid k,m\in\mathbb N_0,\ 2k+m\leq \mathbf N_0\}.
\end{equation}
Henceforth we write $\alpha=(k,m)$. Define the vector bundle over $T^*M$
$$
\mathscr E:=
\bigoplus_{\alpha\in\mathscr V}
\mathscr E_{\alpha},\quad
\mathscr E_{(k,m)}:=\otimes^m T^*(T^*M)
$$
and its sections composed of the derivatives of the symbols $a^{(k)}_{\mathbf q,r}$:
\begin{equation}
  \label{e:A-r-der-def}
\mathbf A_{\mathbf q,r}\in C^\infty(T^*M;\mathscr E),\quad
\mathbf A_{\mathbf q,r}:=(\nabla^m a^{(k)}_{\mathbf q,r})_{(k,m)\in\mathscr V},\quad
r=1,\dots,n.
\end{equation}
That is, in the biindex $(k,m)$, $k$ is the power of $h$
and $m$ is the number of derivatives taken. We denote by
$$
\iota_\alpha:\mathscr E_\alpha\to \mathscr E,\quad
\pi_\alpha:\mathscr E\to \mathscr E_\alpha
$$
the natural embedding and projection maps.

The iterative rules~\eqref{e:iterated-symbol} together
with the chain rule imply the relations
\begin{equation}
  \label{e:itesym-1}
\mathbf A_{\mathbf q,r}(\rho)=\mathbf M_{q_{n-r}}(\rho)\mathbf A_{\mathbf q,r-1}(\varphi_1(\rho)),\quad
r=2,\dots,n,\quad
\rho\in T^*M\setminus 0
\end{equation}
where the coefficients of the operators $L_{j,q}$ determine the homomorphisms
$$
\mathbf M_q\in C^\infty(T^* M\setminus 0;\Hom(\varphi_1^*\mathscr E;\mathscr E)),\quad
q\in\mathscr A.
$$
That is,
$\mathbf M_q(\rho)$ is a linear map
$\mathscr E(\varphi_1(\rho))\to \mathscr E(\rho)$
depending smoothly on $\rho\in T^*M\setminus 0$.

%%%%%%%%%%%%%%%%%%%%%%%%%%%%%%%%%%%%%%%%%%%%%%%%%%%%%%%%%%%%%%%%%%%%%%%%%%%%%%%%
\begin{figure}
\includegraphics{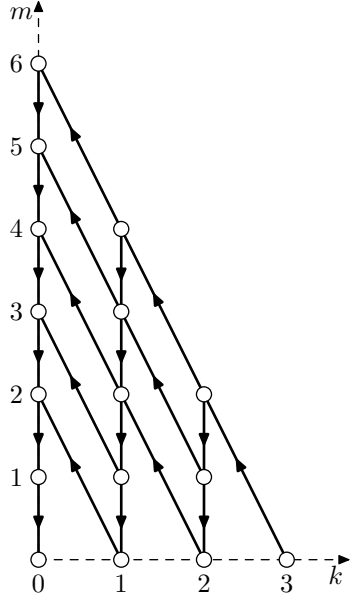}
\caption{A subgraph $\widehat{\mathscr G}$ of the reduced graph $\widetilde{\mathscr G}$ for
$\mathbf N_0=6$, with edges $(k,m)\to (k,m-1)$ and $(k,m)\to (k-1,m+2)$.
 The full graph $\widetilde{\mathscr G}$ is obtained
as follows: there is an edge from $\alpha$ to $\alpha'$ in $\widetilde{\mathscr G}$
if and only if there is a nontrivial path from $\alpha$ to $\alpha'$
in $\widehat{\mathscr G}$.}
\label{f:graph}
\end{figure}
%%%%%%%%%%%%%%%%%%%%%%%%%%%%%%%%%%%%%%%%%%%%%%%%%%%%%%%%%%%%%%%%%%%%%%%%%%%%%%%%

Define the directed graph%
\footnote{A directed graph is a pair $(V,E)$ where $V$ is a finite set of \emph{vertices}
and $E\subset V\times V$ is the set of \emph{edges}. There is an edge going from the
vertex $v_1$ to the vertex $v_2$ if and only if $(v_1,v_2)\in E$.}
 $\mathscr G$ with the set of vertices $\mathscr V$,
which has an edge from $\alpha=(k,m)$ to $\alpha'=(k',m')$ if and only if
\begin{equation}
  \label{e:edge-graph}
2k'+m'\leq 2k+m,\quad
k'\leq k.
\end{equation}
If~\eqref{e:edge-graph} holds then we write
$$
\alpha\rightarrow \alpha'.
$$
The homomorphisms $\mathbf M_q$ are subordinate to the graph $\mathscr G$ in the following
sense: we may write them in the `block matrix' form
\begin{equation}
  \label{e:matrix-reloaded}
\mathbf M_q=\sum_{\alpha\rightarrow \alpha'}
\iota_\alpha \mathbf M_{q,\alpha,\alpha'} \pi_{\alpha'}
\end{equation}
where
\begin{equation}
  \label{e:MPAA}
\mathbf M_{q,\alpha,\alpha'}:=\pi_\alpha\mathbf M_q\iota_{\alpha'}
\in C^\infty(T^*M\setminus 0;\Hom(\varphi_1^*\mathscr E_{\alpha'}; \mathscr E_\alpha)).
\end{equation}
That is, if $\nabla^{m}a_{\mathbf q,r}^{(k)}(\rho)$ depends
on $\nabla^{m'} a_{\mathbf q,r-1}^{(k')}(\varphi_1(\rho))$ in~\eqref{e:iterated-symbol},
then~\eqref{e:edge-graph} holds. This is straightforward to see using~\eqref{e:iterated-symbol}
and the chain rule.

It will be important for our analysis to separate out the `diagonal' part of $\mathbf M_q$,
consisting of the homomorphisms $\iota_\alpha \mathbf M_{q,\alpha,\alpha}\pi_\alpha$
corresponding to the loops $\alpha\to\alpha$ in the graph $\mathscr G$.
Using~\eqref{e:iterated-symbol} (recalling that $L_{0,q}=a_q$) and the chain rule we compute
\begin{equation}
  \label{e:diagonal-term}
\mathbf M_{q,\alpha,\alpha}(\rho)=a_q(\rho)\cdot (d\varphi_1(\rho)^T)^{\otimes m},\quad
\alpha=(k,m).
\end{equation}
The remaining components of $\mathbf M_q$ correspond to the \emph{reduced graph}
$\widetilde{\mathscr G}$, obtained by removing all the loops $\alpha\to \alpha$
from $\mathscr G$, see Figure~\ref{f:graph}.

%%%%%%%%%%%%%%%%%%%%%%%%%%%%%%%%%%%%%%%%%%%%%%%%%%%%%%%%%%%%%%%%%%%%%%%%%%%%%%%%
\subsubsection{Long paths and end of the proof}
  \label{s:long-paths}

We now restrict to the case $r=n$ in Lemma~\ref{l:symbols-bounds}, proving the bounds
\begin{equation}
  \label{e:symbols-bounds-tot}
\|a^{(\tilde k)}_{\mathbf q,n}\|_{C^{\widetilde m}}\leq C_{\tilde k\widetilde m}n^{4\tilde k+2\widetilde m}(\mathcal J^-_{\mathbf q})^{2\tilde k+\widetilde m},\quad
\tilde k,\widetilde m\in\mathbb N_0.
\end{equation}
The general case follows from here by replacing $\mathbf q$ with $q_{n-r}\dots q_{n-1}$. 

By~\eqref{e:C-m-norm} and the support property~\eqref{e:supp-symbol} see that~\eqref{e:symbols-bounds-tot} follows from
\begin{equation}
  \label{e:symbols-bounds-tot-2}
\sup_{\rho\in\mathcal V^-_{\mathbf q}\cap \{{1\over 4}\leq |\xi|_g\leq 4\}}\|\mathbf A_{\mathbf q,n}(\rho)\|\leq C_{\mathbf N_0} n^{2\mathbf N_0}(\mathcal J^-_{\mathbf q})^{\mathbf N_0}.
\end{equation}
Here $\mathbf N_0$ was the natural number used in~\eqref{e:V-graph} and thus in the definition~\eqref{e:A-r-der-def}
of $\mathbf A_{\mathbf q,n}$. To obtain~\eqref{e:symbols-bounds-tot} we put $\mathbf N_0:=2\tilde k+\widetilde m$.

In the rest of this section we prove~\eqref{e:symbols-bounds-tot-2}.
Iterating~\eqref{e:itesym-1} we get the following formula for $\mathbf A_{\mathbf q,n}$:
\begin{equation}
  \label{e:itesym-2}
\mathbf A_{\mathbf q,n}(\rho)=\mathbf M_{q_0}(\rho)\mathbf M_{q_1}(\varphi_1(\rho))
\cdots\mathbf M_{q_{n-2}}(\varphi_{n-2}(\rho))\mathbf A_{\mathbf q,1}(\varphi_{n-1}(\rho)).
\end{equation}
Using the decomposition~\eqref{e:matrix-reloaded} we write
\begin{equation}
  \label{e:itesym-3}
\mathbf A_{\mathbf q,n}(\rho)=\sum_{\vec\alpha\in \mathscr P}\iota_{\alpha_1}\mathbf M_{\mathbf q,\vec\alpha}(\rho)
\pi_{\alpha_n}\mathbf A_{\mathbf q,1}(\varphi_{n-1}(\rho))
\end{equation}
where  
\begin{equation}
  \label{e:path-def}
\mathscr P:=\{\vec\alpha=\alpha_1\dots\alpha_n\in\mathscr V^n\mid
\alpha_j\rightarrow \alpha_{j+1}\text{ for all }j=1,\dots,n-1\}
\end{equation}
is the set of \emph{paths} of length~$n-1$ in the graph~$\mathscr G$ and
\begin{equation}
  \label{e:MM-def}
\mathbf M_{\mathbf q,\vec\alpha}(\rho):=\mathbf M_{q_{0},\alpha_1,\alpha_2}(\rho)
\mathbf M_{q_1,\alpha_2,\alpha_3}(\varphi_1(\rho))\cdots
\mathbf M_{q_{n-2},\alpha_{n-1},\alpha_n}(\varphi_{n-2}(\rho)).
\end{equation}
Since $\sup_{T^*M}\|\mathbf A_{\mathbf q,1}\|\leq C$, using the triangle inequality
in~\eqref{e:itesym-3} we get for all $\rho\in T^*M$
\begin{equation}
  \label{e:itesym-4}
\|\mathbf A_{\mathbf q,n}(\rho)\|\leq
C\sum_{\vec\alpha\in\mathscr P}\|\mathbf M_{\mathbf q,\vec\alpha}(\rho)\|
\leq
 C\#(\mathscr P)\cdot \max_{\vec\alpha\in\mathscr P}\|\mathbf M_{\mathbf q,\vec\alpha}(\rho)\|.
\end{equation}
Thus to show~\eqref{e:symbols-bounds-tot-2}
(and thus finish the proof of Lemma~\ref{l:symbols-bounds})
it remains to prove the following
%%%%%%%%%%%%%%%%%%%%%%%%%%%%%%%%%%%%%%%%%%%%%%%%%%%%%%%%%%%%%%%%%%%%%%%%%%%%%%%%
\begin{lemm}
  \label{l:ehr-comb}
There exists a constant $C$ depending on $\mathbf N_0$ but not on $n,\mathbf q$ such that
\begin{align}
  \label{e:ehr-comb-1}
\#(\mathscr P)&\leq Cn^{2\mathbf N_0},\\
  \label{e:ehr-comb-2}
\max_{\vec\alpha\in\mathscr P}
\sup_{\rho\in\mathcal V^-_{\mathbf q}\cap \{{1\over 4}\leq |\xi|_g\leq 4\}}\|\mathbf M_{\mathbf q,\vec\alpha}(\rho)\|
&\leq C (\mathcal J^-_{\mathbf q})^{\mathbf N_0}.
\end{align}
\end{lemm}
%%%%%%%%%%%%%%%%%%%%%%%%%%%%%%%%%%%%%%%%%%%%%%%%%%%%%%%%%%%%%%%%%%%%%%%%%%%%%%%%
\begin{proof}
1.
For each path $\vec\alpha\in\mathscr P$ we define the corresponding \emph{reduced path}
$$
\mathscr R(\vec\alpha)=\beta_1\dots\beta_{\ell+1}\in\mathscr V^{\ell+1},\quad
\beta_j\neq\beta_{j+1}\quad\text{for all }j
$$
obtained by removing all the loops in $\vec\alpha$: that is, $\vec\alpha$ has the form
\begin{equation}
  \label{e:alpha-form}
\vec\alpha=\beta_1^{s_{(1)}-s_{(0)}}\beta_2^{s_{(2)}-s_{(1)}}\dots\beta_{\ell+1}^{s_{(\ell+1)}-s_{(\ell)}}
\end{equation}
where $\beta^s=\beta\beta\dots\beta$ is the path obtained by repeating $\beta\in\mathscr V$
for $s$ times and $(s_{(j)})$ is a sequence such that
$$
0=s_{(0)}<s_{(1)}<s_{(2)}<\ldots<s_{(\ell)}<s_{(\ell+1)}=n.
$$
See Figure~\ref{f:reducedpath}.

For every $\vec\alpha\in\mathscr P$, $\mathscr R(\vec\alpha)$ is a path
in the reduced graph $\widetilde{\mathscr G}$. The latter graph is acyclic,
indeed if~\eqref{e:edge-graph} holds and $(k,m)\neq (k',m')$, then
$3k'+m'<3k+m$. Since $0\leq 3k+m\leq {3\mathbf N_0\over 2}\leq 2\mathbf N_0$ for all $(k,m)\in\mathscr V$,
we see that the length $\ell$ of any path in $\widetilde{\mathscr G}$
is bounded above by~$2\mathbf N_0$.

Now, the size of the range of $\mathscr R$ is bounded above by the number of
paths in $\widetilde{\mathscr G}$, which is finite (since $\widetilde{\mathscr G}$
is acyclic) and depends only on $\mathbf N_0$.
On the other hand, if $\vec\beta$ is a fixed path in $\widetilde{\mathscr G}$
then elements of $\mathscr R^{-1}(\vec\beta)$ are determined by $s_{(1)},\dots,s_{(\ell)}$,
thus they are in one to one correspondence with size $\ell$ subsets of $\{1,\dots,n-1\}$.
Thus $\mathscr R^{-1}(\vec\beta)$ has $\binom{n-1}{\ell}\leq n^{2\mathbf N_0}$ elements.
Together these two statements give~\eqref{e:ehr-comb-1}.

%%%%%%%%%%%%%%%%%%%%%%%%%%%%%%%%%%%%%%%%%%%%%%%%%%%%%%%%%%%%%%%%%%%%%%%%%%%%%%%%
\begin{figure}
\includegraphics{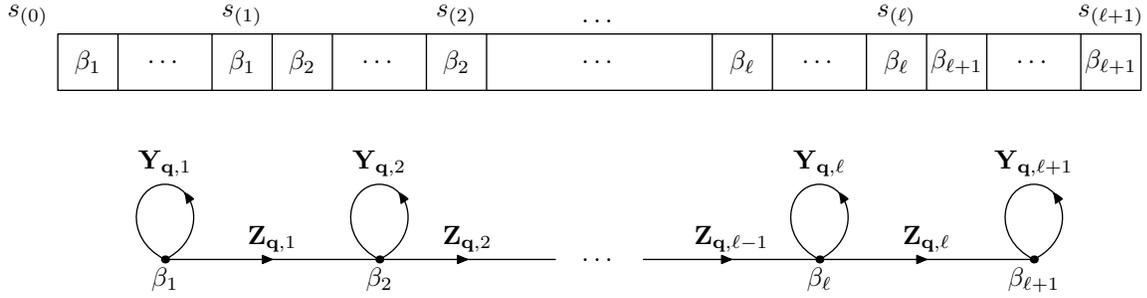}
\caption{Top: the decomposition~\eqref{e:alpha-form} of a path in $\mathscr G$,
with the indices $s_{(j)}$ marked. Bottom: a representation of this decomposition
as a combination of loops and a path in the reduced graph $\widetilde{\mathscr G}$,
with the homomorphisms in the right-hand side of~\eqref{e:M-dec}.}
\label{f:reducedpath}
\end{figure}
%%%%%%%%%%%%%%%%%%%%%%%%%%%%%%%%%%%%%%%%%%%%%%%%%%%%%%%%%%%%%%%%%%%%%%%%%%%%%%%%

\noindent 2. Take $\rho\in \mathcal V^-_{\mathbf q}\cap \{{1\over 4}\leq |\xi|_g\leq 4\}$
and $\vec\alpha\in\mathscr P$. Writing $\vec\alpha$ in the form~\eqref{e:alpha-form},
we have
\begin{equation}
  \label{e:M-dec}
\mathbf M_{\mathbf q,\vec\alpha}(\rho)=\mathbf Y_{\mathbf q,1}(\rho) \mathbf Z_{\mathbf q,1}(\rho)\cdots
\mathbf Y_{\mathbf q,\ell}(\rho)\mathbf Z_{\mathbf q,\ell}(\rho) \mathbf Y_{\mathbf q,\ell+1}(\rho)
\end{equation}
where
$$
\begin{aligned}
\mathbf Y_{\mathbf q,j}(\rho) &:= \mathbf M_{q_{s_{(j-1)}},\beta_j,\beta_j}(\varphi_{s_{(j-1)}}(\rho))\cdots\mathbf M_{q_{s_{(j)}-2},\beta_j,\beta_j}(\varphi_{s_{(j)}-2}(\rho)),\\
\mathbf Z_{\mathbf q,j}(\rho) &:= \mathbf M_{q_{s_{(j)}-1},\beta_j,\beta_{j+1}}(\varphi_{s_{(j)}-1}(\rho)).
\end{aligned}
$$
That is, the factors $\mathbf Y_{\mathbf q,j}$ correspond to loops in the path $\vec\alpha$
and the factors $\mathbf Z_{\mathbf q,j}$, to `true jumps' between the loops.
See Figure~\ref{f:reducedpath}.

Using the formula~\eqref{e:diagonal-term} for the `diagonal terms' $\mathbf M_{q,\alpha,\alpha}$
we compute
\begin{equation}
  \label{e:Y-computed}
\mathbf Y_{\mathbf q,j}(\rho)=\bigg(\prod_{r=s_{(j-1)}}^{s_{(j)}-2} a_{q_r}(\varphi_r(\rho))\bigg)
\cdot (d\varphi_{s_{(j)}-1-s_{(j-1)}}(\varphi_{s_{(j-1)}}(\rho))^T)^{\otimes m_j}
\end{equation}
where $\beta_j=(k_j,m_j)$.
Define the words
$$
\mathbf q_j:=q_{s_{(j-1)}}\dots q_{s_{(j)}-1},\quad
j=1,\dots,\ell+1,
$$
and note that $\mathbf q$ can be written as the concatenation
\begin{equation}
  \label{e:q-concat}
\mathbf q=\mathbf q_1 \mathbf q_2\dots
\mathbf q_{\ell+1}.
\end{equation}
Since $\sup |a_q|\leq 1$
and $\varphi_{s_{(j-1)}}(\rho)\in \mathcal V^-_{\mathbf q_j}\cap \{{1\over 4}\leq |\xi|_g\leq 4\}$, we obtain from~\eqref{e:global-expansion-rate-1}
$$
\|\mathbf Y_{\mathbf q,j}(\rho)\|\leq C\|d\varphi_{s_{(j)}-s_{(j-1)}}(\varphi_{s_{(j-1)}}(\rho))\|^{m_j}
\leq C(\mathcal J^-_{\mathbf q_j})^{m_j}
\leq C(\mathcal J^-_{\mathbf q_j})^{\mathbf N_0}.
$$
We have $\|\mathbf Z_{\mathbf q,j}(\rho)\|\leq C$ and the product~\eqref{e:M-dec}
has $2\ell+1\leq 4\mathbf N_0+1$ elements. Therefore
by~\eqref{e:jacobians-concat} and~\eqref{e:q-concat}
\begin{equation}
  \label{e:mastoedont}
\|\mathbf M_{\mathbf q,\vec\alpha}(\rho)\|\leq C(\mathcal J^-_{\mathbf q_1}\cdots \mathcal J^-_{\mathbf q_{\ell+1}})^{\mathbf N_0}
\leq C(\mathcal J^-_{\mathbf q})^{\mathbf N_0}
\end{equation}
giving~\eqref{e:ehr-comb-2}.
\end{proof}
%%%%%%%%%%%%%%%%%%%%%%%%%%%%%%%%%%%%%%%%%%%%%%%%%%%%%%%%%%%%%%%%%%%%%%%%%%%%%%%%

%%%%%%%%%%%%%%%%%%%%%%%%%%%%%%%%%%%%%%%%%%%%%%%%%%%%%%%%%%%%%%%%%%%%%%%%%%%%%%%%
\subsection{Summing over many words}
\label{s:ehr-sum}

We finally give the proof of Proposition~\ref{l:ehrenfest-prop-sum}.
By~\eqref{e:jacobian-bounds} for $h$ small enough we have the following
bound on the length of words with Jacobians less than $C_0h^{-\delta}\leq h^{-1/2}$:
$$
\mathcal J^-_{\mathbf p}\leq C_0h^{-\delta},\
\mathcal J^+_{\mathbf r}\leq C_0h^{-\delta}
\quad\Longrightarrow\quad
|\mathbf p|,|\mathbf r|\leq C_1\log(1/h),\quad
C_1:={1\over 2\Lambda_0}.
$$
We now split the operator $A_F$ from~\eqref{e:eps} into pieces by
the length of the words involved:
$$
A_F=\sum_{n_-,n_+\leq C_1\log(1/h)} A_{F_{n_-,n_+}},\quad
F_{n_-,n_+}(\mathbf p,\mathbf r):=
\begin{cases} F(\mathbf p,\mathbf r)&\text{if }\mathbf p\in\mathscr A^{n_-},
\mathbf r\in \mathscr A^{n_+};\\
0,&\text{otherwise}.
\end{cases}
$$
Using the triangle inequality we see that Proposition~\ref{l:ehrenfest-prop-sum}
follows from
%%%%%%%%%%%%%%%%%%%%%%%%%%%%%%%%%%%%%%%%%%%%%%%%%%%%%%%%%%%%%%%%%%%%%%%%%%%%%%%%
\begin{prop}
  \label{e:ehrenfest-prop-sum-mod}
Let $n_\pm\leq C_1\log(1/h)$, fix $\delta\in [0,{1\over 2})$ and $C_0>0$, and define
$$
\mathscr A^\pm_{\delta}:=\{\mathbf q\in \mathscr A^{n_\pm}\mid \mathcal J^\pm_{\mathbf q}\leq C_0h^{-\delta}\}.
$$
Assume that
$$
F:\mathscr A^-_\delta\times\mathscr A^+_\delta\to \mathbb C,\quad
\sup |F|\leq 1.
$$
Then there exists a constant $C$ depending only on $\delta,C_0,A_1,\dots,A_Q$ such that
$$
\|A_F\|_{L^2\to L^2}\leq C\quad\text{where}\quad
A_F:=\sum_{(\mathbf p,\mathbf r)\in\mathscr A^-_\delta\times\mathscr A^+_\delta}F(\mathbf p,\mathbf r)A^-_{\mathbf p}A^+_{\mathbf r}.
$$
\end{prop}
%%%%%%%%%%%%%%%%%%%%%%%%%%%%%%%%%%%%%%%%%%%%%%%%%%%%%%%%%%%%%%%%%%%%%%%%%%%%%%%%
\begin{proof}
The proof proceeds by writing $A_F$ as a pseudodifferential operator and estimating its full symbol.
The complications arising from the fact that $A_F$ is the sum over polynomially many in~$h$ terms
are handled similarly to the proof of Lemma~\ref{l:cq-log}.

1. Let $\mathbf p\in\mathscr A^-_\delta$, $\mathbf r\in\mathscr A^+_\delta$
and fix $\mathbf N\in\mathbb N$ to be chosen at the end of the proof in~\eqref{e:the-chosen-one}.
Following the analysis in~\S\S\ref{s:ehr-step}--\ref{s:iterated-symbol-estimates}
(and its immediate analog for the operators $A^+$) we write similarly to~\eqref{e:iterated-operator}
and~\eqref{e:iterated-remainder-estimated}
\begin{equation}
  \label{e:pyromancer-1}
\begin{aligned}
A^-_{\mathbf p}&=\Op_h\bigg(\sum_{k=0}^{\mathbf N-1}h^{k}a^{(k)}_{\mathbf p,-}\bigg)+\mathcal O(h^{\mathbf N-(2\mathbf N+17)\delta-})_{L^2\to L^2},\\
A^+_{\mathbf r}&=\Op_h\bigg(\sum_{k=0}^{\mathbf N-1}h^{k}a^{(k)}_{\mathbf r,+}\bigg)+\mathcal O(h^{\mathbf N-(2\mathbf N+17)\delta-})_{L^2\to L^2}
\end{aligned}
\end{equation}
where (note we put $a^{(k)}_{\mathbf p,-}:=a^{(k)}_{\mathbf p,n_-}$ in the notation of~\S\ref{s:ehr-step}):
\begin{itemize}
\item $a^{(k)}_{\mathbf p,-},a^{(k)}_{\mathbf r,+}\in \CIc(T^*M)$ satisfy the support conditions
\begin{equation}
  \label{e:pyromancer-support}
\supp a^{(k)}_{\mathbf p,-}\subset \mathcal V^-_{\mathbf p}\cap \{\textstyle{1\over 4}<|\xi|_g<4\},\quad
\supp a^{(k)}_{\mathbf r,+}\subset \mathcal V^+_{\mathbf r}\cap \{\textstyle{1\over 4}<|\xi|_g<4\}
\end{equation}
and the derivative bounds similar to~\eqref{e:symbols-bounds-used}
\begin{equation}
  \label{e:pyromancer-2}
\|a^{(k)}_{\mathbf p,-}\|_{C^m},\|a^{(k)}_{\mathbf r,+}\|_{C^m}
=\mathcal O(h^{-(2k+m)\delta-});
\end{equation}
\item if we fix $\mathbf N_\pm\leq 2\mathbf N$ and denote similarly to~\eqref{e:A-r-der-def}
\begin{equation}
  \label{e:pyromancer-2.5}
\mathbf A^-_{\mathbf p}:=(\nabla^m a^{(k)}_{\mathbf p,-})_{(k,m)\in\mathscr V_-},\quad
\mathbf A^+_{\mathbf r}:=(\nabla^m a^{(k)}_{\mathbf r,+})_{(k,m)\in\mathscr V_+},
\end{equation}
where $\mathscr V_\pm:=\{(k,m)\mid k,m\in\mathbb N_0,\ 2k+m\leq \mathbf N_\pm\}$,
then for each $\rho\in T^*M\setminus 0$ we have similarly to~\eqref{e:itesym-4}
\begin{equation}
  \label{e:pyromancer-3}
\|\mathbf A^-_{\mathbf p}(\rho)\|\leq C\sum_{\vec\alpha\in \mathscr P_-}\|\mathbf M^-_{\mathbf p,\vec\alpha}(\rho)\|,\quad
\|\mathbf A^+_{\mathbf r}(\rho)\|\leq C\sum_{\vec\alpha\in \mathscr P_+}\|\mathbf M^+_{\mathbf r,\vec\alpha}(\rho)\|
\end{equation}
where $\mathscr P_\pm$ are the sets of paths of length $n_\pm-1$ in the corresponding graphs
(see~\eqref{e:path-def});
\item the homomorphisms $\mathbf M^-_{\mathbf p,\vec\alpha}(\rho)$,
$\mathbf M^+_{\mathbf r,\vec\alpha}(\rho)$
are defined similarly to~\eqref{e:MM-def}: if $\vec\alpha^\pm=\alpha_1^\pm\dots \alpha^\pm_{n_\pm}\in \mathscr P_\pm$ then
$$
\begin{aligned}
\mathbf M^-_{\mathbf p,\vec\alpha^-}(\rho)&=\mathbf M^-_{p_0,\alpha_1^-,\alpha_2^-}(\rho)
\mathbf M^-_{p_1,\alpha_2^-,\alpha_3^-}(\varphi_1(\rho))
\cdots
\mathbf M^-_{p_{n_--2},\alpha_{n_--1}^-,\alpha_{n_-}^-}(\varphi_{n_--2}(\rho)),\\
\mathbf M^+_{\mathbf r,\vec\alpha^+}(\rho)&=
\mathbf M^+_{r_1,\alpha_1^+,\alpha_2^+}(\rho)
\mathbf M^+_{r_2,\alpha_2^+,\alpha_3^+}(\varphi_{-1}(\rho))
\cdots
\mathbf M^+_{r_{n_+-1},\alpha_{n_+-1}^+,\alpha_{n_+}^+}(\varphi_{-(n_+-2)}(\rho));
\end{aligned}
$$
\item finally, the homomorphisms
$$
\mathbf M^\pm_{q,\alpha,\alpha'}
\in C^\infty(T^*M\setminus 0;\Hom(\varphi_{\mp 1}^*\mathscr E_{\alpha'};\mathscr E_\alpha)),\quad
q\in \mathscr A,\
\alpha,\alpha'\in\mathscr V_\pm,\
\alpha\to \alpha'
$$
are defined similarly to~\eqref{e:MPAA}, in particular we have
similarly to~\eqref{e:diagonal-term}
$$
\begin{aligned}
\mathbf M^-_{q,\alpha,\alpha}(\rho)&=a_q(\rho)\cdot (d\varphi_1(\rho)^T)^{\otimes m},\\
\mathbf M^+_{q,\alpha,\alpha}(\rho)&=a_q(\varphi_{-1}(\rho))\cdot (d\varphi_{-1}(\rho)^T)^{\otimes m}
\end{aligned}
$$
where $\alpha=(k,m)$.
\end{itemize}

\noindent 2. Using~\eqref{e:pyromancer-1}--\eqref{e:pyromancer-2} together with
the precise version of the product formula, Lemma~\ref{l:prod-mfld}, we obtain
$$
A^-_{\mathbf p}A^+_{\mathbf r}=\Op_h\bigg(\sum_{k_\pm,i\geq 0\atop
k_-+k_++i<\mathbf N} h^{k_-+k_++i} L_i(a^{(k_-)}_{\mathbf p,-}\otimes a^{(k_+)}_{\mathbf r,+})|_{\Diag}\bigg)
+\mathcal O(h^{\mathbf N-(2\mathbf N+17)\delta-})_{L^2\to L^2}
$$
where each $L_i$ is a differential operator of order~$2i$ on $T^*M\times T^*M$.
Recalling that $\mathscr A=\{1,\dots,Q\}$, we have
$$
\#(\mathscr A^\pm_\delta)\leq h^{-C_2}\quad\text{where}\quad
C_2:=C_1\log Q.
$$
Summing over $(\mathbf p,\mathbf r)$, we get
\begin{equation}
  \label{e:pyromancer-4}
A_F=\Op_h\bigg(\sum_{k_\pm,i\geq 0\atop
k_-+k_++i<\mathbf N}h^{k_-+k_++i} a_{k_-,k_+,i}\bigg)
+\mathcal O(h^{\mathbf N-(2\mathbf N+17)\delta-2C_2-})_{L^2\to L^2}
\end{equation}
where
$$
a_{k_-,k_+,i}:=\sum_{(\mathbf p,\mathbf r)\in\mathscr A^-_\delta\times\mathscr A^+_\delta}
F(\mathbf p,\mathbf r)L_i(a^{(k_-)}_{\mathbf p,-}\otimes a^{(k_+)}_{\mathbf r,+})|_{\Diag}.
$$

\noindent 3. 
We now estimate the derivatives of the symbols $a_{k_-,k_+,i}$. We first compute the principal term $a_{0,0,0}$,
using that $a_{\mathbf p,-}^{(0)}=a_{\mathbf p}^-$, $a_{\mathbf r,+}^{(0)}=a_{\mathbf r}^+$ similarly
to the line following~\eqref{e:supp-symbol}:
$$
a_{0,0,0}=\sum_{\mathbf p,\mathbf r}F(\mathbf p,\mathbf r) a^-_{\mathbf p}a^+_{\mathbf r}
$$
which, recalling that $\sup|F|\leq 1$, $a_1,\dots,a_Q\geq 0$, and $a_1+\cdots+a_Q\leq 1$, implies\begin{equation}
  \label{e:wildfire-1}
\sup|a_{0,0,0}|\leq 1.
\end{equation}
To estimate the higher derivatives of $a_{0,0,0}$, as well as the other symbols
$a_{k_-,k_+,i}$, we argue similarly to Lemma~\ref{l:ehr-comb}, handling the sum over words
similarly to the proof of Lemma~\ref{l:cq-log}.
By the triangle inequality and since $\sup|F|\leq 1$ we have for any $m$
\begin{equation}
  \label{e:wildfire-2}
\|a_{k_-,k_+,i}\|_{C^m}\leq C
\sup_{\rho\in \{{1\over 4}\leq |\xi|_g\leq 4\}}
\max_{m_\pm\geq 0\atop m_-+m_+\leq m+2i}
\sum_{\mathbf p,\mathbf r}
\big(\|\nabla^{m_-}a^{(k_-)}_{\mathbf p,-}(\rho)\|\cdot \|\nabla^{m_+}a^{(k_+)}_{\mathbf r,+}(\rho)\|\big).
\end{equation}
Fix $m_\pm\geq 0$ such that $m_-+m_+\leq m+2i$ and put
$$
\mathbf N_\pm := 2k_\pm+m_\pm,\quad
\mathbf N_-+\mathbf N_+\leq 2(k_-+k_++i)+m.
$$
By~\eqref{e:pyromancer-3}
we then have
for each $\rho\in \{{1\over 4}\leq |\xi|_g\leq 4\}$
\begin{equation}
  \label{e:wildfire-3}
\begin{aligned}
\|\nabla^{m_-}a^{(k_-)}_{\mathbf p,-}(\rho)\|\cdot \|\nabla^{m_+}a^{(k_+)}_{\mathbf r,+}(\rho)\|
&\leq C\|\mathbf A^-_{\mathbf p}(\rho)\|\cdot \|\mathbf A^+_{\mathbf r}(\rho)\|
\\&\leq
C\sum_{\vec\alpha^\pm\in\mathscr P_\pm}
\big(\|\mathbf M^-_{\mathbf p,\vec\alpha^-}(\rho)\|\cdot \|\mathbf M^+_{\mathbf r,\vec\alpha^+}(\rho)\|\big).
\end{aligned}
\end{equation}
Fix two paths $\vec\alpha^\pm\in\mathscr P_\pm$ and write them in the form~\eqref{e:alpha-form}:
$$
\vec\alpha^\pm=\beta_{1,\pm}^{s_{(1)}^\pm-s_{(0)}^\pm}
\beta_{2,\pm}^{s_{(2)}^\pm-s_{(1)}^\pm}\dots \beta_{\ell_\pm+1,\pm}^{s_{(\ell_\pm+1)}^\pm-s_{(\ell_\pm)}^\pm}
$$
for some sequences $0=s_{(0)}^\pm<s_{(1)}^\pm<\dots<s_{(\ell_\pm)}^\pm<s_{(\ell_\pm+1)}^\pm=n_\pm$.
Define
$$
S^-_{\vec\alpha^-}:=\{s_{(1)}^--1,\dots,s^-_{(\ell_-+1)}-1\},\quad
S^+_{\vec\alpha^+}:=\{s_{(1)}^+,\dots,s_{(\ell_++1)}^+\}.
$$
Arguing similarly to~\eqref{e:mastoedont}, but keeping track of the symbols
$a_{q_r}$ in~\eqref{e:Y-computed} (rather than simply using the inequalities $|a_q|\leq 1$) and
recalling the support properties~\eqref{e:pyromancer-support}
we get for all $\rho\in \supp a^{(k_-)}_{\mathbf p,-}\cap \supp a^{(k_+)}_{\mathbf r,+}
\subset \mathcal V^-_{\mathbf p}\cap \mathcal V^+_{\mathbf r}
\cap \{{1\over 4}<|\xi|_g<4\}$
$$
\begin{aligned}
\|\mathbf M^-_{\mathbf p,\vec\alpha^-}(\rho)\|\leq C (\mathcal J^-_{\mathbf p})^{\mathbf N_-}
\tilde a^-_{\mathbf p,\vec\alpha^-}(\rho),\quad
\|\mathbf M^+_{\mathbf r,\vec\alpha^+}(\rho)\|\leq C (\mathcal J^+_{\mathbf r})^{\mathbf N_+}
\tilde a^+_{\mathbf r,\vec\alpha^+}(\rho)
\end{aligned}
$$
where we define the nonnegative functions $\tilde a^-_{\mathbf p,\vec\alpha^-}$,
$\tilde a^+_{\mathbf r,\vec\alpha^+}$ by removing certain factors
in the definitions~\eqref{e:a-pm-def} of $a^-_{\mathbf p}$, $a^+_{\mathbf r}$
(denoting $\mathbf p=p_0\dots p_{n_--1}$, $\mathbf r=r_1\dots r_{n_+}$):
$$
\tilde a^-_{\mathbf p,\vec\alpha^-}:=\prod_{0\leq j<n_-,\ j\notin S^-_{\vec\alpha^-}}
(a_{p_j}\circ\varphi_j),\quad
\tilde a^+_{\mathbf r,\vec\alpha^+}:=\prod_{1\leq j\leq n_+,\ j\notin S^+_{\vec\alpha^+}}
(a_{r_j}\circ\varphi_{-j}).
$$
Since $\mathcal J^-_\mathbf p,\mathcal J^+_{\mathbf r}\leq C_0h^{-\delta}$, we have
for all $\rho\in \supp a^{(k_-)}_{\mathbf p,-}\cap \supp a^{(k_+)}_{\mathbf r,+}$
$$
\|\mathbf M^-_{\mathbf p,\vec\alpha^-}(\rho)\|
\cdot
\|\mathbf M^+_{\mathbf r,\vec\alpha^+}(\rho)\|
\leq Ch^{-(2(k_-+k_++i)+m)\delta}\tilde a^-_{\mathbf p,\vec\alpha^-}(\rho)
\tilde a^+_{\mathbf r,\vec\alpha^+}(\rho).
$$
Combining this with~\eqref{e:wildfire-2}--\eqref{e:wildfire-3} we obtain
\begin{equation}
  \label{e:wildfire-4}
\|a_{k_-,k_+,i}\|_{C^m}\leq Ch^{-(2(k_-+k_++i)+m)\delta}\sup_{\rho\in \{{1\over 4}\leq |\xi|_g\leq 4\}}
\sum_{\vec\alpha^\pm\in\mathscr P_\pm}\sum_{\mathbf p,\mathbf r}
\big(\tilde a^-_{\mathbf p,\vec\alpha^-}(\rho)
\tilde a^+_{\mathbf r,\vec\alpha^+}(\rho)\big).
\end{equation}
Now, we have for all $\vec\alpha_\pm$
and $\rho$
\begin{equation}
  \label{e:wildfire-5}
\sum_{(\mathbf p,\mathbf r)\in\mathscr A^{n_-}\times\mathscr A^{n_+}}
\big(\tilde a^-_{\mathbf p,\vec\alpha^-}(\rho)
\tilde a^+_{\mathbf r,\vec\alpha^+}(\rho)\big)
\leq Q^{4(k_-+k_++i)+2m+2}\leq C.
\end{equation}
Indeed, we write the left-hand side as the product of sums over the individual
digits $p_{j_-}$, $r_{j_+}$. Since $a_1+\dots+a_Q\leq 1$, each such sum is bounded by $Q$
when $j_\pm\in S^\pm_{\vec\alpha^\pm}$ and by~1 otherwise. It remains to
recall from Step~1 of the proof of Lemma~\ref{l:ehr-comb} that
$\ell_\pm\leq 2\mathbf N_\pm$ and thus
$\#(S^-_{\vec\alpha^-})+\#(S^+_{\vec\alpha^+})\leq 2\mathbf N_-+2\mathbf N_++2
\leq 4(k_-+k_++i)+2m+2$.

Substituting~\eqref{e:wildfire-5} into~\eqref{e:wildfire-4} and using
the bound~\eqref{e:ehr-comb-1} on $\#(\mathscr P_\pm)$, we finally get the bound
\begin{equation}
  \label{e:wildfire-6}
\|a_{k_-,k_+,i}\|_{C^m}=\mathcal O(h^{-(2(k_-+k_++i)+m)\delta-}).
\end{equation}

\noindent 4. 
The bounds~\eqref{e:wildfire-1} and~\eqref{e:wildfire-6} give
$$
a_{0,0,0}=\mathcal O(1)_{S^{\comp}_{\delta+}},\quad
a_{k_-,k_+,i}=\mathcal O(h^{-2(k_-+k_++i)\delta-})_{S^{\comp}_\delta}.
$$
From the $L^2$ boundedness of pseudodifferential operators with symbols in $S^{\comp}_\delta$
we see that the first term on the right-hand side of~\eqref{e:pyromancer-4} is bounded
by a constant in $L^2\to L^2$ norm. The remainder in~\eqref{e:pyromancer-4} is also
bounded by a constant if we choose $\mathbf N$ large enough so that
\begin{equation}
  \label{e:the-chosen-one}
\mathbf N(1-2\delta)>17\delta+2C_2.
\end{equation}
Thus $\|A_F\|_{L^2\to L^2}\leq C$, finishing the proof.
\end{proof}
%%%%%%%%%%%%%%%%%%%%%%%%%%%%%%%%%%%%%%%%%%%%%%%%%%%%%%%%%%%%%%%%%%%%%%%%%%%%%%%%

\appendix

%%%%%%%%%%%%%%%%%%%%%%%%%%%%%%%%%%%%%%%%%%%%%%%%%%%%%%%%%%%%%%%%%%%%%%%%%%%%%%%%
%%%%%%%%%%%%%%%%%%%%%%%%%%%%%%%%%%%%%%%%%%%%%%%%%%%%%%%%%%%%%%%%%%%%%%%%%%%%%%%%
\section{Semiclassical calculus on a surface}
  \label{s:semi-detail}

In this appendix we provide versions of several standard statements from semiclassical analysis
(product and commutator rules,
Egorov's Theorem) with explicit expressions for the resulting symbols and for the $L^2\to L^2$
norms of the remainders. These are used in the proofs
of Egorov's Theorems up to minimal Ehrenfest time (Lemma~\ref{l:egorov-mild})
and local Ehrenfest time (\S\ref{s:ops-Aq}).

We restrict to the case of dimension~$n=2$. The statements below apply in the general case
but the number of derivatives needed to get an $\mathcal O(h^{\mathbf N})$ remainder%
\footnote{As in~\S\ref{s:ops-Aq}, we use boldface $\mathbf N$ here to avoid confusion
with~\eqref{e:prop-times}.}
will take the form $2\mathbf N+C_n$ where $C_n$ is a constant depending only on the dimension.
The precise values of the constants $C_n$ (which we compute for $n=2$) are not important.
We do not attempt to prove optimal bounds. This is already evident in the case
of Lemma~\ref{l:l2-bdd-Rn} below which does not recover boundedness of
pseudodifferential operators in $\Psi^{\comp}_\delta(\mathbb R^2)$.

To shorten the formulas below, we introduce the following notation:
$$
\mathbf D^k_\bullet a
$$
denotes the result of applying some differential operator of order~$k$ to $a$.
The specific operator varies from place to place, with coefficients depending on
the objects listed in `$\bullet$' but not on~$h$ or~$a$.
Next, for an operator $A$ on~$L^2$ we write
$$
A=\mathcal O_{\bullet}(h^{\mathbf N})
$$
to mean $\|A\|_{L^2\to L^2}\leq Ch^{\mathbf N}$ where the constant $C$ depends on the objects
listed in `$\bullet$'.

%%%%%%%%%%%%%%%%%%%%%%%%%%%%%%%%%%%%%%%%%%%%%%%%%%%%%%%%%%%%%%%%%%%%%%%%%%%%%%%%
\subsection{Operators on $\mathbb R^2$}
\label{s:ehr-R2}

We first discuss pseudodifferential calculus on~$\mathbb R^2$.
We use the standard quantization given by
\begin{equation}
  \label{e:Op-h-Rn}
\Op^0_h(a)f(x)=(2\pi h)^{-2}\int_{\mathbb R^4}e^{{i\over h}\langle x-y,\xi\rangle}a(x,\xi)f(y)\,dyd\xi,\quad
a\in\mathscr S(T^*\mathbb R^2).
\end{equation}
We start with a quantitative version of the basic $L^2$ boundedness statement
which follows from the proof of~\cite[Theorem~4.21]{e-z}:
%%%%%%%%%%%%%%%%%%%%%%%%%%%%%%%%%%%%%%%%%%%%%%%%%%%%%%%%%%%%%%%%%%%%%%%%%%%%%%%%
\begin{lemm}
  \label{l:l2-bdd-Rn}
We have for some global constant $C$ and all $a\in\mathscr S(T^*\mathbb R^2)$
$$
\|\Op^0_h(a)\|_{L^2(\mathbb R^2)\to L^2(\mathbb R^2)}\leq C\max_{|\alpha|,|\beta|\leq 3}\sup|\xi^\alpha\partial^\beta_\xi a|.
$$
\end{lemm}
%%%%%%%%%%%%%%%%%%%%%%%%%%%%%%%%%%%%%%%%%%%%%%%%%%%%%%%%%%%%%%%%%%%%%%%%%%%%%%%%
The next statement is a quantitative version of the product formula. To prove it we
write $\Op^0_h(a)\Op^0_h(b)=\Op^0_h(a\#b)$, where $a\#b$ is determined by oscillatory
testing~\cite[Theorem~4.19]{e-z} and estimated via quadratic stationary phase~\cite[Theorem~3.13]{e-z}, and apply Lemma~\ref{l:l2-bdd-Rn}.
%%%%%%%%%%%%%%%%%%%%%%%%%%%%%%%%%%%%%%%%%%%%%%%%%%%%%%%%%%%%%%%%%%%%%%%%%%%%%%%%
\begin{lemm}
  \label{l:prod-Rn}
Let $\mathbf N\in\mathbb N_0$, $R>0$. Then for all $a,b\in\CIc(T^*\mathbb R^2)$, $\supp a\cup\supp b\subset B(0,R)$, we have
\begin{equation}
  \label{e:prod-Rn}
\Op^0_h(a)\Op^0_h(b)=\Op^0_h\bigg(\sum_{|\alpha|<\mathbf N}{(-ih)^{|\alpha|}\over\alpha!} \partial^\alpha_\xi a\,\partial^\alpha_x b\bigg) +\mathcal O_{\mathbf N,R}(
\|a\|_{C^{\mathbf N+6}}\|b\|_{C^{\mathbf N+6}}h^{\mathbf N}).
\end{equation}
\end{lemm}
%%%%%%%%%%%%%%%%%%%%%%%%%%%%%%%%%%%%%%%%%%%%%%%%%%%%%%%%%%%%%%%%%%%%%%%%%%%%%%%%
\Remark
It is also useful to discuss composition of pseudodifferential operators with
multiplication operators. Assume that $a\in \CIc(T^*\mathbb R^2)$,
$b\in \CIc(\mathbb R^2)$, and $\supp a\subset B_{T^*\mathbb R^2}(0,R)$,
$\supp b\subset B_{\mathbb R^2}(0,R)$. Denote by~$\Op^0_h(b)$ the multiplication
operator by $b$. From~\eqref{e:Op-h-Rn} we see that
$\Op^0_h(b)\Op^0_h(a)=\Op^0_h(ab)$. Moreover, Lemma~\ref{l:prod-Rn} still applies
with the same proof.

We finally give a quantitative version of the change of variables formula.
We follow~\cite[\S E.1.6]{dizzy}. The statement below is proved by
following the proof of~\cite[Proposition~E.10]{dizzy} using
the method of stationary phase with explicit remainder~\cite[Theorem~3.16]{e-z}
and applying Lemma~\ref{l:l2-bdd-Rn}. We use the notation
\begin{equation}
  \label{e:varphi-*}
\varphi^{-*}:=(\varphi^{-1})^*,\quad
\varphi^{-*}f=f\circ\varphi^{-1}.
\end{equation}
%%%%%%%%%%%%%%%%%%%%%%%%%%%%%%%%%%%%%%%%%%%%%%%%%%%%%%%%%%%%%%%%%%%%%%%%%%%%%%%%
\begin{lemm}
  \label{l:chvar-Rn}
Assume that $\varphi:U\to V$ is a diffeomorphism where $U,V\subset\mathbb R^2$
are open sets and $\chi_1,\chi_2\in\CIc(U)$. Put
\begin{equation}
  \label{e:tilde-phi}
\widetilde \varphi:T^*U\to T^*V,\quad
\widetilde\varphi(x,\xi)=(\varphi(x),(d\varphi(x))^{-T}\xi).
\end{equation}
Let $\mathbf N\in\mathbb N$, $R>0$. Then
for all $a\in\CIc(T^*\mathbb R^2)$, $\supp a\subset B(0,R)$, we have
$$
\begin{aligned}
\chi_1\varphi^*\Op^0_h(a)\varphi^{-*}\chi_2=\,&\Op^0_h\bigg(\chi_1\Big(\chi_2+\sum_{j=1}^{\mathbf N-1} h^j \mathbf D^{2j}_{\varphi,\chi_2}\Big)
(a\circ\widetilde\varphi)\bigg)\\
&+\mathcal O_{\mathbf N,R,\varphi,\chi_1,\chi_2}(\|a\|_{C^{2\mathbf N+12}}h^{\mathbf N}).
\end{aligned}
$$
Here the operators $\mathbf D^{2j}_{\varphi,\chi_2}$ are supported in $\supp\chi_2$.
\end{lemm}
%%%%%%%%%%%%%%%%%%%%%%%%%%%%%%%%%%%%%%%%%%%%%%%%%%%%%%%%%%%%%%%%%%%%%%%%%%%%%%%%

%%%%%%%%%%%%%%%%%%%%%%%%%%%%%%%%%%%%%%%%%%%%%%%%%%%%%%%%%%%%%%%%%%%%%%%%%%%%%%%%
\subsection{Operators on a compact surface}
\label{s:ehr-mfld}

We now study operators on a compact Riemannian surface $(M,g)$.
We define a (non-canonical) quantization procedure similarly to~\cite[Proposition~E.15]{dizzy}:
\begin{equation}
  \label{e:Op-h-M}
\Op_h(a)=\sum_\ell \chi'_\ell\varphi_\ell^*\Op_h^0\big((\chi'_\ell a)\circ \widetilde\varphi_\ell^{-1}\big)\varphi_\ell^{-*}\chi_\ell
\end{equation}
where we use the notation~\eqref{e:varphi-*},
$\Op_h^0(\bullet)$ on the right-hand side is defined by~\eqref{e:Op-h-Rn},
$\varphi_\ell:U_\ell\to V_\ell$, $U_\ell\subset M$, $V_\ell\subset\mathbb R^2$,
is a finite collection of coordinate charts with $M=\bigcup_\ell U_\ell$, the cutoff functions
$\chi_\ell,\chi'_\ell\in\CIc(U_\ell)$ satisfy
\begin{equation}
  \label{e:mfld-cutoffs}
1=\sum_\ell\chi_\ell,\quad
\supp\chi_\ell\cap\supp (1-\chi'_\ell)=\emptyset,
\end{equation}
and $\widetilde\varphi_\ell:T^*U_\ell\to T^*V_\ell$ is defined by~\eqref{e:tilde-phi}.
To simplify the formulas below we denote 
$$
\Xi:=\{(M,g)\}\cup\{ (\varphi_\ell,\chi_\ell,\chi'_\ell)\}_\ell.
$$
For each $j\in\mathbb N_0$ we fix some norm $\|\bullet\|_{C^j}$ on functions on $T^*M$ supported
in $\{|\xi|_g\leq 10\}$.

We first give an $L^2$ boundedness and pseudolocality statement:
%%%%%%%%%%%%%%%%%%%%%%%%%%%%%%%%%%%%%%%%%%%%%%%%%%%%%%%%%%%%%%%%%%%%%%%%%%%%%%%%
\begin{lemm}
  \label{l:l2-bdd-mfld}
Assume that $a\in \CIc(T^*M)$ and $\supp a\subset \{|\xi|_g\leq 10\}$.
Then
\begin{equation}
  \label{e:l2-bdd-mfld-1}
\Op_h(a)=\mathcal O_\Xi(\|a\|_{C^3}).
\end{equation}
Moreover, if $\chi_1,\chi_2\in C^\infty(M)$ and $\supp\chi_1\cap\supp\chi_2=\emptyset$,
then for every $\mathbf N\in\mathbb N_0$
\begin{equation}
  \label{e:l2-bdd-mfld-2}
\chi_1\Op_h(a)\chi_2=\mathcal O_{\mathbf N,\Xi,\chi_1,\chi_2}(\|a\|_{C^{\mathbf N+6}}h^{\mathbf N}).
\end{equation}
\end{lemm}
%%%%%%%%%%%%%%%%%%%%%%%%%%%%%%%%%%%%%%%%%%%%%%%%%%%%%%%%%%%%%%%%%%%%%%%%%%%%%%%%
\begin{proof}
The bound~\eqref{e:l2-bdd-mfld-1} follows immediately from~\eqref{e:Op-h-M} and
Lemma~\ref{l:l2-bdd-Rn}.
The bound~\eqref{e:l2-bdd-mfld-2} for the quantization $\Op_h^0$
on $\mathbb R^2$ and $\chi_1,\chi_2\in\CIc(\mathbb R^2)$ follows from the remark following Lemma~\ref{l:prod-Rn};
for the quantization $\Op_h$ it then follows from~\eqref{e:Op-h-M}.
\end{proof}
%%%%%%%%%%%%%%%%%%%%%%%%%%%%%%%%%%%%%%%%%%%%%%%%%%%%%%%%%%%%%%%%%%%%%%%%%%%%%%%%
We next give an auxiliary statement used in the proof of Lemma~\ref{l:prod-mfld} below.
We introduce the following notation: for $a\in\CIc(T^*M)$
\begin{equation}
  \label{e:op-l-def}
\Op_h^\ell(a):=\Op_h^0\big((\chi'_\ell a)\circ\widetilde\varphi_\ell^{-1}\big):L^2(\mathbb R^2)\to L^2(\mathbb R^2).
\end{equation}
%%%%%%%%%%%%%%%%%%%%%%%%%%%%%%%%%%%%%%%%%%%%%%%%%%%%%%%%%%%%%%%%%%%%%%%%%%%%%%%%
\begin{lemm}
  \label{l:psidorep}
Assume that
\begin{equation}
  \label{e:psidorep-1}
A=\sum_r \chi'_r\varphi_r^* \Op_h^r(a_r)\varphi_r^{-*}\chi_r:L^2(M)\to L^2(M)
\end{equation}
for some $a_r\in \CIc(T^*M)$ such that $\supp a_r\subset\{|\xi|_g\leq 10\}$.
Put
\begin{equation}
  \label{e:psidorep-2}
A_\ell:=\varphi_\ell^{-*}\chi'_\ell A\chi'_\ell \varphi_\ell^*:L^2(\mathbb R^2)\to L^2(\mathbb R^2).
\end{equation}
Then for every $\mathbf N\in\mathbb N$ we have
\begin{align}
  \label{e:psidorep-s-1}
A_\ell&=\Op_h^\ell\bigg(\sum_r\Big(\chi'_\ell\chi_r+\sum_{j=1}^{\mathbf N-1}h^j  \mathbf D^{2j}_{\ell,r,\Xi}\Big)\chi'_r a_r\bigg)+
\mathcal O_{\mathbf N,\Xi}\big(\max_r \|a_r\|_{C^{2\mathbf N+12}}h^{\mathbf N}\big),
\\
  \label{e:psidorep-s-2}
A&=\sum_\ell \chi'_\ell \varphi_\ell^* A_\ell \varphi_\ell^{-*}\chi_\ell+\mathcal O_{\mathbf N,\Xi}\big(\max_r \|a_r\|_{C^{\mathbf N+6}}h^{\mathbf N}\big),
\\
  \label{e:psidorep-s-3}
A&=\Op_h\bigg(\sum_r\Big(\chi_r+\sum_{j=1}^{\mathbf N-1}h^j \mathbf D^{2j}_{r,\Xi}\Big)a_r\bigg)+\mathcal O_{\mathbf N,\Xi}\big(\max_r\|a_r\|_{C^{2\mathbf N+12}}h^{\mathbf N}\big).
\end{align}
Here the operators $\mathbf D^{2j}_{\ell,r,\Xi}$ from~\eqref{e:psidorep-s-1}
and $\mathbf D^{2j}_{r,\Xi}$ from~\eqref{e:psidorep-s-3} are supported in $\supp\chi_r$.
\end{lemm}
%%%%%%%%%%%%%%%%%%%%%%%%%%%%%%%%%%%%%%%%%%%%%%%%%%%%%%%%%%%%%%%%%%%%%%%%%%%%%%%%
\Remark The expression~\eqref{e:psidorep-1} is the general form of
a pseudodifferential operator on~$M$, with $\Op_h(a)$ obtained by putting $a_r:=a$ for all~$r$.
The operator $A_\ell$ is the localization of $A$ to the $\ell$-th coordinate chart.
The statement~\eqref{e:psidorep-s-1} shows that each localization is a pseudodifferential
operator on~$\mathbb R^2$; \eqref{e:psidorep-s-2} reconstructs~$A$ from its localizations;
and~\eqref{e:psidorep-s-3} writes a general pseudodifferential operator in the form $\Op_h(a)$
for some~$a$.
%%%%%%%%%%%%%%%%%%%%%%%%%%%%%%%%%%%%%%%%%%%%%%%%%%%%%%%%%%%%%%%%%%%%%%%%%%%%%%%%
\begin{proof}
The expansion~\eqref{e:psidorep-s-1} follows immediately from Lemma~\ref{l:chvar-Rn},
with $\varphi:=\varphi_r\circ\varphi_\ell^{-1}$,
$\chi_1:=(\chi_\ell'\chi_r')\circ\varphi_\ell^{-1}$,
$\chi_2:=(\chi_\ell'\chi_r)\circ\varphi_\ell^{-1}$,
and $a:=(\chi'_r a_r)\circ\widetilde\varphi_r^{-1}$.

To show~\eqref{e:psidorep-s-2} we write by~\eqref{e:mfld-cutoffs}
$$
A-\sum_\ell \chi'_\ell \varphi_\ell^* A_\ell \varphi_\ell^{-*}\chi_\ell=\sum_\ell
(1-(\chi'_\ell)^2)A\chi_\ell
$$
and estimate the right-hand side similarly to~\eqref{e:l2-bdd-mfld-2}.

To show~\eqref{e:psidorep-s-3}, we introduce a bit more notation.
For a vector of symbols $\mathbf a=\{a_r\}_r$ indexed by the coordinate charts used in~\eqref{e:Op-h-M}, let $\Op'_h(\mathbf a)$ be the operator defined in~\eqref{e:psidorep-1}.
Next, put
$$
\iota(a)=\{a\}_r,\quad
\pi(\mathbf a)=\sum_r \chi_r a_r.
$$
Recalling~\eqref{e:Op-h-M}, we have for any $a\in\CIc(T^*M)$
$$
\Op_h(a)=\Op'_h(\iota(a)).
$$
Therefore, for each vector $\mathbf a=\{a_r\}_r$ with
$a_r\in\CIc(T^*M)$, $\supp a_r\subset \{|\xi|_g\leq 10\}$,
we have $\Op'_h(\mathbf a)-\Op_h(\pi(\mathbf a))=\Op'_h(\mathbf b)$
where $\mathbf b:=\mathbf a-\iota(\pi(\mathbf a))$.
We apply~\eqref{e:psidorep-s-1} and~\eqref{e:psidorep-s-2} to this operator to write
it in the form $\Op'_h(\mathbf c)$ for some vector of symbols~$\mathbf c$ (modulo
a remainder);
note that by~\eqref{e:psidorep-s-1} the leading term of~$\mathbf c$
is zero since $\pi(\mathbf b)=0$. This implies
\begin{equation}
  \label{e:psidorep-i}
\Op'_h(\mathbf a)=\Op_h(\pi(\mathbf a))+\Op'_h\bigg(\sum_{j=1}^{\mathbf N-1}h^j \mathbf D^{2j}_{\Xi}\mathbf a\bigg)+\mathcal O_{\mathbf N,\Xi}(\|\mathbf a\|_{C^{2\mathbf N+12}}h^{\mathbf N})
\end{equation}
where the differential operators $\mathbf D^{2j}_\Xi$ act on vectors of symbols.
We iteratively apply~\eqref{e:psidorep-i} to the second term on the right-hand side
and obtain~\eqref{e:psidorep-s-3}.
\end{proof}
%%%%%%%%%%%%%%%%%%%%%%%%%%%%%%%%%%%%%%%%%%%%%%%%%%%%%%%%%%%%%%%%%%%%%%%%%%%%%%%%
We can now give the product and commutator formulas for the quantization on~$M$:
%%%%%%%%%%%%%%%%%%%%%%%%%%%%%%%%%%%%%%%%%%%%%%%%%%%%%%%%%%%%%%%%%%%%%%%%%%%%%%%%
\begin{lemm}
  \label{l:prod-mfld}
Assume that $a,b\in\CIc(T^*M)$ and $\supp a\cup\supp b\subset \{|\xi|_g\leq 10\}$.
Then for every $\mathbf N\in\mathbb N$ we have
\begin{gather}
  \label{e:prod-mfld}
\begin{aligned}
\Op_h(a)\Op_h(b)=\,&\Op_h\Big(ab+\sum_{j=1}^{\mathbf N-1}h^j\mathbf D_{\Xi}^{2j-2}(d^1 a\otimes d^1 b)|_{\Diag}\Big)
\\&+\mathcal O_{\mathbf N,\Xi}(\|a\otimes b\|_{C^{2\mathbf N+15}}h^{\mathbf N}),
\end{aligned}
\\
  \label{e:comm-mfld}
\begin{aligned}\relax
[\Op_h(a),\Op_h(b)]=\,&\Op_h\Big(-ih\{a,b\}
+
\sum_{j=2}^{\mathbf N-1}h^j\mathbf D_{\Xi}^{2j-4}(d^2 a\otimes d^2 b)|_{\Diag}\Big)
\\&+\mathcal O_{\mathbf N,\Xi}(\|a\otimes b\|_{C^{2\mathbf N+15}}h^{\mathbf N}),
\end{aligned}
\end{gather}
where $a\otimes b\in \CIc(T^*M\times T^*M)$ is defined by
$(a\otimes b)(\rho,\rho')=a(\rho)b(\rho')$, $\Diag\subset T^*M\times T^*M$ denotes the diagonal,
and $d^kb$ denotes the vector $(\partial^\alpha b)_{|\alpha|\leq k}$.
\end{lemm}
%%%%%%%%%%%%%%%%%%%%%%%%%%%%%%%%%%%%%%%%%%%%%%%%%%%%%%%%%%%%%%%%%%%%%%%%%%%%%%%%
\Remarks 1. The expression $\mathbf D^{2j-2}(d^1a\otimes d^1b)|_{\Diag}$ in~\eqref{e:prod-mfld}
is a linear combination of products $\partial^\alpha a\,\partial^\beta b$ where
$|\alpha|+|\beta|\leq 2j$ and $|\alpha|,|\beta|\leq 2j-1$. That is,
the symbol in product formula does not feature terms of the form
$h^j (\mathbf D^{2j} a)b$ or $h^j a(\mathbf D^{2j}b)$. This is not obvious,
in fact the proof needs us to use the same quantization procedures $\Op_h$
on both sides of~\eqref{e:prod-mfld}.

Here is an informal explanation:
in a fixed coordinate chart we have $\Op_h(a)=\Op_h^0(\tilde a)$,
$\Op_h(b)=\Op_h^0(\tilde b)$,
$\Op_h(ab)=\Op_h^0(\tilde c)$,
where
$\tilde a=a+\sum_{j\geq 1}h^jL_ja$,
$\tilde b=b+\sum_{j\geq 1}h^jL_jb$, and
$\tilde c=ab+\sum_{j\geq 1}h^jL_j(ab)$; here each $L_j$ is a differential
operator of order~$2j$ (depending on the chart chosen). Denote by $\tilde a\#\tilde b$ the Moyal product from~\eqref{e:prod-Rn}.
If we denote by `\dots' terms of the form $h^j \mathbf D^{2j-2}(d^1 a\otimes d^1b)|_{\Diag}$, then
$\tilde a\#\tilde b=\tilde a\tilde b+\dots=ab+\sum_{j\geq 1}h^j((L_ja)b+a(L_jb))+\dots$
and Leibniz's Rule shows that $\tilde c=ab+\sum_{j\geq 1}h^j((L_ja)b+a(L_jb))+\dots$ as well.

Similarly in the commutator formula~\eqref{e:comm-mfld}
the expression $\mathbf D^{2j-4}(d^2a\otimes d^2b)|_{\Diag}$ consists
of products $\partial^\alpha a\,\partial^\alpha b$ where
$|\alpha|+|\beta|\leq 2j$ and $|\alpha|,|\beta|\leq 2j-2$.

\noindent 2. We immediately deduce from~\eqref{e:comm-mfld}
the formula~\eqref{e:commfor} used
in the proof of Egorov's Theorem up to global Ehrenfest time:
it suffices to take $b\in S^{\comp}_0(T^*M)$ such that $P=\Op_h(b)+\mathcal O(h^\infty)$ and
choose $\mathbf N$ large enough so that~$(1-2\delta)\mathbf N>2+13\delta$.
Note that $h^j\mathbf D^{2j-4}(d^2a\otimes d^2b)|_{\Diag}\in h^{1+(j-1)(1-2\delta)}S^{\comp}_\delta(T^*M)$
when $a\in S^{\comp}_\delta(T^*M)$.
The expansion~\eqref{e:comm-mfld} is crucial in the proof of the precise version
of Egorov's Theorem in Lemma~\ref{l:egorov-precise} below.
%%%%%%%%%%%%%%%%%%%%%%%%%%%%%%%%%%%%%%%%%%%%%%%%%%%%%%%%%%%%%%%%%%%%%%%%%%%%%%%%
\begin{proof}
1. Fix cutoff functions
$$
\chi''_\ell\in\CIc(U_\ell),\quad
\supp\chi_\ell\cap \supp(1-\chi''_\ell)=\supp\chi''_\ell\cap \supp(1-\chi'_\ell)=\emptyset.
$$
We write
\begin{gather}
\label{e:rope1}
\begin{aligned}
\Op_h(a)\Op_h(b)=\,&\sum_\ell (\chi'_\ell)^2\Op_h(a)\chi''_\ell\Op_h(b)\chi_\ell
\\&+\sum_\ell (1-(\chi'_\ell)^2)\Op_h(a)\chi''_\ell\Op_h(b)\chi_\ell\\
&+\sum_\ell \Op_h(a)(1-\chi''_\ell)\Op_h(b)\chi_\ell,
\end{aligned}\\
\label{e:rope2}
\Op_h(ab)=\sum_\ell (\chi'_\ell)^2\Op_h(ab)\chi_\ell+
\sum_\ell (1-(\chi'_\ell)^2)\Op_h(ab)\chi_\ell.
\end{gather}
The last two terms on the right-hand side of~\eqref{e:rope1}
and the last term on the right-hand side of~\eqref{e:rope2}
are estimated using~\eqref{e:l2-bdd-mfld-1}
and~\eqref{e:l2-bdd-mfld-2}.
Rewriting the first terms on the right-hand sides of~\eqref{e:rope1}--\eqref{e:rope2}, we get
\begin{equation}
  \label{e:prodigy1}
\begin{aligned}
\Op_h(a)\Op_h(b)-\Op_h(ab)=&\,
\sum_\ell \chi'_\ell\varphi_\ell^*(A_\ell B_\ell-C_\ell)\varphi_\ell^{-*}\chi_\ell
\\&+\mathcal O_{\mathbf N,\Xi}(\|a\otimes b\|_{C^{\mathbf N+9}}h^{\mathbf N}),
\end{aligned}
\end{equation}
where (note we use the notation $A_\ell$ in a slightly different way than Lemma~\ref{l:psidorep})
$$
A_\ell:=\varphi_\ell^{-*}\chi'_\ell \Op_h(a)\chi''_\ell\varphi_\ell^*,\quad
B_\ell:=\varphi_\ell^{-*}\chi'_\ell \Op_h(b)\chi''_\ell\varphi_\ell^*,\quad
C_\ell:=\varphi_\ell^{-*}\chi'_\ell \Op_h(ab)\chi''_\ell\varphi_\ell^*.
$$

\noindent 2.
Similarly to~\eqref{e:psidorep-s-1} we write for every $\mathbf N$ using the notation~\eqref{e:op-l-def}
\begin{equation}
  \label{e:rope0}
A_\ell=\Op_h^\ell\bigg(\sum_{j=0}^{\mathbf N-1}h^jL_{j,\ell}a\bigg)
+\mathcal O_{\mathbf N,\Xi}(\|a\|_{C^{2\mathbf N+12}}h^{\mathbf N})
\end{equation}
where each $L_{j,\ell}$ is a differential operator of order~$2j$ supported
in $\supp\chi''_\ell$ and $L_{0,\ell}=\chi''_\ell$. Same is true for $B_\ell,C_\ell$,
with the same operators $L_{j,\ell}$.

Using~\eqref{e:rope0} and the bound~\eqref{e:l2-bdd-mfld-1} we get
\begin{equation}
  \label{e:prodigy3}
A_\ell B_\ell=\sum_{j,k\geq 0\atop j+k<\mathbf N}h^{j+k}\Op_h^\ell(L_{j,\ell}a)
\Op_h^\ell(L_{k,\ell}b)+
\mathcal O_{\mathbf N,\Xi}(\|a\otimes b\|_{C^{2\mathbf N+15}}h^{\mathbf N}).
\end{equation}
We next use the product formula for the standard quantization (Lemma~\ref{l:prod-Rn})
and the fact that $L_{j,\ell}a,L_{k,\ell}b$ are supported in
$\supp\chi''_\ell$ which does not intersect $\supp(1-\chi'_\ell)$, to write
\begin{equation}
  \label{e:rope2.5}
\begin{aligned}
\Op_h^\ell(L_{j,\ell}a)
\Op_h^\ell(L_{k,\ell}b)=&\,
\Op_h^\ell\bigg((L_{j,\ell}a)(L_{k,\ell}b)+\sum_{s=1}^{\mathbf N-j-k-1}h^s\mathbf D^{s,s}_{\ell,\Xi}(L_{j,\ell}a\otimes L_{k,\ell}b)|_{\Diag}\bigg)
\\&+
\mathcal O_{\mathbf N,\Xi}(\|a\otimes b\|_{C^{2\mathbf N+12}}h^{\mathbf N-j-k}).
\end{aligned}
\end{equation}
Here $\mathbf D^{s,s}$ denotes a differential operator of order~$2s$ on $T^*M\times T^*M$
which has no more than~$s$ differentiations in either component of the product.
This implies
\begin{equation}
  \label{e:rope3}
\begin{aligned}
A_\ell B_\ell-C_\ell=&\,\Op_h^\ell\bigg(\chi''_\ell (\chi''_\ell-1)ab
+\sum_{j=1}^{\mathbf N-1}h^j\big((\chi''_\ell a)(L_{j,\ell}b)+(L_{j,\ell}a)(\chi''_\ell b)-L_{j,\ell}(ab)\big)
\\
&+\sum_{j=1}^{\mathbf N-1} h^j \mathbf D^{2j-2}_{\ell,\Xi}(d^1a\otimes d^1b)|_{\Diag}\bigg)
+\mathcal O_{\mathbf N,\Xi}(\|a\otimes b\|_{C^{2\mathbf N+15}}h^{\mathbf N})
\end{aligned}
\end{equation}
where the second line includes all the terms in~\eqref{e:rope2.5} such that
$s\geq 1$ or $j\cdot k>0$.
Using Leibniz's Rule for the operators $L_{j,\ell}$, $j\geq 1$,
$$
L_{j,\ell}(ab)=a(L_{j,\ell}b)+(L_{j,\ell}a)b+\mathbf D^{2j-2}_{j,\ell,\Xi}(d^1a\otimes d^1b)|_{\Diag}
$$
we see that
the restriction of the first line on the right-hand side of~\eqref{e:rope3}
to $T^*M\setminus\supp(1-\chi''_\ell)\supset \supp\chi_\ell$
has the form
$\sum_{j=1}^{\mathbf N-1}h^j \mathbf D^{2j-2}_{\ell,\Xi}(d^1a\otimes d^1 b)|_{\Diag}$.
From here and~\eqref{e:psidorep-s-3} (using that the operators $\mathbf D^{2j}_{\Xi,r}$ there
are supported in $\supp\chi_r$) we get the product formula~\eqref{e:prod-mfld}.

\noindent 3.
To obtain the commutator formula~\eqref{e:comm-mfld} we write similarly to~\eqref{e:prodigy1}
$$
\begin{gathered}\relax
\begin{aligned}\relax
[\Op_h(a),\Op_h(b)]+ih\Op_h(\{a,b\})=&\,\sum_\ell \chi'_\ell\varphi_\ell^*([A_\ell,B_\ell]-E_\ell)
\varphi_\ell^{-*}\chi_\ell
\\&+\mathcal O_{\mathbf N,\Xi}(\|a\otimes b\|_{C^{\mathbf N+9}}h^{\mathbf N}),
\end{aligned}
\\
E_\ell:=\varphi_\ell^{-*}\chi'_\ell\Op_h(-ih\{a,b\})\chi''_\ell\varphi_\ell^*.
\end{gathered}
$$
Similarly to~\eqref{e:prodigy3} we get
$$
[A_\ell,B_\ell]=\sum_{j,k\geq 0\atop j+k<\mathbf N}h^{j+k}[\Op_h^\ell(L_{j,\ell}a),
\Op_h^\ell(L_{k,\ell}b)]+\mathcal O_{\mathbf N,\Xi}(\|a\otimes b\|_{C^{2\mathbf N+15}}h^{\mathbf N}).
$$
By Lemma~\ref{l:prod-Rn} we have the following analog of~\eqref{e:rope2.5}:
$$
\begin{aligned}\relax
[\Op_h^\ell(L_{j,\ell}a),
\Op_h^\ell(L_{k,\ell}b)]=&\,
\Op_h^\ell\bigg(-ih\{L_{j,\ell}a,L_{k,\ell}b\}+\sum_{s=2}^{\mathbf N-j-k-1}h^s\mathbf D^{s,s}_{\ell,\Xi}(L_{j,\ell}a\otimes L_{k,\ell}b)|_{\Diag}\bigg)
\\&+
\mathcal O_{\mathbf N,\Xi}(\|a\otimes b\|_{C^{2\mathbf N+12}}h^{\mathbf N-j-k}).
\end{aligned}
$$
This gives the following analog of~\eqref{e:rope3}:
$$
\begin{aligned}\relax
[A_\ell,B_\ell]-E_\ell=&\,\Op_h^\ell\bigg(ih\big(\chi''_\ell\{a,b\}-\{\chi''_\ell a,\chi''_\ell b\}\big)
\\&
+\sum_{j=1}^{\mathbf N-2}ih^{j+1}\big(
L_{j,\ell}\{a,b\}-\{\chi''_\ell a,L_{j,\ell}b\}-\{L_{j,\ell}a,\chi''_\ell b\}
\big)
\\&+
\sum_{j=2}^{\mathbf N-1}h^j\mathbf D^{2j-4}_{\ell,\Xi}(d^2a\otimes d^2b)|_{\Diag}
\bigg)
+
\mathcal O_{\mathbf N,\Xi}(\|a\otimes b\|_{C^{2\mathbf N+15}}h^{\mathbf N})
\end{aligned}
$$
where the third line includes all terms such that $s\geq 2$ or $j\cdot k>0$.
To get~\eqref{e:comm-mfld} it remains to
argue as at the end of Step~2 
using the following Leibniz's rule for the Poisson bracket:
$$
L_{j,\ell}\{a,b\}=\{a,L_{j,\ell}b\}+\{L_{j,\ell}a,b\}+\mathbf D^{2j-2}_{\ell,\Xi}
(d^2a\otimes d^2b)|_{\Diag}.\qedhere
$$
\end{proof}
%%%%%%%%%%%%%%%%%%%%%%%%%%%%%%%%%%%%%%%%%%%%%%%%%%%%%%%%%%%%%%%%%%%%%%%%%%%%%%%%

%%%%%%%%%%%%%%%%%%%%%%%%%%%%%%%%%%%%%%%%%%%%%%%%%%%%%%%%%%%%%%%%%%%%%%%%%%%%%%%%
\subsection{Egorov's Theorem}
\label{s:ehr-egorov}

We finally give a quantitative version of Egorov's Theorem~\eqref{e:egorov-basic}.
The proof below applies to more general situations but we restrict ourselves to
the case of the propagator $U(t)=\exp(-itP/h)$, where $P$ is defined in~\eqref{e:U-t-def},
and the flow $\varphi_t$ defined in~\eqref{e:phi-def}.
%%%%%%%%%%%%%%%%%%%%%%%%%%%%%%%%%%%%%%%%%%%%%%%%%%%%%%%%%%%%%%%%%%%%%%%%%%%%%%%%
\begin{lemm}
  \label{l:egorov-precise}
Assume that $a\in\CIc(T^*M)$ and $\supp a\subset \{{1\over 4}\leq|\xi|_g\leq 4\}$.
Then we have for all $\mathbf N\in\mathbb N$ and $0\leq t\leq 1$
\begin{equation}
  \label{e:egorov-precise}
U(-t)\Op_h(a)U(t)=\Op_h\bigg(\Big(a+\sum_{j=1}^{\mathbf N-1}h^j\mathbf D^{2j}_{t,\Xi} a\Big)\circ\varphi_t\bigg)
+\mathcal O_{\mathbf N,\Xi}\big(\|a\|_{C^{2\mathbf N+17}}h^{\mathbf N}\big).
\end{equation}
\end{lemm}
%%%%%%%%%%%%%%%%%%%%%%%%%%%%%%%%%%%%%%%%%%%%%%%%%%%%%%%%%%%%%%%%%%%%%%%%%%%%%%%%
\begin{proof}
1. 
We first recall from~\eqref{e:funcal} and~\eqref{e:U-t-def} that
$$
P=\Op_h(p_0+hp')+\mathcal O(h^\infty)_{L^2\to L^2},\quad
p_0=p\quad\text{on}\quad \{\textstyle{1\over 4}\leq |\xi|_g\leq 4\}
$$
where $p_0,p'$ are classical symbols on $T^*M$ supported inside
$\{{1\over 5}<|\xi|_g<5\}$.
Here $p(x,\xi)=|\xi|_g$ and $\varphi_t=\exp(tH_p)$.

By the commutator formula~\eqref{e:comm-mfld},
for any $\tilde a\in \CIc(T^*M)$, $\supp \tilde a\subset \{{1\over 4}\leq |\xi|_g\leq 4\}$,
$$
{i\over h}[P,\Op_h(\tilde a)]=\Op_h\bigg(H_p \tilde a+\sum_{j=1}^{\mathbf N-1}h^j\mathbf D^{2j}_{\Xi} \tilde a\bigg)
+\mathcal O_{\mathbf N,\Xi}(\|\tilde a\|_{C^{2\mathbf N+17}}h^{\mathbf N}).
$$
Here we use that $p'$ is classical, i.e. has an expansion in powers of~$h$, and incorporate
the terms in that expansion into the operators $\mathbf D^{2j}_\Xi$.

Therefore, for any
family of symbols $a_t\in \CIc(T^*M)$ depending smoothly
on~$t\in [0,1]$ and such that $\supp a_t\subset \{{1\over 4}\leq|\xi|_g\leq 4\}$,
and for any $\mathbf N\in\mathbb N$
\begin{equation}
  \label{e:commune}
\begin{aligned}
\partial_t\Op_h(a_t\circ\varphi_t)-{i\over h}[P,\Op_h(a_t\circ\varphi_t)]
=&\,\Op_h\bigg(\Big(\partial_t a_t-\sum_{j=1}^{\mathbf N-1} h^j L_{j,t} a_t\Big)\circ\varphi_t\bigg)
\\&+\mathcal O_{\mathbf N,\Xi}(\|a_t\|_{C^{2\mathbf N+17}}h^{\mathbf N})
\end{aligned}
\end{equation}
where each $L_{j,t}$ is a differential operator of order~$2j$ on $T^*M$ with coefficients
depending on~$t,\Xi$.

\noindent 2. We now construct $t$-dependent families of symbols
$a_t^{(j)}\in\CIc(T^*M)$, $t\in [0,1]$,  $j=0,\dots,\mathbf N-1$, using the following iterative procedure:
$$
a_t^{(0)}:=a;\qquad
a_t^{(j)}:=\sum_{k=0}^{j-1}\int_0^t L_{j-k,s}a_s^{(k)}\,ds,\quad j=1,\dots,\mathbf N-1.
$$
Note that $a_t^{(j)}$ has the form $\mathbf D^{2j}_{t,\Xi} a$.
Put
$$
\tilde a^{(\mathbf N)}_t:=\sum_{j=0}^{\mathbf N-1}h^j a^{(j)}_t,
$$
then~\eqref{e:commune} implies
\begin{equation}
  \label{e:commune2}
\partial_t\Op_h(\tilde a_t^{(\mathbf N)}\circ\varphi_t)-{i\over h}[P,\Op_h(\tilde a_t^{(\mathbf N)}\circ\varphi_t)]
=\mathcal O_{\mathbf N,\Xi}(\|a\|_{C^{2\mathbf N+17}}h^{\mathbf N}).
\end{equation}

\noindent 3. From~\eqref{e:commune2} and the unitarity of $U(t)$ we obtain for $t\in [0,1]$
$$
\partial_t \big(U(t)\Op_h(\tilde a_t^{(\mathbf N)}\circ\varphi_t)U(-t)\big)=\mathcal O_{\mathbf N,\Xi}(\|a\|_{C^{2\mathbf N+17}}h^{\mathbf N}).
$$
Integrating this and using that $\tilde a_0^{(\mathbf N)}=a$ we have
$$
U(t)\Op_h(\tilde a_t^{(\mathbf N)}\circ\varphi_t )U(-t)=\Op_h(a)+\mathcal O_{\mathbf N,\Xi}(\|a\|_{C^{2\mathbf N+17}}h^{\mathbf N}).
$$
Conjugating this by $U(t)$ we get~\eqref{e:egorov-precise}.
\end{proof}
%%%%%%%%%%%%%%%%%%%%%%%%%%%%%%%%%%%%%%%%%%%%%%%%%%%%%%%%%%%%%%%%%%%%%%%%%%%%%%%%

%%%%%%%%%%%%%%%%%%%%%%%%%%%%%%%%%%%%%%%%%%%%%%%%%%%%%%%%%%%%%%%%%%%%%%%%%%%%%%%%
%%%%%%%%%%%%%%%%%%%%%%%%%%%%%%%%%%%%%%%%%%%%%%%%%%%%%%%%%%%%%%%%%%%%%%%%%%%%%%%%
\section{Fourier localization of Lagrangian states}
\label{s:fourloc-proof}

In this appendix we prove Proposition~\ref{l:fourloc-lag}. We use the following interpolation
inequality in the classes~$C^k$. It is standard (see for instance~\cite[Lemma~7.7.2]{Hormander1}
for a special case) but we provide a proof for the reader's convenience.
%%%%%%%%%%%%%%%%%%%%%%%%%%%%%%%%%%%%%%%%%%%%%%%%%%%%%%%%%%%%%%%%%%%%%%%%%%%%%%%%
\begin{lemm}
  \label{l:interpolation-Ck}
Assume that $U\subset\mathbb R^n$ is an open set, $K\subset U$,
$d(K,\mathbb R^n\setminus U)>r_0>0$,
and $f\in C^\infty(U)$. Denote
$$
\|f\|_m:=\max_{|\alpha|\leq m}\sup_U |\partial^\alpha f|,\quad
m\in\mathbb N_0.
$$
Let $0<\ell<m$. Then there exists a constant $C$ depending only on~$m,r_0$
such that
\begin{equation}
  \label{e:interpolation-Ck}
\max_{|\alpha|\leq \ell}\sup_K |\partial^\alpha f|\leq C \|f\|_0^{1-\ell/ m}
\|f\|_m^{\ell/ m}.
\end{equation}
\end{lemm}
%%%%%%%%%%%%%%%%%%%%%%%%%%%%%%%%%%%%%%%%%%%%%%%%%%%%%%%%%%%%%%%%%%%%%%%%%%%%%%%%
\begin{proof}
Since $\|f\|_0\leq\|f\|_m$ it suffices to show~\eqref{e:interpolation-Ck} for $|\alpha|=\ell$.
Then~\eqref{e:interpolation-Ck} holds once we prove the following inequality
for all $x_0\in K$:
\begin{equation}
  \label{e:interpolation-Ck-1}
\max_{|\alpha|=\ell}|\partial^\alpha f(x_0)|\leq C R_0^{1-\ell/m}R_m^{\ell/m},\quad
R_k:=\max_{|\alpha|\leq k}\sup_{B(x_0,r_0)}|\partial^\alpha f|.
\end{equation}
By Taylor's inequality we have for all $y\in B(0,r_0)$ and some constant $C_m$ depending only on~$m$
$$
\bigg|f(x_0+y)-\sum_{\ell=0}^{m-1}P_\ell(y)\bigg|\leq C_mR_m\, |y|^m,\quad
P_\ell(y):=\sum_{|\alpha|=\ell}{\partial^\alpha f(x_0)\over \alpha!}y^\alpha.
$$
Substituting
$$
y:=\Big({R_0\over R_m}\Big)^{1/m}r\theta,\quad
\theta\in\mathbb S^{n-1},\quad
0\leq r\leq r_0
$$
and using that $|f(x_0+y)|\leq R_0$ we get
$$
\sup_{r\in [0,r_0]}\bigg|\sum_{\ell=0}^{m-1} \Big({R_0\over R_m}\Big)^{\ell/m}P_\ell(\theta)r^\ell \bigg|
\leq (1+C_mr_0^m)R_0.
$$
The expression on the left-hand side is the sup-norm on the interval $[0,r_0]$
of a polynomial of degree~$m-1$ in~$r$.
Using this sup-norm to estimate the coefficients of this polynomial, we obtain
$$
\sup_{\theta\in\mathbb S^{n-1}}|P_\ell(\theta)|\leq C_{m,r_0}R_0^{1-\ell/m}R_m^{\ell/m}\quad\text{for all}\quad \ell=0,\dots,m-1
$$
where the constant $C_{m,r_0}$ depends only on~$m,r_0$.
This implies~\eqref{e:interpolation-Ck-1}.
\end{proof}
%%%%%%%%%%%%%%%%%%%%%%%%%%%%%%%%%%%%%%%%%%%%%%%%%%%%%%%%%%%%%%%%%%%%%%%%%%%%%%%%
We are now ready to give
%%%%%%%%%%%%%%%%%%%%%%%%%%%%%%%%%%%%%%%%%%%%%%%%%%%%%%%%%%%%%%%%%%%%%%%%%%%%%%%%
\begin{proof}[Proof of Proposition~\ref{l:fourloc-lag}]
We show the following stronger estimate:
\begin{equation}
  \label{e:fll-0}
|\hat u(\xi/h)|\leq C'_Nh^{N+n/2}\,\langle\xi\rangle^{-n},\quad
\xi\in\mathbb R^n\setminus \Omega_\Phi(C_0^{-1}h').
\end{equation}
Take arbitrary $\xi\in\mathbb R^n\setminus\Omega_\Phi(C_0^{-1}h')$ and put
$$
s:=d(\xi,\Omega_\Phi)\geq C_0^{-1}h'.
$$
We have
\begin{equation}
  \label{e:fll-1}
\hat u(\xi/h)=\int_U e^{i\Phi_\xi(x)/h}a(x)\,dx,\quad
\Phi_\xi(x):=\Phi(x)-\langle x,\xi\rangle.
\end{equation}
In the rest of the proof we put
$$
N_0:=\bigg\lceil{2N+n\over 1-\tau}\bigg\rceil,\quad
N':=N_0+1
$$
and denote by~$C$ constants which depend
only on $\tau,n,N,C_0,C_{N'}$, whose precise value might change
from place to place.

We integrate by parts in~\eqref{e:fll-1} using the differential operator
$L$ defined by
$$
Lf(x)=\sum_{j=1}^n b_j(x)\partial_jf(x),\quad
b_j(x):=-i{\partial_j\Phi_\xi(x)\over |d\Phi_\xi(x)|^2}.
$$
Integrating by parts $N_0$ times and using that $hLe^{i\Phi_\xi(x)/h}=e^{i\Phi_\xi(x)/h}$ we get
\begin{equation}
  \label{e:fll-2}
\big|\hat u(\xi/h)\big|=\bigg|\int_{U}e^{i\Phi_\xi(x)/h}(hL^t)^{N_0}a(x)\,dx\bigg|
\leq C_0h^{N_0}\sup_K \big|(L^t)^{N_0}a\big|
\end{equation}
where $L^t$ is the transpose operator:
$$
L^tf(x)=-\sum_{j=1}^n \partial_j\big(b_j(x)f(x)\big).
$$
To estimate the function $(L^t)^{N_0}a$ we bound the derivatives of $\Phi_\xi$.
Since $\diam\Omega_\Phi\leq C_0h'\leq C_0^2s$ we have
$$
s\leq |d\Phi_\xi(x)|\leq Cs\quad\text{for all }x\in U.
$$
By Lemma~\ref{l:interpolation-Ck} applied to the first derivatives
of $\Phi_\xi$ we obtain the derivative bounds
for $0\leq\ell\leq N_0$
\begin{equation}
  \label{e:derbs}
\max_{|\alpha|=\ell+1}\sup_K |\partial^\alpha \Phi_\xi|\leq
C s^{1-\ell/N_0}
\leq Csh^{-(1-\tau)\ell/2}
\end{equation}
where in the last inequality we used the definition of $N_0$ and the fact that $s\geq C_0^{-1}h^\tau\geq C_0^{-1}h$.
This implies the derivative bounds for $0\leq \ell\leq N_0$
\begin{equation}
  \label{e:derbs2}
\max_{|\alpha|=\ell}\sup_K |\partial^\alpha b_j|\leq Cs^{-1}h^{-(1-\tau)\ell/2}.
\end{equation}
This gives an estimate on the right-hand side of~\eqref{e:fll-2}, implying
\begin{equation}
  \label{e:fll-3}
|\hat u(\xi/h)|\leq Ch^{(1+\tau)N_0/2}s^{-N_0}.
\end{equation}
We have $s\geq C^{-1}h^\tau$, thus (using again the definition of~$N_0$)
$$
|\hat u(\xi/h)|\leq Ch^{(1-\tau)N_0/2}\leq Ch^{N+n/2}.
$$
This gives~\eqref{e:fll-0} for $|\xi|\leq C$. On the other hand,
if $\xi$ is large enough then $s\geq \langle\xi\rangle/2$
in which case~\eqref{e:fll-0} follows from~\eqref{e:fll-3} as well
since $N_0\geq n$.
\end{proof}
%%%%%%%%%%%%%%%%%%%%%%%%%%%%%%%%%%%%%%%%%%%%%%%%%%%%%%%%%%%%%%%%%%%%%%%%%%%%%%%%

%%%%%%%%%%%%%%%%%%%%%%%%%%%%%%%%%%%%%%%%%%%%%%%%%%%%%%%%%%%%%%%%%%%%%%%%%%%%%%%%
%%%%%%%%%%%%%%%%%%%%%%%%%%%%%%%%%%%%%%%%%%%%%%%%%%%%%%%%%%%%%%%%%%%%%%%%%%%%%%%%
\medskip\noindent\textbf{Acknowledgements.}
We are extremely grateful to both anonymous referees for their very careful reading, and the numerous suggestions
to improve the presentation.
SD was supported by Clay Research Fellowship, Sloan Fellowship,
and NSF CAREER Grant DMS-1749858.
LJ was supported by Recruitment Program of Young Overseas Talent Plan.
SN acknowledges partial support from
the Agence National de la Recherche, through the grants
GERASIC-ANR-13-BS01-0007-02 and ISDEEC-ANR-16-CE40-0013.
This project was initiated during the {\it IAS Emerging Topics Workshop on
 Quantum Chaos and Fractal Uncertainty Principle} in October 2017.

%%%%%%%%%%%%%%%%%%%%%%%%%%%%%%%%%%%%%%%%%%%%%%%%%%%%%%%%%%%%%%%%%%%%%%%%%%%%%%%%

\end{document}